\documentclass[oneside,french,english]{amsart}
\usepackage[T1]{fontenc}
\usepackage[latin9]{inputenc}
\usepackage{geometry}
\usepackage{babel}
\makeatletter
\addto\extrasfrench{%
   \providecommand{\fg}{\ifdim\lastskip>\z@\unskip\fi~\frqq}%
}

\makeatother
\usepackage{units}
\usepackage{mathrsfs}
\usepackage{amstext}
\usepackage{amsthm}
\usepackage{amssymb}
\usepackage{stmaryrd}
\usepackage{graphicx}
\usepackage{esint}
\usepackage{xargs}[2008/03/08]
\usepackage[unicode=true,pdfusetitle,
 bookmarks=true,bookmarksnumbered=false,bookmarksopen=false,
 breaklinks=false,pdfborder={0 0 1},backref=false,colorlinks=false]
 {hyperref}
\usepackage{breakurl}

\makeatletter
\numberwithin{equation}{section}
\numberwithin{figure}{section}
\theoremstyle{plain}
\newtheorem{thm}{\protect\theoremname}[section]
  \theoremstyle{definition}
  \newtheorem{defn}[thm]{\protect\definitionname}
  \theoremstyle{remark}
  \newtheorem{rem}[thm]{\protect\remarkname}
  \theoremstyle{plain}
  \newtheorem{prop}[thm]{\protect\propositionname}
  \theoremstyle{plain}
  \newtheorem{lem}[thm]{\protect\lemmaname}
  \theoremstyle{plain}
  \newtheorem{cor}[thm]{\protect\corollaryname}
  \theoremstyle{plain}
  \newtheorem{fact}[thm]{\protect\factname}
  \theoremstyle{plain}
  \newtheorem*{fact*}{\protect\factname}

\makeatother

  \addto\captionsenglish{\renewcommand{\corollaryname}{Corollary}}
  \addto\captionsenglish{\renewcommand{\definitionname}{Definition}}
  \addto\captionsenglish{\renewcommand{\factname}{Fact}}
  \addto\captionsenglish{\renewcommand{\lemmaname}{Lemma}}
  \addto\captionsenglish{\renewcommand{\propositionname}{Proposition}}
  \addto\captionsenglish{\renewcommand{\remarkname}{Remark}}
  \addto\captionsenglish{\renewcommand{\theoremname}{Theorem}}
  \addto\captionsfrench{\renewcommand{\corollaryname}{Corollaire}}
  \addto\captionsfrench{\renewcommand{\definitionname}{Définition}}
  \addto\captionsfrench{\renewcommand{\factname}{Fait}}
  \addto\captionsfrench{\renewcommand{\lemmaname}{Lemme}}
  \addto\captionsfrench{\renewcommand{\propositionname}{Proposition}}
  \addto\captionsfrench{\renewcommand{\remarkname}{Remarque}}
  \addto\captionsfrench{\renewcommand{\theoremname}{Théorème}}
  \providecommand{\corollaryname}{Corollary}
  \providecommand{\definitionname}{Definition}
  \providecommand{\factname}{Fact}
  \providecommand{\lemmaname}{Lemma}
  \providecommand{\propositionname}{Proposition}
  \providecommand{\remarkname}{Remark}
\providecommand{\theoremname}{Theorem}

\begin{document}

\global\long\def\eps{\varepsilon}

\global\long\def\pain{{\displaystyle \left(P_{j}\right)}}

\global\long\def\tt#1{\mathtt{#1}}

\global\long\def\tx#1{\mathrm{#1}}

\global\long\def\cal#1{\mathcal{#1}}

\global\long\def\scrib#1{\mathscr{#1}}

\global\long\def\frak#1{\mathfrak{#1}}

\global\long\def\math#1{{\displaystyle #1}}


\global\long\def\pare#1{\Big({\displaystyle #1}\Big)}

\global\long\def\acc#1{\left\{  #1\right\}  }

\global\long\def\cro#1{\left[#1\right]}

\newcommandx\dd[1][usedefault, addprefix=\global, 1=]{\tx d#1}

\global\long\def\DD#1{\tx D#1}

\global\long\def\ppp#1#2{\frac{\partial#1}{\partial#2}}

\global\long\def\ddd#1#2{\frac{\tx d#1}{\tx d#2}}

\global\long\def\norm#1{\left\Vert #1\right\Vert }

\global\long\def\abs#1{\left|#1\right|}

\global\long\def\ps#1{\left\langle #1\right\rangle }

\global\long\def\crocro#1{\left\llbracket #1\right\rrbracket }

\global\long\def\ord#1{\tx{ord}\left(#1\right)}


\global\long\def\wc{\mathbb{C}}

\global\long\def\ww#1{\mathbb{#1}}

\global\long\def\wr{\mathbb{R}}

\global\long\def\wn{\mathbb{N}}

\global\long\def\wn{\mathbb{Z}}

\global\long\def\wq{\mathbb{Q}}

\global\long\def\wcc{\left(\ww C^{2},0\right)}

\global\long\def\wccc{\left(\ww C^{3},0\right)}

\global\long\def\wcn{\left(\ww C^{n},0\right)}

\global\long\def\cn{\ww C^{n},0}


\global\long\def\germ#1{\ww C\left\{  #1\right\}  }

\global\long\def\germinv#1{\ww C\left\{  #1\right\}  ^{\times} }

\global\long\def\form#1{\ww C\left\llbracket #1\right\rrbracket }

\global\long\def\forminv#1{\ww C\left\llbracket #1\right\rrbracket ^{\times}}

\global\long\def\maxid{\mathfrak{m}}

\newcommandx\jet[1][usedefault, addprefix=\global, 1=k]{\mbox{\ensuremath{\tx J}}_{#1}}

\global\long\def\hol#1{\mathcal{O}\left(#1\right)}

\newcommandx\formdiff[1][usedefault, addprefix=\global, 1=]{\widehat{\Omega^{#1}}}


\global\long\def\pp#1{\frac{\partial}{\partial#1}}

\newcommandx\der[2][usedefault, addprefix=\global, 1=\mathbf{x}, 2=]{\tx{Der}_{#2}\left(\form{#1}\right)}

\newcommandx\vf[1][usedefault, addprefix=\global, 1={\ww C^{3},0}]{\chi\left(#1\right)}

\newcommandx\fvf[1][usedefault, addprefix=\global, 1={\ww C^{3},0}]{\widehat{\chi}\left(#1\right)}

\newcommandx\sfvf[1][usedefault, addprefix=\global, 1={\ww C^{3},0}]{\widehat{\chi}_{\omega}\left(#1\right)}

\newcommandx\fisot[2][usedefault, addprefix=\global, 1=\tx{fib}]{\widehat{\tx{Isot}}_{#1}\left(#2\right)}

\global\long\def\isotsect#1#2#3{\tx{Isot}_{\tx{fib}}\left(#1,\cal S_{#2,#3};\tx{Id}\right)}

\global\long\def\ynorm{Y_{\tx{norm}}}

\global\long\def\lie#1{\cal L_{\math{#1}}}


\newcommandx\fdiff[1][usedefault, addprefix=\global, 1={\ww C^{3},0}]{\tx{Diff}_{\tx{fib}}\left(#1\right)}

\newcommandx\diff[1][usedefault, addprefix=\global, 1={\ww C^{3},0}]{\tx{Diff}\left(#1\right)}

\newcommandx\diffid[1][usedefault, addprefix=\global, 1={\ww C^{3},0}]{\tx{Diff}\left(#1;\tx{Id}\right)}

\newcommandx\fdiffid[1][usedefault, addprefix=\global, 1={\ww C^{3},0}]{\tx{Diff}_{\tx{fib}}\left(#1;\tx{Id}\right)}

\newcommandx\diffsect[2][usedefault, addprefix=\global, 1=\theta, 2=\eta]{\tx{Diff}_{\tx{fib}}\left(\cal S_{#1,#2};\tx{Id}\right)}

\newcommandx\diffform[1][usedefault, addprefix=\global, 1={\ww C^{3},0}]{\widehat{\tx{Diff}}\left(#1\right)}

\newcommandx\fdiffform[1][usedefault, addprefix=\global, 1={\ww C^{3},0}]{\widehat{\tx{Diff}}_{\tx{fib}}\left(#1\right)}

\newcommandx\fdiffformid[1][usedefault, addprefix=\global, 1={\ww C^{3},0}]{\widehat{\tx{Diff}}_{\tx{fib}}\left(#1;\tx{Id}\right)}

\newcommandx\gid[1][usedefault, addprefix=\global, 1=k]{\cal G_{\tx{Id}}^{\left(#1\right)}}

\newcommandx\diffformid[1][usedefault, addprefix=\global, 1={\ww C^{3},0}]{\widehat{\tx{Diff}}\left(#1;\tx{Id}\right)}

\newcommandx\sdiff[1][usedefault, addprefix=\global, 1={\ww C^{3},0}]{\tx{Diff}_{\omega}\left(#1\right)}

\newcommandx\sdiffid[1][usedefault, addprefix=\global, 1={\ww C^{3},0}]{\tx{Diff}_{\omega}\left(#1;\tx{Id}\right)}

\newcommandx\sdiffform[1][usedefault, addprefix=\global, 1={\ww C^{3},0}]{\widehat{\tx{Diff}}_{\omega}\left(#1\right)}

\newcommandx\sdiffformid[1][usedefault, addprefix=\global, 1={\ww C^{3},0}]{\widehat{\tx{Diff}}_{\omega}\left(#1;\tx{Id}\right)}


\global\long\def\sect#1#2#3{S\left(#1,#2,#3\right)}

\global\long\def\germsect#1#2{\cal S_{#1,#2}}

\global\long\def\asympsect#1#2{\cal{AS}_{#1,#2}}

\global\long\def\fol#1{\mathcal{F}_{#1}}


\global\long\def\rrel{\mathcal{R}}

\global\long\def\surj{\twoheadrightarrow}

\global\long\def\inj{\hookrightarrow}

\global\long\def\bij{\simeq}

\global\long\def\quotient#1#2{\bigslant{#1}{#2}}


\global\long\def\res#1{\tx{res}\left(#1\right)}

\global\long\def\param{\cal P}

\global\long\def\po{\cal P_{0}}

\global\long\def\pfib{\cal P_{\tx{fib}}}

\global\long\def\pw{\cal P_{\omega}}

\global\long\def\porb{\cal P_{\tx{orb}}}

\global\long\def\ls#1{\Gamma_{#1\lambda}}

\global\long\def\lsp#1{\Gamma'_{#1\lambda}}

\global\long\def\lsr#1#2{\Gamma_{#1\lambda}\left(#2\right)}

\global\long\def\sn{\mathbb{\cal{SN}}}

\global\long\def\sndiag{\mathbb{\cal{SN}}_{\tx{diag}}}

\global\long\def\sndiagnd{\mathbb{\cal{SN}}_{\tx{diag,nd}}}

\global\long\def\snodiag{\cal{SN}_{\tx{diag},0}}

\global\long\def\sns{\mathbb{\cal{SN}}_{\omega}}

\global\long\def\snsdiag{\mathbb{\cal{SN}}_{\mbox{\ensuremath{\tx{diag}}},\omega}}

\global\long\def\snsnd{\mathbb{\cal{SN}}_{\tx{snd}}}

\global\long\def\sno{\cal{SN}_{0}}

\global\long\def\snfib{\cal{SN}_{\tx{fib}}}

\global\long\def\snofib{\cal{SN}_{\tx{fib},0}}

\global\long\def\snoid{\cal{SN}_{\tx{Id},0}}

\global\long\def\fsn{\mathbb{\widehat{\cal{SN}}}}

\global\long\def\fsns{\widehat{\mathbb{\cal{SN}}}_{\omega}}

\global\long\def\fsndiag{\mathbb{\widehat{\cal{SN}}}_{\tx{diag}}}

\global\long\def\fsnnd{\mathbb{\widehat{\cal{SN}}}_{\tx{nd}}}

\global\long\def\fsnfibnd{\mathbb{\widehat{\cal{SN}}}_{\tx{fib,nd}}}

\global\long\def\fsndiagnd{\mathbb{\widehat{\cal{SN}}}_{\tx{diag,nd}}}

\global\long\def\fsno{\cal{\widehat{SN}}_{*}}

\global\long\def\fsnfib{\cal{\widehat{SN}}_{\tx{fib}}}

\global\long\def\mfibid{\mathscr{M}_{\tx{fib,Id}}}

\global\long\def\mfib{\mathscr{M}_{\tx{fib}}}

\global\long\def\mspaceid{\mathscr{M}_{\tx{Id}}}

\global\long\def\ms{\mathscr{M}_{\omega}}

\newcommand{\bigslant}[2]{{\raisebox{.3em}{$#1$}\left/\raisebox{-.3em}{$#2$}\right.}}

\title[Analytic classification of doubly-resonant saddle-nodes]{Doubly-resonant saddle-nodes in $\left(\ww C^{3},0\right)$ and
the fixed singularity at infinity in Painlevé equations: analytic
classification}

\author{To appear in \emph{Annales de l'Institut Fourier}\\
\\
Amaury Bittmann}

\address{IRMA, Université de Strasbourg, 7 rue René Descartes, 67084 Strasbourg
Cedex, France}

\email{\href{mailto:amaury.bittmann@ac-strasbourg.fr}{amaury.bittmann@ac-strasbourg.fr}}
\begin{abstract}
In this work, we consider germs of analytic singular vector fields
in $\ww C^{3}$ with an isolated and doubly-resonant singularity of
saddle-node type at the origin. Such vector fields come from irregular
two-dimensional differential systems with two opposite non-zero eigenvalues,
and appear for instance when studying the irregular singularity at
infinity in Painlevé equations $\left(P_{j}\right)_{j=I,\dots,V}$
for generic values of the parameters. Under suitable assumptions,
we prove a theorem of analytic normalization over sectorial domains,
analogous to the classical one due to Hukuhara-Kimura-Matuda for saddle-nodes
in $\ww C^{2}$. We also prove that these maps are in fact the Gevrey-1
sums of the formal normalizing map, the existence of which has been
proved in a previous paper. Finally we provide an analytic classification
under the action of fibered diffeomorphisms, based on the study of
the so-called \emph{Stokes diffeomorphisms} obtained by comparing
consecutive sectorial normalizing maps \emph{à la} Martinet-Ramis~/~Stolovitch
for 1-resonant vector fields.
\end{abstract}

\maketitle

\section{Introduction}

~

As in \cite{bittmann1}, we consider (germs of) singular vector fields
$Y$ in $\ww C^{3}$ which can be written in appropriate coordinates
$\left(x,\mathbf{y}\right):=\left(x,y_{1},y_{2}\right)$ as 
\begin{eqnarray}
Y & = & x^{2}\pp x+\Big(-\lambda y_{1}+F_{1}\left(x,\mathbf{y}\right)\Big)\pp{y_{1}}+\Big(\lambda y_{2}+F_{2}\left(x,\mathbf{y}\right)\Big)\pp{y_{2}}\,\,\,\,\,,\label{eq: intro}
\end{eqnarray}
where $\lambda\in\ww C^{*}$ and $F_{1},\,F_{2}$ are germs of holomorphic
functions in $\left(\ww C^{3},0\right)$ of homogeneous valuation
(order) at least two. They represent irregular two-dimensional differential
systems having two opposite non-zero eigenvalues:
\[
\begin{cases}
{\displaystyle x^{2}\ddd{y_{1}\left(x\right)}x=-\lambda y_{1}\left(x\right)+F_{1}\left(x,\mathbf{y}\left(x\right)\right)}\\
{\displaystyle x^{2}\ddd{y_{2}\left(x\right)}x=\lambda y_{2}\left(x\right)+F_{2}\left(x,\mathbf{y}\left(x\right)\right)} & .
\end{cases}
\]
 These we call doubly-resonant vector fields of saddle-node type (or
simply \textbf{doubly-resonant saddle-nodes}). We will impose more
(non-generic) conditions in the sequel. The motivation for studying
such vector fields is at least of two types.
\begin{enumerate}
\item There are two independent resonance relations between the eigenvalues
(here $0$, $-\lambda$ and $\lambda$): we generalize then the study
in \cite{MR82,MR83}. More generally, this work is aimed at understanding
singularities of vector fields in $\ww C^{3}$. According to a theorem
of resolution of singularities in dimension less than three in \cite{mcquillan2013almost},
there exists a list of ``final models'' for singularities (\emph{log-canonical})\emph{
}obtained after a finite procedure of \emph{weighted blow-ups} for
three dimensional singular analytic vector fields. In this list, we
find in particular doubly-resonant saddles-nodes, as those we are
interested in. In dimension $2$, these final models have been intensively
studied (for instance by Martinet, Ramis, Ecalle, Ilyashenko, Teyssier,
...) from the view point of both formal and analytic classification
(some important questions remain unsolved, though). In dimension $3$,
the problems of formal and analytic classification are still open
questions, although Stolovitch has performed such a classification
for 1-resonant vector fields in any dimension \cite{Stolo}. The presence
of two kinds of resonance relations brings new difficulties.
\item Our second main motivation is the study of the irregular singularity
at infinity in Painlevé equations$\left(P_{j}\right)_{j=I,\dots,V}$,
for generic values of the parameters (\emph{cf.} \cite{Yoshida85}).
These equations were discovered by Paul Painlevé~\cite{Painleve}
because the only movable singularities of the solutions are poles
(the so-called \emph{Painlevé property}). Their study has become a
rich domain of research since the important work of Okamoto~\cite{OkamotoSpace}.
The fixed singularities of the Painlevé equations, and more particularly
those at infinity, where notably investigated by Boutroux with his
famous \emph{tritronquées} solutions \cite{Boutroux13}. Recently,
several authors provided more complete information about such singularities,
studying ``quasi-linear Stokes phenomena'' and also giving connection
formulas; we refer to the following (non-exhaustive) sources \cite{Joshi,Kapaev,KapaevKitaev,KitaevJosi,ClarksonMcLeod,Costin15}.
Stokes coefficients are invariant under local changes of analytic
coordinates, but do not form a complete invariant of the vector field.
To the best of our knowledge there currently does not exist a general
analytic classification for doubly-resonant saddle-nodes. Such a classification
would provide a new framework allowing to analyze Stokes phenomena
in that class of singularities.
\end{enumerate}
In this paper we provide an analytic classification under the action
of fibered diffeomorphisms for a specific (to be defined later on)
class of doubly-resonant saddle-nodes which contains the Painlevé
case. For this purpose, the main tool is a theorem of analytic normalization
over sectorial domain \emph{(à la }Hukuhara-Kimura-Matuda \cite{HKM}
for saddle-nodes in $\left(\ww C^{2},0\right)$) for a specific class
(to be defined later on) of doubly-resonant saddle-nodes which contains
the Painlevé case. The analytic classification for this class of vector
fields, inspired by the important works \cite{MR82,MR83,Stolo} for
1-resonant vector fields, is based on the study of so-called \emph{Stokes
diffeomorphisms, }which are the transition maps between different
sectorial domains\emph{ }for the normalization\emph{.}\bigskip{}

In \cite{Yoshida84,Yoshida85} Yoshida shows that doubly-resonant
saddle-nodes arising from the compactification of Painlevé equations
$\left(P_{j}\right)_{j=I,\dots,V}$ (for generic values for the parameters)
are conjugate to vector fields of the form: 
\begin{eqnarray}
 & Z= & x^{2}\pp x+\pare{-\left(1+\gamma y_{1}y_{2}\right)+a_{1}x}y_{1}\pp{y_{1}}\nonumber \\
 &  & +\pare{1+\gamma y_{1}y_{2}+a_{2}x}y_{2}\pp{y_{2}}\,\,\,\,\,,\label{eq: normal form Yoshida-1}
\end{eqnarray}
with $\gamma\in\ww C^{*}$ and $\left(a_{1},a_{2}\right)\in\ww C^{2}$
such that $a_{1}+a_{2}=1$. One should notice straight away that this
``conjugacy'' does not agree with what is traditionally (in particular
in this paper) meant by conjugacy, for Yoshida's transform $\Psi\left(x,\mathbf{y}\right)=\left(x,\psi_{1}\left(x,\mathbf{y}\right),\psi_{2}\left(x,\mathbf{y}\right)\right)$
takes the form
\begin{eqnarray}
\psi_{i}\left(x,\mathbf{y}\right) & = & y_{i}\left(1+\sum_{\substack{\left(k_{0},k_{1},k_{2}\right)\in\ww N^{3}\\
k_{1}+k_{2}\geq1
}
}\frac{q_{i,\mathbf{k}}\left(x\right)}{x^{k_{0}}}y_{1}^{k_{1}+k_{0}}y_{2}^{k_{1}+k_{0}}\right)\,\,\,\,\,,\label{eq: Yoshida-1}
\end{eqnarray}
where each $q_{i,\mathbf{k}}$ is formal power series although $x$
\emph{appears with negative exponents}. This expansion may not even
be a formal Laurent series. It is, though, the asymptotic expansion
along $\left\{ x=0\right\} $ of a function analytic in a domain 
\begin{eqnarray*}
 & \Big\{\left(x,\mathbf{z}\right)\in\text{\ensuremath{S\times\mathbf{D}\left(0,\mathbf{r}\right)\mid\abs{z_{1}z_{2}}<\nu\abs x}}\Big\}
\end{eqnarray*}
for some small $\nu>0$, where $S$ is a sector of opening greater
than $\pi$ with vertex at the origin and $\mathbf{D}\left(0,\mathbf{r}\right)$
is a polydisc of small poly-radius $\mathbf{r}=\left(r_{1},r_{2}\right)\in\left(\ww R_{>0}\right)^{2}$.
Moreover the $\left(q_{i,\mathbf{k}}\left(x\right)\right)_{i,\mathbf{k}}$
are actually Gevrey-1 power series. The drawback here is that the
transforms are convergent on regions so small that taken together
they cannot cover an entire neighborhood of the origin in $\ww C^{3}$
(which seems to be problematic to obtain an analytic classification
\emph{à la }Martinet-Ramis).

\bigskip{}

Several authors studied the problem of convergence of formal transformations
putting vector fields as in $\left(\mbox{\ref{eq: intro}}\right)$
into ``normal forms''. Shimomura, improving on a result of Iwano
\cite{Iwano}, shows in \cite{Shimo} that analytic doubly-resonant
saddle-nodes satisfying more restrictive conditions are conjugate
(formally and over sectors) to vector fields of the form
\begin{eqnarray*}
 &  & x^{2}\pp x+\left(-\lambda+a_{1}x\right)y_{1}\pp{y_{1}}+\left(\lambda+a_{2}x\right)y_{2}\pp{y_{2}}\,\,\,\,
\end{eqnarray*}
\emph{via} a diffeomorphism whose coefficients have asymptotic expansions
as $x\rightarrow0$ in sectors of opening greater than $\pi$. 

Stolovitch then generalized this result to any dimension in \cite{Stolo}.
More precisely, Stolovitch's work offers an analytic classification
of vector fields in $\ww C^{n+1}$ with an irregular singular point,
without further hypothesis on eventual additional resonance relations
between eigenvalues of the linear part. However, as Iwano and Shimomura
did, he needed to impose other assumptions, among which the condition
that the restriction of the vector field to the invariant hypersurface
$\acc{x=0}$ is a linear vector field. In \cite{DeMaesschalck}, the
authors obtain a \emph{Gevrey-1 summable} ``normal form'', though
not as simple as Stolovitch's one and not unique \emph{a priori},
but for more general kind of vector field with one zero eigenvalue.
However, the same assumption on hypersurface $\acc{x=0}$ is required
(the restriction is a linear vector field). Yet from \cite{Yoshida85}
(and later \cite{bittmann1}) stems the fact that this condition is
not met in the case of Painlevé equations $\left(P_{j}\right)_{j=I,\dots,V}$. 

In comparison, we merely ask here that the restricted vector field
be orbitally linearizable (see Definition \ref{def: asympt hamil}),
\emph{i.e.} the foliation\emph{ }induced by $Y$ on $\acc{x=0}$ (and
not the vector field $Y_{|\acc{x=0}}$ itself) be linearizable. The
fact that this condition is fulfilled by the singularities of Painlevé
equations formerly described is well-known. As discussed in Remark
\ref{rem:New difficulties}, the more general context also introduces
new phenomena and technical difficulties as compared to prior classification
results. 

\subsection{Scope of the paper}

~

The action of local analytic~/~formal diffeomorphisms $\Psi$ fixing
the origin on local holomorphic vector fields $Y$ of type $\left(\mbox{\ref{eq: intro}}\right)$
by change of coordinates is given by
\begin{eqnarray*}
\Psi_{*}Y & := & \mathrm{D}\Psi\left(Y\right)\circ\Psi^{-1}~.
\end{eqnarray*}
In \cite{bittmann1} we performed the formal classification of such
vector fields by exhibiting an explicit universal family of vector
fields for the action of formal changes of coordinates at $0$ (called
a family of normal forms). Such a result seems currently out of reach
in the analytic category: it is unlikely that an explicit universal
family for the action of local analytic changes of coordinates be
described anytime soon. If we want to describe the space of equivalent
classes (of germs of a doubly-resonant saddle-node under local analytic
changes of coordinates) with same formal normal form, we therefore
need to find a complete set of invariants which is of a different
nature. We call \textbf{moduli space} this quotient space and would
like to give it a (non-trivial) presentation based on functional invariants
\emph{à la} Martinet-Ramis \cite{MR82,MR83}. 

We only deal here with $x$-fibered local analytic conjugacies acting
on vector fields of the form $\left(\mbox{\ref{eq: intro}}\right)$
with some additional assumptions detailed further down (see Definitions
\ref{def: drsn}, \ref{def: non-deg} and \ref{def: asympt hamil}).
Importantly, these hypothesis are met in the case of Painlevé equations
mentioned above. The classification under the action of general (not
necessarily fibered) diffeomorphisms can be found in \cite{bittmann:tel-01367968}).

\bigskip{}

First we prove a theorem of analytic sectorial normalizing map (over
a pair of opposite ``wide'' sectors of opening greater than $\pi$
whose union covers a full punctured neighborhood of $\left\{ x=0\right\} $).
Then we attach to each vector field a complete set of invariants given
as transition maps (over ``narrow'' sectors of opening less than
$\pi$) between the sectorial normalizing maps. Although this viewpoint
has become classical since the work of Martinet and Ramis, and has
latter been generalized by Stolovitch as already mentioned, our approach
has some geometric flavor. For instance, we avoid the use of fixed-point
methods altogether to establish the existence of the normalizing maps,
and generalize instead the approach of Teyssier~\cite{teyssier2004equation,Teyssier03}
relying on path-integration of well-chosen $1$-forms (following Arnold's
method of characteristics~\cite{Arnold}).

As a by-product of this normalization we deduce that the normalizing
sectorial diffeomorphisms are Gevrey-$1$ asymptotic to the normalizing
formal power series of \cite{bittmann1}, retrospectively proving
their $1$-summability (with respect to the $x-$coordinate). When
the vector field additionally supports a symplectic transverse structure
(which is again the case of Painlevé equations) we prove that the
(essentially unique) sectorial normalizing map is performed by a transversally
symplectic diffeomorphism. We deduce from this a theorem of analytic
classification under the action of \emph{transversally symplectic}
diffeomorphisms.

\subsection{Definitions and main results}

~

To state our main results we need to introduce some notations and
nomenclature. 
\begin{itemize}
\item For $n\in\ww N_{>0}$, we denote by $\left(\ww C^{n},0\right)$ an
(arbitrary small) open neighborhood of the origin in $\ww C^{n}$.
\item We denote by $\germ{x,\mathbf{y}}$, with $\mathbf{y}=\left(y_{1},y_{2}\right)$,
the $\ww C$-algebra of germs of holomorphic functions at the origin
of $\ww C^{3}$, and by $\germ{x,\mathbf{y}}^{\times}$ the group
of invertible elements for the multiplication (also called units),
\emph{i.e. }elements $U$ such that $U\left(0\right)\neq0$.
\item $\vf$ is the Lie algebra of germs of singular holomorphic vector
fields at the origin $\ww C^{3}$. Any vector field in $\vf$ can
be written as 
\[
Y={\displaystyle b\left(x,y_{1},y_{2}\right)\pp x+b_{1}\left(x,y_{1},y_{2}\right)\pp{y_{1}}+b_{2}\left(x,y_{1},y_{2}\right)\pp{y_{2}}}
\]
with $b,b_{1},b_{2}\in\germ{x,y_{1},y_{2}}$ vanishing at the origin.
\item $\diff$ is the group of germs of a holomorphic diffeomorphism fixing
the origin of $\ww C^{3}$. It acts on $\vf$ by conjugacy: for all
\[
\left(\Phi,Y\right)\in\diff\times\vf
\]
we define the push-forward of $Y$ by $\Phi$ by
\begin{equation}
\Phi_{*}\left(Y\right):=\left(\mbox{D}\Phi\cdot Y\right)\circ\Phi^{-1}\qquad,\label{eq: push forward intro}
\end{equation}
where $\mbox{D}\Phi$ is the Jacobian matrix of $\Phi$.
\item $\fdiff$ is the subgroup of $\diff$ of fibered diffeomorphisms preserving
the $x$-coordinate, \emph{i.e. }of the form $\left(x,\mathbf{y}\right)\mapsto\left(x,\phi\left(x,\mathbf{y}\right)\right)$.
\item We denote by $\fdiffid$ the subgroup of $\fdiff$ formed by diffeomorphisms
tangent to the identity.
\end{itemize}
All these concepts have \emph{formal} analogues, where we only suppose
that the objects are defined with formal power series, not necessarily
convergent near the origin. 
\begin{defn}
\label{def: drsn}A \textbf{diagonal doubly-resonant saddle-node}
is a vector field $Y\in\vf$ of the form 
\begin{eqnarray*}
Y & = & x^{2}\pp x+\Big(-\lambda y_{1}+F_{1}\left(x,\mathbf{y}\right)\Big)\pp{y_{1}}+\Big(\lambda y_{2}+F_{2}\left(x,\mathbf{y}\right)\Big)\pp{y_{2}}\,\,\,\,\,,
\end{eqnarray*}
with $\lambda\in\ww C^{*}$ and $F_{1},F_{2}\in\germ{x,\mathbf{y}}$
of order at least two. We denote by $\sndiag$ the set of such vector
fields.
\end{defn}

\begin{rem}
One can also define the foliation associate to a diagonal doubly-resonant
saddle-node in a geometric way. A vector field $Y\in\vf$ is orbitally
equivalent to a diagonal doubly-resonant saddle-node $\Big($\emph{i.e.}
$Y$ is conjugate to some $VX$, where $V\in\germ{x,\mathbf{y}}^{\times}$
and $X\in\snfib$$\Big)$ if and only if the following conditions
hold:

\begin{enumerate}
\item $\tx{Spec}\left(\tx D_{0}Y\right)=\acc{0,-\lambda,\lambda}$ with
$\lambda\neq0$;
\item there exists a germ of irreducible analytic hypersurface $\mathscr{H}_{0}=\acc{S=0}$
which is transverse to the eigenspace $E_{0}$ (corresponding to the
zero eigenvalue) at the origin, and which is stable under the flow
of $Y$;
\item $\cal L_{Y}\left(S\right)=U.S^{2}$, where $\cal L_{Y}$ is the Lie
derivative of $Y$ and $U\in\germ{x,\mathbf{y}}^{\times}$. 
\end{enumerate}
\end{rem}

By Taylor expansion up to order $1$ with respect to $\mathbf{y}$,
given a vector field $Y\in\sndiag$ written as in (\ref{eq: intro})
we can consider the associate 2-dimensional system:
\begin{equation}
x^{2}\ddd{\mathbf{y}}x=\mathbf{\alpha}\left(x\right)+\mathbf{A}\left(x\right)\mathbf{y}\left(x\right)+\mathbf{F}\left(x,\mathbf{y}\left(x\right)\right)\qquad,\label{eq: system doubly resonant saddle node}
\end{equation}
with ${\displaystyle \mathbf{y}}=\left(y_{1},y_{2}\right)$, such
that the following conditions hold: 
\begin{itemize}
\item ${\displaystyle \alpha\left(x\right)=\left(\begin{array}{c}
\alpha_{1}\left(x\right)\\
\alpha_{2}\left(x\right)
\end{array}\right)},$ with $\alpha_{1},\alpha_{2}\in\germ x$ and ${\displaystyle {\displaystyle \alpha_{1},\alpha_{2}\in\tx O\left(x^{2}\right)}}$
\item ${\displaystyle \mathbf{A}\left(x\right)\in\mbox{Mat}_{2,2}\left(\germ x\right)}$
with ${\displaystyle \mathbf{A}\left(0\right)=\tx{diag}\left(-\lambda,\lambda\right)}$,
$\lambda\in\ww C^{*}$
\item ${\displaystyle \mathbf{F}\left(x,\mathbf{y}\right)=\left(\begin{array}{c}
F_{1}\left(x,\mathbf{y}\right)\\
F_{2}\left(x,\mathbf{y}\right)
\end{array}\right)}$, with ${\displaystyle F_{1},F_{2}\in\germ{x,\mathbf{y}}}$ and ${\displaystyle F_{1},F_{2}\in\tx O\left(\norm{\mathbf{y}}^{2}\right)}$.
\end{itemize}
Based on this expression, we state:
\begin{defn}
\label{def: non-deg} The \textbf{residue} of $Y\in\sndiag$ is the
complex number
\[
{\displaystyle \tx{res}\left(Y\right):=\left(\frac{\mbox{Tr}\left(\mathbf{A}\left(x\right)\right)}{x}\right)_{\mid x=0}}\qquad.
\]
We say that $Y$ is\textbf{ non-degenerate }(\emph{resp. }\textbf{strictly
non-degenerate) }if $\tx{res}\left(Y\right)\notin\ww Q_{\leq0}$ (\emph{resp.
}$\Re\left(\tx{res}\left(Y\right)\right)>0$).
\end{defn}

\begin{rem}
It is obvious that there is an action of $\fdiff[\ww C^{3},0,\tx{Id}]$
on $\sndiag$. The residue is an invariant of each orbit of $\snfib$
under the action of $\fdiff[\ww C^{3},0,\tx{Id}]$ by conjugacy (see~\cite{bittmann1}).
\end{rem}

The main result of \cite{bittmann1} can now be stated as follows:
\begin{thm}
\cite{bittmann1}\label{thm: forme normalel formelle} Let $Y\in\sndiag$
be non-degenerate. Then there exists a unique formal fibered diffeomorphism
$\hat{\Phi}$ tangent to the identity such that: 
\begin{eqnarray}
\hat{\Phi}_{*}\left(Y\right) & = & x^{2}\pp x+\left(-\lambda+a_{1}x+c_{1}\left(y_{1}y_{2}\right)\right)y_{1}\pp{y_{1}}\nonumber \\
 &  & +\left(\lambda+a_{2}x+c_{2}\left(y_{1}y_{2}\right)\right)y_{2}\pp{y_{2}}\,\,\,\,,\label{eq: fibered normal form-1-2}
\end{eqnarray}
where $\lambda\in\ww C^{*}$, ${\displaystyle c_{1},c_{2}\in\form v}$
are formal power series in $v=y_{1}y_{2}$ without constant term and
$a_{1},a_{2}\in\ww C$ are such that ${\displaystyle a_{1}+a_{2}=\tx{res}\left(Y\right)\in\ww C\backslash\ww Q_{\leq0}}$.
\end{thm}

\begin{defn}
The vector field obtained in (\ref{eq: fibered normal form-1-2})
is called the \textbf{formal normal form }of $Y$. The formal fibered
diffeomorphism $\hat{\Phi}$ is called the \textbf{formal normalizing
map} of $Y$.
\end{defn}

The above result is valid for formal objects, without considering
problems of convergence. The first main result in this work states
that this formal normalizing map is analytic in sectorial domains,
under some additional assumptions that we are now going to precise.
\begin{defn}
\label{def: asympt hamil}~

\begin{itemize}
\item We say that a germ of a vector field $X$ in $\left(\ww C^{2},0\right)$
is \textbf{orbitally linear} if 
\[
X=U\left(\mathbf{y}\right)\left(\lambda_{1}y_{1}\pp{y_{1}}+\lambda_{2}y_{2}\pp{y_{2}}\right)\,\,,
\]
for some ${\displaystyle U\left(\mathbf{y}\right)\in\germ{\mathbf{y}}^{\times}}$
and $\left(\lambda_{1},\lambda_{2}\right)\in\ww C^{2}$.
\item We say that a germ of vector field $X$ in $\left(\ww C^{2},0\right)$
is analytically (\emph{resp. formally}) \textbf{orbitally linearizable}
if $X$ is analytically (\emph{resp.} formally) conjugate to an orbitally
linear vector field.
\item We say that a diagonal doubly-resonant saddle-node $Y\in\sndiag$
is \textbf{div-integrable} if $Y_{\mid\acc{x=0}}\in\vf[\ww C^{2},0]$
is (analytically) orbitally linearizable. 
\end{itemize}
\end{defn}

\begin{rem}
Alternatively we could say that the foliation associated to ${\displaystyle Y_{\mid\acc{x=0}}}$
is linearizable. Since ${\displaystyle Y_{\mid\acc{x=0}}}$ is analytic
at the origin of $\ww C^{2}$ and has two opposite eigenvalues, it
follows from a classical result of Brjuno (see \cite{Martinet}),
that $Y_{\mid\acc{x=0}}$ is analytically orbitally linearizable if
and only if it is formally orbitally linearizable.
\end{rem}

\begin{defn}
We denote by $\snodiag$ the set of strictly non-degenerate diagonal
doubly-resonant saddle-nodes which are div-integrable.
\end{defn}

The vector field corresponding to the irregular singularity at infinity
in the Painlevé equations $\left(P_{j}\right)_{j=I,\dots,V}$ is orbitally
equivalent to an element of $\snofib$, for generic values of the
parameters (see \cite{Yoshida85}).

We can now state the first main result of our paper (we refer to section
$2.$ for more details on\emph{ 1-summability}). 
\begin{thm}
\label{Th: Th drsn}Let $Y\in\snodiag$ and let $\hat{\Phi}$ (given
by Theorem \ref{thm: forme normalel formelle}) be the unique formal
fibered diffeomorphism tangent to the identity such that 
\begin{eqnarray*}
\hat{\Phi}_{*}\left(Y\right) & = & x^{2}\pp x+\left(-\lambda+a_{1}x+c_{1}\left(y_{1}y_{2}\right)\right)y_{1}\pp{y_{1}}+\left(\lambda+a_{2}x+c_{2}\left(y_{1}y_{2}\right)\right)y_{2}\pp{y_{2}}\\
 & =: & \ynorm\,\,,
\end{eqnarray*}
where $\lambda\neq0$ and ${\displaystyle c_{1}\left(v\right),c_{2}\left(v\right)\in v\form v}$
are formal power series without constant term. Then:

\begin{enumerate}
\item the normal form $\ynorm$ is analytic \emph{(i.e. ${\displaystyle c_{1},c_{2}\in\germ v}$),
}and it also is div-integrable,\emph{ }i.e. $c_{1}+c_{2}=0$;
\item \label{enu:the-formal-normalizing}the formal normalizing map $\hat{\Phi}$
is 1-summable (with respect to $x$) in every direction $\theta\neq\arg\left(\pm\lambda\right)$.
\item \label{enu:there-exist-analytic}there exist analytic sectorial fibered
diffeomorphisms $\Phi_{+}$ and $\Phi_{-}$, (asymptotically) tangent
to the identity, defined in sectorial domains of the form ${\displaystyle S_{+}\times\left(\ww C^{2},0\right)}$
and ${\displaystyle S_{-}\times\left(\ww C^{2},0\right)}$ respectively,
where 
\begin{eqnarray*}
S_{+} & := & \acc{x\in\ww C\mid0<\abs x<r\mbox{ and }\abs{\arg\left(\frac{x}{i\lambda}\right)}<\frac{\pi}{2}+\epsilon}\\
S_{-} & := & \acc{x\in\ww C\mid0<\abs x<r\mbox{ and }\abs{\arg\left(\frac{-x}{i\lambda}\right)}<\frac{\pi}{2}+\epsilon}
\end{eqnarray*}
(for any ${\displaystyle \epsilon\in\left]0,\frac{\pi}{2}\right[}$
and some $r>0$ small enough), which admit $\hat{\Phi}$ as weak Gevrey-1
asymptotic expansion in these respective domains, and which conjugate
$Y$ to $\ynorm$. Moreover $\Phi_{+}$ and $\Phi_{-}$ are the unique
such germs of analytic functions in sectorial domains (see Definition
\ref{def: sectorial diff}).
\end{enumerate}
\end{thm}

\begin{rem}
Although item \ref{enu:there-exist-analytic} above is a straightforward
consequence of the\emph{ 1-summability} of $\hat{\Phi}$ (item \ref{enu:the-formal-normalizing}
above), we will in fact start by proving item \ref{enu:there-exist-analytic}
in Corollary \ref{cor: existence normalisations sectorielles}, and
establish the 1-summability of item \ref{enu:the-formal-normalizing}
in a second step (see Proposition \ref{prop: sommabilit=0000E9 normalisation formelle}).
What we will obtain at first directly is only the \emph{weak} 1-summability
(see subsection \ref{subsec:Weak-Gevrey-1-power}) of $\hat{\Phi}$
(see Proposition \ref{prop: Weak sectorial normalizations}), and
not immediately the 1-summability. To obtain the ``true'' 1-summability,
we will need to prove that the transition maps between $\Phi_{+}$
and $\Phi_{-}$ are exponentially close to the identity (see Proposition
\ref{prop: isotropies plates}), and then to use a fundamental theorem
of Martinet and Ramis (see Theorem \ref{th: martinet ramis}).
\end{rem}

\begin{defn}
\label{def: sectorial normalizations}We call $\Phi_{+}$ and $\Phi_{-}$
the \textbf{sectorial normalizing maps} of $Y\in\snodiag$. 
\end{defn}

They are the 1-sums of $\hat{\Phi}$ along the respective directions
$\arg\left(i\lambda\right)$ and $\arg\left(-i\lambda\right)$. Notice
that $\Phi_{+}$ and $\Phi_{-}$ are \emph{germs of analytic sectorial
fibered diffeomorphisms}, \emph{i.e.} they are of the form
\begin{eqnarray*}
\Phi_{+}:S_{+}\times\left(\ww C^{2},0\right) & \longrightarrow & S_{+}\times\left(\ww C^{2},0\right)\\
\left(x,\mathbf{y}\right) & \longmapsto & \left(x,\Phi_{+,1}\left(x,\mathbf{y}\right),\Phi_{+,2}\left(x,\mathbf{y}\right)\right)
\end{eqnarray*}
and 
\begin{eqnarray*}
\Phi_{-}:S_{-}\times\left(\ww C^{2},0\right) & \longrightarrow & S_{-}\times\left(\ww C^{2},0\right)\\
\left(x,\mathbf{y}\right) & \longmapsto & \left(x,\Phi_{-,1}\left(x,\mathbf{y}\right),\Phi_{-,2}\left(x,\mathbf{y}\right)\right)
\end{eqnarray*}
(see section 2. for a precise definition of \emph{germ of analytic
sectorial fibered diffeomorphism}). The fact that they are\emph{ }also
\emph{(asymptotically) tangent to the identity }means that we have:
\[
{\displaystyle \Phi_{\pm}\left(x,\mathbf{y}\right)=\tx{Id}\left(x,\mathbf{y}\right)+\tx O\left(\norm{\left(x,\mathbf{y}\right)}^{2}\right)}\,\,.
\]

In fact, we can prove the uniqueness of the sectorial normalizing
maps under weaker assumptions.
\begin{prop}
\label{prop: unique normalizations}Let $\varphi_{+}$ and $\varphi_{-}$
be two germs of sectorial fibered diffeomorphisms in ${\displaystyle S_{+}\times\left(\ww C^{2},0\right)}$
and ${\displaystyle S_{-}\times\left(\ww C^{2},0\right)}$ respectively,
where $S_{+}$ and $S_{-}$ are as in Theorem \ref{Th: Th drsn},
which are (asymptotically) tangent to the identity and such that 
\[
\left(\varphi_{\pm}\right)_{*}\left(Y\right)=\ynorm\,\,.
\]
Then, they necessarily coincide with the the sectorial normalizing
maps $\Phi_{+}$ and $\Phi_{-}$ defined above.
\end{prop}

Since two analytically conjugate vector fields are also formally conjugate,
we fix now a normal form 
\[
\ynorm=x^{2}\pp x+\left(-\lambda+a_{1}x-c\left(v\right)\right)y_{1}\pp{y_{1}}+\left(\lambda+a_{2}x+c\left(v\right)\right)y_{2}\pp{y_{2}}\,\,\,\,,
\]
with $\lambda\in\ww C^{*}$, $\Re\left(a_{1}+a_{2}\right)>0$ and
$c\in v\germ v$ vanishing at the origin.
\begin{defn}
\label{def: ynorm class}We denote by ${\displaystyle \cro{\ynorm}}$
the set of germs of holomorphic doubly-resonant saddle-nodes in $\left(\ww C^{3},0\right)$
which are formally conjugate to $\ynorm$ by formal fibered diffeomorphisms
tangent to the identity, and denote by ${\displaystyle \quotient{\cro{\ynorm}}{{\fdiff[\ww C^{3},0,\tx{Id}]}}}$
the set of orbits of the elements in this set under the action of
${\displaystyle \fdiff[\ww C^{3},0,\tx{Id}]}$.
\end{defn}

According to Theorem \ref{Th: Th drsn}, to any ${\displaystyle Y\in{\displaystyle \cro{\ynorm}}}$
we can associate two sectorial normalizing maps $\Phi_{+},\Phi_{-}$,
which can in fact extend analytically in domains $S_{+}\times\left(\ww C^{2},0\right)$
and $S_{-}\times\left(\ww C^{2},0\right)$, where $S_{\pm}$ is an
asymptotic sector in the direction $\arg\left(\pm i\lambda\right)$
with opening $2\pi$ (see Definition \ref{def: asymptotitc sector}):
\[
\left(S_{+},S_{-}\right)\in\asympsect{\arg\left(i\lambda\right)}{2\pi}\times\asympsect{\arg\left(-i\lambda\right)}{2\pi}\,\,.
\]
Then, we consider two germs of sectorial fibered diffeomorphisms $\Phi_{\lambda},\Phi_{-\lambda}$
analytic in $S_{\lambda},S_{-\lambda}$, with 
\begin{eqnarray}
S_{\lambda} & := & S_{+}\cap S_{-}\cap\acc{\Re\left(\frac{x}{\lambda}\right)>0}\in\asympsect{\arg\left(\lambda\right)}{\pi}\label{eq: petits secteurs}\\
S_{-\lambda} & := & S_{+}\cap S_{-}\cap\acc{\Re\left(\frac{x}{\lambda}\right)<0}\in\asympsect{\arg\left(-\lambda\right)}{\pi}\,\,,\nonumber 
\end{eqnarray}
defined by: 
\[
\begin{cases}
\Phi_{\lambda}:=\left(\Phi_{+}\circ\Phi_{-}^{-1}\right)_{\mid S_{\lambda}\times\left(\ww C^{2},0\right)}\in\diffsect[\arg\left(\lambda\right)][\epsilon] & ,\,\forall\epsilon\in\left[0,\pi\right[\\
\Phi_{-\lambda}:=\left(\Phi_{-}\circ\Phi_{+}^{-1}\right)_{\mid S_{-\lambda}\times\left(\ww C^{2},0\right)}\diffsect[\arg\left(-\lambda\right)][\epsilon] & ,\,\forall\epsilon\in\left[0,\pi\right[\,\,.
\end{cases}
\]
Notice that $\Phi_{\lambda},\Phi_{-\lambda}$ are \emph{isotropies}
of $\ynorm$, \emph{i.e. }they satisfy:
\begin{eqnarray*}
\left(\Phi_{\pm\lambda}\right)_{*}\left(\ynorm\right) & = & \ynorm\,\,.
\end{eqnarray*}
\begin{defn}
\label{def: Stokes diffeo}With the above notations, we define ${\displaystyle \Lambda_{\lambda}\left(\ynorm\right)}$
${\displaystyle \left(\mbox{\emph{resp.} }\Lambda_{-\lambda}\left(\ynorm\right)\right)}$
as the group of germs of sectorial fibered isotropies of $\ynorm$,
tangent to the identity, and admitting the identity as Gevrey-1 asymptotic
expansion (see Definition \ref{def:Gevrey-1_expansion}) in sectorial
domains of the form ${\displaystyle S_{\lambda}\times\left(\ww C^{2},0\right)}$
${\displaystyle \left(resp.\,S_{-\lambda}\times\left(\ww C^{2},0\right)\right)}$,
with $S_{\pm\lambda}\in\asympsect{\arg\left(\pm\lambda\right)}{\pi}$. 

The two sectorial isotropies $\Phi_{\lambda}$ and $\Phi_{-\lambda}$
defined above are called the \textbf{Stokes diffeomorphisms} associate
to ${\displaystyle Y\in{\displaystyle \cro{\ynorm}}}$.
\end{defn}

Our second main result gives the moduli space for the analytic classification
that we are looking for.
\begin{thm}
\label{thm: espace de module}The map
\begin{eqnarray*}
\quotient{\cro{\ynorm}}{{\fdiff[\ww C^{3},0,\tx{Id}]}} & \longrightarrow & \Lambda_{\lambda}\left(\ynorm\right)\times\Lambda_{-\lambda}\left(\ynorm\right)\\
Y & \longmapsto & \left(\Phi_{\lambda},\Phi_{-\lambda}\right)\,\,
\end{eqnarray*}
is well-defined and bijective.
\end{thm}

In particular, the result states that Stokes diffeomorphisms only
depend on the class of $Y\in\cro{\ynorm}$ in the quotient ${\displaystyle \quotient{\cro{\ynorm}}{{\fdiff[\ww C^{3},0,\tx{Id}]}}}$.
We will give a description of this set of invariants in terms of power
series in the \emph{space of leaves} in section \ref{sec: analytic classification}.
\begin{rem}
\label{rem:New difficulties}In this paper we start by proving a theorem
of sectorial normalizing map analogous to the classical one due to
Hukuhara-Kimura-Matuda for saddle-nodes in $\left(\ww C^{2},0\right)$
\cite{HKM}, generalized later by Stolovitch in any dimension in \cite{Stolo}.
Unlike the method based on a fixed point theorem used by these authors,
we have used a more geometric approach (following the works of Teyssier~\cite{Teyssier03,teyssier2004equation})
based on the resolution of an homological equation, by integrating
a well chosen 1-form along asymptotic paths. This latter approach
turned out to be more efficient to deal with the fact that $Y_{\mid\acc{x=0}}$
is not necessarily linearizable. Indeed, if we look at \cite{Stolo}
in details, one of the first problem is that in the irregular systems
that needs to be solved by a fixed point method (for instance equation
$\left(2.7\right)$ in the cited paper), the non-linear terms would
not be divisible by the ``time'' variable $t$ in our situation.
This would complicate the different estimations that are done later
in the cited work. This is the first main new phenomena we have met. 

Then we will see that the sectorial normalizing maps $\Phi_{+},\Phi_{-}$
in the corollary above admit in fact the unique formal normalizing
map $\hat{\Phi}$ given by Theorem \ref{thm: forme normalel formelle}
as ``true'' Gevrey-1 asymptotic expansion in $\cal S_{+}\in\cal S_{\arg\left(\lambda\right),\eta}$
and $\cal S_{-}\in\cal S_{\arg\left(-\lambda\right),\eta}$ respectively.
This will be proved by studying $\Phi_{+}\circ\left(\Phi_{-}\right)^{-1}$
in $\cal S_{+}\cap\cal S_{-}$ (and more generally any germ of sectorial
fibered isotropy of $\ynorm$ in ``narrow'' sectorial neighborhoods
$\cal S_{\pm\lambda}\subset\cal S_{+}\cap\cal S_{-}$ which admits
the identity as weak Gevrey-1 asymptotic expansion). The main difficulty
is to prove that such a sectorial isotropy of $\ynorm$ over the ``narrow''
sectors described above is necessarily exponentially close to the
identity (see Lemma \ref{lem: isotropies exp plates}). This will
be done \emph{via} a detailed analysis of these maps in the space
of leaves (see Definition \ref{def: space of leaves}). In fact, this
is the second main new difficulty we have met, which is due to the
presence of the ``resonant'' term 
\[
\frac{c_{m}\left(y_{1}y_{2}\right)^{m}\log\left(x\right)}{x}
\]
in the exponential expression of the first integrals of the vector
field (see $\left(\mbox{\ref{eq: integrales premieres}}\right)$).
In \cite{Stolo}, similar computations are done in Subsection $3.4.1$.
In this part of the paper, infinitely many irregular differential
equations appear when identifying terms of same homogeneous degree.
The fact that $Y_{\mid\acc{x=0}}$ is linear implies that these differential
equations are all linear and independent of each others (\emph{i.e.
}they are not mixed together). In our situation, this is not the case
and then more complicated.
\end{rem}

\subsection{\label{subsec: transversally symplectic}Painlevé equations: the
transversally Hamiltonian case }

~

In \cite{Yoshida85} Yoshida shows that a vector field in the class
$\snofib$ naturally appears after a suitable compactification (given
by the so-called Boutroux coordinates \cite{Boutroux13}) of the phase
space of Painlevé equations $\left(P_{j}\right)_{j=I,\dots,V}$, for
generic values of the parameters. In these cases the vector field
presents an additional transverse Hamiltonian structure. Let us illustrate
these computations in the case of the first Painlevé equation: 
\begin{eqnarray*}
\left(P_{I}\right)\,\,\,\,\,\,\,\,\,\,\,\,\,\,\,\,\,\,\,\,\,\ddd{^{2}z_{1}}{t^{2}} & = & 6z_{1}^{2}+t\,\,\,\,\,\,\,\,\,\,\,\,\,\,.
\end{eqnarray*}
As is well known since Okamoto \cite{Okamoto}, $\left(P_{I}\right)$
can be seen as a non-autonomous Hamiltonian system 
\[
\begin{cases}
{\displaystyle \ppp{z_{1}}t=-\ppp H{z_{2}}}\\
{\displaystyle \ppp{z_{2}}t=\ppp H{z_{1}}}
\end{cases}
\]
with Hamiltonian 
\begin{eqnarray*}
H\left(t,z_{1},z_{2}\right) & := & 2z_{1}^{3}+tz_{1}-\frac{z_{2}^{2}}{2}.
\end{eqnarray*}
More precisely, if we consider the standard symplectic form $\omega:=dz_{1}\wedge dz_{2}$
and the vector field 
\begin{eqnarray*}
Z & := & \pp t-\ppp H{z_{2}}\pp{z_{1}}+\ppp H{z_{1}}\pp{z_{2}}
\end{eqnarray*}
induced by $\left(P_{I}\right)$, then the Lie derivative 
\[
\cal L_{Z}\left(\omega\right)=\left(\ppp{^{2}H}{t\partial z_{1}}\mbox{d}z_{1}+\ppp{^{2}H}{t\partial z_{2}}\mbox{d}z_{2}\right)\wedge\mbox{d}t=\mbox{d}z_{1}\wedge\mbox{d}t
\]
 belongs to the ideal $\ps{\tx dt}$ generated by $\tx dt$ in the
exterior algebra $\Omega^{*}\left(\ww C^{3}\right)$ of differential
forms in variables $\left(t,z_{1},z_{2}\right)$. Equivalently, for
any $t_{1},t_{2}\in\ww C$ the flow of $Z$ at time $\left(t_{2}-t_{1}\right)$
acts as a \emph{symplectomorphism} between fibers $\acc{t=t_{1}}$
and $\acc{t=t_{2}}$. 

The weighted compactification given by the Boutroux coordinates~\cite{Boutroux13}
defines a chart near $\acc{t=\infty}$ as follows: 
\[
\begin{cases}
{\displaystyle z_{2}=y_{2}x^{-\frac{3}{5}}}\\
{\displaystyle z_{1}=y_{1}x^{-\frac{2}{5}}}\\
{\displaystyle t=x^{-\frac{4}{5}}} & .
\end{cases}
\]
In the coordinates $\left(x,y_{1},y_{2}\right)$, the vector field
$Z$ is transformed, up to a translation $y_{1}\leftarrow y_{1}+\zeta$
with ${\displaystyle \zeta=\frac{i}{\sqrt{6}}}$, to the vector field
\begin{eqnarray}
\tilde{Z} & = & -\frac{5}{4x^{\frac{1}{5}}}Y\label{eq: P1 ham at infinity}
\end{eqnarray}
where
\begin{eqnarray*}
Y & = & x^{2}\pp x+\left(-\frac{4}{5}y_{2}+\frac{2}{5}xy_{1}+\frac{2\zeta}{5}x\right)\pp{y_{1}}+\left(-\frac{24}{5}y_{1}^{2}-\frac{48\zeta}{5}y_{1}+\frac{3}{5}xy_{2}\right)\pp{y_{2}}\,\,\,\,\,\,\,\,\,.
\end{eqnarray*}
We observe that $Y$ is a\emph{ }strictly non-degenerate doubly-resonant
saddle-node as in Definitions \ref{def: drsn} and \ref{def: non-deg}
with residue $\tx{res}\left(Y\right)=1$. Furthermore we have: 
\[
\begin{cases}
{\displaystyle \tx dt} & {\displaystyle =-\frac{4}{5}5^{\frac{4}{5}}x^{-\frac{9}{5}}\tx dx}\\
{\displaystyle \tx dz_{1}\wedge\tx dz_{2}} & {\displaystyle =\frac{1}{x}\left(\tx dy_{1}\wedge\tx dy_{2}\right)+\frac{1}{5x^{2}}\left(2y_{1}\tx dy_{2}-3y_{2}\tx dy_{1}\right)\wedge\tx dx}\\
 & {\displaystyle \in\,\frac{1}{x}\left(\tx dy_{1}\wedge\tx dy_{2}\right)+\ps{\tx dx}}
\end{cases}\,\,\,\,\,,
\]
where $\ps{\mbox{d}x}$ denotes the ideal generated by $\mbox{d}x$
in the algebra of holomorphic forms in $\ww C^{*}\times\ww C^{2}$.
We finally obtain 
\[
\begin{cases}
{\displaystyle {\displaystyle \cal L_{Y}\left(\frac{\tx dy_{1}\wedge\tx dy_{2}}{x}\right)}=\frac{1}{5x}\left(3y_{2}\mbox{d}y_{1}-\left(2\zeta+2y_{1}\right)\mbox{d}y_{2}\right)\wedge\mbox{d}x}\\
{\displaystyle \cal L_{Y}\left(\mbox{d}x\right)=2x\mbox{d}x}
\end{cases}\quad.
\]
 Therefore, both ${\displaystyle \cal L_{Y}\left(\omega\right)}$
and $\cal L_{Y}\left(\mbox{d}x\right)$ are differential forms who
lie in the ideal $\ps{\tx dx}$, in the algebra of germs of meromorphic
1-forms in $\left(\ww C^{3},0\right)$ with poles only in $\acc{x=0}$.
This motivates the following:
\begin{defn}
\label{def: intro}Consider the rational 1-form 
\begin{eqnarray*}
\omega & := & \frac{\mbox{d}y_{1}\wedge\mbox{d}y_{2}}{x}~.
\end{eqnarray*}
We say that vector field $Y$ is \textbf{transversally Hamiltonian
}(with respect to $\omega$ and $\mbox{dx}$) if 
\begin{eqnarray*}
\mathcal{L}_{Y}\left(\tx dx\right)\in\left\langle \tx dx\right\rangle  & \mbox{ and } & \mathcal{L}_{Y}\left(\omega\right)\in\left\langle \tx dx\right\rangle \qquad.
\end{eqnarray*}
For any open sector $S\subset\ww C^{*}$, we say that a germ of sectorial
fibered diffeomorphism $\Phi$ in $S\times\left(\ww C^{2},0\right)$
is \textbf{transversally symplectic} (with respect to $\omega$ and
$\mbox{d}x$) if
\[
\Phi^{*}\left(\omega\right)\in\,\omega+\ps{\tx dx}\qquad
\]
 (Here $\Phi^{*}\left(\omega\right)$ denotes the pull-back of $\omega$
by $\Phi$). 

We denote by $\sdiffid$ the group of transversally symplectic diffeomorphisms
which are tangent to the identity.
\end{defn}

\begin{rem}
~

\begin{enumerate}
\item The flow of a transversally Hamiltonian vector field $X$ defines
a map between fibers $\acc{x=x_{1}}$ and $\acc{x=x_{2}}$ which sends
$\omega_{\mid x=x_{1}}$ onto $\omega_{\mid x=x_{2}}$, since 
\[
{\displaystyle \left(\exp\left(X\right)\right)^{*}\left(\omega\right)\in\,\omega+\ps{\mbox{d}x}}\,\,\,.
\]
\item A fibered sectorial diffeomorphism $\Phi$ is transversally symplectic
if and only if $\det\left(\mbox{D}\Phi\right)=1$.
\end{enumerate}
\end{rem}

\begin{defn}
\label{def: nc dr th}A \textbf{transversally Hamiltonian doubly-resonant
saddle-node} is a transversally Hamiltonian vector field which is
conjugate, \emph{via }a transversally symplectic diffeomorphism, to
one of the form 
\begin{eqnarray*}
Y & = & x^{2}\pp x+\Big(-\lambda y_{1}+F_{1}\left(x,\mathbf{y}\right)\Big)\pp{y_{1}}+\Big(\lambda y_{2}+F_{2}\left(x,\mathbf{y}\right)\Big)\pp{y_{2}}\,\,\,\,\,,
\end{eqnarray*}
with $\lambda\in\ww C^{*}$ and $f_{1},f_{2}$ analytic in $\left(\ww C^{3},0\right)$
and of order at least $2$.
\end{defn}

Notice that a transversally Hamiltonian doubly-resonant saddle-node
is necessarily strictly non-degenerate (since its residue is always
equal to $1$), and also div-integrable (see section 3). It follows
from Yoshida's work \cite{Yoshida85} that the doubly-resonant saddle-nodes
at infinity in Painlevé equations$\left(P_{j}\right)_{j=I,\dots,V}$
(for generic values of the parameters) all are transversally Hamiltonian. 

We recall the second main result from \cite{bittmann1}. 
\begin{thm}
\cite{bittmann1}\label{thm: Th ham formel}~

Let $Y\in\sndiag$ be a diagonal doubly-resonant saddle-node which
is supposed to be transversally Hamiltonian. Then, there exists a
unique formal fibered transversally symplectic diffeomorphism $\hat{\Phi}$,
tangent to identity, such that: 
\begin{eqnarray}
\hat{\Phi}_{*}\left(Y\right) & = & x^{2}\pp x+\left(-\lambda+a_{1}x-c\left(y_{1}y_{2}\right)\right)y_{1}\pp{y_{1}}+\left(\lambda+a_{2}x+c\left(y_{1}y_{2}\right)\right)y_{2}\pp{y_{2}}\nonumber \\
 & =: & \ynorm\,\,,\label{eq: fibered normal form-1-1-1}
\end{eqnarray}
where $\lambda\in\ww C^{*}$, $c\left(v\right)\in v\form v$ a formal
power series in $v=y_{1}y_{2}$ without constant term and $a_{1},a_{2}\in\ww C$
are such that $a_{1}+a_{2}=1$.
\end{thm}

As a consequence of Theorem \ref{thm: Th ham formel}, Theorem \ref{Th: Th drsn}
we have the following:
\begin{thm}
\label{thm: Th ham}Let $Y$ be a transversally Hamiltonian doubly-resonant
saddle-node and let $\hat{\Phi}$ be the unique formal normalizing
map given by Theorem \ref{thm: Th ham formel}. Then the associate
sectorial normalizing maps $\Phi_{+}$ and $\Phi_{-}$ are also transversally
symplectic.
\end{thm}

\begin{proof}
Since $\hat{\Phi}$ is 1-summable in $S_{\pm}\times\left(\ww C^{2},0\right)$,
the formal power series $\det\left(\mbox{D}\hat{\Phi}\right)$ is
also 1-summable in $S_{\pm}\times\left(\ww C^{2},0\right)$, and its
asymptotic expansion has to be the constant $1$. By uniqueness of
the 1-sum, we thus have $\det\left(\mbox{D}\Phi_{\pm}\right)=1$.
\end{proof}
Let us fix a normal form $\ynorm$ as in Theorem \ref{thm: Th ham},
and consider two sectorial domains ${\displaystyle S_{\lambda}\times\left(\ww C^{2},0\right)}$
and $S_{-\lambda}\times\left(\ww C^{2},0\right)$ as in (\ref{eq: petits secteurs}).
Then, the Stokes diffeomorphisms $\left(\Phi_{\lambda},\Phi_{-\lambda}\right)$
defined in the previous subsection as 
\[
\begin{cases}
\Phi_{\lambda}:=\left(\Phi_{+}\circ\Phi_{-}^{-1}\right)_{\mid S_{\lambda}\times\left(\ww C^{2},0\right)}\\
\Phi_{-\lambda}:=\left(\Phi_{-}\circ\Phi_{+}^{-1}\right)_{\mid S_{-\lambda}\times\left(\ww C^{2},0\right)} & ,
\end{cases}
\]
are also transversally symplectic. 
\begin{defn}
We denote by $\Lambda_{\lambda}^{\omega}\left(\ynorm\right)$ $\big($\emph{resp.
}$\Lambda_{-\lambda}^{\omega}\left(\ynorm\right)$$\big)$ the group
of germs of sectorial fibered isotropies of $\ynorm$, admitting the
identity as Gevrey-1 asymptotic expansion in sectorial domains of
the form ${\displaystyle S_{\lambda}\times\left(\ww C^{2},0\right)}$
${\displaystyle \left(resp.\,S_{-\lambda}\times\left(\ww C^{2},0\right)\right)}$,
and which are transversally symplectic. 
\end{defn}

Let us denote by $\cro{\ynorm}_{\omega}$ the set of germs of vector
fields which are formally conjugate to $\ynorm$\emph{ via }(formal)
transversally symplectic diffeomorphisms tangent to the identity.
As a consequence of Theorems (\ref{thm: espace de module}) and (\ref{thm: Th ham}),
we can now state the following result.
\begin{thm}
\label{th: espace de module symplectic}The map 
\begin{eqnarray*}
\quotient{\cro{\ynorm}_{\omega}}{{\sdiffid}} & \longrightarrow & \Lambda_{\lambda}^{\omega}\left(\ynorm\right)\times\Lambda_{-\lambda}^{\omega}\left(\ynorm\right)\\
Y & \longmapsto & \left(\Phi_{\lambda},\Phi_{-\lambda}\right)\,\,
\end{eqnarray*}
is a well-defined bijection.
\end{thm}

\subsection{Outline of the paper}

~

In section 2, we introduce the different tools we need concerning
asymptotic expansion, Gevrey-1 series and 1-summability. We will in
particular introduce a notion of ``\textbf{weak}'' 1-summability.

In section 3, we prove Proposition \ref{prop: forme pr=0000E9par=0000E9e ordre N},
which states that we can always formally conjugate a non-degenerate
doubly-resonant saddle-node which is also div-integrable to its normal
form up to remaining terms of order $\tx O\left(x^{N}\right)$, for
all $N\in\ww N_{>0}$, and the conjugacy is actually $1$-summable.

In section 4, we prove that for all $Y\in\snofib$, there exists a
unique pair of sectorial normalizing maps $\left(\Phi_{+},\Phi_{-}\right)$
tangent to the identity which conjugates $Y$ to its normal form $\ynorm$
over sectors with opening greater than $\pi$ and arbitrarily close
to $2\pi$. The existence is given by Corollary \ref{cor: existence normalisations sectorielles},
while the uniqueness clause stated in Proposition \ref{prop: unique normalizations}
is proved thanks to Proposition \ref{prop: isot sect}. Moreover,
we will see that $\Phi_{+}$ and $\Phi_{-}$ both admit the unique
formal normalizing map $\hat{\Phi}$ given by Theorem \ref{thm: forme normalel formelle}
as weak Gevrey-1 asymptotic expansion (see Proposition \ref{prop: Weak sectorial normalizations}).

In section 5, we show that the Stokes diffeomorphisms $\Phi_{\lambda}$
and $\Phi_{-\lambda}$, which admit \emph{a priori} the identity only
as weak Gevrey-1 asymptotic expansion, admit in fact the identity
as ``true'' Gevrey-1 asymptotic expansion. This will be done by
studying more generally the germs of sectorial isotropies of the normal
form in sectorial domains with ``narrow'' opening (see Corollary
\ref{prop: isotropies plates}). Using a theorem by Martinet and Ramis~\cite{MR82}
reformulated in Theorem \ref{th: martinet ramis}, which is a ``non-abelian''
version of the Ramis-Sibuya theorem, we will obtain the fact that
$\hat{\Phi}$ is 1-summable in every direction $\theta\neq\arg\left(\pm\lambda\right)$,
of 1-sums $\Phi_{+}$ and $\Phi_{-}$ respectively in the corresponding
domains (see Corollary \ref{prop: sommabilit=0000E9 normalisation formelle}).
We then give a short proof of Theorem \ref{Th: Th drsn}, just by
using the different lemmas and propositions needed and proved earlier
in this paper. After that, we will once again use Theorem \ref{th: martinet ramis}
in order to obtain both Theorems \ref{thm: espace de module} and
\ref{th: espace de module symplectic}. We give in Proposition \ref{prop: description espace de module}
a description of the moduli space of analytic classification in terms
of some spaces of power series in the space of leaves. \tableofcontents{}

\section{Background}

We refer the reader to \cite{MR82,malgrange1995sommation,ramis1993divergent,DeMaesschalck}
for a detailed introduction to the theory of asymptotic expansion,
Gevrey series and summability (see also \cite{Stolo} for a useful
discussion of these concepts), where one can find the proofs of the
classical results we recall (but we do not prove here). We call $x\in\ww C$
the \emph{independent} variable and ${\displaystyle \mathbf{y}:=\left(y_{1},\dots,y_{n}\right)\in\ww C^{n}}$,
$n\in\ww N$, the \emph{dependent} variables. As usual we define ${\displaystyle \mathbf{y^{k}}:=y_{1}^{k_{1}}\dots y_{n}^{k_{n}}}$
for ${\displaystyle \mathbf{k}=\left(k_{1},\dots,k_{n}\right)\in\ww N^{n}}$,
and ${\displaystyle \abs{\mathbf{k}}=k_{1}+\dots+k_{n}}$. The notions
of asymptotic expansion, Gevrey-1 power series and 1-summability presented
here are always considered with respect to the independent variable
$x$ living in (open) sectors, the dependent variable $\mathbf{y}$
belonging to poly-discs 
\[
\mathbf{D\left(0,r\right)}:=\acc{\mathbf{y}=\left(y_{1},\dots,y_{n}\right)\in\ww C^{n}\mid\abs{y_{1}}<r_{1},\dots\abs{y_{n}}<r_{n}}\,\,\,,
\]
of poly-radius ${\displaystyle \mathbf{r}=\left(r_{1},\dots,r_{n}\right)\in\left(\ww R_{>0}\right)^{n}}$.
Given an open subset 
\[
\cal U\subset\ww C^{n+1}=\acc{\left(x,\mathbf{y}\right)\in\ww C\times\ww C^{n}}
\]
 we denote by $\cal O\left({\cal U}\right)$ the algebra of holomorphic
function in $\cal U$. The algebra of germs of analytic functions
of $m$ variables $\mathbf{x}:=\left(x_{1},\dots,x_{m}\right)$ at
the origin is denoted by $\germ{\mathbf{x}}$.

The results recalled in this section are valid when $n=0$. Some statements
for which we do not give a proof can be proved exactly as in the classical
case $n=0$, uniformly in the dependent multi-variable $\mathbf{y}$.
For convenience and homogeneity reasons we will present some classical
results not in their original (and more general) form, but rather
in more specific cases which we will need. Finally, we will introduce
a notion of \emph{weak }Gevrey-1 summability, which we will compare
to the classical notion of 1-summability.

\subsection{Sectorial germs}

~

Given $r>0$, and $\alpha,\beta\in\ww R$ with $\alpha<\beta$, we
denote by $\sect r{\alpha}{\beta}$ the following open sector:
\[
S\left(r,\alpha,\beta\right):=\acc{x\in\ww C\mid0<\abs x<r\mbox{ and }\alpha<\arg\left(x\right)<\beta}\,\,.
\]
Let $\theta\in\ww R$, $\eta\in\ww R_{\geq0}$ and $n\in\ww N$.
\begin{defn}
\label{def: sectorial germs}

\begin{enumerate}
\item An\emph{ x-sectorial neighborhood }(or simply \emph{sectorial neighborhood})
\emph{of the origin} (in $\ww C^{n+1}$) \emph{in the direction $\theta$
with opening $\eta$} is an open set $\cal S\subset\ww C^{n+1}$ such
that
\[
\cal S\supset S\left(r,\theta-\frac{\eta}{2}-\epsilon,\theta+\frac{\eta}{2}+\epsilon\right)\times\mathbf{D\left(0,r\right)}
\]
for some $r>0$, $\mathbf{r}\in\left(\ww R_{>0}\right)^{n}$ and $\epsilon>0$.
We denote by $\left(\germsect{\theta}{\eta},\leq\right)$ the directed
set formed by all such neighborhoods, equipped with the order relation
\begin{eqnarray*}
S_{1}\leq S_{2} & \Longleftrightarrow & S_{1}\supset S_{2}\,\,.
\end{eqnarray*}
\item The algebra of \emph{germs of holomorphic functions in a sectorial
neighborhood of the origin in the direction $\theta$ with opening
$\eta$} is the direct limit 
\[
\cal O\left(\cal S_{\theta,\eta}\right):=\underrightarrow{\lim}\,\cal O\left(\cal S\right)
\]
with respect to the directed system defined by $\acc{\cal O\left(\cal S\right):\cal S\in\germsect{\theta}{\eta}}$.
\end{enumerate}
\end{defn}

We now give the definition of \emph{a (germ of a) sectorial diffeomorphism}.
\begin{defn}
\label{def: sectorial diff}

\begin{enumerate}
\item Given an element $\cal S\in\cal S_{\theta,\eta}$, we denote by $\fdiff[\cal S,\tx{Id}]$
the set of holomorphic fibered diffeomorphisms of the form 
\begin{eqnarray*}
\Phi:\cal S & \rightarrow & \Phi\left(\cal S\right)\\
\left(x,\mathbf{y}\right) & \mapsto & \left(x,\phi_{1}\left(x,\mathbf{y}\right),\phi_{2}\left(x,\mathbf{y}\right)\right)\,\,,
\end{eqnarray*}
such that ${\displaystyle \Phi\left(x,\mathbf{y}\right)-\tx{Id}\left(x,\mathbf{y}\right)=\tx O\left(\norm{x,\mathbf{y}}^{2}\right),\,\,\mbox{ as }\left(x,\mathbf{y}\right)\rightarrow\left(0,\mathbf{0}\right)\mbox{ in }\cal S.}$
\footnote{This condition implies in particular that $\Phi\left(\cal S\right)\in\germsect{\theta}{\eta}$.}
\item The set of \emph{germs of (fibered) sectorial diffeomorphisms in the
direction $\theta$ with opening $\eta$, tangent to the identity},
is the direct limit 
\[
\diffsect[\theta][\eta]:=\underrightarrow{\lim}\,\fdiff[\cal S,\tx{Id}]
\]
with respect to the directed system defined by $\acc{{{\fdiff[\cal S,\tx{Id}]}}:\cal S\in\cal S_{\theta,\eta}}$.
We equip $\diffsect[\theta][\eta]$ with a group structure as follows:
given two germs $\Phi,\Psi\in{\displaystyle \diffsect}$ we chose
corresponding representatives $\Phi_{0}\in\fdiff[\cal S,\tx{Id}]$
and $\Psi_{0}\in\fdiff[\cal T,\tx{Id}]$ with $\cal S,\cal T\in\cal S_{\theta,\eta}$
such that $\cal T\subset\Phi_{0}\left(\cal S\right)$ and let $\Psi\circ\Phi$
be the germ defined by $\Psi_{0}\circ\Phi_{0}$. \footnote{One can prove that this definition is independent of the choice of
the representatives}
\end{enumerate}
\end{defn}

We will also need the notion of \emph{asymptotic sectors.}
\begin{defn}
\label{def: asymptotitc sector}An\emph{ (open) asymptotic sector
of the origin in the direction $\theta$ and with opening $\eta$}
is an open set $S\subset\ww C$ such that 
\[
S\in\bigcap_{0\leq\eta'<\eta}\cal S_{\theta,\eta'}\,\,.
\]
We denote by $\cal{AS}_{\theta,\eta}$ the set of all such (open)
asymptotic sectors.
\end{defn}

\subsection{\label{subsec: Strong-Gevrey-1-power} Gevrey-1 power series and
1-summability}

~

\subsubsection{Gevrey-1 asymptotic expansions}

~

In this subsection we fix a formal power series which we write under
two forms: 
\[
{\displaystyle \hat{f}\left(x,\mathbf{y}\right)=\sum_{k\geq0}f_{k}\left(\mathbf{y}\right)x^{k}=\sum_{\left(j_{0},\mathbf{j}\right)\in\ww N^{n+1}}f_{j_{0},\mathbf{j}}x^{j_{0}}\mathbf{y^{j}}\in\form{x,\mathbf{y}}}\,\,,
\]
using the canonical identification $\form{x,\mathbf{y}}=\form x\left\llbracket \mathbf{y}\right\rrbracket =\form{\mathbf{y}}\left\llbracket x\right\rrbracket $.
We also fix a norm $\norm{\cdot}$ in $\ww C^{n+1}$.
\begin{defn}
\label{def:Gevrey-1_expansion}~

\begin{itemize}
\item A function $f$ analytic in a domain ${\displaystyle \sect r{\alpha}{\beta}\times\mathbf{D\left(0,r\right)}}$
admits $\hat{f}$ as asymptotic expansion \emph{in the sense of Gérard-Sibuya}
in this domain if for all closed sub-sector $S'\subset S\left(r,\alpha,\beta\right)$
and compact $\mathbf{K}\subset\mathbf{D\left(0,r\right)}$, for all
$N\in\ww N$, there exists a constant $C_{S',K,N}>0$ such that: 
\[
\abs{f\left(x,\mathbf{y}\right)-\sum_{j_{0}+j_{1}+\dots j_{n}\leq N}f_{j_{0},\mathbf{j}}x^{j_{0}}\mathbf{y^{j}}}\leq C_{S',K,N}\norm{\left(x,\mathbf{y}\right)}^{N+1}
\]
for all $\left(x,\mathbf{y}\right)\in S'\times K$.
\item A function $f$ analytic in a domain ${\displaystyle \sect r{\alpha}{\beta}\times\mathbf{D\left(0,r\right)}}$
admits $\hat{f}$ as\emph{ asymptotic expansion (with respect to $x$)
}if for all $k\in\ww N$, $f_{k}\left(\mathbf{y}\right)$ is analytic
in $\mathbf{D\left(0,r\right)}$, and if for all closed sub-sector
$S'\subset S\left(r,\alpha,\beta\right)$, compact subset $\mathbf{K}\subset\mathbf{D\left(0,r\right)}$
and $N\in\ww N$, there exists $A_{S',K,N}>0$ such that: 
\[
\abs{f\left(x,\mathbf{y}\right)-\sum_{k\geq0}^{N}f_{k}\left(\mathbf{y}\right)x^{k}}\leq A_{S',K,N}\abs x^{N+1}
\]
for all $\left(x,\mathbf{y}\right)\in S'\times K$.
\item An analytic function $f$ in a sectorial domain ${\displaystyle \sect r{\alpha}{\beta}\times\mathbf{D\left(0,r\right)}}$
admits $\hat{f}$ as \emph{Gevrey-1 asymptotic expansion} in this
domain, if for all $k\in\ww N$, $f_{k}\left(\mathbf{y}\right)$ is
analytic in $\mathbf{D\left(0,r\right)}$, and if for all closed sub-sector
$S'\subset\sect r{\alpha}{\beta}$, there exists $A,C>0$ such that:
\[
\abs{f\left(x,\mathbf{y}\right)-\sum_{k=0}^{N-1}f_{k}\left(\mathbf{y}\right)x^{k}}\leq AC^{N}\left(N!\right)\abs x^{N}
\]
for all $N\in\ww N$ and $\left(x,\mathbf{y}\right)\in S'\times\mathbf{D}\left(\mathbf{0},\mathbf{r}\right)$.
\end{itemize}
\end{defn}

\begin{rem}
~

\begin{enumerate}
\item If a function admits $\hat{f}$ as Gevrey-1 asymptotic expansion in
${\displaystyle \sect r{\alpha}{\beta}\times\mathbf{D\left(0,r\right)}}$,
then it also admits $\hat{f}$ as asymptotic expansion. 
\item If a function admits $\hat{f}$ as asymptotic expansion in ${\displaystyle \sect r{\alpha}{\beta}\times\mathbf{D\left(0,r\right)}}$,
then it also admits $\hat{f}$ as asymptotic expansion in the the
sense of Gérard-Sibuya. 
\item An asymptotic expansion (in any of the different senses described
above) is unique. 
\end{enumerate}
\end{rem}

As a consequence of Stirling formula, we have the following characterization
for functions admitting $0$ as Gevrey-1 asymptotic expansion.
\begin{prop}
\label{prop: dev asympt nul expo plat}The set of analytic functions
admitting $0$ as Gevrey-1 asymptotic expansion at the origin in a
sectorial domain ${\displaystyle \sect r{\alpha}{\beta}\times\mathbf{D\left(0,r\right)}}$
is exactly the set of of analytic functions $f$ in ${\displaystyle \sect r{\alpha}{\beta}\times\mathbf{D\left(0,r\right)}}$
such that for all closed sub-sector $S'\subset\sect r{\alpha}{\beta}$
and all compact $\mathbf{K}\subset\mathbf{D\left(0,r\right)}$, there
exist $A_{S',K},B_{S',K}>0$ such that:
\[
\abs{f\left(x,\mathbf{y}\right)}\leq A_{S',K}\exp\left(-\frac{B_{S',K}}{\abs x}\right)\,\,.
\]
We say that such a function is exponentially flat at the origin in
the corresponding domain.
\end{prop}

\subsubsection{Borel transform and Gevrey-1 power series}

~
\begin{defn}
~\label{def:Borel_transform}

\begin{itemize}
\item We define the Borel transform $\cal B\left(\hat{f}\right)$ of $\hat{f}$
as: 
\[
\cal B\left(\hat{f}\right)\left(t,\mathbf{y}\right):=\sum_{k\geq0}\frac{f_{k}\left(\mathbf{y}\right)}{k!}t^{k}\,\,.
\]
\item We say that $\hat{f}$ is Gevrey-1 if $\cal B\left(\hat{f}\right)$
is convergent in a neighborhood of the origin in $\ww C\times\ww C^{n}$.
Notice that in this case the $f_{k}\left(\mathbf{y}\right),\,k\geq0$,
are all analytic in a same polydisc $\mathbf{D}\left(\mathbf{0},\mathbf{r}\right)$,
of poly-radius ${\displaystyle \mathbf{r}=\left(r_{n}\dots,r_{n}\right)\in\left(\ww R_{>0}\right)^{n}}$,
so that $\cal B\left(\hat{f}\right)$ is analytic in $\mbox{D}\left(0,\rho\right)\times\mathbf{D}\left(\mathbf{0},\mathbf{r}\right)$,
for some $\rho>0$. Possibly by reducing $\rho,r_{1},\dots,r_{n}>0$,
we can assume that $\cal B\left(\hat{f}\right)$ is bounded in $\mbox{D}\left(0,\rho\right)\times\mathbf{D}\left(\mathbf{0},\mathbf{r}\right)$.
\end{itemize}
\end{defn}

\begin{rem}
~

\begin{enumerate}
\item If a sectorial function $f$ admits $\hat{f}$ for Gevrey-$1$ asymptotic
expansion as in Definition~\ref{def:Gevrey-1_expansion} then $\hat{f}$
is a Gevrey-1 formal power series.
\item The set of all Gevrey-$1$ formal power series is an algebra closed
under partial derivatives $\pp x,\pp{y_{1}},\dots,\pp{y_{n}}$.
\end{enumerate}
\end{rem}

\bigskip{}
\begin{rem}
\label{rem: deifinition bis transformee de Borel}For technical reasons
we will sometimes need to use another definition of the Borel transform,
that is:
\[
\widetilde{\cal B}\left(\hat{f}\right)\left(t,\mathbf{y}\right):=\sum_{k\geq0}f_{k+1}\left(\mathbf{y}\right)\frac{t^{k}}{k!}\,\,.
\]
The first definition we gave has the advantage of being ``directly''
invertible (\emph{via }the Laplace transform) for all 1-summable formal
power series (see next subsection), but behaves not so good with respect
to the product. On the contrary, the second definition will be ``directly''
invertible only for 1-summable formal power series with null constant
term (otherwise a translation is needed). However, the advantage of
the second Borel transform is that it changes a product into a convolution
product: 
\[
\widetilde{\cal B}\left(\hat{f}\hat{g}\right)=\left(\widetilde{{\cal B}}\left(\hat{f}\right)\ast\widetilde{{\cal B}}\left(\hat{g}\right)\right)\,\,,
\]
where the convolution product of two analytic functions $h_{1}h_{2}$
is defined by 
\[
\left(h_{1}\ast h_{2}\right)\left(t,\mathbf{y}\right):=\int_{0}^{t}h_{1}\left(s\right)h_{2}\left(s-t\right)\dd[s]\,\,.
\]
The property of being Gevrey-1 or not does not depend on the choice
of the definition we take for the Borel transform.
\end{rem}

\subsubsection{Directional 1-summability and Borel-Laplace summation}

~
\begin{defn}
\label{def: 1_summability}Given $\theta\in\ww R$ and $\delta>0$,
we define the infinite sector in the direction $\theta$ with opening
$\delta$ as the set 
\[
\cal A_{\theta,\delta}^{\infty}:=\acc{t\in\ww C^{*}\mid\abs{\arg\left(t\right)-\theta}<\frac{\delta}{2}}\,\,.
\]
We say that $\hat{f}$ is \emph{1-summable in the direction} $\theta\in\ww R$,
if the following three conditions holds:

\begin{itemize}
\item $\hat{f}$ is a Gevrey-1 formal power series;
\item $\cal B\left(\hat{f}\right)$ can be analytically continued to a domain
of the form $\cal A_{\theta,\delta}^{\infty}\times\mathbf{D}\left(\mathbf{0},\mathbf{r}\right)$;
\item there exists $\lambda>0,M>0$ such that:
\[
\forall\left(t,\mathbf{y}\right)\in\cal A_{\theta,\delta}^{\infty}\times\mathbf{D}\left(\mathbf{0},\mathbf{r}\right),\,\abs{\cal B\left(\hat{f}\right)\left(t,\mathbf{y}\right)}\leq M\exp\left(\lambda\abs t\right)\,\,\,\,.
\]
\end{itemize}
In this case we set ${\displaystyle \Delta_{\theta,\delta,\rho}:=\cal A_{\theta,\delta}^{\infty}\cup\mbox{D}\left(0,\rho\right)}$
and 
\[
\norm{\hat{f}}_{\lambda,\theta,\delta,\rho,\mathbf{r}}:=\underset{\left(t,\mathbf{y}\right)\in\Delta_{\theta,\delta,\rho}\times\mathbf{D}\left(\mathbf{0},\mathbf{r}\right)}{\sup}\abs{\cal B\left(\hat{f}\right)\left(t,\mathbf{y}\right)\exp\left(-\lambda\abs t\right)}\quad.
\]
If the domain is clear from the context we will simply write ${\displaystyle \norm{\hat{f}}_{\lambda}}$.
\end{defn}

\begin{rem}
\label{rem: norme bis borel}~

\begin{enumerate}
\item For fixed $\left(\lambda,\theta,\delta,\rho,\mathbf{r}\right)$ as
above, the set $\mathfrak{B}_{\lambda,\theta,\delta,\rho,\mathbf{r}}$
of formal power series $\hat{f}$ 1-summable in the direction $\theta$
and such that ${\displaystyle \norm{\hat{f}}_{\lambda,\theta,\delta,\rho,\mathbf{r}}<+\infty}$
is a Banach vector space for the norm $\norm{\cdot}_{\lambda,\theta,\delta,\rho,\mathbf{r}}$.
We simply write ${\displaystyle \left(\mathfrak{B}_{\lambda},\norm{\cdot}_{\lambda}\right)}$
when there is no ambiguity.
\item We will also need a norm well-adapted to the second Borel transform
$\widetilde{B}$ (\emph{cf. }Remark \ref{rem: deifinition bis transformee de Borel}),
that is:
\[
\norm{\hat{f}}_{\lambda,\theta,\delta,\rho,\mathbf{r}}^{\tx{bis}}:=\underset{\left(t,\mathbf{y}\right)\in\Delta_{\theta,\delta,\rho}\times\mathbf{D}\left(\mathbf{0},\mathbf{r}\right)}{\sup}\abs{\cal B\left(\hat{f}\right)\left(t,\mathbf{y}\right)\left(1+\lambda^{2}\abs t^{2}\right)\exp\left(-\lambda\abs t\right)}\,\,.
\]
We write then $\mathfrak{B}_{\lambda,\theta,\delta,\rho,\mathbf{r}}^{\tx{bis}}$
the set space of formal power series $\hat{f}$ which are 1-summable
in the direction $\theta$ and such that ${\displaystyle \norm{\hat{f}}_{\lambda,\theta,\delta,\rho,\mathbf{r}}^{\tx{bis}}<+\infty}$.
\item If $\lambda'\geq\lambda$, then ${\displaystyle \mathfrak{B}_{\lambda,\theta,\delta,\rho,\mathbf{r}}}\subset\mathfrak{B}_{\lambda',\theta,\delta,\rho,\mathbf{r}}$
and ${\displaystyle \mathfrak{B}_{\lambda,\theta,\delta,\rho,\mathbf{r}}^{\tx{bis}}}\subset\mathfrak{B}_{\lambda',\theta,\delta,\rho,\mathbf{r}}^{\tx{bis}}$.
\end{enumerate}
\end{rem}

\begin{prop}[{\foreignlanguage{french}{\cite[Proposition 4.]{DeMaesschalck}}}]
\label{prop: norme d'algebre} If $\hat{f},\hat{g}\in\mathfrak{B}_{\lambda,\theta,\delta,\rho,\mathbf{r}}^{\tx{bis}}$,
then $\hat{f}\hat{g}\in\mathfrak{B}_{\lambda,\theta,\delta,\rho,\mathbf{r}}^{\tx{bis}}$
and:
\[
\norm{\hat{f}\hat{g}}_{\lambda,\theta,\delta,\rho,\mathbf{r}}^{\tx{bis}}\leq\frac{4\pi}{\lambda}\norm{\hat{f}}_{\lambda,\theta,\delta,\rho,\mathbf{r}}^{\tx{bis}}\norm{\hat{g}}_{\lambda,\theta,\delta,\rho,\mathbf{r}}^{\tx{bis}}\,\,.
\]
\end{prop}

\begin{rem}
If $\lambda\geq4\pi$, then $\norm{\cdot}_{\lambda,\theta,\delta,\rho,\mathbf{r}}^{\tx{bis}}$
is a sub-multiplicative norm, \emph{i.e. 
\[
\norm{\hat{f}\hat{g}}_{\lambda,\theta,\delta,\rho,\mathbf{r}}^{\tx{bis}}\leq\norm{\hat{f}}_{\lambda,\theta,\delta,\rho,\mathbf{r}}^{\tx{bis}}\norm{\hat{g}}_{\lambda,\theta,\delta,\rho,\mathbf{r}}^{\tx{bis}}\,\,.
\]
}
\end{rem}

\begin{defn}
\label{def:Laplace}Let $g$ be analytic in a domain and $\mbox{\ensuremath{\cal A_{\theta,\delta}^{\infty}\times\mathbf{D}\left(\mathbf{0},\mathbf{r}\right)} and let }\lambda>0,M>0$
such that
\[
\forall\left(t,\mathbf{y}\right)\in\cal A_{\theta,\delta}^{\infty}\times\mathbf{D}\left(\mathbf{0},\mathbf{r}\right),\,\abs{g\left(t,\mathbf{y}\right)}\leq M\exp\left(\lambda\abs t\right)\,\,\,\,.
\]
We define the \emph{Laplace transform} of $g$ in the direction $\theta$
as: 
\[
\cal L_{\theta}\left(g\right)\left(x,\mathbf{y}\right):=\int_{e^{i\theta}\ww R_{>0}}g\left(t,\mathbf{y}\right)\exp\left(-\frac{t}{x}\right)\frac{\mbox{d}t}{x}\,,
\]
which is absolutely convergent for all $x\in\ww C$ with $\Re\left(\frac{e^{i\theta}}{x}\right)>\lambda$
and $\mathbf{y\in\mathbf{D}\left(\mathbf{0},\mathbf{r}\right)}$,
and analytic with respect to $\left(x,\mathbf{y}\right)$ in this
domain.
\end{defn}

\begin{rem}
As for the Borel transform, there also exists another definition of
the Laplace transform, that is:
\[
\widetilde{\cal L}_{\theta}\left(g\right)\left(x,\mathbf{y}\right):=\int_{e^{i\theta}\ww R_{>0}}g\left(t,\mathbf{y}\right)\exp\left(-\frac{t}{x}\right)\dd[t]\,\,.
\]
\end{rem}

\begin{prop}
\label{prop: Borel_Laplace_summation}A formal power series $\hat{f}\in\form{x,\mathbf{y}}$
is 1-summable in the direction $\theta$ if and only if there exists
a germ of a sectorial holomorphic function $f_{\theta}\in\cal O\left(\cal S_{\theta,\pi}\right)$
which admits $\hat{f}$ as Gevrey-1 asymptotic expansion in some $\cal S\in\cal S_{\theta,\pi}$.
Moreover, $f_{\theta}$ is unique $\big($as a germ in ${\displaystyle \cal O\left(\cal S_{\theta,\pi}\right)}$$\big)$
and in particular 
\[
f_{\theta}=\cal L_{\theta}\left(\cal B\left(\hat{f}\right)\right)\,\,.
\]
The function (germ) $f_{\theta}$ is called the \emph{1-sum of $\hat{f}$
in the direction $\theta$}.
\end{prop}

\begin{rem}
With the second definitions of Borel and Laplace transforms given
above, we have a similar result for formal power series of the form
${\displaystyle \hat{f}\left(x,\mathbf{y}\right)=\sum_{k}f_{k}\left(\mathbf{y}\right)x^{k}}$
with: 
\[
f_{\theta}=\widetilde{\cal L}_{\theta}\left(\widetilde{\cal B}\right)\left(\hat{f}\right)+\hat{f}\left(0,\mathbf{y}\right)\,\,.
\]
\end{rem}

We recall the following well-known result.
\begin{lem}
\label{lem:diff_alg_and_summability}The set $\Sigma_{\theta}\subset\form{x,\mathbf{y}}$
of $1$-summable power series in the direction $\theta$ is an algebra
closed under partial derivatives. Moreover the map
\begin{eqnarray*}
\Sigma_{\theta} & \longrightarrow & \cal O\left(\cal S_{\theta,\pi}\right)\\
\hat{f} & \longmapsto & f_{\theta}
\end{eqnarray*}
is an injective morphism of differential algebras.
\end{lem}

\begin{defn}
A formal power series $\hat{f}\in\form{x,\mathbf{y}}$ is called \emph{1-summable}
if it is 1-summable in all but a finite number of directions, called
\emph{Stokes directions}. In this case, if ${\displaystyle \theta_{1},\dots,\theta_{k}\in\nicefrac{\ww R}{2\pi\ww Z}}$
are the possible Stokes directions, we say that $\hat{f}$ is 1-summable
except for\textbf{ }$\theta_{1},\dots,\theta_{k}$.

More generally, we say that an $m-$uple $\left(f_{1},\dots,f_{m}\right)\in\form{x,\mathbf{y}}^{m}$
is Gevrey-1 (\emph{resp. }1-summable in direction $\theta$) if this
property holds for each component $f_{j},j=1,\dots,m$. Similarly,
a formal vector field (or\emph{ }diffeomorphism) is said to be Gevrey-1
(\emph{resp. }1-summable in direction $\theta$) if each one of its
components has this property.
\end{defn}

The following classical result deals with composition of 1-summable
power series (an elegant way to prove it is to use an important theorem
of Ramis-Sibuya).
\begin{prop}
\label{prop: compositon summable}Let $\hat{\Phi}\left(x,\mathbf{y}\right)\in\form{x,\mathbf{y}}$
be 1-summable in directions $\theta$ and $\theta-\pi$, and let $\Phi_{+}\left(x,\mathbf{y}\right)$
and $\Phi_{-}\left(x,\mathbf{y}\right)$ be its 1-sums directions
$\theta$ and $\theta-\pi$ respectively. Let also ${\displaystyle \hat{f}_{1}\left(x,\mathbf{z}\right),\dots,\hat{f}_{n}\left(x,\mathbf{z}\right)}$
be 1-summable in directions $\theta$, $\theta-\pi$, and $f_{1,+},\dots,f_{n,+}$,
and $f_{1,-},\dots,f_{n,-}$ be their 1-sums in directions $\theta$
and $\theta-\pi$ respectively. Assume that 
\begin{equation}
\hat{f}_{j}\left(0,\mathbf{0}\right)=0\mbox{, for all }j=1,\dots,n\,\,.\label{eq: condition composition sommable}
\end{equation}
 Then 
\[
\hat{\Psi}\left(x,\mathbf{z}\right):=\hat{\Phi}\left(x,\hat{f}_{1}\left(x,\mathbf{z}\right),\dots,\hat{f}_{n}\left(x,\mathbf{z}\right)\right)
\]
 is 1-summable in directions $\theta,\theta-\pi$, and its 1-sum in
the corresponding direction is 
\[
\Psi_{\pm}\left(x,\mathbf{z}\right):=\Phi_{\pm}\left(x,f_{1,\pm}\left(x,\mathbf{z}\right),\dots,f_{n,\pm}\left(x,\mathbf{z}\right)\right)\,\,\,,
\]
which is a germ of a sectorial holomorphic function in this direction.
\end{prop}

Consider $\hat{Y}$ a formal singular vector field at the origin and
a formal fibered diffeomorphism $\hat{\varphi}:\left(x,\mathbf{y}\right)\mapsto\left(x,\hat{\phi}\left(x,\mathbf{y}\right)\right)$.
Assume that both $\hat{Y}$ and $\hat{\varphi}$ are 1-summable in
directions $\theta$ and $\theta-\pi$, for some $\theta\in\ww R$,
and denote by $Y_{+},Y_{-}$ (\emph{resp. $\varphi_{+},\varphi_{-}$})
their 1-sums in directions $\theta$ and $\theta-\pi$ respectively.
As a consequence of Proposition \ref{prop: compositon summable} and
Lemma \ref{lem:diff_alg_and_summability}, we can state the following:
\begin{cor}
\label{cor: summability push-forward}Under the assumptions above,
$\hat{\varphi}_{*}\left(\hat{Y}\right)$ is 1-summable in both directions
$\theta$ and $\theta-\pi$, and its 1-sums in these directions are
$\varphi_{+}\left(Y_{+}\right)$ and $\varphi_{-}\left(Y_{-}\right)$
respectively.
\end{cor}

\subsubsection{An important result by Martinet and Ramis}

~

We are going to make an essential use of an isomorphism theorem proved
in \cite{MR82}. This result is of paramount importance in the present
paper since it will be a key tool in the proofs of both Theorems~\ref{Th: Th drsn}
and \ref{thm: espace de module} (see section~\ref{sec: analytic classification}).

Let us consider two open asymptotic sectors $\cal S$ and $\cal S'$
at the origin in directions $\theta$ and $\theta-\pi$ respectively,
both of opening $\pi$: 
\begin{eqnarray*}
S & \in & \cal{AS}_{\theta,\pi}\\
S' & \in & \cal{AS}_{\theta-\pi,\pi}
\end{eqnarray*}
(see Definition \ref{def: asymptotitc sector}). In this particular
setting, the cited theorem can be stated as follows. 
\begin{thm}
\label{th: martinet ramis}\cite[Théorème 5.2.1]{MR82} Consider a
pair of germs of sectorial diffeomorphisms 
\[
\left(\varphi,\varphi'\right)\in\diffsect[\theta][0]\times\diffsect[\theta-\pi][0]
\]
such that $\varphi$ and $\varphi'$ extend analytically and admit
the identity as Gevrey-1 asymptotic expansion in $S\times\left(\ww C^{2},0\right)$
and $S'\times\left(\ww C^{2},0\right)$ respectively. Then, there
exists a pair $\left(\phi_{+},\phi_{-}\right)$ of germs of sectorial
fibered diffeomorphisms 
\[
\left(\phi_{+},\phi_{-}\right)\in\diffsect[\theta+\frac{\pi}{2}][\eta]\times\diffsect[\theta-\frac{\pi}{2}][\eta]
\]
with $\eta\in\left]\pi,2\pi\right[$, which extend analytically in
$S_{+}\times\left(\ww C^{2},0\right)$ and $S_{-}\times\left(\ww C^{2},0\right)$
respectively, for some ${\displaystyle S_{+}\in\cal{AS}_{\theta+\frac{\pi}{2},2\pi}}$
and ${\displaystyle S_{-}\in\cal{AS}_{\theta-\frac{\pi}{2},2\pi}}$
, such that:
\[
\begin{cases}
\phi_{+}\circ\left(\phi_{-}\right)_{\mid S\times\left(\ww C^{2},0\right)}^{-1}=\varphi\\
\phi_{+}\circ\left(\phi_{-}\right)_{\mid S'\times\left(\ww C^{2},0\right)}^{-1}=\varphi' & \,\,.
\end{cases}
\]
There also exists a formal diffeomorphism $\hat{\phi}$ which is tangent
to the identity, such that $\phi_{+}$ and $\phi_{-}$ both admit
$\hat{\phi}$ as Gevrey-1 asymptotic expansion in $S_{+}\times\left(\ww C^{2},0\right)$
and $S_{-}\times\left(\ww C^{2},0\right)$ respectively.
\end{thm}

In particular, in the theorem above $\hat{\phi}$ is $1-$summable
in every direction except $\theta$ and $\theta-\pi$, and its 1-sums
in directions $\theta+\frac{\pi}{2}$ and $\theta-\frac{\pi}{2}$
respectively are $\phi_{+}$ and $\phi_{-}$. For future use, we are
going to prove a \emph{``transversally symplectic'' }version of
the above theorem.
\begin{cor}
\label{cor: MR symplectic}With the assumptions and notations of Theorem
\ref{th: martinet ramis}, if $\varphi$ and $\varphi'$ both are
transversally symplectic (see Definition \ref{def: intro}), then
there exists a germ of an analytic fibered diffeomorphism $\psi\in\fdiff[\ww C^{3},0,\tx{Id}]$
(tangent to the identity), such that
\[
\sigma_{+}:=\phi_{+}\circ\psi\mbox{ and }\sigma_{-}:=\phi_{-}\circ\psi
\]
 both are transversally symplectic. Moreover we also have:
\[
\begin{cases}
\sigma_{+}\circ\left(\sigma_{-}\right)_{\mid S\times\left(\ww C^{2},0\right)}^{-1}=\varphi\\
\sigma_{+}\circ\left(\sigma_{-}\right)_{\mid S'\times\left(\ww C^{2},0\right)}^{-1}=\varphi' & \,\,.
\end{cases}
\]
\end{cor}

\begin{proof}
We recall that for any germ $\varphi$ of a sectorial fibered diffeomorphism
which is tangent to the identity, $\varphi$ is transversally symplectic
if and only if $\det\left(\tx D\varphi\right)=1$. 

First of all, let us show that 
\[
\det\left(\tx D\phi_{+}\right)=\det\left(\tx D\phi_{-}\right)\mbox{ in }\left(S_{+}\cap S_{-}\right)\times\left(\ww C^{2},0\right)\,\,\,.
\]
Since $\phi_{+}$ and $\phi_{-}$ both are sectorial fibered diffeomorphism
which are tangent to the identity and transversally symplectic, then
\[
\det\left(\phi_{+}\circ\left(\phi_{-}\right)_{\mid\left(S_{+}\cap S_{-}\right)\times\left(\ww C^{2},0\right)}^{-1}\right)=1\,\,.
\]
The \emph{chain rule} implies immediately that 
\[
\det\left(\tx D\phi_{+}\right)=\det\left(\tx D\phi_{-}\right)\mbox{ in }\left(S_{+}\cap S_{-}\right)\times\left(\ww C^{2},0\right)\,\,\,.
\]
Thus, this equality allows us to define (thanks to the Riemann's Theorem
of removable singularities) a germ of analytic function $f\in\cal O\left(\ww C^{3},0\right)$.
Notice that $f\left(0,0,0\right)=1$ because $\phi_{+}$ and $\phi_{-}$
are tangent to the identity. Now, let us look for an element $\psi\in\fdiff[\ww C^{3},0,\tx{Id}]$
of the form 
\begin{equation}
\psi:\left(x,y_{1},y_{2}\right)\mapsto\left(x,\psi_{1}\left(x,\mathbf{y}\right),y_{2}\right)\label{eq: rectififcation symplectic}
\end{equation}
 such that 
\[
\sigma_{+}:=\phi_{+}\circ\psi\mbox{ and }\sigma_{-}:=\phi_{-}\circ\psi
\]
 both be transversally symplectic. An easy computation gives: 
\[
\det\left(\sigma_{\pm}\right)=\left(\det\left(\tx D\phi_{\pm}\right)\circ\psi\right)\det\left(\tx D\psi\right)=1
\]
\emph{i.e.}
\[
\left(f\circ\psi\right)\det\left(\mbox{D}\psi\right)=1\,\,\,.
\]
According to (\ref{eq: rectififcation symplectic}), we must have:
\begin{equation}
\left(f\circ\psi\right)\ppp{\psi_{1}}{y_{1}}=1\,\,.\label{eq: equa diff rectification symplectic}
\end{equation}
Let us define 
\[
F\left(x,y_{1},y_{2}\right):=\int_{0}^{y_{1}}f\left(x,z,y_{2}\right)\tx dz\,\,,
\]
so that (\ref{eq: equa diff rectification symplectic}) can be integrated
as 
\[
F\circ\psi=y_{1}+h\left(x,y_{2}\right)\,\,\,,
\]
for some $h\in\germ{x,y_{2}}$. Notice that 
\[
\ppp F{y_{1}}\left(0,0,0\right)=1
\]
since $f\left(0,0,0\right)=1$. Let us chose $h=0$. Then, we have
to solve 
\[
F\circ\psi=y_{1}\,\,\,,
\]
with unknown $\psi\in\fdiff[\ww C^{3},0,\tx{Id}]$ as in (\ref{eq: equa diff rectification symplectic}).
If we define
\[
\Phi:\left(x,\mathbf{y}\right)\mapsto\left(x,F\left(x,\mathbf{y}\right),y_{2}\right)\,\,,
\]
the latter problem is equivalent to find $\psi$ as above such that:
\[
\Phi\circ\psi=\tx{Id}\,\,\,.
\]
Since $\tx D\Phi_{0}=\tx{Id}$, the inverse function theorem gives
us the existence of the germ $\psi=\Phi^{-1}\in\fdiff[\ww C^{3},0,\tx{Id}]$
.
\end{proof}

\subsection{\label{subsec:Weak-Gevrey-1-power}Weak Gevrey-1 power series and
weak 1-summability}

~

We present here a weaker notion of 1-summability that we will also
need. Any function $f\left(x,\mathbf{y}\right)$ analytic in a domain
$\cal U\times\mathbf{D}\left(\mathbf{0},\mathbf{r}\right)$, where
$\cal U\subset\ww C$ is open, and bounded in any domain $\cal U\times\overline{\mathbf{D}}\left(\mathbf{0},\mathbf{r'}\right)$
with $r'_{1}<r_{1},\dots,r'_{n}<r_{n}$, can be written 
\begin{equation}
f\left(x,\mathbf{y}\right)=\sum_{\mathbf{j}\in\ww N^{n}}F_{\mathbf{j}}\left(x\right)\mathbf{y}^{\mathbf{j}}\qquad,\label{eq: developpement selon y-1}
\end{equation}
where for all $\mathbf{j}\in\ww N^{n}$, $F_{\mathbf{j}}$ is analytic
and bounded on $\cal U$, and defined \emph{via }the Cauchy formula:
\[
F_{\mathbf{j}}\left(x\right)=\frac{1}{\left(2i\pi\right)^{n}}\int_{\abs{z_{1}}=r'_{1}}\dots\int_{\abs{z_{n}}=r'_{n}}\frac{f\left(x,\mathbf{z}\right)}{\left(z_{1}\right)^{j_{1}+1}\dots\left(z_{n}\right)^{j_{n}+1}}\mbox{d}z_{n}\dots\mbox{d}z_{1}\qquad.
\]
Notice that the convergence of the function series above is uniform
in every compact with respect to $x$ and $\mathbf{y}$.

In the same way, any formal power series $\hat{f}\left(x,\mathbf{y}\right)\in\form{x,\mathbf{y}}$
can be written as 
\[
{\displaystyle \hat{f}\left(x,\mathbf{y}\right)=\sum_{\mathbf{j}\in\ww N^{n}}\hat{F}_{\mathbf{j}}}\left(x\right)\mathbf{y}^{\mathbf{j}}\,\,.
\]
\begin{defn}
~

\begin{itemize}
\item The formal power series $\hat{f}$ is said to be \textbf{weakly Gevrey-1}
if for all $\mathbf{j}\in\ww N^{n}$, $\hat{F}_{\mathbf{j}}\left(x\right)\in\form x$
is a Gevrey-1 formal power series.
\item A function 
\[
{\displaystyle f\left(x,\mathbf{y}\right)=\sum_{\mathbf{j}\in\ww N^{n}}F_{\mathbf{j}}\left(x\right)\mathbf{y^{j}}}
\]
analytic and bounded in a domain $\sect r{\alpha}{\beta}\times\mathbf{D}\left(\mathbf{0},\mathbf{r}\right)$,
admits $\hat{f}$ as \textbf{weak Gevrey-1 asymptotic expansion} in
$x\in\sect r{\alpha}{\beta}$, if for all $\mathbf{j}\in\ww N^{n}$,
$F_{\mathbf{j}}$ admits $\hat{F}_{\mathbf{j}}$ as Gevrey-1 asymptotic
expansion in $\sect r{\alpha}{\beta}$.
\item The formal power series $\hat{f}$ is said to be \textbf{weakly 1-summable
in the direction $\theta\in\ww R$}, if the following conditions hold:

\begin{itemize}
\item for all $\mathbf{j}\in\ww N^{n}$, $\hat{F}_{\mathbf{j}}\left(x\right)\in\form x$
is 1-summable in the direction $\theta$, whose 1-sum in the direction
$\theta$ is denoted by $F_{\mathbf{j},\theta}$;
\item the series ${\displaystyle f_{\theta}\left(x,\mathbf{y}\right):=\sum_{\mathbf{j}\in\ww N^{n}}F_{\mathbf{j},\theta}\left(x\right)\mathbf{y^{j}}}$
defines a germ of a sectorial holomorphic function in a sectorial
neighborhood attached to the origin in the direction $\theta$ with
opening greater than $\pi$.
\end{itemize}
In this case, $f_{\theta}\left(x,\mathbf{y}\right)$ is called \textbf{the
weak 1-sum of $\hat{f}$ in the direction $\theta$}.
\end{itemize}
\end{defn}

As a consequence to the classical theory of summability and Gevrey
asymptotic expansions, we immediately have the following:
\begin{lem}
\label{lem: proprietes faibles}

\begin{enumerate}
\item The weak Gevrey-1 asymptotic expansion of an analytic function in
a domain $\sect r{\alpha}{\beta}\times\mathbf{D}\left(\mathbf{0},\mathbf{r}\right)$
is unique.
\item The weak 1-sum of a weak 1-summable formal power series in the direction
$\theta$, is unique as a germ in ${\displaystyle \cal O\left(\cal S_{\theta,\pi}\right)}$.
\item \label{enu: derivation faible}The set $\Sigma_{\theta}^{\left(\tx{weak}\right)}\subset\form{x,\mathbf{y}}$
of weakly $1$-summable power series in the direction $\theta$ is
an algebra closed under partial derivatives. Moreover the map
\begin{eqnarray*}
\Sigma_{\theta}^{\left(\tx{weak}\right)} & \longrightarrow & \cal O\left(\cal S_{\theta,\pi}\right)\\
\hat{f} & \longmapsto & f_{\theta}
\end{eqnarray*}
is an injective morphism of differential algebras.
\end{enumerate}
\end{lem}

The following proposition is an analogue of Proposition \ref{prop: compositon summable}
for weak 1-summable formal power series, with the a stronger condition
instead of (\ref{eq: condition composition sommable}).
\begin{prop}
\label{prop: composition faible}Let 
\[
{\displaystyle \hat{\Phi}\left(x,\mathbf{y}\right)=\sum_{\mathbf{j}\in\ww N^{n}}\hat{\Phi}_{\mathbf{j}}\left(x\right)\mathbf{y^{j}}\in\form{x,\mathbf{y}}}
\]
and 
\[
{\displaystyle \hat{f}^{\left(k\right)}\left(x,\mathbf{z}\right)=\sum_{\mathbf{j}\in\ww N^{n}}\hat{F}_{\mathbf{j}}^{\left(k\right)}\left(x\right)\mathbf{z^{j}}\in\form{x,\mathbf{z}}}\,\,,
\]
for $k=1,\dots,n$, be n+1 formal power series which are weakly 1-summable
in directions $\theta$ and $\theta-\pi$. Let us denote by $\Phi_{+},f_{+}^{\left(1\right)},\dots,f_{+}^{\left(n\right)}$
$\big($\emph{resp. }$\Phi_{-},f_{-}^{\left(1\right)},\dots,f_{-}^{\left(n\right)}$$\big)$
their respective weak 1-sums in the direction $\theta$ (\emph{resp.}
$\theta-\pi$). Assume that $\hat{F}_{\mathbf{0}}^{\left(k\right)}=0$
for all $k=1,\dots,n$. Then, 
\[
\hat{\Psi}\left(x,\mathbf{z}\right):=\hat{\Phi}\left(x,\hat{f}^{\left(1\right)}\left(x,\mathbf{z}\right),\dots,\hat{f}^{\left(n\right)}\left(x,\mathbf{z}\right)\right)
\]
 is weakly 1-summable directions $\theta$ and $\theta-\pi$, and
its 1-sum in the corresponding direction is 
\[
\Psi_{\pm}\left(x,\mathbf{z}\right)=\Phi_{\pm}\left(x,f_{\pm}^{\left(1\right)}\left(x,\mathbf{z}\right),\dots,f_{\pm}^{\left(n\right)}\left(x,\mathbf{z}\right)\right)\,\,\,,
\]
which is a germ of a sectorial holomorphic function in this direction
with opening $\pi$.
\end{prop}

\begin{proof}
First of all, 
\[
\hat{\Psi}\left(x,\mathbf{z}\right):=\hat{\Phi}\left(x,\hat{f}^{\left(1\right)}\left(x,\mathbf{z}\right),\dots,\hat{f}^{\left(n\right)}\left(x,\mathbf{z}\right)\right)
\]
 is well defined as formal power series since for all $k=1,\dots,n$,
$\hat{F}_{\mathbf{0}}^{\left(k\right)}=0$. It is also clear that
\[
\Psi_{\pm}\left(x,\mathbf{z}\right):={\displaystyle \Phi_{\pm}\left(x,f_{\pm}^{\left(1\right)}\left(x,\mathbf{z}\right),\dots,f_{\pm}^{\left(n\right)}\left(x,\mathbf{z}\right)\right)}
\]
 is an analytic in a domain $\cal S_{+}\in\cal S_{\theta,\pi}$ (\emph{resp.
}$\cal S_{-}\in\cal S_{\theta-\pi,\pi}$), because $f_{\pm}^{\left(k\right)}\left(x,\mathbf{0}\right)=0$
for all $k=1,\dots,n$. Finally, we check that $\Psi_{\pm}$ admits
$\hat{\Psi}$ as weak Gevrey-1 asymptotic expansion in $\cal S_{\pm}$.
Indeed:
\begin{eqnarray*}
\Psi_{\pm}\left(x,\mathbf{z}\right) & = & \Phi_{\pm}\left(x,f_{\pm}^{\left(1\right)}\left(x,\mathbf{z}\right),\dots,f_{\pm}^{\left(n\right)}\left(x,\mathbf{z}\right)\right)\\
 & = & \sum_{\mathbf{j}\in\ww N^{n}}\left(\Phi_{\mathbf{j}}\right)_{\pm}\left(x\right)\left(f_{\pm}^{\left(1\right)}\left(x,\mathbf{z}\right)\right)^{j_{1}}\dots\left(f_{\pm}^{\left(1\right)}\left(x,\mathbf{z}\right)\right)^{j_{n}}\\
 & = & \sum_{\mathbf{j}\in\ww N^{n}}\left(\Phi_{\mathbf{j}}\right)_{\pm}\left(x\right)\left(\sum_{\abs{\mathbf{l}}\geq1}\left(F_{\mathbf{l}}^{\left(1\right)}\right)_{\pm}\left(x\right)\mathbf{z^{l}}\right)^{j_{1}}\dots\\
 &  & \,\,\,\,\,\,\dots\left(\sum_{\abs{\mathbf{l}}\geq1}\left(F_{\mathbf{l}}^{\left(n\right)}\right)_{\pm}\left(x\right)\mathbf{z^{l}}\right)^{j_{n}}\\
 & = & \sum_{\mathbf{j}\in\ww N^{n}}\left(\Psi_{\mathbf{j}}\right)_{\pm}\left(x\right)\mathbf{y^{j}}
\end{eqnarray*}
where for all $\mathbf{j}\in\ww N^{n}$, $\left(\Psi_{\mathbf{j}}\right)_{\pm}\left(x\right)$
is obtained as a finite number of additions and products of the $\left(\Phi_{\mathbf{k}}\right)_{\pm}$,$\left(F_{\mathbf{k}}^{\left(1\right)}\right)_{\pm}$,$\dots$,$\left(F_{\mathbf{k}}^{\left(n\right)}\right)_{\pm}$,
$\abs{\mathbf{k}}\leq\abs{\mathbf{l}}$. The same computation is valid
for the associated formal power series, and allows us to define the
$\hat{\Psi}_{\mathbf{j}}\left(x\right)$, for all $\mathbf{j}\in\ww N^{n}$.
Then, each $\left(\Psi_{\mathbf{j}}\right)_{\pm}$ has $\hat{\Psi}_{\mathbf{j}}$
as Gevrey-1 asymptotic expansion in $\cal S_{\pm}$.
\end{proof}
As a consequence of Proposition \ref{prop: composition faible} and
Lemma \ref{lem: proprietes faibles}, we have an analogue version
of Corollary (\ref{cor: summability push-forward}) in the weak 1-summable
case. Consider $\hat{Y}$ a formal singular vector field at the origin
and a formal fibered diffeomorphism $\hat{\varphi}:\left(x,\mathbf{y}\right)\mapsto\left(x,\hat{\phi}\left(x,\mathbf{y}\right)\right)$
such that $\hat{\phi}\left(x,\mathbf{0}\right)=\mathbf{0}$. Assume
that both $\hat{Y}$ and $\hat{\varphi}$ are weakly 1-summable in
directions $\theta$ and $\theta-\pi$, for some $\theta\in\ww R$,
and denote by $Y_{+},Y_{-}$ (\emph{resp. $\varphi_{+},\varphi_{-}$})
their weak 1-sums in directions $\theta$ and $\theta-\pi$ respectively.
\begin{cor}
\label{cor: summability push-forward faible} Under the assumptions
above, $\hat{\varphi}_{*}\left(\hat{Y}\right)$ is weakly 1-summable
in both directions $\theta$ and $\theta-\pi$, and its 1-sums in
these directions are $\varphi_{+}\left(Y_{+}\right)$ and $\varphi_{-}\left(Y_{-}\right)$
respectively.
\end{cor}

\subsection{Weak 1-summability \emph{versus} 1-summability}

~

As in the previous subsection, let a formal power series $\hat{f}\left(x,\mathbf{y}\right)\in\form{x,\mathbf{y}}$
which is written as 
\[
{\displaystyle \hat{f}\left(x,\mathbf{y}\right)=\sum_{\mathbf{j}\in\ww N^{n}}\hat{F}_{\mathbf{j}}}\left(x\right)\mathbf{y}^{\mathbf{j}}\,\,,
\]
so that its Borel transform is 
\[
{\displaystyle \cal B\left(\hat{f}\right)\left(t,\mathbf{y}\right)=\sum_{\mathbf{j}\in\ww N^{n}}\cal B\left(\hat{F}_{\mathbf{j}}\right)\left(t\right)\mathbf{y^{j}}\,\,.}
\]
The next lemma is immediate.
\begin{lem}
\label{lem: defition equivalent 1-summable}~

\begin{enumerate}
\item The power series $\cal B\left(\hat{f}\right)\left(t,\mathbf{y}\right)$
is convergent in a neighborhood of the origin in $\ww C^{n+1}$ if
and only if the $\cal B\left(\hat{F}_{\mathbf{j}}\right),\,\mathbf{j}\in\ww N^{n}$,
are all analytic and bounded in a same disc $\mbox{D}\left(0,\rho\right)$,
$\rho>0$, and if there exists $B,L>0$ such that for all $\mathbf{j}\in\ww N^{n}$,
${\displaystyle \underset{t\in\mbox{D}\left(0,\rho\right)}{\sup}\abs{\cal B\left(\hat{F}_{\mathbf{j}}\right)\left(t\right)}\leq L.B^{\abs{\mathbf{j}}}}$. 
\item If $1.$ is satisfied, then $\cal B\left(\hat{f}\right)$ can be analytically
continued to a domain $\cal A_{\theta,\delta}^{\infty}\times\mathbf{D}\left(\mathbf{0},\mathbf{r}\right)$
if and only if for all $\mathbf{j}\in\ww N^{n}$, $\cal B\left(\hat{F}_{\mathbf{j}}\right)$
can be analytically continued to $\cal A_{\theta,\delta}^{\infty}$
and if for all compact $K\subset\cal A_{\theta,\delta}^{\infty}$,
there exists $B,L>0$ such that for all $\mathbf{j}\in\ww N^{n}$,
${\displaystyle \underset{t\in K}{\sup}\abs{\cal B\left(\hat{F}_{\mathbf{j}}\right)\left(t\right)}\leq L.B^{\abs{\mathbf{j}}}}$.
\item If and $1.$ and $2.$ are satisfied, then there exists $\lambda,M>0$
such that:
\[
\forall\left(t,\mathbf{y}\right)\in\cal A_{\theta,\delta}^{\infty}\times\mathbf{D}\left(\mathbf{0},\mathbf{r}\right),\,\abs{\cal B\left(\hat{f}\right)\left(t,\mathbf{y}\right)}\leq M.\exp\left(\lambda\abs t\right)
\]
if and only if there exists $\lambda,L,B>0$ such that for all $\mathbf{j}\in\ww N^{n}$,
\[
\forall t\in\cal A_{\theta,\delta}^{\infty},\,\abs{\cal B\left(\hat{F}_{\mathbf{j}}\right)\left(t\right)}\leq L.B^{\abs{\mathbf{j}}}\exp\left(\lambda\abs t\right)\qquad.
\]
\end{enumerate}
\end{lem}

\begin{rem}
~

\begin{enumerate}
\item Condition $1.$ above states that the formal power series $\hat{f}$
is Gevrey-1.
\item As usual, there exists an equivalent lemma for the second definitions
of the Borel transform (see Remark \ref{rem: deifinition bis transformee de Borel}).
\end{enumerate}
\end{rem}

The following corollary gives a link between 1-summability and weak
1-summability.
\begin{cor}
\label{cor: faible et forte sommabilite}Let 
\begin{eqnarray*}
{\displaystyle \hat{f}\left(x,\mathbf{y}\right)} & = & \sum_{\mathbf{j}\in\ww N^{n}}\hat{F}_{\mathbf{j}}\left(x\right)\mathbf{y^{j}}\in\form{x,\mathbf{y}}
\end{eqnarray*}
 be a formal power series. Then, $\hat{f}$ is 1-summable in the direction
$\theta\in\ww R$, of 1-sum ${\displaystyle f\in\cal O\left(\cal S_{\theta,\pi}\right)}$,
if and only if the following two conditions hold:

\begin{itemize}
\item $\hat{f}$ is weakly 1-summable in the direction $\theta$;
\item there exists $\lambda,\delta,\rho$ such that for all $\mathbf{j}\in\ww N^{n}$,
${\displaystyle \norm{\hat{F}_{\mathbf{j}}}_{\lambda,\theta,\delta,\rho}<\infty}$
and the power series ${\displaystyle \sum_{\mathbf{j}\in\ww N^{n}}\norm{\hat{F}_{\mathbf{j}}}_{\lambda,\theta,\delta,\rho}\mathbf{y}^{\mathbf{j}}}$\emph{
}\textup{\emph{is convergent near the origin of $\ww C^{n}$.}}
\end{itemize}
\end{cor}

\begin{proof}
This is an immediate consequence of Lemma \ref{lem: defition equivalent 1-summable}.
\end{proof}
\begin{rem}
We can replace the norm $\norm{\cdot}_{\lambda,\theta,\delta,\rho}$
by $\norm{\cdot}_{\lambda,\theta,\delta,\rho}^{\tx{bis}}$in the second
point of the above corollary.
\end{rem}

Notice that there exists formal power series which are weakly 1-summable
in some direction but which are not Gevrey-1: for instance, the series
\[
{\displaystyle \hat{f}:=\sum_{j}\hat{F}_{j}\left(x\right)y^{j}}\,\,,
\]
where for all $j\in\ww N$, $\hat{F}_{j}\left(x\right)$ is such that
${\displaystyle \cal B\left(\hat{F}_{j}\right)\left(t\right)=\frac{1}{t+\frac{1}{j}}}$,
is weakly 1-summable in the direction $0\in\ww R$, but is \emph{not}
Gevrey-1, since $\cal B\left(\hat{F}_{j}\right)$ has a pole in every
${\displaystyle -\frac{1}{j}\underset{j\rightarrow+\infty}{\longrightarrow}0}$.

\subsection{Some useful tools on 1-summability of solutions of singular linear
differential equations}

~

For future reuse, we give here two results on the 1-summability of
formal solutions to some singular linear differential equations with
1-summable right hand side, which generalize (and precise) a similar
result proved in \cite{MR82} (Proposition p. 126). The authors use
a norm $\norm{\cdot}_{\beta}$, but we will need to use a norm $\norm{\cdot}_{\beta}^{\tx{bis}}$
later (in the proof of Proposition \ref{prop: preparation}).
\begin{prop}
\label{prop: solution borel sommable precise}Let $\hat{b}$ be a
formal power series 1-summable in the direction $\theta$; consider
a domain $\Delta_{\theta,\delta,\rho}$ as in Definition~\ref{def: 1_summability}.
Let use denote by $b_{\theta}$ its 1-sum in this direction $\theta$.
Let us also fix $\alpha,k\in\ww C$.

\begin{enumerate}
\item \label{enu:Assume--and}Assume $\norm{\hat{b}}_{\beta}^{\tx{bis}}<+\infty$
and that $k\in\ww C\backslash\acc 0$ is such that $d_{k}:=\mbox{dist}\left(-k,\Delta_{\theta,\delta,\rho}\right)>0$
and 
\[
\beta d_{k}>C\abs{\alpha k}\qquad,
\]
where $C>0$ is a constant large enough, independent from parameters
$k,\beta,\theta,\delta,\rho$ (for instance, one can take $C=\frac{2\exp\left(2\right)}{5}+5$).
Then, the irregular singular differential equation 
\begin{equation}
x^{2}\ddd ax\left(x\right)+\left(1+\alpha x\right)ka\left(x\right)=\hat{b}\left(x\right)\label{eq: equation borel sommable}
\end{equation}
has a unique formal solution $\hat{a}$ such that $\hat{a}\left(0\right)=\frac{1}{k}\hat{b}\left(0\right)$.
Moreover, $\hat{a}$ is 1-summable in the direction $\theta$, and
\begin{equation}
\norm{\hat{a}}_{\beta}\leq\frac{\beta}{\beta d_{k}-C\abs{\alpha k}}\norm{\hat{b}}_{\beta}\qquad.\label{eq: inegalite norme borel}
\end{equation}
Finally, the 1-sum $a_{\theta}$ of $\hat{a}$ in the direction $\theta$
is the only solution to
\[
x^{2}\ddd{a_{\theta}}x\left(x\right)+\left(1+\alpha x\right)ka_{\theta}\left(x\right)=b_{\theta}\left(x\right)
\]
which is bounded in some $S_{\theta,\pi}\in\germsect{\theta}{\pi}$.
\item \label{enu:Assume--and-1}Assume $\norm{\hat{b}}_{\beta}<+\infty$
and that $\Re\left(k\right)>0$. Then the regular singular differential
equation 
\begin{equation}
x\ddd ax\left(x\right)+ka\left(x\right)=\hat{b}\left(x\right)\label{eq: regular diff equation}
\end{equation}
admits a unique formal solution $\hat{a}$ which is also 1-summable
in the direction $\theta$, of 1-sum $a_{\theta}$. Moreover, $a_{\theta}$
is the only germ of solution to the differential equation 
\[
x\ddd ax\left(x\right)+ka\left(x\right)=b_{\theta}\left(x\right)
\]
which is bounded in some $S_{\theta,\pi}\in\germsect{\theta}{\pi}$. 
\end{enumerate}
\end{prop}

\begin{proof}
~

(\ref{enu:Assume--and}) Since $\hat{b}$ is 1-summable in the direction
$\theta$, we can choose $\rho>0$ and $\delta>0$ such that $\widetilde{\cal B}\left(\hat{b}\right)$
can be analytically continued to (and is bounded in) any domain of
the form $\Delta_{\theta,\delta,\rho}\cap\overline{\mbox{D}}\left(0,R\right)$,
$R>0$.\\
Let us apply the Borel transform $\widetilde{\cal B}$ to equation
(\ref{eq: equation borel sommable}): we obtain 
\begin{equation}
\left(t+k\right)\widetilde{\cal B}\left(\hat{a}\right)\left(t\right)+\alpha k\int_{0}^{t}\widetilde{\cal B}\left(\hat{a}\right)\left(s\right)\mbox{d}s=\widetilde{\cal B}\left(\hat{b}\right)\left(t\right)\qquad.\label{eq: borel equa diff}
\end{equation}
The derivative with respect to $t$ of this equation shows that $\widetilde{\cal B}\left(\hat{a}\right)$
is solution of a linear differential equation, with only one (regular)
singularity at $t=-k$ (but this singularity is not in $\Delta_{\theta,\delta,\rho}$
by assumption):
\[
\left(t+k\right)\ddd{\widetilde{\cal B}\left(\hat{a}\right)}t\left(t\right)+\left(1+\alpha k\right)\widetilde{\cal B}\left(\hat{a}\right)\left(t\right)=\ddd{\widetilde{\cal B}\left(\hat{b}\right)}t\left(t\right)\qquad.
\]
 Since $\widetilde{\cal B}\left(\hat{b}\right)$ can be analytically
continued to $\Delta_{\theta,\delta,\rho}$, the same goes for ${\displaystyle \ddd{\widetilde{\cal B}\left(\hat{b}\right)}t\left(t\right)}$
and then for $\widetilde{\cal B}\left(\hat{a}\right)$. Since ${\displaystyle \widetilde{\cal B}\left(\hat{a}\right)\left(0\right)=\frac{\widetilde{\cal B}\left(\hat{b}\right)\left(0\right)}{k}=\frac{\hat{b}'\left(0\right)}{k}}$,
we can write: {\small{}
\begin{eqnarray*}
\widetilde{\cal B}\left(\hat{a}\right)\left(t\right) & = & \left(t+k\right)^{-1-\alpha k}\left(\hat{b}'\left(0\right).k^{\alpha k}+\int_{0}^{t}\ddd{\widetilde{\cal B}\left(\hat{b}\right)}s\left(s\right).\left(s+k\right)^{\alpha k}\mbox{d}s\right)\\
 & = & \left(t+k\right)^{-1-\alpha k}\Bigg(\hat{b}'\left(0\right).k^{\alpha k}+\widetilde{\cal B}\left(\hat{b}\right)\left(t\right).\left(t+k\right)^{\alpha k}-\widetilde{\cal B}\left(\hat{b}\right)\left(0\right).k^{\alpha k}\\
 &  & \,\,\,\,\,\,-\alpha k\int_{0}^{t}\widetilde{\cal B}\left(\hat{b}\right)\left(s\right).\left(s+k\right)^{\alpha k-1}\mbox{d}s\Bigg)\\
 & = & \left(t+k\right)^{-1-\alpha k}\Bigg(\widetilde{\cal B}\left(\hat{b}\right)\left(t\right).\left(t+k\right)^{\alpha k}\\
 &  & \,\,\,\,\,\,-\alpha k\int_{0}^{t}\widetilde{\cal B}\left(\hat{b}\right)\left(s\right).\left(s+k\right)^{\alpha k-1}\mbox{d}s\Bigg)\\
\widetilde{\cal B}\left(\hat{a}\right) & = & \frac{\widetilde{\cal B}\left(\hat{b}\right)\left(t\right)}{\left(t+k\right)}-\alpha k.\left(t+k\right)^{-1-\alpha k}\int_{0}^{t}\widetilde{\cal B}\left(\hat{b}\right)\left(s\right).\left(s+k\right)^{\alpha k-1}\mbox{d}s\,\,.
\end{eqnarray*}
}The fact that $\widetilde{\cal B}\left(\hat{b}\right)$ is bounded
in any domain of the form $\Delta_{\theta,\delta,\rho}\cap\overline{\mbox{D}}\left(0,R\right)$,
$R>0$, implies that the same goes for $\widetilde{\cal B}\left(\hat{a}\right)$.
Let us prove inequality (\ref{eq: inegalite norme borel}). For all
$R>0$, for all Gevrey-1 series $\hat{f}\in\form{x,\mathbf{y}}$ such
that $\cal B\left(\hat{f}\right)$ can be analytically continued to
$\Delta_{\theta,\delta,r}$, we set: 
\[
\norm{\hat{f}}_{\beta,R}^{\tx{bis}}:=\underset{t\in\Delta_{\theta,\delta,\rho}\cap\mbox{\ensuremath{\overline{\mbox{D}}\left(0,R\right)}}}{\sup}\acc{\abs{\widetilde{\cal B}\left(\hat{f}\right)\left(t\right)\left(1+\beta^{2}\abs t^{2}\right)\exp\left(-\beta\abs t\right)}}\in\ww R\cup\acc{\infty}\,\,.
\]
Notice that ${\displaystyle \norm{\hat{f}}_{\beta}^{\tx{bis}}=\underset{R>0}{\sup}\acc{\norm{\hat{f}}_{\beta,R}^{\tx{bis}}}}$
for all $\hat{f}$ as above, and that for all $R>0$, $\norm{\hat{a}}_{\beta,R}^{\tx{bis}}<+\infty$,
since $\widetilde{\cal B}\left(\hat{a}\right)$ is bounded in any
domain of the form $\Delta_{\theta,\delta,\rho}\cap\overline{\mbox{D}}\left(0,R\right)$.
Fix some $R>0$, and let $t\in\Delta_{\theta,\delta,\rho}\cap\mbox{\ensuremath{\overline{\mbox{D}}\left(0,R\right)}}$.
From equation (\ref{eq: borel equa diff}) we obtain 
\begin{eqnarray*}
\widetilde{\cal B}\left(\hat{a}\right)\left(t\right) & = & \frac{1}{\left(t+k\right)}\left(\widetilde{{\cal B}}\left(\hat{b}\right)\left(t\right)-\alpha k\int_{0}^{t}\widetilde{{\cal B}}\left(\hat{a}\right)\left(s\right)\mbox{d}s\right)
\end{eqnarray*}
an then
\begin{eqnarray*}
\abs{\widetilde{\cal B}\left(\hat{a}\right)\left(t\right)} & \leq & \frac{1}{\abs{t+k}}\cro{\norm{\hat{b}}_{\beta}^{\tx{bis}}\frac{\exp\left(\beta\abs t\right)}{1+\beta^{2}\abs t^{2}}+\abs{\alpha k}.\norm{\hat{a}}_{\beta,R}^{\tx{bis}}\int_{0}^{\abs t}\frac{\exp\left(\beta u\right)}{1+\beta^{2}u^{2}}{{\dd[u]}}}\\
 & \leq & \frac{1}{d_{k}}\frac{\exp\left(\beta\abs t\right)}{1+\beta^{2}\abs t^{2}}\cro{\norm{\hat{b}}_{\beta}^{\tx{bis}}+\abs{\alpha k}\norm{\hat{a}}_{\beta,R}^{\tx{bis}}\frac{C}{\beta}}\,\,,
\end{eqnarray*}
with $C=\frac{2\exp\left(2\right)}{5}+5$. Here we use the following:

\begin{fact}
There exists a constant $C>0$ (\emph{e.g.} $C=\frac{2\exp\left(2\right)}{5}+5$),
such that for all $\beta>0$, we have:
\[
\forall t\geq0,\int_{0}^{t}\frac{\exp\left(\beta u\right)}{1+\beta^{2}u^{2}}\tx du\leq\frac{C}{\beta}\frac{\exp\left(\beta t\right)}{1+\beta^{2}t^{2}}\,\,.
\]
\end{fact}

\begin{proof}
Let $F:u\mapsto\frac{\exp\left(\beta u\right)}{1+\beta^{2}u^{2}}$,
for $u\geq0$. For $t\in\cro{0,\frac{2}{\beta}}$, we have:
\[
\int_{0}^{t}F\left(u\right)\tx du\leq\frac{\exp\left(2\right)}{5}.\frac{2}{\beta}\,\,,
\]
since $F$ is an increasing function over $\wr_{+}$: 
\[
F'\left(u\right)=\beta F\left(u\right).\frac{\left(1-\beta u\right)^{2}}{1+\beta^{2}u^{2}}\geq0\,\,.
\]
Moreover for all $t\geq0$, we have $F\left(t\right)\geq F\left(0\right)=1$.
Hence for all $t\in\cro{0,\frac{2}{\beta}}$:
\[
\int_{0}^{t}F\left(u\right)\tx du\leq\frac{\exp\left(2\right)}{5}.\frac{2}{\beta}.F\left(t\right)\,\,.
\]
For $t\geq\frac{2}{\beta}$, the following inequality holds:
\begin{eqnarray*}
\int_{0}^{t}F\left(u\right)\tx du & \leq & \frac{\exp\left(2\right)}{5}.\frac{2}{\beta}F\left(t\right)+\int_{\frac{2}{\beta}}^{t}F\left(u\right)\tx du\,\,.
\end{eqnarray*}
In addition, if $u\geq\frac{\beta}{2}$, then: 
\begin{eqnarray*}
\frac{\left(1-\beta u\right)^{2}}{1+\beta^{2}u^{2}} & \geq & \frac{1}{5}\,\,,
\end{eqnarray*}
Therefore, for all $u\geq\frac{\beta}{2}$:
\[
F'\left(u\right)=\beta F\left(u\right).\frac{\left(1-\beta u\right)^{2}}{1+\beta^{2}u^{2}}\geq\frac{\beta}{5}F\left(u\right)\,\,.
\]
Hence:
\begin{eqnarray*}
\int_{0}^{t}F\left(u\right)\tx du & \leq & \int_{0}^{\frac{2}{\beta}}F\left(u\right)\tx du+\int_{\frac{2}{\beta}}^{t}F\left(u\right)\tx du\,\,.\\
 & \leq & \frac{\exp\left(2\right)}{5}.\frac{2}{\beta}F\left(t\right)+\frac{5}{\beta}\int_{\frac{2}{\beta}}^{t}F'\left(u\right)\tx du\\
 & \leq & \frac{F\left(t\right)}{5\beta}.\left(2\exp\left(2\right)+25\right)\,\,.
\end{eqnarray*}
\end{proof}
Let us go back to the proof of the lemma. Finally, we have: 
\begin{eqnarray*}
\norm{\hat{a}}_{\beta,R}^{\tx{bis}} & \leq & \frac{1}{d_{k}}\cro{\norm{\hat{b}}_{\beta}^{\tx{bis}}+\frac{C.\abs{\alpha k}.\norm{\hat{a}}_{\beta,R}^{\tx{bis}}}{\beta}}\,\,,
\end{eqnarray*}
and consequently:
\[
\norm{\hat{a}}_{\beta,R}^{\tx{bis}}\leq\frac{\beta}{\beta d_{k}-C\abs{\alpha k}}\norm{\hat{b}}_{\beta}^{\tx{bis}}\,\,.
\]
As a conclusion:
\[
\norm{\hat{a}}_{\beta}^{\tx{bis}}\leq\frac{\beta}{\beta d_{k}-C\abs{\alpha k}}\norm{\hat{b}}_{\beta}^{\tx{bis}}\,\,,
\]
and $a_{\theta}$ is the 1-sum of $\hat{a}$ in the direction $\theta$.\bigskip{}

(\ref{enu:Assume--and-1}) Let us write ${\displaystyle \hat{b}\left(x\right)=\sum_{j\geq0}b_{j}x^{j}}$.
A direct computation shows that the only formal solution to equation
$\left(\mbox{\ref{eq: regular diff equation}}\right)$ is ${\displaystyle \hat{a}\left(x\right)=\sum_{j\geq0}a_{j}x^{j}}$
where for all $j\in\ww N$, ${\displaystyle a_{j}=\frac{b_{j}}{j+k}}$:
it exists since $k\notin\ww Z_{\leq0}$, and then $k+j\neq0$. In
particular, we see immediately that $\hat{a}$ is Gevrey-1, because
the same goes for $\hat{b}$. In other words, the Borel transform
$\cal B\left(\hat{a}\right)$ is analytic in some disc $D\left(0,\rho\right)$,
$\rho>0$. In $D\left(0,\rho\right)$, $\cal B\left(\hat{a}\right)$
satisfies:
\begin{eqnarray*}
t\ddd{\cal B\left(\hat{a}\right)}t\left(t\right)+k\cal B\left(\hat{a}\right)\left(t\right) & = & \cal B\left(\hat{b}\right)\left(t\right)\,\,.
\end{eqnarray*}
The general solution near the origin to this equation is 
\[
y\left(t\right)=\frac{c}{t^{k}}+\frac{1}{t^{k}}\int_{0}^{t}\cal B\left(\hat{b}\right)\left(s\right)s^{k-1}\mbox{d}s\,\,,\,c\in\ww C.
\]
In particular, the only solution analytic in $D\left(0,\rho\right)$
is the one with $c=0$, \emph{i.e.}
\[
\cal B\left(\hat{a}\right)\left(t\right)=\frac{1}{t^{k}}\int_{0}^{t}\cal B\left(\hat{b}\right)\left(s\right)s^{k-1}\mbox{d}s\,\,.
\]
 Since $\cal B\left(\hat{b}\right)$ can be analytically continued
to an infinite domain that have denoted by $\Delta_{\theta,\delta,\rho}$
bisected by $\ww R_{+}e^{i\theta}$ (because $\hat{b}$ is 1-summable
in the direction $\theta$), $\cal B\left(\hat{a}\right)$ can also
be analytically continued to the same domain. Moreover, there exists
$\beta>0$ such that $\norm{\hat{b}}_{\beta}<+\infty$, \emph{i.e.
$\forall t\in\Delta_{\theta,\delta,\rho}$}:\emph{
\[
\abs{\cal B\left(\hat{b}\right)\left(t\right)}\leq\norm{\hat{b}}_{\beta}\exp\left(\beta\abs t\right)\,\,.
\]
}Thus, for all $t\in\Delta_{\theta,\delta,\rho}$, we have:
\begin{eqnarray*}
\abs{\exp\left(-\beta\abs t\right)\cal B\left(\hat{a}\right)\left(t\right)} & \leq & \frac{1}{\abs{t^{k}}}\int_{0}^{\abs t}\abs{\exp\left(-\beta\abs t\right)}\abs{\cal B\left(\hat{b}\right)\left(s\mbox{e}^{i\arg\left(t\right)}\right)}\abs{s^{k-1}\mbox{e}^{i\left(k-1\right)\arg\left(t\right)}}\mbox{d}s\\
 & \leq & \frac{1}{\abs t^{\Re\left(k\right)}}\int_{0}^{\abs t}\abs{\exp\left(-\beta s\right)}\abs{\cal B\left(\hat{b}\right)\left(s\mbox{e}^{i\arg\left(t\right)}\right)}s^{\Re\left(k\right)-1}\mbox{d}s\\
 & \leq & \frac{\norm{\hat{b}}_{\beta}}{\abs t^{\Re\left(k\right)}}\int_{0}^{\abs t}s^{\Re\left(k\right)-1}\mbox{d}s\\
 & = & \frac{\norm{\hat{b}}_{\beta}}{\Re\left(k\right)}\,\,.
\end{eqnarray*}
Thus, $\hat{a}$ is 1-summable in the direction $\theta$.
\end{proof}

\section{1-summable preparation up to any order $N$}

The aim of this section is to prove that we can always formally conjugate
a non-degenerate doubly-resonant saddle-node, which is also div-integrable,
to its normal form up to a remainder of order $\tx O\left(x^{N}\right)$
for every $N\in\ww N_{>0}$. Moreover, we prove that this conjugacy
is in fact 1-summable in every direction $\theta\neq\arg\left(\pm\lambda\right)$,
hence analytic over sectorial domains of opening at least $\pi$.
\begin{prop}
\label{prop: forme pr=0000E9par=0000E9e ordre N}Let $Y\in\sndiag$
be a non-degenerate diagonal doubly-resonant saddle-node which is
div-integrable, with $\tx D_{0}Y=\tx{diag}\left(0,-\lambda,\lambda\right)$,
$\lambda\neq0$. Then, for all $N\in\ww N_{>0}$, there exists a formal
fibered diffeomorphism $\Psi^{\left(N\right)}\in\fdiffformid$ tangent
to the identity and 1-summable in every direction $\theta\neq\arg\left(\pm\lambda\right)$
such that:
\begin{eqnarray*}
\left(\Psi^{\left(N\right)}\right)_{*}\left(Y\right) & = & x^{2}\pp x+\left(\left(-\left(\lambda+d^{\left(N\right)}\left(y_{1}y_{2}\right)\right)+a_{1}x\right)+x^{N}F_{1}^{\left(N\right)}\left(x,\mathbf{y}\right)\right)y_{1}\pp{y_{1}}\\
 &  & +\left(\left(\lambda+d^{\left(N\right)}\left(y_{1}y_{2}\right)+a_{2}x\right)+x^{N}F_{2}^{\left(N\right)}\left(x,\mathbf{y}\right)\right)y_{2}\pp{y_{2}}\\
 & =: & Y^{\left(N\right)}\,\,\,,
\end{eqnarray*}
where $\lambda\in\ww C^{*}$, ${\displaystyle \left(a_{1}+a_{2}\right)=\tx{res}\left(Y\right)\in\ww C\backslash\ww Q_{\leq0}}$,
$d^{\left(N\right)}\left(v\right)\in v\germ v$ is an analytic germ
vanishing at the origin, and $F_{1}^{\left(N\right)},F_{2}^{\left(N\right)}\in\form{x,\mathbf{y}}$
are 1-summable in the direction $\theta$, and of order at least one
with respect to $\mathbf{y}$. Moreover, one can choose $d^{\left(2\right)}=\dots=d^{\left(N\right)}$
for all $N\geq2$. 
\end{prop}

\begin{defn}
A vector field $Y^{\left(N\right)}$ as is the proposition above is
said to be \emph{normalized up to order $N$}.
\end{defn}

\begin{rem}
~

\begin{enumerate}
\item Observe that this result does not require the more restrictive assumption
of being ``strictly non-degenerate'' (\emph{i.e.} $\Re\left(a_{1}+a_{2}\right)>0$)
.
\item As a consequence of Corollary \ref{cor: summability push-forward},
the 1-sum $\Psi_{\theta}^{\left(N\right)}$ of $\Psi^{\left(N\right)}$
in the direction $\theta$ is a germ of sectorial fibered diffeomorphism
tangent to the identity, \emph{i.e.} ${\displaystyle \Psi_{\theta}^{\left(N\right)}\in\diffsect[\theta][\pi]}$,
which conjugates $Y$ to the 1-sum $Y_{\theta}^{\left(N\right)}$
of $Y^{\left(N\right)}$ in the direction $\theta$.
\end{enumerate}
\end{rem}

In order to prove this result we will proceed in several steps and
use after each step Proposition \ref{prop: compositon summable} and
Corollary \ref{cor: summability push-forward} in order to prove the
1-summability in every direction $\theta\neq\arg\left(\pm\lambda\right)$
of the different objects. First, we will normalize analytically the
vector field restricted to $\acc{x=0}$. Then, we will straighten
the formal separatrix to $\acc{y_{1}=y_{2}=0}$ in suitable coordinates.
Next, we will simplify the linear terms with respect to $\mathbf{y}$.
After that, we will straighten two invariant hypersurfaces to $\acc{y_{1}=0}$
and $\acc{y_{2}=0}$. Finally, we will conjugate the vector field
to its final normal form up to remaining terms of order $\tx O\left(x^{N}\right)$.

\subsection{Analytic normalization on the hyperplane $\protect\acc{x=0}$}

~

\subsubsection{Transversally Hamiltonian \emph{versus} div-integrable}

~

We start by proving that an element of $\sndiag$ which is transversally
Hamiltonian is necessarily div-integrable.
\begin{prop}
If $Y\in\sndiag$ is transversally Hamiltonian, then $Y$ is div-integrable. 
\end{prop}

\begin{proof}
Let us consider more generally a diagonal doubly-resonant saddle-node
$Y\in\sndiag$ such that $Y_{\mid\acc{x=0}}$ is a Hamiltonian vector
field with respect to $\mbox{d}y_{1}\wedge\mbox{d}y_{2}$ (this is
the case if $Y$ is transversally Hamiltonian): there exists a Hamiltonian
${\displaystyle H\left(\mathbf{y}\right)=\lambda y_{1}y_{2}+O\left(\norm{\mathbf{y}}^{3}\right)\in\germ{\mathbf{y}}}$,
such that 
\[
Y=x^{2}\pp x+\left(\left(-\ppp H{y_{2}}+xF_{1}\left(x,\mathbf{y}\right)\right)\pp{y_{1}}+\left(\ppp H{y_{1}}+xF_{2}\left(x,\mathbf{y}\right)\right)\pp{y_{2}}\right)\,\,\,,
\]
where $F_{1},F_{2}\in\germ{x,\mathbf{y}}$ are vanishing at the origin.
If we define ${\displaystyle J:=\left(\begin{array}{cc}
0 & -1\\
1 & 0
\end{array}\right)\in M_{2}\left(\ww C\right)}$ and $\nabla H:=\,^{t}\left(\mbox{D}H\right)$, then ${\displaystyle Y_{\mid\acc{x=0}}=J\nabla H}$.
According to the Morse lemma for holomorphic functions, there exists
a germ of an analytic change of coordinates $\varphi\in\diff[\ww C^{2},0]$
given by
\begin{eqnarray*}
\mathbf{y}=\left(y_{1},y_{2}\right) & \mapsto & \varphi\left(\mathbf{y}\right)=\left(y_{1}+O\left(\left\Vert \mathbf{y}\right\Vert ^{2}\right),y_{2}+O\left(\left\Vert \mathbf{y}\right\Vert ^{2}\right)\right)\,\,\,\,\,,
\end{eqnarray*}
such that ${\displaystyle {\displaystyle \widetilde{H}\left(\mathbf{y}\right):=H\left(\varphi^{-1}\left(\mathbf{y}\right)\right)=y_{1}y_{2}}}$.
Let us now recall a trivial result from linear algebra.

\begin{fact*}
Let ${\displaystyle J:=\left(\begin{array}{cc}
0 & -1\\
1 & 0
\end{array}\right)\in M_{2}\left(\ww C\right)}$, and $P\in M_{2}\left(\ww C\right)$. Then, ${\displaystyle PJ\,^{t}P=\mbox{det}\left(P\right)J}$.
\end{fact*}
As a consequence we have: 

\begin{cor}
\label{cor: hamiltonian system}Let $H\in\ww C\acc{\mathbf{y}}$ be
a germ of an analytic function at $0$, $Y_{0}:=J\nabla H$ the associated
Hamiltonian vector field in $\ww C^{2}$ (for the usual symplectic
form $dy_{1}\wedge dy_{2}$), and an analytic diffeomorphism near
the origin denoted by $\varphi$. Then: 
\begin{eqnarray*}
\varphi_{*}\left(Y_{0}\right) & := & \left(\tx D\varphi\circ\varphi^{-1}\right)\cdot\left(Y_{0}\circ\varphi^{-1}\right)=\det\left(\tx D\varphi\circ\varphi^{-1}\right)J\nabla\widetilde{H}\,\,\,\,,
\end{eqnarray*}
where $\widetilde{H}:=H\circ\varphi^{-1}$.
\end{cor}

As a conclusion we have proved that $Y$ is div-integrable.
\end{proof}

\subsubsection{General case}

~

Now we prove how to normalize the restriction to $\acc{x=0}$ of a
div-integrable element of $\sndiag$.
\begin{prop}
\label{prop: preparation sur hypersurface invariante}Let $Y\in\sndiag$
be div-integrable. Then, there exists $\psi\in\fdiffid$ of the form
\[
\psi:\left(x,\mathbf{y}\right)\mapsto\left(x,y_{1}+O\left(\norm{\mathbf{y}}^{2}\right),y_{2}+O\left(\norm{\mathbf{y}}^{2}\right)\right)
\]
such that 
\[
\psi_{*}\left(Y\right)=x^{2}\pp x+\left(-\left(\lambda+d\left(v\right)\right)y_{1}+xT_{1}\left(x,\mathbf{y}\right)\right)\pp{y_{1}}+\left(\left(\lambda+d\left(v\right)\right)y_{2}+xT_{2}\left(x,\mathbf{y}\right)\right)\pp{y_{2}}\,\,\,,
\]
with $v:=y_{1}y_{2}$, $d\left(v\right)\in v\germ v$ and $T_{1},T_{2}\in\germ{x,\mathbf{y}}$
vanishing at the origin.
\end{prop}

\begin{proof}
By assumption, and according to a theorem due to Brjuno (\emph{cf
}\cite{Martinet}), up to a first transformation analytic at the origin
in $\ww C^{2}$, we can suppose that 
\[
Y_{\mid\acc{x=0}}=\left(\lambda+h\left(\mathbf{y}\right)\right)\left(-y_{1}\pp{y_{1}}+y_{2}\pp{y_{2}}\right)\,\,\,.
\]
Then, it remains to apply the following lemma to $Y_{\mid\acc{x=0}}$.
\end{proof}
\begin{lem}
\label{lem: preparation sur hypersurface invariante}Let $Y_{0}$
be a germ of analytic vector field in $\left(\ww C^{2},0\right)$
of the form
\begin{eqnarray*}
Y_{0} & = & \left(\lambda+h\left(\mathbf{y}\right)\right)\left(-y_{1}\pp{y_{1}}+y_{2}\pp{y_{2}}\right)\,\,\,,
\end{eqnarray*}
with $h\in\germ{\mathbf{y}}$ vanishing at the origin. Then there
exists ${\displaystyle \phi\in\diff[\ww C^{2},0,\tx{Id}]}$ such that
\begin{eqnarray*}
\phi_{*}\left(Y_{0}\right) & = & \left(\lambda+d\left(v\right)\right)\left(-y_{1}\pp{y_{1}}+y_{2}\pp{y_{2}}\right)\,\,\,\,,
\end{eqnarray*}
with $v:=y_{1}y_{2}$ and $d\in v\ww C\acc v$.
\end{lem}

\begin{rem}
In other words, we have removed every non-resonant term in $h\left(\mathbf{y}\right)$.
In fact, we re-obtain here a particular case (with one vector field
in dimension 2) of the principal result in \cite{StoloChampsCommutants}
(which is itself inspired of Vey's works).
\end{rem}

\begin{proof}
We claim that $\phi$ can be chosen of the form 
\[
{\displaystyle \phi\left(\mathbf{y}\right)=\left(y_{1}e^{-\gamma\left(\mathbf{y}\right)},y_{2}e^{\gamma\left(\mathbf{y}\right)}\right)}\,\,,
\]
for a conveniently chosen $\gamma\in\germ{\mathbf{y}}$. Indeed, let
us study how such a diffeomorphism acts on $Y_{0}$. Let us write
$U:=\left(\lambda+h\left(v\right)\right)$ and $L:=\left(-y_{1}\pp{y_{1}}+y_{2}\pp{y_{2}}\right)$,
such that $Y_{0}=UL$. An easy computation shows:
\begin{eqnarray*}
\phi_{*}\left(Y_{0}\right) & = & \phi_{*}\left(U.L\right)\\
 & = & \left(\cro{U\cdot\left(1-{\cal L_{L}\left(\gamma\right)}\right)}\circ\phi^{-1}\right)L\,\,,
\end{eqnarray*}
where $\cal L_{L}$ is the Lie derivative of associate to $L$. We
want to to find $\gamma$ such that the unit 
\[
D:=\cro{U\left(1-{\cal L_{L}\left(\gamma\right)}\right)}\circ\phi^{-1}
\]
 is free from \emph{non-resonant} terms, \emph{i.e.} is of the form
\[
D=\lambda+d\left(y_{1}y_{2}\right)=\lambda+\sum_{k\geq1}d{}_{k}\left(y_{1}y_{2}\right)^{k}\,\,.
\]
Notice that if a unit ${\displaystyle W=\sum_{k\geq0}W_{k}\left(y_{1}y_{2}\right)^{k}\germ{\mathbf{y}}}^{\times}$
is free from non-resonant terms, then: 
\[
\begin{cases}
W\circ\phi^{-1}=W\\
\cal L_{L}\left(W\right)=0 & .
\end{cases}
\]
Thus, let us split both $U$ and $\gamma$ in a ``resonant'' and
a ``non-resonant'' part:
\[
\begin{cases}
U=U_{\tx{res}}+U_{\tx{n-res}}\\
\gamma=\gamma_{\tx{res}}+\gamma_{\tx{n-res}}
\end{cases}
\]
where 
\[
\begin{cases}
U_{\tx{n-res}}=\underset{\substack{k_{1}\neq k_{2}}
}{\sum}U_{k_{1},k_{2}}y_{1}^{k_{1}}y_{2}^{k_{2}}\\
U_{\tx{res}}=\underset{\substack{k}
}{\sum}U_{k,k}\left(y_{1}y_{2}\right)^{k}\\
\gamma_{\tx{n-res}}=\underset{\substack{k_{1}\neq k_{2}}
}{\sum}\gamma_{k_{1},k_{2}}y_{1}^{k_{1}}y_{2}^{k_{2}}\\
\gamma_{\tx{res}}=\underset{k}{\sum}\gamma_{k,k}\left(y_{1}y_{2}\right)^{k} & .
\end{cases}
\]
Then the non-resonant terms of $U\left(1-{\cal L_{L}\left(\gamma\right)}\right)$
are 
\[
\left(U_{\tx{n-res}}-\left(U_{\tx{n-res}}+U_{\tx{res}}\right){\cal L_{L}\left(\gamma_{\tx{n-res}}\right)}\right)\circ\phi^{-1}\,\,.
\]
Hence, the partial differential equation we want to solve is: 
\[
\cal L_{L}\left(\gamma\right)=\frac{U_{\tx{n-res}}}{U_{\tx{res}}+U_{\tx{n-res}}}\,\,.
\]
One sees immediately that this equation admit an analytic solution
(and even infinitely many solutions) $\gamma\in\germ{\mathbf{y}}$,
since the unit $U\in\germ{\mathbf{y}}$ is analytic. 
\end{proof}

\subsection{1-summable simplification of the ``dependent'' affine part}

~

We are concerned by studying vector fields of the form 
\begin{equation}
Y=x^{2}\pp x+\left(-\lambda y_{1}+f_{1}\left(x,\mathbf{y}\right)\right)\pp{y_{1}}+\left(\lambda y_{2}+f_{2}\left(x,\mathbf{y}\right)\right)\pp{y_{2}}\,\,\,\,,\label{eq: Y asympt hamilt}
\end{equation}
with
\[
\begin{cases}
{\displaystyle f_{1}\left(x,\mathbf{y}\right)=-d\left(y_{1}y_{2}\right)y_{1}+xT_{1}\left(x,\mathbf{y}\right)}\\
{\displaystyle f_{2}\left(x,\mathbf{y}\right)=d\left(y_{1}y_{2}\right)y_{2}+xT_{2}\left(x,\mathbf{y}\right)} & ,
\end{cases}
\]
where ${\displaystyle d\left(v\right)\in v\germ v}$ and ${\displaystyle T_{1},T_{2}\in\germ{x,\mathbf{y}}}$
are of order at least one.
\begin{prop}
\label{prop: diag prep}Let $Y\in\sndiag$ be a doubly-resonant saddle-node
of the form 
\[
Y=x^{2}\pp x+\left(-\lambda y_{1}+f_{1}\left(x,\mathbf{y}\right)\right)\pp{y_{1}}+\left(\lambda y_{2}+f_{2}\left(x,\mathbf{y}\right)\right)\pp{y_{2}}\,\,\,\,,
\]
where $f_{1},f_{2}\in\germ{x,\mathbf{y}}$ are such that ${\displaystyle f_{1}\left(x,\mathbf{y}\right),f_{2}\left(x,\mathbf{y}\right)=\mbox{\ensuremath{\tx O}}\left(\norm{\left(x,\mathbf{y}\right)}^{2}\right)}$.
Then there exist formal power series $\hat{y}_{1}$, $\hat{y}_{2}$,
$\hat{\alpha}_{1}$, $\hat{\alpha}_{2}$, $\hat{\beta}_{1}$, $\hat{\beta}_{2}\in x\form x$
which are 1-summable in every direction $\theta\neq\arg\left(\pm\lambda\right)$,
such that the formal fibered diffeomorphism 
\[
\hat{\Phi}:\left(x,y_{1},y_{2}\right)\mapsto\left(x,\hat{y}_{1}\left(x\right)+\left(1+\hat{\alpha}_{1}\left(x\right)\right)y_{1}+\hat{\beta}_{1}\left(x\right)y_{2},\hat{y}_{2}\left(x\right)+\hat{\alpha}_{2}\left(x\right)y_{1}+\left(1+\hat{\beta}_{1}\left(x\right)\right)y_{2}\right)\,\,\,,
\]
(which is tangent to the identity and 1- summable in every direction
$\theta\neq\arg\left(\pm\lambda\right)$) conjugates $Y$ to 
\[
\hat{\Phi}_{*}\left(Y\right)=x^{2}\pp x+\left(\left(-\lambda+a_{1}x\right)y_{1}+\hat{F}_{1}\left(x,\mathbf{y}\right)\right)\pp{y_{1}}+\left(\left(\lambda+a_{2}x\right)y_{2}+\hat{F}_{2}\left(x,\mathbf{y}\right)\right)\pp{y_{2}}\,\,\,,
\]
where $a_{1},a_{2}\in\ww C$ and $\hat{F}_{1},\hat{F}_{2}\in\form{x,\mathbf{y}}$
are of order at least 2 with respect to $\mathbf{y}$, and 1-summable
in every direction $\theta\neq\arg\left(\pm\lambda\right)$.
\end{prop}

\begin{rem}
Notice that $\hat{\Phi}_{\mid\acc{x=0}}=\mbox{Id}$, so that $\hat{F}_{i}\left(0,\mathbf{y}\right)=f_{i}\left(0,\mathbf{y}\right)$
for $i=1,2$. Moreover, the residue of $\hat{\Phi}_{*}\left(Y\right)$
is $a_{1}+a_{2}$.
\end{rem}

The proof of Proposition \ref{prop: diag prep} is postponed to subsection
\ref{subsec:Proof-of-Proposition}.

\subsubsection{Technical lemmas on irregular differential systems}

~
\begin{lem}
There exists a pair of formal power series $\left(\hat{y}_{1}\left(x\right),\hat{y}_{2}\left(x\right)\right)\in\left(x\form x\right)^{2}$
which are 1-summable in every direction $\theta\neq\arg\left(\pm\lambda\right)$,
such that the formal diffeomorphism given by 
\[
{\displaystyle \hat{\Phi}_{1}\left(x,y_{1},y_{2}\right)=\left(x,y_{1}-\hat{y}_{1}\left(x\right),y_{2}-\hat{y}_{2}\left(x\right)\right)},
\]
(which is 1-summable in every direction $\theta\neq\arg\left(\pm\lambda\right)$)
conjugates $Y$ in (\ref{eq: Y asympt hamilt}) to
\begin{equation}
\hat{Y}_{1}\left(x,\mathbf{y}\right)=x^{2}\pp x+\left(-\lambda y_{1}+\hat{g}_{1}\left(x,\mathbf{y}\right)\right)\pp{y_{1}}+\left(\lambda y_{2}+\hat{g}_{2}\left(x,\mathbf{y}\right)\right)\pp{y_{2}}\qquad,\label{eq: redressement separatrice}
\end{equation}
where $\hat{g}_{1},\hat{g}_{2}$ are formal power series of order
at least $2$ such that $\hat{g}_{1}\left(x,\mathbf{0}\right)=\hat{g}_{2}\left(x,\mathbf{0}\right)=0$,
and are 1-summable in every direction $\theta\neq\arg\left(\pm\lambda\right)$.
\end{lem}

In other words, in the new coordinates, the curve given by $\left(y_{1},y_{2}\right)=\left(0,0\right)$
is invariant by the flow of the vector field, and contains the origin
in its closure: it is usually called the (formal, 1-summable) \emph{center
manifold.}
\begin{proof}
This is an immediate consequence of an important theorem by Ramis
and Sibuya on the summability of formal solutions to irregular differential
systems \cite{ramis1989hukuhara}. This theorem proves the existence
and the 1-summability in every direction $\theta\neq\arg\left(\pm\lambda\right)$,
of $\hat{y}_{1}$ and $\hat{y}_{2}$: $\left(\hat{y}_{1}\left(x\right),\hat{y}_{2}\left(x\right)\right)$
is defined as the unique formal solution to 
\[
\begin{cases}
{\displaystyle x^{2}\ddd{y_{1}}x=-\lambda y_{1}\left(x\right)+f_{1}\left(x,y_{1}\left(x\right),y_{2}\left(x\right)\right)}\\
{\displaystyle x^{2}\ddd{y_{2}}x=\lambda y_{2}\left(x\right)+f_{2}\left(x,y_{1}\left(x\right),y_{2}\left(x\right)\right)}
\end{cases}\,\,,
\]
such that $\left(\hat{y}_{1}\left(0\right),\hat{y}_{2}\left(0\right)\right)=\left(0,0\right)$.
The 1-summability of $\hat{g}_{1}$ and $\hat{g}_{2}$ comes from
Proposition \ref{prop: compositon summable}.
\end{proof}
The next step is aimed at changing to linear terms with respect to
$\mathbf{y}$ in ``diagonal'' form.
\begin{lem}
There exists a pair of formal power series $\left(\hat{p}_{1},\hat{p}_{2}\right)\in\left(\form x\right)^{2}$
which are 1-summable in every direction $\theta\neq\arg\left(\pm\lambda\right)$,
such that the formal fibered diffeomorphism given by 
\[
\hat{\Phi}_{2}\left(x,y_{1},y_{2}\right)=\left(x,y_{1}+x\hat{p}_{2}\left(x\right)y_{2},y_{2}+x\hat{p}_{1}\left(x\right)y_{1}\right)\,\,,
\]
(which is tangent to the identity and 1-summable in every direction
$\theta\neq\arg\left(\pm\lambda\right)$) conjugates $\hat{Y}_{1}$
in (\ref{eq: redressement separatrice}), to
\begin{eqnarray}
\hat{Y}_{2}\left(x,\mathbf{y}\right) & = & x^{2}\pp x+\left(\left(-\lambda+x\hat{a}_{1}\left(x\right)\right)y_{1}+\hat{H}_{1}\left(x,\mathbf{y}\right)\right)\pp{y_{1}}\nonumber \\
 &  & +\left(\left(\lambda+x\hat{a}_{2}\left(x\right)\right)y_{2}+\hat{H}_{2}\left(x,\mathbf{y}\right)\right)\pp{y_{2}}\qquad,\label{eq: Y semi-diag}
\end{eqnarray}
where $\hat{a}_{1},\hat{a}_{2},\hat{H}_{1},\hat{H}_{2}$ are formal
power series which are 1-summable in every direction $\theta\neq\arg\left(\pm\lambda\right)$
and $\hat{H}_{1},\hat{H}_{2}$ are of order at least $2$ with respect
to $\mathbf{y}$. 
\end{lem}

\begin{proof}
Let us write
\[
\begin{cases}
{\displaystyle \hat{g}_{1}\left(x,\mathbf{y}\right)=x\hat{b}_{1}\left(x\right)y_{1}+x\hat{c}_{1}\left(x\right)y_{2}+\hat{G}_{1}\left(x,\mathbf{y}\right)}\\
{\displaystyle \hat{g}_{2}\left(x,\mathbf{y}\right)=x\hat{c}_{2}\left(x\right)y_{1}+x\hat{b}_{2}\left(x\right)y_{2}+\hat{G}_{2}\left(x,\mathbf{y}\right)}
\end{cases}\qquad,
\]
where $\hat{b}_{1},\hat{b}_{2},\hat{c}_{1},\hat{c}_{2},\hat{G}_{1},\hat{G}_{2}$
are formal power series 1-summable in the direction $\theta\neq\arg\left(\pm\lambda\right)$,
such that $\hat{G}_{1}$ and $\hat{G}_{2}$ are of order at least
$2$ with respect to $\mathbf{y}$. Let us consider the following
irregular differential system naturally associated to $\hat{Y}_{1}$: 

\begin{equation}
x^{2}\ddd{\mathbf{z}}x\left(x\right)=\mathbf{\hat{B}}\left(x\right)\mathbf{z}\left(x\right)+\mathbf{\hat{G}}\left(x,\mathbf{z}\left(x\right)\right)\qquad,\label{eq: systeme diff irreg}
\end{equation}
where 
\[
\mathbf{\hat{B}}\left(x\right)=\left(\begin{array}{cc}
-\lambda+x\hat{b}_{1}\left(x\right) & x\hat{c}_{1}\left(x\right)\\
x\hat{c}_{2}\left(x\right) & \lambda+x\hat{b}_{2}\left(x\right)
\end{array}\right),\quad\mathbf{\hat{G}}\left(x,\mathbf{z}\left(x\right)\right)=\left(\begin{array}{c}
\hat{G}_{1}\left(x,\mathbf{z}\left(x\right)\right)\\
\hat{G}_{2}\left(x,\mathbf{z}\left(x\right)\right)
\end{array}\right)\qquad.
\]
We are looking for 
\[
{\displaystyle \mathbf{\hat{P}}\left(x\right)=\left(\begin{array}{cc}
1 & x\hat{p}_{2}\left(x\right)\\
x\hat{p}_{1}\left(x\right) & 1
\end{array}\right)\in\mbox{GL}_{2}\left(\form x\right)}\,\,,
\]
where $\hat{p}_{1},\hat{p}_{2}$ are 1-summable formal power series
in $x$ every direction $\theta\neq\arg\left(\pm\lambda\right)$,
such that the linear transformation given by $\mathbf{z}\left(x\right)=\hat{\mathbf{P}}\left(x\right)\mathbf{y}\left(x\right)$
changes equation (\ref{eq: systeme diff irreg}) to 
\[
x^{2}\ddd{\mathbf{y}}x\left(x\right)=\mathbf{\hat{A}}\left(x\right)\mathbf{y}\left(x\right)+\mathbf{\hat{H}}\left(x,\mathbf{y}\left(x\right)\right)\qquad,
\]
with 
\[
\hat{\mathbf{A}}\left(x\right)=\left(\begin{array}{cc}
-\lambda+x\hat{a}_{1}\left(x\right) & 0\\
0 & \lambda+x\hat{a}_{2}\left(x\right)
\end{array}\right),\quad\hat{\mathbf{H}}\left(x,\mathbf{y}\left(x\right)\right)=\left(\begin{array}{c}
\hat{H}_{1}\left(x,\mathbf{y}\left(x\right)\right)\\
\hat{H}_{2}\left(x,\mathbf{y}\left(x\right)\right)
\end{array}\right)\quad,
\]
where $\hat{a}_{1},\hat{a}_{2},\hat{H}_{1},\hat{H}_{2}$ are 1-summable
formal power series in $x$ every direction $\theta\neq\arg\left(\pm\lambda\right)$. 

We have

\[
x^{2}\ddd{\mathbf{y}}x\left(x\right)=\hat{\mathbf{P}}\left(x\right)^{-1}\left(\mathbf{\hat{B}}(x)\hat{\mathbf{P}}\left(x\right)-x^{2}\ddd{\hat{\mathbf{P}}}x\left(x\right)\right)\mathbf{y}\left(x\right)+\hat{\mathbf{P}}\left(x\right)^{-1}\mathbf{\hat{G}}\left(x,\hat{\mathbf{P}}\left(x\right)\mathbf{y}\left(x\right)\right)
\]
and we want to determine $\mathbf{\hat{A}}\left(x\right)$ and $\hat{\mathbf{P}}\left(x\right)$
as above so that 
\[
\mathbf{\hat{B}}(x)\hat{\mathbf{P}}\left(x\right)-x^{2}\ddd{\hat{\mathbf{P}}}x\left(x\right)=\mathbf{\hat{A}}\left(x\right)\qquad.
\]
This gives four equations: 
\begin{equation}
\begin{cases}
{\displaystyle \hat{a}_{1}\left(x\right)=\hat{b}_{1}\left(x\right)+x\hat{c}_{1}\left(x\right)\hat{p}_{1}\left(x\right)}\\
{\displaystyle \hat{a}_{2}\left(x\right)=\hat{b}_{2}\left(x\right)+x\hat{c}_{2}\left(x\right)\hat{p}_{2}\left(x\right)}\\
{\displaystyle x^{2}\ddd{\hat{p}_{1}}x\left(x\right)=\left(2\lambda+x\hat{b}_{2}\left(x\right)-x-x\hat{b}_{1}\left(x\right)\right)\hat{p}_{1}\left(x\right)+\hat{c}_{2}\left(x\right)-x^{2}\hat{c}_{1}\left(x\right)\hat{p}_{1}\left(x\right)^{2}}\\
{\displaystyle x^{2}\ddd{\hat{p}_{2}}x\left(x\right)=\left(-2\lambda+x\hat{b}_{1}\left(x\right)-x-x\hat{b}_{2}\left(x\right)\right)\hat{p}_{2}\left(x\right)+\hat{c}_{1}\left(x\right)-x^{2}\hat{c}_{2}\left(x\right)\hat{p}_{2}\left(x\right)^{2}}
\end{cases}\qquad.\label{eq: system 4 eq}
\end{equation}
Thanks to the theorem by Ramis and Sibuya on the summability of formal
solutions to irregular systems \cite{ramis1989hukuhara}, we have
the existence and the 1-summability in every direction $\theta\neq\arg\left(\pm\lambda\right)$,
of $\hat{p}_{1}$ and $\hat{p}_{2}$: $\left(\hat{p}_{1}\left(x\right),\hat{p}_{2}\left(x\right)\right)$
is defined as the unique formal solution to 
\[
\begin{cases}
{\displaystyle x^{2}\ddd{\hat{p}_{1}}x\left(x\right)=\left(2\lambda+x\hat{b}_{2}\left(x\right)-x-x\hat{b}_{1}\left(x\right)\right)\hat{p}_{1}\left(x\right)+\hat{c}_{2}\left(x\right)-x^{2}\hat{c}_{1}\left(x\right)\hat{p}_{1}\left(x\right)^{2}}\\
{\displaystyle x^{2}\ddd{\hat{p}_{2}}x\left(x\right)=\left(-2\lambda+x\hat{b}_{1}\left(x\right)-x-x\hat{b}_{2}\left(x\right)\right)\hat{p}_{2}\left(x\right)+\hat{c}_{1}\left(x\right)-x^{2}\hat{c}_{2}\left(x\right)\hat{p}_{2}\left(x\right)^{2}}
\end{cases}
\]
such that 
\[
{\displaystyle \left(\hat{p}_{1}\left(0\right),\hat{p}_{2}\left(0\right)\right)=\left(\frac{-\hat{c}_{2}\left(0\right)}{2\lambda},\frac{\hat{c}_{1}\left(0\right)}{2\lambda}\right)}\,\,.
\]
Notice that $\hat{a}_{1}$ and $\hat{a}_{2}$ are defined by the first
two equations in (\ref{eq: system 4 eq}). Finally, $\hat{\mathbf{H}}$
is defined by 
\[
{\displaystyle \hat{\mathbf{H}}\left(x,\mathbf{y}\right):=\hat{\mathbf{P}}\left(x\right)^{-1}\mathbf{\hat{G}}\left(x,\hat{\mathbf{P}}\left(x\right)\mathbf{y}\right)}\,\,,
\]
and it is 1-summable in every direction $\theta\neq\arg\left(\pm\lambda\right)$,
by Proposition \ref{prop: compositon summable}.
\end{proof}
The goal of the last following lemma is to transform $\hat{a}_{1}\left(x\right)$
and $\hat{a}_{2}\left(x\right)$ in (\ref{eq: Y semi-diag}) to constant
terms.
\begin{lem}
There exists a pair of formal power series $\left(\hat{q}_{1},\hat{q}_{2}\right)\in\left(\form x\right)^{2}$
with $\hat{q}_{1}\left(0\right)=\hat{q}_{2}\left(0\right)=1$, which
are 1-summable in every direction $\theta\neq\arg\left(\pm\lambda\right)$,
such that the formal fibered diffeomorphism of the form 
\[
\hat{\Phi}_{3}\left(x,y_{1},y_{2}\right)=\left(x,\hat{q}_{1}\left(x\right)y_{1},\hat{q}_{2}\left(x\right)y_{2}\right)\,\,,
\]
(which is tangent to the identity and 1-summable in every direction
$\theta\neq\arg\left(\pm\lambda\right)$) conjugates $\hat{Y}_{2}$
in (\ref{eq: Y semi-diag}), to
\begin{eqnarray*}
\hat{Y}_{3}\left(x,\mathbf{y}\right) & = & x^{2}\pp x+\left(\left(-\lambda+a_{1}x\right)y_{1}+\hat{F}_{1}\left(x,\mathbf{y}\right)\right)\pp{y_{1}}\\
 &  & +\left(\left(\lambda+a_{2}x\right)y_{2}+\hat{F}_{2}\left(x,\mathbf{y}\right)\right)\pp{y_{2}}\qquad,
\end{eqnarray*}
where $\hat{F}_{1},\hat{F}_{2}$ are formal power series of order
at least $2$ with respect to $\mathbf{y}$ which are 1-summable in
every direction $\theta\neq\arg\left(\pm\lambda\right)$ and $\left(a_{1},a_{2}\right)=\left(\hat{a}_{1}\left(0\right),\hat{a}_{2}\left(0\right)\right)$.
\end{lem}

\begin{proof}
We can associate to $\hat{Y}_{2}$ the following irregular differential
system: 
\[
x^{2}\ddd{\mathbf{z}}x\left(x\right)=\mathbf{\hat{A}}\left(x\right)\mathbf{z}\left(x\right)+\mathbf{\hat{H}}\left(x,\mathbf{z}\left(x\right)\right)\qquad,
\]
and we are looking for a change of coordinates of the form ${\displaystyle \mathbf{z}\left(x\right)=\hat{\mathbf{Q}}\left(x\right)\mathbf{y}\left(x\right)}$,
where 
\[
{\displaystyle \mathbf{\hat{Q}}\left(x\right)=\left(\begin{array}{cc}
\hat{q}_{1}\left(x\right) & 0\\
0 & \hat{q}_{2}\left(x\right)
\end{array}\right)\in\mbox{GL}_{2}\left(\form x\right)}
\]
 with $\hat{q}_{1}\left(0\right)=\hat{q}_{2}\left(0\right)=1$, such
that the new system is 
\[
x^{2}\ddd{\mathbf{y}}x\left(x\right)=\mathbf{A}\left(x\right)\mathbf{y}\left(x\right)+\mathbf{\hat{F}}\left(x,\mathbf{y}\left(x\right)\right)\qquad,
\]
with 
\[
\mathbf{A}\left(x\right)=\left(\begin{array}{cc}
-\lambda+a_{1}x & 0\\
0 & \lambda+a_{2}x
\end{array}\right),\quad\hat{\mathbf{F}}\left(x,\mathbf{y}\left(x\right)\right)=\left(\begin{array}{c}
\hat{F}_{1}\left(x,\mathbf{y}\left(x\right)\right)\\
\hat{F}_{2}\left(x,\mathbf{y}\left(x\right)\right)
\end{array}\right)\quad,
\]
and $\left(a_{1},a_{2}\right)=\left(\hat{a}_{1}\left(0\right),\hat{a}_{2}\left(0\right)\right)$. 

We have
\begin{eqnarray*}
x^{2}\ddd{\mathbf{y}}x\left(x\right) & = & \underbrace{\hat{\mathbf{Q}}\left(x\right){}^{-1}\left(\hat{\mathbf{A}}\left(x\right)\hat{\mathbf{Q}}\left(x\right)-x^{2}\ddd{\hat{\mathbf{Q}}}x\left(x\right)\right)}\mathbf{y}(x)+\hat{\mathbf{Q}}(x)^{-1}\hat{\mathbf{H}}\left(x,\hat{\mathbf{Q}}(x)\mathbf{y}(x)\right)\\
 &  & \qquad\,\qquad\qquad\quad=\\
 &  & \qquad\quad\left(\begin{array}{cc}
-\lambda+a_{1}x & 0\\
0 & \lambda+a_{2}x
\end{array}\right)
\end{eqnarray*}
so that

\[
x^{2}\ddd{\hat{\mathbf{Q}}}x(x)=\hat{\mathbf{A}}(x)\hat{\mathbf{Q}}(x)-\hat{\mathbf{Q}}(x)\left(\begin{array}{cc}
-\lambda+a_{1}x & 0\\
0 & \lambda+a_{2}x
\end{array}\right)\qquad
\]
and we obtain:

\begin{eqnarray*}
 & \begin{cases}
{\displaystyle x^{2}\ddd{\hat{q}_{1}}x(x)=x\hat{q}_{1}(x)\left(\hat{a}_{1}\left(x\right)-a_{1}\right)}\\
{\displaystyle x^{2}\ddd{\hat{q}_{2}}x(x)=x\hat{q}_{2}(x)\left(\hat{a}_{2}\left(x\right)-a_{2}\right)}
\end{cases}\\
\Longleftrightarrow & \begin{cases}
{\displaystyle \ddd{\hat{q}_{1}}x(x)=\hat{q}_{1}(x)\left(\frac{\hat{a}_{1}\left(x\right)-a_{1}}{x}\right)}\\
{\displaystyle \ddd{\hat{q}_{2}}x(x)=\hat{q}_{2}(x)\left(\frac{\hat{a}_{2}\left(x\right)-a_{2}}{x}\right)}
\end{cases}\\
\Longleftrightarrow & \begin{cases}
{\displaystyle \hat{q}_{1}(x)=\mbox{exp}\left(\int_{0}^{x}\frac{\hat{a}_{1}\left(s\right)-a_{1}}{s}\mbox{d}s\right)}\\
{\displaystyle \hat{q}_{2}(x)=\mbox{exp}\left(\int_{0}^{x}\frac{\hat{a}_{2}\left(s\right)-a_{2}}{s}\mbox{d}s\right)}
\end{cases} & ,\,\mbox{if we set }\hat{q}_{1}\left(0\right)=\hat{q}_{2}\left(0\right)=1\,\,,
\end{eqnarray*}
and the expression ${\displaystyle \int_{0}^{x}\frac{\hat{a}_{j}\left(s\right)-a_{j}}{s}\mbox{d}s}$,
for $j=1,2$, means the only anti-derivative of ${\displaystyle \frac{\hat{a}_{j}\left(s\right)-a_{j}}{s}}$
without constant term. Since $\hat{a}_{1}$ and $\hat{a}_{2}$ are
1-summable in every direction $\theta\neq\arg\left(\pm\lambda\right)$,
the same goes for $\hat{q}_{1}$ and $\hat{q}_{2}$, and then for
$\hat{F}_{1}$ and $\hat{F}_{2}$ by Proposition \ref{prop: compositon summable}.
\end{proof}

\subsubsection{\label{subsec:Proof-of-Proposition}Proof of Proposition \ref{prop: diag prep}.}

~

We are now able to prove Proposition \ref{prop: diag prep}.
\begin{proof}[Proof of Proposition \ref{prop: diag prep}\emph{.}]

We have to use successively Lemma \ref{lem: preparation sur hypersurface invariante}
(with $Y_{0}:=Y_{\mid\acc{x=0}}$), followed by Proposition \ref{prop: diag prep},
then Proposition \ref{prop: preparation} and finally Proposition
\ref{prop: ordre N avec eq homologique}, using at each step Corollary
\ref{cor: summability push-forward} to obtain the 1-summability.
\end{proof}

\subsection{1-summable straightening of two invariant hypersurfaces}

~

For any $\theta\in\ww R$, we recall that we denote by $F_{\theta}$
the 1-sum of a 1-summable series $\hat{F}$ in the direction $\theta$.

Let $\theta\in\ww R$ with $\theta\neq\arg\left(\pm\lambda\right)$
and consider a formal vector field $\hat{Y}$, 1-summable in the direction
$\theta$ of 1-sum $Y_{\theta}$, of the form 
\begin{equation}
\hat{Y}=x^{2}\pp x+\left(\lambda_{1}\left(x\right)y_{1}+\hat{F}_{1}\left(x,\mathbf{y}\right)\right)\pp{y_{1}}+\left(\lambda_{2}\left(x\right)y_{2}+\hat{F}_{2}\left(x,\mathbf{y}\right)\right)\pp{y_{2}}\qquad,\label{eq: Y semi prep}
\end{equation}
where:
\begin{itemize}
\item $\lambda_{1}\left(x\right)=-\lambda+a_{1}x$
\item $\lambda_{2}\left(x\right)=\lambda+a_{2}x$
\item $\lambda\neq0$
\item $a_{1},a_{2}\in\ww C$
\item for $j=1,2$, 
\[
{\displaystyle \hat{F}_{j}\left(x,\mathbf{y}\right)=\sum_{\substack{\mathbf{n}\in\ww N^{2}\\
\abs{\mathbf{n}}\geq2
}
}\hat{F}_{\mathbf{n}}^{\left(j\right)}\left(x\right)\mathbf{y^{n}}}\in\form{x,\mathbf{y}}
\]
 is 1-summable in the direction $\theta$ of 1-sum 
\[
{\displaystyle F_{j,\theta}\left(x,\mathbf{y}\right)=\sum_{\substack{\mathbf{n}\in\ww N^{2}\\
\abs{\mathbf{n}}\geq2
}
}F_{j,\mathbf{n},\theta}\left(x\right)\mathbf{y^{n}}}\,\,.
\]
In particular, there exists $A,B,\mu>0$ such that for all $\mathbf{n}\in\ww N^{2}$,
$\abs n\geq2$, for $j=1,2$: 
\[
\forall t\in\Delta_{\theta,\epsilon,\rho},\,\abs{\widetilde{\cal B}\left(\hat{F}_{j,\mathbf{n}}\right)\left(t\right)}\leq A.B^{\abs{\mathbf{n}}}\frac{\exp\left(\mu\abs t\right)}{1+\mu^{2}\abs t^{2}}\qquad,
\]
for some $\rho>0$ and $\epsilon>0$ such that $\left(\ww R.\lambda\right)\cap\cal A_{\theta,\epsilon}=\emptyset$
(see Definition \ref{def: 1_summability} and Remark \ref{rem: norme bis borel}
for the notations). Notice that $F_{j,\theta}$ is analytic and bounded
in some sectorial neighborhood $\cal S\in{\cal S}_{\theta,\pi}$ of
the origin. For technical reasons, we use in this subsection the alternative
definition of the Borel transform $\widetilde{\cal B}$, with its
associate norm $\norm{\cdot}_{\mu}^{\tx{bis}}$ (see Remarks \ref{rem: deifinition bis transformee de Borel}
and \ref{rem: norme bis borel} and Proposition \ref{prop: norme d'algebre}
\end{itemize}
\begin{prop}
\label{prop: preparation}Under the assumptions above, there exists
a pair of formal power series $\left(\hat{\phi}_{1},\hat{\phi}_{2}\right)\in\left(\form{x,\mathbf{y}}\right)^{2}$
of order at least two with respect to $\mathbf{y}$ which are 1-summable
in every direction $\theta\neq\arg\left(\pm\lambda\right)$, such
that the formal fibered diffeomorphism 
\[
\hat{\Phi}\left(x,\mathbf{y}\right)=\left(x,y_{1}+\hat{\phi}_{1}\left(x,\mathbf{y}\right),y_{2}+\hat{\phi}_{2}\left(x,\mathbf{y}\right)\right)\qquad,
\]
(which is tangent to the identity and 1-summable in every direction
$\theta\neq\arg\left(\pm\lambda\right)$) conjugates $\hat{Y}$ in
(\ref{eq: Y semi prep}) to 
\[
\hat{Y}_{\tx{prep}}=x^{2}\pp x+\left(\left(-\lambda+a_{1}x\right)+y_{2}\hat{R}_{1}\left(x,\mathbf{y}\right)\right)y_{1}\pp{y_{1}}+\left(\left(\lambda+a_{2}x\right)+y_{1}\hat{R}_{2}\left(x,\mathbf{y}\right)\right)y_{2}\pp{y_{2}}\qquad,
\]
where ${\displaystyle \hat{R}_{1},\hat{R}_{2}\in\form{x,\mathbf{y}}}$
are 1-summable in every direction $\theta\neq\arg\left(\pm\lambda\right)$.
\end{prop}

\begin{proof}
We follow and adapt the proof of analytic straightening of invariant
curves for resonant saddles in two dimensions in \cite{MatteiMoussu}.

We are looking for 
\[
\hat{\Psi}\left(x,\mathbf{y}\right)=\left(x,y_{1}+\hat{\psi}_{1}\left(x,\mathbf{y}\right),y_{2}+\hat{\psi}_{2}\left(x,\mathbf{y}\right)\right)\quad,
\]
with $\hat{\psi}_{1},\hat{\psi}_{2}$ of order at least 2, and $\hat{R}_{1},\hat{R}_{2}$
as above such that: 
\[
\hat{\Psi}_{*}\left(\hat{Y}_{\tx{prep}}\right)=\hat{Y}\qquad,
\]
\emph{i.e. 
\begin{eqnarray}
\mbox{D}\hat{\Psi}\cdot\hat{Y}_{\tx{prep}} & = & \hat{Y}\circ\hat{\Psi}\qquad.\label{eq: conjugaison forme prep}
\end{eqnarray}
}Then, we will set $\Phi:=\Psi^{-1}$. Let us write 
\begin{eqnarray*}
\hat{T}_{1} & := & y_{1}y_{2}\hat{R}_{1}=\sum_{\abs{\mathbf{n}}\geq2}\hat{T}_{1,\mathbf{n}}\left(x\right)\mathbf{y^{n}}\\
\hat{T}_{2} & := & y_{1}y_{2}\hat{R}_{2}=\sum_{\abs{\mathbf{n}}\geq2}\hat{T}_{2,\mathbf{n}}\left(x\right)\mathbf{y^{n}}\qquad\\
\hat{\psi}_{1} & = & \sum_{\abs{\mathbf{n}}\geq2}\hat{\psi}_{1,\mathbf{n}}\left(x\right)\mathbf{y^{n}}\qquad\\
\hat{\psi}_{2} & = & \sum_{\abs{\mathbf{n}}\geq2}\hat{\psi}_{2,\mathbf{n}}\left(x\right)\mathbf{y^{n}}\qquad,
\end{eqnarray*}
so that equation $\left(\mbox{\ref{eq: conjugaison forme prep}}\right)$
becomes:
\begin{eqnarray*}
 &  & x^{2}\ppp{\hat{\psi}_{1}}{x^{2}}+\left(1+\ppp{\hat{\psi}_{1}}{y_{1}}\right)\left(\lambda_{1}\left(x\right)y_{1}+\hat{T}_{1}\right)+\ppp{\hat{\psi}_{1}}{y_{2}}\left(\lambda_{2}\left(x\right)y_{2}+\hat{T}_{2}\right)\\
 &  & =\lambda_{1}\left(x\right)\left(y_{1}+\hat{\psi}_{1}\right)+\hat{F}_{1}\left(x,y_{1}+\hat{\psi}_{1},y_{2}+\hat{\psi}_{2}\right)
\end{eqnarray*}
and 
\begin{eqnarray*}
 &  & x^{2}\ppp{\hat{\psi}_{2}}{x^{2}}+\ppp{\hat{\psi}_{2}}{y_{1}}\left(\lambda_{1}\left(x\right)y_{1}+\hat{T}_{1}\right)+\left(1+\ppp{\hat{\psi}_{2}}{y_{2}}\right)\left(\lambda_{2}\left(x\right)y_{2}+\hat{T}_{2}\right)\\
 &  & =\lambda_{2}\left(x\right)\left(y_{2}+\hat{\psi}_{2}\right)+\hat{F}_{2}\left(x,y_{1}+\hat{\psi}_{1},y_{2}+\hat{\psi}_{2}\right)\,\,\,.
\end{eqnarray*}
These equations can be written as:
\begin{equation}
\begin{cases}
{\displaystyle \underset{\left|\mathbf{n}\right|\geq2}{\sum}\left(\delta_{1,\mathbf{n}}(x)\hat{\psi}_{1,\mathbf{n}}\left(x\right)+x^{2}\ddd{\hat{\psi}_{1,\mathbf{n}}}x(x)+\hat{T}_{1,\mathbf{n}}\left(x\right)\right)\mathbf{y}^{\mathbf{n}}}\\
{\displaystyle =\hat{F}_{1}\left(x,y_{1}+\hat{\psi}_{1}\left(x,\mathbf{y}\right),y_{2}+\hat{\psi}_{2}\left(x,\mathbf{y}\right)\right)-\hat{T}_{1}(x)\frac{\partial\hat{\psi}_{1}}{\partial y_{1}}(x,\mathbf{y})-\hat{T}_{2}(x)\frac{\partial\hat{\psi}_{1}}{\partial y_{2}}(x,\mathbf{y})}\\
{\displaystyle =:\underset{\left|\mathbf{n}\right|\geq2}{\sum}\zeta_{1,\mathbf{n}}(x)\mathbf{y}^{\mathbf{n}}}\\
{\displaystyle \underset{\left|\mathbf{n}\right|\geq2}{\sum}\left(\delta_{2,\mathbf{n}}(x)\hat{\psi}_{2,\mathbf{n}}\left(x\right)+x^{2}\ddd{\hat{\psi}_{2,\mathbf{n}}}x(x)+\hat{T}_{2,\mathbf{n}}\left(x\right)\right)\mathbf{y}^{\mathbf{n}}}\\
{\displaystyle =\hat{F}_{2}\left(x,y_{1}+\hat{\psi}_{1}\left(x,\mathbf{y}\right),y_{2}+\hat{\psi}_{2}\left(x,\mathbf{y}\right)\right)-\hat{T}_{1}(x)\frac{\partial\hat{\psi}_{2}}{\partial y_{1}}(x,\mathbf{y})-\hat{T}_{2}(x)\frac{\partial\hat{\psi}_{2}}{\partial y_{2}}(x,\mathbf{y})}\\
{\displaystyle =:\underset{\left|\mathbf{n}\right|\geq2}{\sum}\zeta_{2,\mathbf{n}}(x)\mathbf{y}^{\mathbf{n}}}
\end{cases}\label{eq: =0000E9quation conjugaison serie}
\end{equation}
where ${\displaystyle \delta_{j,\mathbf{n}}(x)=\lambda_{1}(x)n_{1}+\lambda_{2}(x)n_{2}-\lambda_{j}(x)}$,
$j=1,2$. We are looking for $\hat{T}_{1},\hat{T}_{2}$ such that
\[
\begin{cases}
\hat{T}_{1,\mathbf{n}}=0 & \mbox{, if }n_{1}=0\mbox{ or }n_{2}=0\\
\hat{T}_{2,\mathbf{n}}=0 & \mbox{, if }n_{1}=0\mbox{ or }n_{2}=0
\end{cases}\mbox{ }\qquad.
\]
Notice that $\zeta_{j,\mathbf{n}}$, for $j=1,2$ and $\abs{\mathbf{n}}\geq2$,
depends only on the $\hat{\psi}_{i,\mathbf{k}}\mbox{'s}$ and the
$\hat{F}_{i,\mathbf{k}}\mbox{'s}$, for $i=1,2$, $\abs{\mathbf{k}}<\mathbf{n}$.
We can then determine the coefficients $\hat{\psi}_{j,\mathbf{n}}$
and $\hat{T}_{j,\mathbf{n}}$, $j=1,2$, $\abs{\mathbf{n}}\geq2$,
by induction on $\abs{\mathbf{n}}$, setting 
\[
\begin{cases}
\hat{T}_{1,\mathbf{n}}=0 & \mbox{, if }n_{1}=0\mbox{ or }n_{2}=0\\
\hat{T}_{2,\mathbf{n}}=0 & \mbox{, if }n_{1}=0\mbox{ or }n_{2}=0\\
\hat{\psi}_{1,\mathbf{n}}=0 & ,\mbox{ if }n_{1}\geq1\mbox{ and }n_{2}\geq1\\
\hat{\psi}_{2,\mathbf{n}}=0 & ,\mbox{ if }n_{1}\geq1\mbox{ and }n_{2}\geq1
\end{cases}\qquad,
\]
and solving for each $\mathbf{n}=\left(n_{1},n_{2}\right)\in\ww N^{2}$
with $\abs{\mathbf{n}}\geq2$, the equations
\[
\begin{cases}
{\displaystyle \delta_{1,\mathbf{n}}(x)\hat{\psi}_{1,\mathbf{n}}\left(x\right)+x^{2}\ddd{\hat{\psi}_{1,\mathbf{n}}}x(x)=\zeta_{1,\mathbf{n}}\left(x\right)} & \mbox{, if }n_{1}=0\mbox{ or }n_{2}=0\\
{\displaystyle \delta_{2,\mathbf{n}}(x)\hat{\psi}_{2,\mathbf{n}}\left(x\right)+x^{2}\ddd{\hat{\psi}_{2,\mathbf{n}}}x(x)=\zeta_{2,\mathbf{n}}\left(x\right)} & \mbox{, if }n_{1}=0\mbox{ or }n_{2}=0
\end{cases}\,\,\,\,\,.
\]

\begin{lem}
There exists $\beta>4\pi,M>0$ such that for all $\mathbf{n}\in\ww N^{2}$
with $\abs{\mathbf{n}}\geq2$, and for $j=1,2$, $\norm{\zeta_{j,\mathbf{n}}}_{\beta}^{\tx{bis}}<+\infty$
and: 
\[
\norm{\hat{\psi}_{j,\mathbf{n}}}_{\beta}^{\tx{bis}}\leq M.\norm{\zeta_{j,\mathbf{n}}}_{\beta}^{\tx{bis}}\qquad,
\]
where the norm corresponds to the domain $\triangle_{\theta,\epsilon,\rho}$
(see Definition \ref{def: 1_summability}).
\end{lem}

\begin{proof}
For $\mathbf{n}=\left(n_{1},n_{2}\right)\in\ww N^{2}$ with $n_{1}+n_{2}\geq2$
we want to solve: 
\[
\begin{cases}
{\displaystyle \delta_{1,\mathbf{n}}(x)=\lambda_{1}(x)\left(n_{1}-1\right)+\lambda_{2}(x)n_{2}=} & \begin{cases}
\lambda\left(n_{2}+1\right)+x\left(-a_{1}+a_{2}n_{2}\right) & \mbox{ , if }n_{1}=0\\
-\lambda\left(n_{1}-1\right)+a_{1}x\left(n_{1}-1\right) & \mbox{ , if }n_{2}=0
\end{cases}\\
{\displaystyle \delta_{2,\mathbf{n}}(x)=\lambda_{2}(x)\left(n_{2}-1\right)+\lambda_{1}(x)n_{1}=} & \begin{cases}
\lambda\left(n_{2}-1\right)+a_{2}x\left(n_{2}-1\right) & \mbox{ , if }n_{1}=0\\
-\lambda\left(n_{1}+1\right)+x\left(-a_{2}+a_{1}n_{1}\right) & \mbox{ , if }n_{2}=0\,\,.
\end{cases}
\end{cases}
\]
We will only deal with $\delta_{1,\mathbf{n}}(x)$ (the case of $\delta_{2,\mathbf{n}}(x)$
being similar). Notice that we are exactly in the situation of Proposition
\ref{prop: solution borel sommable precise}. In particular, using
notation in this definition, we respectively have:

\begin{eqnarray*}
 &  & \begin{cases}
{\displaystyle k=\lambda\left(n_{2}+1\right),\,\alpha=\frac{\left(-a_{1}+a_{2}n_{2}\right)}{\lambda\left(n_{2}+1\right)},}\\
{\displaystyle d_{k}=\min\acc{\abs{\lambda\left(n_{2}+1\right)}-\rho,\abs{\lambda\left(n_{2}+1\right)}\abs{\sin\left(\theta+\epsilon\right)},\abs{\lambda\left(n_{2}+1\right)}\abs{\sin\left(\theta-\epsilon\right)}}}
\end{cases}\\
 &  & (\mbox{when }n_{1}=0)
\end{eqnarray*}
and

\begin{eqnarray*}
 &  & \begin{cases}
{\displaystyle k=-\lambda\left(n_{1}+1\right),\,\alpha=\frac{\left(-a_{2}+a_{1}n_{1}\right)}{-\lambda\left(n_{1}+1\right)},}\\
{\displaystyle d_{k}=\min\acc{\abs{\lambda\left(n_{1}+1\right)}-\rho,\abs{\lambda\left(n_{1}+1\right)}\abs{\sin\left(\theta+\epsilon\right)},\abs{\lambda\left(n_{1}+1\right)}\abs{\sin\left(\theta-\epsilon\right)}}}
\end{cases}\\
 &  & (\mbox{when }n_{2}=0).
\end{eqnarray*}
We can chose the domain $\Delta_{\theta,\epsilon,\rho}$ corresponding
to the 1-summability of $\hat{F}_{1}$ and $\hat{F}_{2}$ with $0<\rho<\abs{\lambda}$,
so that $d_{k}>0$, since $\epsilon>0$ is such that $\left(\ww R.\lambda\right)\cap\cal A_{\theta,\epsilon}=\emptyset$.
Finally, we chose 
\[
\beta>\frac{C\left(\abs{a_{1}}+\abs{a_{2}}\right)}{\min\acc{\abs{\lambda}-\rho,\abs{\lambda\sin\left(\theta+\epsilon\right)},\abs{\lambda\sin\left(\theta-\epsilon\right)}}}>0\,\,,
\]
(with $C=\frac{2\exp\left(2\right)}{5}+5)$, so that $\norm{\hat{F}_{1}}_{\beta}^{\tx{bis}}<+\infty$.
This choice of $\beta$ implies $\beta d_{k}>C\abs{\alpha k}$ as
needed in Proposition \ref{prop: solution borel sommable precise},
in both considered situations, namely $n_{1}=0$ and $n_{2}=0$ respectively.
Since for $j=1,2$ and $\abs{\mathbf{n}}\geq2$, $\zeta_{j,\mathbf{n}}$
depends only on the $\hat{\psi}_{i,\mathbf{k}}$'s and the $\hat{F}_{i,\mathbf{k}}$'s,
for $i=1,2$, $\abs{\mathbf{k}}<\mathbf{n}$, we deduce by induction
that 
\[
\begin{cases}
{\displaystyle \norm{\zeta_{1,\mathbf{n}}}_{\beta}^{\tx{bis}}<+\infty} & \mbox{ , if }n_{1}=0\mbox{ or }n_{2}=0\\
{\displaystyle \norm{\zeta_{2,\mathbf{n}}}_{\beta}^{\tx{bis}}<+\infty} & \mbox{ , if }n_{1}=0\mbox{ or }n_{2}=0
\end{cases}
\]
and then, thanks to Proposition \ref{prop: solution borel sommable precise}:
\[
{\displaystyle \norm{\hat{\psi}_{j,\mathbf{n}}}_{\beta}^{\tx{bis}}\leq\left(\frac{\beta}{\beta\left(\abs{\lambda}-\rho\right)-C\left(\abs{a_{1}}+\abs{a_{2}}\right)}\right).\norm{\zeta_{j,\mathbf{n}}}_{\beta}^{\tx{bis}}}\,\,,\mbox{ for \ensuremath{j=1,2}.}
\]
The lemma is proved, with 
\[
M=\left(\frac{\beta}{\beta\min\acc{\abs{\lambda}-\rho,\abs{\lambda\sin\left(\theta+\epsilon\right)},\abs{\lambda\sin\left(\theta-\epsilon\right)}}-C\left(\abs{a_{1}}+\abs{a_{2}}\right)}\right)\,\,.
\]
\end{proof}
In order to finish the proof of Proposition \ref{prop: preparation},
we have to prove that for $j=1,2$, the series ${\displaystyle \overline{\hat{\psi}_{j}}:=\sum_{\mathbf{n}\in\ww N^{2}}\norm{\hat{\psi}_{j,\mathbf{n}}}_{\beta}^{\tx{bis}}\mathbf{y}^{\mathbf{n}}}$
is convergent in a poly-disc $\mathbf{D\left(0,r\right)}$, with ${\displaystyle \mathbf{r}=\left(r_{1},r_{2}\right)\in\left(\ww R_{>0}\right)^{2}}$
(then, Corollary \ref{cor: faible et forte sommabilite} gives 1-summability).
We will prove this by using a method of dominant series. Let us introduce
some useful notations. If $\left(\mathfrak{B},\norm{\cdot}\right)$
is a Banach algebra, for any formal power series ${\displaystyle f\left(\mathbf{y}\right)=\sum_{\mathbf{n}}f_{\mathbf{n}}\mathbf{y^{n}}}\in\mathfrak{B}\left\llbracket \mathbf{y}\right\rrbracket $,
we define ${\displaystyle \overline{f}:=\sum_{\mathbf{n}}\norm{f_{\mathbf{n}}}\mathbf{y^{n}}},$
and $\overline{\overline{f}}\left(y\right):=\overline{f}\left(y,y\right)$.
If ${\displaystyle g=\sum_{\mathbf{n}}g_{\mathbf{n}}\mathbf{y^{n}}}\in\mathfrak{B}\left\llbracket \mathbf{y}\right\rrbracket $
is another formal power series, we write $\overline{f}\prec\overline{g}$
if for all $\mathbf{n}\in\ww N^{2}$, we have $\norm{f_{\mathbf{n}}}\leq\norm{g_{\mathbf{n}}}$.
We remind the following classical result (the proof is performed in
\cite{rebelo_reis} when $\left(\mathfrak{B},\norm{\cdot}\right)=\left(\ww C,\abs{\cdot}\right)$
, but the same proof works for any Banach algebra).

\begin{lem}
\cite[Theorem 2.2 p.48]{rebelo_reis} For $j=1,2$, let ${\displaystyle f_{j}=\sum_{\abs{\mathbf{n}}\geq2}f_{j,\mathbf{n}}\mathbf{y^{n}}}\in\mathfrak{B}\left\llbracket \mathbf{y}\right\rrbracket $
be two formal power series with coefficients in a Banach algebra $\left(\mathfrak{B},\norm{\cdot}\right)$,
and of order at least two. Consider also two other series ${\displaystyle g_{j}=\sum_{\abs{\mathbf{n}}\geq2}g_{j,\mathbf{n}}\mathbf{y^{n}}}\in\mathfrak{B}\acc{\ww{\mathbf{y}}}$,
$j=1,2$, of order at least two, which have a non-zero radius of convergence
at the origin. Assume that there exists $\sigma>0$ such that for
$j=1,2$: 
\[
\sigma\overline{f_{j}}\prec\overline{g_{j}}\left(y_{1}+\overline{f_{1}},y_{2}+\overline{f_{2}}\right)\qquad.
\]
 Then, $f_{1}$ and $f_{2}$ have a non-zero radius of convergence.
\end{lem}

Taking $\beta>4\pi$, according to Proposition \ref{prop: norme d'algebre},
for all $\hat{f},\hat{g}\in\mathfrak{B}_{\beta}^{\tx{bis}}$, we have:
\[
\norm{\hat{f}\hat{g}}_{\beta}^{\tx{bis}}\leq\norm{\hat{f}}_{\beta}^{\tx{bis}}\norm{\hat{g}}_{\beta}^{\tx{bis}}\,\,.
\]
This implies that $\left(\mathfrak{B}_{\beta}^{\tx{bis}},\norm{\cdot}_{\beta}^{\tx{bis}}\right)$
is a Banach algebra as needed in the above lemma. It remains to prove
that there exists $\sigma>0$ such that for $j=1,2$:
\[
\sigma\overline{\hat{\psi}_{j}}\prec\overline{\hat{F}_{j}}\left(y_{1}+\overline{\hat{\psi}_{1}},y_{2}+\overline{\hat{\psi}_{2}}\right)\qquad.
\]
Remember that there exists $M>0$ such that for $j=1,2$:
\[
\norm{\hat{\psi}_{j,\mathbf{n}}}_{\beta}^{\tx{bis}}\leq M.\norm{\zeta_{j,\mathbf{n}}}_{\beta}^{\tx{bis}}\qquad
\]
where
\[
\begin{cases}
{\displaystyle \zeta_{1}:=\underset{\left|\mathbf{n}\right|\geq2}{\sum}\zeta_{1,\mathbf{n}}(x)\mathbf{y}^{\mathbf{n}}}\\
{\displaystyle =\hat{F}_{1}\left(x,y_{1}+\hat{\psi}_{1}\left(x,\mathbf{y}\right),y_{2}+\hat{\psi}_{2}\left(x,\mathbf{y}\right)\right)-\hat{T}_{1}(x)\frac{\partial\hat{\psi}_{1}}{\partial y_{1}}(x,\mathbf{y})-\hat{T}_{2}(x)\frac{\partial\hat{\psi}_{1}}{\partial y_{2}}(x,\mathbf{y})}\\
{\displaystyle \zeta_{2}:=\underset{\left|\mathbf{n}\right|\geq2}{\sum}\zeta_{2,\mathbf{n}}(x)\mathbf{y}^{\mathbf{n}}}\\
{\displaystyle =\hat{F}_{2}\left(x,y_{1}+\hat{\psi}_{1}\left(x,\mathbf{y}\right),y_{2}+\hat{\psi}_{2}\left(x,\mathbf{y}\right)\right)-\hat{T}_{1}(x)\frac{\partial\hat{\psi}_{2}}{\partial y_{1}}(x,\mathbf{y})-\hat{T}_{2}(x)\frac{\partial\hat{\psi}_{2}}{\partial y_{2}}(x,\mathbf{y})} & \,\,\,.
\end{cases}
\]
If we set $\sigma:=\frac{1}{M}$, then we have
\[
\begin{cases}
{\displaystyle \sigma\overline{\hat{\psi}_{1}}\prec\overline{\zeta_{1}}\prec\overline{\hat{F}}_{1}\left(x,y_{1}+\overline{\hat{\psi}_{1}}\left(x,\mathbf{y}\right),y_{2}+\overline{\hat{\psi}_{2}}\left(x,\mathbf{y}\right)\right)+\overline{\hat{T}_{1}}(x)\frac{\partial\overline{\hat{\psi}_{1}}}{\partial y_{1}}(x,\mathbf{y})+\overline{\hat{T}_{2}}(x)\frac{\partial\overline{\hat{\psi}_{1}}}{\partial y_{2}}(x,\mathbf{y})}\\
{\displaystyle \sigma\overline{\hat{\psi}_{2}}\prec\overline{\zeta_{2}}\prec\overline{\hat{F}}_{2}\left(x,y_{1}+\overline{\hat{\psi}_{1}}\left(x,\mathbf{y}\right),y_{2}+\overline{\hat{\psi}_{2}}\left(x,\mathbf{y}\right)\right)+\overline{\hat{T}_{1}}(x)\frac{\partial\overline{\hat{\psi}_{2}}}{\partial y_{1}}(x,\mathbf{y})+\overline{\hat{T}_{2}}(x)\frac{\partial\overline{\hat{\psi}_{2}}}{\partial y_{2}}(x,\mathbf{y})}
\end{cases}\qquad.
\]
Moreover, we recall that 
\[
\begin{cases}
\hat{T}_{1,\mathbf{n}}=0 & \mbox{, if }n_{1}=0\mbox{ or }n_{2}=0\\
\hat{T}_{2,\mathbf{n}}=0 & \mbox{, if }n_{1}=0\mbox{ or }n_{2}=0\\
\hat{\psi}_{1,\mathbf{n}}=0 & ,\mbox{ if }n_{1}\geq1\mbox{ and }n_{2}\geq1\\
\hat{\psi}_{2,\mathbf{n}}=0 & ,\mbox{ if }n_{1}\geq1\mbox{ and }n_{2}\geq1
\end{cases}\qquad,
\]
so that we have in fact more precise dominant relations: 
\[
\begin{cases}
{\displaystyle \sigma\overline{\hat{\psi}_{1}}\prec\overline{\zeta_{1}}\prec\overline{\hat{F}}_{1}\left(x,y_{1}+\overline{\hat{\psi}_{1}}\left(x,\mathbf{y}\right),y_{2}+\overline{\hat{\psi}_{2}}\left(x,\mathbf{y}\right)\right)}\\
{\displaystyle \sigma\overline{\hat{\psi}_{2}}\prec\overline{\zeta_{2}}\prec\overline{\hat{F}}_{2}\left(x,y_{1}+\overline{\hat{\psi}_{1}}\left(x,\mathbf{y}\right),y_{2}+\overline{\hat{\psi}_{2}}\left(x,\mathbf{y}\right)\right)}
\end{cases}\qquad.
\]
It remains the apply the lemma above to conclude.
\end{proof}
\begin{rem}
\label{rem: redressement hypersurfaces cas ham}In the previous proposition,
assume that for $j=1,2$, 
\[
{\displaystyle \hat{F}_{j}\left(x,\mathbf{y}\right)=\sum_{\mathbf{n}\in\ww N^{2},\,\abs{\mathbf{n}}\geq2}\hat{F}_{\mathbf{n}}^{\left(j\right)}\left(x\right)\mathbf{y^{n}}}
\]
 in the expression of $\hat{Y}$ satisfies
\[
\begin{cases}
\hat{F}_{\mathbf{n}}^{\left(1\right)}\left(0\right)=0 & ,\,\forall\mathbf{n}=\left(n_{1},n_{2}\right)\mid n_{1}+n_{2}\geq2\mbox{ and }\big(n_{1}=0\mbox{ or }n_{2}=0\big)\\
\hat{F}_{\mathbf{n}}^{\left(2\right)}\left(0\right)=0 & ,\,\forall\mathbf{n}=\left(n_{1},n_{2}\right)\mid n_{1}+n_{2}\geq2\mbox{ and }\big(n_{1}=0\mbox{ or }n_{2}=0
\end{cases}\,\,\,.
\]
Then, the diffeomorphism $\hat{\Phi}$ in the proposition can be chosen
to be the identity on $\acc{x=0}$, so that
\[
\begin{cases}
y_{1}y_{2}\hat{R}_{1}\left(x,\mathbf{y}\right) & =\hat{F}_{1}\left(0,\mathbf{y}\right)+x\hat{S}_{1}\left(x,\mathbf{y}\right)\\
y_{1}y_{2}\hat{R}_{2}\left(x,\mathbf{y}\right) & =\hat{F}_{2}\left(0,\mathbf{y}\right)+x\hat{S}_{2}\left(x,\mathbf{y}\right)\qquad,
\end{cases}
\]
 where $\hat{S}_{1},\hat{S}_{2}$ are 1-summable in the direction
$\theta\neq\arg\left(\pm\lambda\right)$ and ${\displaystyle \hat{F}_{1}\left(0,\mathbf{y}\right),\hat{F}_{2}\left(0,\mathbf{y}\right)\in\germ{\mathbf{y}}}$
are convergent in neighborhood of the origin in $\ww C^{2}$. Indeed,
we easily see by induction on $\abs{\mathbf{n}}=n_{1}+n_{2}\geq2$
that $\hat{\psi}_{1}$ and $\hat{\psi}_{2}$ can be chosen ``divisible''
by $x$, and that $\zeta_{1},\zeta_{2}$ are such that $\zeta_{j,\mathbf{n}}\left(x\right)$
is also ``divisible'' by $x$ if $n_{1}=0$ or $n_{2}=0$.
\end{rem}

\subsection{\label{subsec:Summable-normal-form}1-summable normal form up to
arbitrary order $N$}

~

We consider now a (formal) non-degenerate diagonal doubly-resonant
saddle node, which is supposed to be div-integrable and 1-summable
in every direction $\theta\neq\arg\left(\pm\lambda\right)$, of the
form 
\begin{eqnarray*}
\hat{Y}_{\tx{prep}} & = & x^{2}\pp x+\left(-\lambda+a_{1}x-d\left(y_{1}y_{2}\right)+x\hat{S}_{1}\left(x,\mathbf{y}\right)\right)y_{1}\pp{y_{1}}\\
 &  & +\left(\lambda+a_{2}x+d\left(y_{1}y_{2}\right)+x\hat{S}_{2}\left(x,\mathbf{y}\right)\right)y_{2}\pp{y_{2}}\qquad,
\end{eqnarray*}
where:
\begin{itemize}
\item $\lambda\in\ww C\backslash\acc 0$;
\item $\hat{S}_{1},\hat{S}_{2}\in\form{x,\mathbf{y}}$ are of order at least
one with respect to $\mathbf{y}$ and 1-summable in every direction
$\theta\in\ww R$ with $\theta\neq\arg\left(\pm\lambda\right)$;
\item $a:=\tx{res}\left(\hat{Y}_{\tx{prep}}\right)=a_{1}+a_{2}\notin\ww Q_{\leq0}$
;
\item $d\left(v\right)\in v\ww C\acc v$ is the germ of an analytic function
in $v:=y_{1}y_{2}$ vanishing at the origin.
\end{itemize}
As usual, we denote by $Y_{\tx{prep},\theta},S_{1,\theta},S_{2,\theta}$
the respective 1-sums of $\hat{Y},\hat{S}_{1},\hat{S}_{2}$ in the
direction $\theta$. Let us introduce some useful notations:

\begin{eqnarray*}
\hat{Y}_{\tx{prep}} & = & Y_{0}+D\overrightarrow{\mathcal{C}}+R\overrightarrow{\mathcal{R}}\qquad,
\end{eqnarray*}
where
\begin{itemize}
\item $\overrightarrow{\cal C}:=-y_{1}\pp{y_{1}}+y_{2}\pp{y_{2}}$
\item $\overrightarrow{\cal R}:=y_{1}\pp{y1}+y_{2}\pp{y_{2}}$
\item $Y_{0}:=\lambda\overrightarrow{\cal C}+x\left(x\pp x+a_{1}y_{1}\pp{y_{1}}+a_{2}y_{2}\pp{y2}\right)$ 
\item ${\displaystyle D\left(x,\mathbf{y}\right)=d\left(y_{1}y_{2}\right)+xD^{\left(1\right)}\left(x,\mathbf{y}\right)=d\left(y_{1}y_{2}\right)+x\left(\frac{\hat{S}_{2}-\hat{S}_{1}}{2}\right)}$
is 1-summable in the direction $\theta$ of 1-sum $D_{\theta}:$ it
is called the ``\emph{tangential}'' part. $D_{\theta}$ is also
dominated by $\norm{\mathbf{y}}=\max\left(\abs{y_{1}},\abs{y_{2}}\right)$
($D$ is of order at least one with respect to $\mathbf{y}$).
\item ${\displaystyle R\left(x,\mathbf{y}\right)=xR^{\left(1\right)}\left(x,\mathbf{y}\right)=x\left(\frac{\hat{S}_{2}+\hat{S}_{1}}{2}\right)}$
is 1-summable in the direction $\theta$ of 1-sum $R_{\theta}$: it
is called the ``\emph{radial}'' part. $R_{\theta}$ is also dominated
by $\norm{\mathbf{y}}_{\infty}=\max\left(\abs{y_{1}},\abs{y_{2}}\right)$
($R$ is of order at least one with respect to $\mathbf{y}$).
\end{itemize}
The following proposition gives the existence of a 1-summable normalizing
map, up to any order $N\in\ww N_{>0}$, with respect to $x$.
\begin{prop}
\label{prop: ordre N avec eq homologique}Let 
\begin{eqnarray*}
\hat{Y}_{\tx{prep}} & = & Y_{0}+D\overrightarrow{\mathcal{C}}+R\overrightarrow{\mathcal{R}}\qquad
\end{eqnarray*}
be as above. 

Then for all $N\in\ww N_{>0}$ there exist $d^{\left(N\right)}\left(v\right)\in\ww C\acc v$
of order at least one and $\Phi^{\left(N\right)}\in\fdiff[\ww C^{3},0,\tx{Id}]$
which conjugates $\hat{Y}_{\tx{prep}}$ (\emph{resp. }its 1-sums $Y_{\tx{prep},\theta}$
in the direction $\theta$) to 
\begin{eqnarray*}
Y^{\left(N\right)} & = & Y_{0}+\left(d^{\left(N\right)}\left(y_{1}y_{2}\right)+x^{N}D^{\left(N\right)}\left(x,\mathbf{y}\right)\right)\overrightarrow{\cal C}+x^{N}R^{\left(N\right)}\left(x,\mathbf{y}\right)\overrightarrow{\cal R}\\
\bigg(\emph{resp. }\,Y_{\theta}^{\left(N\right)} & = & Y_{0}+\left(d^{\left(N\right)}\left(y_{1}y_{2}\right)+x^{N}D_{\theta}^{\left(N\right)}\left(x,\mathbf{y}\right)\right)\overrightarrow{\cal C}+x^{N}R_{\theta}^{\left(N\right)}\left(x,\mathbf{y}\right)\overrightarrow{\cal R}\bigg)\qquad,
\end{eqnarray*}
where $D^{\left(N\right)},R^{\left(N\right)}$ are 1-summable in the
direction $\theta$, of order at least one with respect to $\mathbf{y}$,
of 1-sums $D_{\theta}^{\left(N\right)},R_{\theta}^{\left(N\right)}$
in the direction $\theta$. Moreover, one can choose $d^{\left(2\right)}=\dots=d^{\left(N\right)}$
for all $N\geq2$, and $d^{\left(1\right)}=d$.
\end{prop}

\begin{proof}
The proof is performed by induction on $N$.

\begin{itemize}
\item The case $N=1$ is the initial situation here, and is already proved
with $\hat{Y}_{\tx{prep}}=Y^{\left(1\right)}$.
\item Assume that the result holds for $N\in\ww N_{>0}$. We will proceed
in three steps.
\begin{enumerate}
\item First step: let us write
\[
R^{\left(N\right)}\left(x,\mathbf{y}\right)=\sum_{n_{1}+n_{2}\geq1}R_{n_{1},n_{2}}^{\left(N\right)}\left(x\right)y_{1}^{n_{1}}y_{2}^{n_{2}}\,\,\,.
\]
We are looking for an analytic solution $\tau$ to the equations:
\begin{eqnarray}
\cal L_{Y^{\left(N\right)}}\left(\tau\right) & = & -x^{N}R^{\left(N\right)}+\left(x^{N+1}\tilde{R}^{\left(N+1\right)}\right)\circ\Lambda_{\tau}\label{eq: eq homologic rec}\\
\cal L_{Y_{\theta}^{\left(N\right)}}\left(\tau\right) & = & -x^{N}R_{\theta}^{\left(N\right)}+\left(x^{N+1}\tilde{R}_{\theta}^{\left(N+1\right)}\right)\circ\Lambda_{\tau}\qquad,\nonumber 
\end{eqnarray}
for a convenient choice of $\tilde{R}^{\left(N+1\right)},\tilde{R}_{\theta}^{\left(N+1\right)}$,
with 
\[
\Lambda_{\tau}\left(x,\mathbf{y}\right):=\left(x,y_{1}\exp\left(\tau\left(x,\mathbf{y}\right)\right),y_{2}\exp\left(\tau\left(x,\mathbf{y}\right)\right)\right)\qquad,
\]
and 
\[
\tau\left(x,\mathbf{y}\right)=x^{N-1}\tau_{0}\left(y_{1}y_{2}\right)+x^{N}\tau_{1}\left(\mathbf{y}\right)\quad,
\]
where ${\displaystyle \tau_{1}\left(\mathbf{y}\right)=\sum_{j_{1}\neq j_{2}}\tau_{1,j_{1}j_{2}}y_{1}^{j_{1}}y_{2}^{j_{2}}}$.
More concretely, $\Lambda_{\tau}$ is the formal flow of $\overrightarrow{\cal R}$
at ``time'' $\tau\left(x,\mathbf{y}\right)$. \\
If we admit for a moment that such an analytic solution $\tau$ exists,
then $\Lambda_{\tau}\in\fdiffid$ and therefore $\Lambda_{\tau}^{-1}\in\fdiffid$.
Consider $d^{\left(N\right)}$ and $\tilde{D}^{\left(N-1\right)}$
such that
\begin{eqnarray*}
 & d^{\left(N+1\right)}\left(z_{1}z_{2}\right)+x^{N-1}\tilde{D}^{\left(N-1\right)}\left(x,\mathbf{z}\right):=\left(d^{\left(N\right)}\left(y_{1}y_{2}\right)+x^{N}D^{\left(N\right)}\left(x,\mathbf{y}\right)\right)\circ\Lambda_{\tau}^{-1}\left(x,\mathbf{z}\right)\\
 & d^{\left(N+1\right)}\left(z_{1}z_{2}\right)+x^{N-1}\tilde{D_{\theta}}^{\left(N-1\right)}\left(x,\mathbf{z}\right):=\left(d^{\left(N\right)}\left(y_{1}y_{2}\right)+x^{N}D_{\theta}^{\left(N\right)}\left(x,\mathbf{y}\right)\right)\circ\Lambda_{\tau}^{-1}\left(x,\mathbf{z}\right),
\end{eqnarray*}
with $\tilde{D}^{\left(N-1\right)}=0$ if $N=1$. Consequently, the
two equations given in $\left(\mbox{\ref{eq: eq homologic rec}}\right)$
imply that 
\begin{eqnarray*}
\left(\Lambda_{\tau}\right)_{*}\left(Y^{\left(N\right)}\right) & = & Y_{0}+\left(d^{\left(N+1\right)}\left(z_{1}z_{2}\right)+x^{N-1}\tilde{D}^{\left(N-1\right)}\left(x,\mathbf{z}\right)\right)\overrightarrow{\cal C}\\
 &  & +x^{N+1}\tilde{R}^{\left(N+1\right)}\left(x,\mathbf{z}\right)\overrightarrow{\cal R}\\
\left(\Lambda_{\tau}\right)_{*}\left(Y_{\theta}^{\left(N\right)}\right) & = & Y_{0}+\left(d^{\left(N+1\right)}\left(z_{1}z_{2}\right)+x^{N-1}\tilde{D}_{\theta}^{\left(N-1\right)}\left(x,\mathbf{z}\right)\right)\overrightarrow{\cal C}\\
 &  & +x^{N+1}\tilde{R}_{\theta}^{\left(N+1\right)}\left(x,\mathbf{z}\right)\overrightarrow{\cal R}\qquad.
\end{eqnarray*}
Indeed: 
\begin{eqnarray*}
 &  & \mbox{D}\Lambda_{\tau}\cdot Y^{\left(N\right)}=\left(\begin{array}{c}
\cal L_{Y^{\left(N\right)}}\left(x\right)\\
{\cal L}_{Y^{\left(N\right)}}\left(y_{1}\exp\left(\tau\left(x,\mathbf{y}\right)\right)\right)\\
\cal L_{Y^{\left(N\right)}}\left(y_{2}\exp\left(\tau\left(x,\mathbf{y}\right)\right)\right)
\end{array}\right)\\
 &  & =\left(\begin{array}{c}
x^{2}\\
\left(\cal L_{Y^{\left(N\right)}}\left(y_{1}\right)+y_{1}\left(\cal L_{Y^{\left(N\right)}}\left(\tau\right)\right)\right)\exp\left(\tau\left(x,\mathbf{y}\right)\right)\\
\left(\cal L_{Y^{\left(N\right)}}\left(y_{2}\right)+y_{2}\left(\cal L_{Y^{\left(N\right)}}\left(\tau\right)\right)\right)\exp\left(\tau\left(x,\mathbf{y}\right)\right)
\end{array}\right)\\
 &  & =\left(Y_{0}+\left(d^{\left(N+1\right)}+x^{N-1}\tilde{D}^{\left(N-1\right)}\right)\overrightarrow{\cal C}+x^{N+1}\tilde{R}^{\left(N+1\right)}\overrightarrow{\cal R}\right)\circ\Lambda_{\tau}\left(x,\mathbf{y}\right)\,\,\,.
\end{eqnarray*}
These computations are also true with the corresponding 1-sums of
formal objects considered here, \emph{i.e.} with $Y_{\theta}^{\left(N\right)},D_{\theta}^{\left(N\right)},\tilde{D}_{\theta}^{\left(N-1\right)},\tilde{R}_{\theta}^{\left(N+1\right)}$
instead of $Y^{\left(N\right)},D^{\left(N\right)},\tilde{D}^{\left(N-1\right)},\tilde{R}^{\left(N+1\right)}$
respectively. We use Proposition \ref{prop: compositon summable}
to obtain the 1-summability of the objects defined by compositions.\\
Let us prove that there exists a germ of analytic function of the
form 
\[
\tau\left(x,\mathbf{y}\right)=x^{N-1}\tau_{0}\left(y_{1}y_{2}\right)+x^{N}\tau_{1}\left(\mathbf{y}\right)\quad,
\]
of order ate least one with respect to $\mathbf{y}$ in the origin,
with 
\[
{\displaystyle \tau_{1}\left(\mathbf{y}\right)=\sum_{j_{1}\neq j_{2}}\tau_{1,j_{1}j_{2}}y_{1}^{j_{1}}y_{2}^{j_{2}}}
\]
satisfying equation $\left(\mbox{\ref{eq: eq homologic rec}}\right)$.
This equation can be written 
\begin{eqnarray*}
 &  & x^{2}\ppp{\tau}x+\left(-\lambda+a_{1}x-d^{\left(N\right)}\left(y_{1}y_{2}\right)-x^{N}D^{\left(N\right)}\left(x,\mathbf{y}\right)+x^{N}R^{\left(N\right)}\left(x,\mathbf{y}\right)\right)y_{1}\ppp{\tau}{y_{1}}\\
 &  & +\left(\lambda+a_{2}x+d^{\left(N\right)}\left(y_{1}y_{2}\right)+x^{N}D^{\left(N\right)}\left(x,\mathbf{y}\right)+x^{N}R^{\left(N\right)}\left(x,\mathbf{y}\right)\right)y_{2}\ppp{\tau}{y_{2}}\\
 &  & =-x^{N}R^{\left(N\right)}+\left(x^{N+1}\tilde{R}^{\left(N+1\right)}\right)\circ\Lambda_{\tau}\quad,
\end{eqnarray*}
or equivalently
\begin{eqnarray*}
 & x^{2}\ppp{\tau}x+a_{1}xy_{1}\ppp{\tau}{y_{1}}+a_{2}xy_{2}\ppp{\tau}{y_{2}}+\left(\lambda+d^{\left(N\right)}\left(y_{1}y_{2}\right)+x^{N}D^{\left(N\right)}\left(x,\mathbf{y}\right)\right)\cal L_{\overrightarrow{\cal C}}\left(\tau\right)\\
 & +\left(x^{N}R^{\left(N\right)}\left(x,\mathbf{y}\right)\right)\cal L_{\overrightarrow{\cal R}}\left(\tau\right)=-x^{N}R^{\left(N\right)}+\left(x^{N+1}\tilde{R}^{\left(N+1\right)}\right)\circ\Lambda_{\tau}\quad.
\end{eqnarray*}
Let us consider terms of degree $N$ with respect to $x$:
\begin{eqnarray}
 & \left(N-1\right)\tau_{0}\left(y_{1}y_{2}\right)+\left(a_{1}+a_{2}+2\delta_{N,1}R^{\left(N\right)}\left(0,\mathbf{y}\right)\right)y_{1}y_{2}\ppp{\tau_{0}}v\left(y_{1}y_{2}\right)\nonumber \\
 & +\left(\lambda+d^{\left(N\right)}\left(y_{1}y_{2}\right)\right)\cal L_{\overrightarrow{\cal C}}\left(\tau_{1}\right)=-R^{\left(N\right)}\left(0,\mathbf{y}\right)\label{eq: eq homo rec degre 1 wrt x}
\end{eqnarray}
(here $\delta_{N,1}$ is the Kronecker notation), and let us define
\[
R_{\mbox{res}}^{\left(N\right)}\left(0,v\right):=\sum_{k\geq1}R_{k,k}^{\left(N\right)}\left(0\right)v^{k}\quad.
\]
We use now the fact that $\tx{Im}\left(\lie{\overrightarrow{\cal C}}\right)\varoplus\tx{Ker}\left(\lie{\overrightarrow{\cal C}}\right)$
is a direct sum, and that $\tx{Ker}\left(\lie{\overrightarrow{\cal C}}\right)$
is the set of formal power series in the resonant monomial $v=y_{1}y_{2}$.
Isolating the term $\lie{\overrightarrow{\cal C}}\left(\tau_{1}\right)$
on the one hand, and the others on the other hand, the direct sum
above gives us: 
\[
\begin{cases}
{\displaystyle v\left(a_{1}+a_{2}+2\delta_{N,1}R_{\mbox{res}}^{\left(N\right)}\left(0,v\right)\right)\ddd{\tau_{0}}v\left(v\right)+\left(N-1\right)\tau_{0}\left(v\right)=-R_{\mbox{res}}^{\left(N\right)}\left(0,v\right)}\\
{\displaystyle \tau_{0}\left(0\right)=0}
\end{cases}
\]
 and
\[
\begin{cases}
{\displaystyle \cal L_{\overrightarrow{\cal C}}\left(\tau_{1}\right)=\frac{-1}{\lambda+d^{\left(N\right)}\left(y_{1}y_{2}\right)}\Bigg(\left(2\delta_{N,1}\left(R^{\left(N\right)}\left(0,\mathbf{y}\right)-R_{\mbox{res}}^{\left(N\right)}\left(0,v\right)\right)\right)y_{1}y_{2}\ddd{\tau_{0}}v\left(y_{1}y_{2}\right)}\\
{\displaystyle \qquad\,\,+R^{\left(N\right)}\left(0,\mathbf{y}\right)-R_{\mbox{res}}^{\left(N\right)}\left(0,v\right)\Bigg)}\\
{\displaystyle \tau_{1}\left(0\right)=0\,\,.}
\end{cases}
\]
Since $R^{\left(N\right)}$ is analytic with respect to $\mathbf{y}$,
$R_{\mbox{res}}^{\left(N\right)}\left(0,v\right)$ is analytic near
$v=0$. Furthermore, as $R_{\mbox{res}}^{\left(N\right)}\left(0,0\right)=0$
and $a_{1}+a_{2}\notin\ww Q_{\leq0}$, the first of the two equation
above has a unique formal solution $\tau_{0}$ with $\tau_{0}\left(0\right)$,
and this solution is convergent in a neighborhood of the origin. Once
$\tau_{0}$ is determined, there exists a unique formal solution $\tau_{1}$
to the second equation satisfying ${\displaystyle \tau_{1}\left(\mathbf{y}\right)=\sum_{j_{1}\neq j_{2}}\tau_{1,j_{1}j_{2}}y_{1}^{j_{1}}y_{2}^{j_{2}}}$,
which is moreover convergent in a neighborhood of the origin of $\ww C^{2}$
.\\
Therefore $\Lambda_{\tau}$ is a germ of analytic diffeomorphism fixing
the origin, fibered, tangent to the identity and conjugates $Y^{\left(N\right)}$
\emph{$\big($resp.} $Y_{\theta}^{\left(N\right)}$$\big)$ to ${\displaystyle \tilde{Y}^{\left(N\right)}:=\left(\Lambda_{\tau}\right)_{*}\left(Y^{\left(N\right)}\right)}$
$\big($\emph{resp.} ${\displaystyle \tilde{Y}_{\theta}^{\left(N\right)}:=\left(\Lambda_{\tau}\right)_{*}\left(Y_{\theta}^{\left(N\right)}\right)}$$\big)$.\\
Equation $\left(\mbox{\ref{eq: eq homo rec degre 1 wrt x}}\right)$
implies that ${\displaystyle \left(\cal L_{Y^{\left(N\right)}}\left(\tau\right)+x^{N}R^{\left(N\right)}\right)}$
and ${\displaystyle \left(\cal L_{Y_{\theta}^{\left(N\right)}}\left(\tau\right)+x^{N}R_{\theta}^{\left(N\right)}\right)}$
are divisible by $x^{N+1}$, so that we can define:
\begin{eqnarray*}
\tilde{R}^{\left(N+1\right)}\left(x,\mathbf{z}\right) & := & \left(\frac{\cal L_{Y^{\left(N\right)}}\left(\tau\right)+x^{N}R^{\left(N\right)}}{x^{N+1}}\right)\circ\Lambda_{\tau}^{-1}\left(x,\mathbf{z}\right)\\
\tilde{R}_{\theta}^{\left(N+1\right)}\left(x,\mathbf{z}\right) & := & \left(\frac{\cal L_{Y_{\theta}^{\left(N\right)}}\left(\tau\right)+x^{N}R_{\theta}^{\left(N\right)}}{x^{N+1}}\right)\circ\Lambda_{\tau}^{-1}\left(x,\mathbf{z}\right)\quad.
\end{eqnarray*}
By Proposition \ref{prop: compositon summable}, $\tilde{R}^{\left(N+1\right)}$
$\big($\emph{resp.} $\tilde{D}^{\left(N-1\right)}$$\big)$ is 1-summable
in the direction $\theta$, of 1-sum $\tilde{R}_{\theta}^{\left(N+1\right)}$
$\big($\emph{resp.} $\tilde{D}_{\theta}^{\left(N-1\right)}$ $\big)$.\\
Finally, notice that ${\displaystyle d^{\left(N+1\right)}\circ\Lambda_{\tau}\left(0,\mathbf{y}\right)=d^{\left(N\right)}\left(y_{1},y_{2}\right)}$,
$\tau\left(0,\mathbf{y}\right)=0$ and then $\Lambda_{\tau}\left(0,\mathbf{y}\right)=\left(0,y_{1},y_{2}\right)$
if $N>1$, so that $d^{\left(N+1\right)}=d^{\left(N\right)}$ when
$N>1$.
\item Second step: exactly as in the previous step which dealt with the
``radial part'' (in fact the computations are even easier here),
we can prove the existence of a germ of an analytic function $\sigma$,
solution to the equation: 
\begin{eqnarray}
\cal L_{\tilde{Y}^{\left(N\right)}}\left(\sigma\right) & = & -x^{N-1}\tilde{D}^{\left(N-1\right)}+\left(x^{N}\tilde{\tilde{D}}^{\left(N\right)}\right)\circ\Gamma_{\sigma}\nonumber \\
\cal L_{\tilde{Y}_{\theta}^{\left(N\right)}}\left(\sigma\right) & = & -x^{N-1}\tilde{D}_{\theta}^{\left(N-1\right)}+\left(x^{N}\tilde{\tilde{D}}_{\theta}^{\left(N\right)}\right)\circ\Gamma_{\sigma}\qquad,\label{eq: eq homologic rec-1}
\end{eqnarray}
for a good choice of $\tilde{\tilde{D}}^{\left(N\right)},\tilde{\tilde{D}}_{\theta}^{\left(N\right)}$,
with 
\[
\Gamma_{\sigma}\left(x,\mathbf{z}\right):=\left(x,y_{1}\exp\left(-\sigma\left(x,\mathbf{z}\right)\right),y_{2}\exp\left(\sigma\left(x,\mathbf{z}\right)\right)\right)\qquad
\]
and 
\[
\sigma\left(x,\mathbf{z}\right)=x^{N-2}\sigma_{0}\left(z_{1}z_{2}\right)+x^{N-1}\sigma_{1}\left(\mathbf{z}\right)\qquad,
\]
where ${\displaystyle \sigma_{1}\left(\mathbf{z}\right)=\sum_{j_{1}\neq j_{2}}\sigma_{1,j_{1},j_{2}}z_{1}^{j_{1}}z_{2}^{j_{2}}}$.
Here, we take $\sigma_{0}=0$ if $N=1$. Notice that $\Gamma_{\sigma}$
is the formal flow of $\overrightarrow{\cal C}$ at ``time'' $\sigma\left(x,\mathbf{z}\right)$.\\
Again, as in the first step with the ``radial part'', we have on
a $\Gamma_{\sigma}\in\fdiffid$ and then also $\Gamma_{\sigma}^{-1}\in\fdiffid$.
If we consider $\tilde{\tilde{R}}^{\left(N+1\right)}$ and $\tilde{\tilde{R}}_{\theta}^{\left(N+1\right)}$
such that
\begin{eqnarray*}
\tilde{\tilde{R}}^{\left(N+1\right)}\left(x,\mathbf{y}\right) & := & \tilde{R}^{\left(N+1\right)}\circ\Gamma_{\sigma}^{-1}\left(x,\mathbf{y}\right)\\
\tilde{\tilde{R}}_{\theta}^{\left(N+1\right)}\left(x,\mathbf{z}\right) & := & \tilde{R}_{\theta}^{\left(N+1\right)}\circ\Gamma_{\sigma}^{-1}\left(x,\mathbf{y}\right)\qquad,
\end{eqnarray*}
then it follows from $\left(\mbox{\ref{eq: eq homologic rec-1}}\right)$
that 
\begin{eqnarray*}
\left(\Gamma_{\sigma}\right)_{*}\left(\tilde{Y}^{\left(N\right)}\right) & = & Y_{0}+\left(d^{\left(N+1\right)}\left(y_{1}y_{2}\right)+x^{N}\tilde{\tilde{D}}^{\left(N\right)}\left(x,\mathbf{y}\right)\right)\overrightarrow{\cal C}\\
 &  & +x^{N+1}\tilde{\tilde{R}}^{\left(N+1\right)}\left(x,\mathbf{y}\right)\overrightarrow{\cal R}\\
\left(\Gamma_{\sigma}\right)_{*}\left(\tilde{Y}_{\theta}^{\left(N\right)}\right) & = & Y_{0}+\left(d^{\left(N+1\right)}\left(y_{1}y_{2}\right)+x^{N}\tilde{\tilde{D}}_{\theta}^{\left(N\right)}\left(x,\mathbf{y}\right)\right)\overrightarrow{\cal C}\\
 &  & +x^{N+1}\tilde{\tilde{R}}_{\theta}^{\left(N+1\right)}\left(x,\mathbf{y}\right)\overrightarrow{\cal R}\qquad.
\end{eqnarray*}
Notice that the degree of the monomial $x^{N+1}$ in front of $\tilde{\tilde{R}}^{\left(N+1\right)}$
is indeed $N+1$ (and not $N$): this essentially comes form the fact
that $\Gamma_{\sigma}$ (and $\Gamma_{\sigma}^{-1}$) preserves the
resonant monomial $v=y_{1}y_{2}$. We choose:
\begin{eqnarray*}
\tilde{\tilde{D}}^{\left(N\right)}\left(x,\mathbf{y}\right) & := & \left(\frac{\cal L_{\tilde{Y}^{\left(N\right)}}\left(\sigma\right)+x^{N-1}\tilde{D}^{\left(N-1\right)}}{x^{N}}\right)\circ\Gamma_{\sigma}^{-1}\left(x,\mathbf{y}\right)\\
\tilde{\tilde{D}}_{\theta}^{\left(N\right)}\left(x,\mathbf{y}\right) & := & \left(\frac{\cal L_{\tilde{Y}_{\theta}^{\left(N\right)}}\left(\sigma\right)+x^{N-1}\tilde{D}_{\theta}^{\left(N-1\right)}}{x^{N}}\right)\circ\Gamma_{\sigma}^{-1}\left(x,\mathbf{y}\right)\quad.
\end{eqnarray*}
By Proposition \ref{prop: compositon summable}, $\tilde{\tilde{D}}^{\left(N\right)}$
$\big($\emph{resp.} $\tilde{\tilde{R}}^{\left(N+1\right)}$$\big)$
is 1-summable in the direction $\theta$, of 1-sum $\tilde{\tilde{D}}_{\theta}^{\left(N\right)}$
$\big($\emph{resp.} $\tilde{\tilde{R}}_{\theta}^{\left(N+1\right)}$
$\big)$. We finally define $\tilde{\tilde{Y}}^{\left(N\right)}:=\left(\Gamma_{\sigma}\right)_{*}\left(\tilde{Y}^{\left(N\right)}\right)$,
and $\tilde{\tilde{Y}}_{\theta}^{\left(N\right)}:=\left(\Gamma_{\sigma}\right)_{*}\left(\tilde{Y}_{\theta}^{\left(N\right)}\right)$.
\item Third (and last) step: as in both previous steps, we can prove the
existence of a germ of an analytic function $\varphi$, solution to
the equation: 
\begin{eqnarray*}
\cal L_{\tilde{\tilde{Y}}^{\left(N\right)}}\left(\varphi\right) & = & -x^{N}\tilde{\tilde{D}}^{\left(N\right)}+\left(x^{N+1}D^{\left(N+1\right)}\right)\circ\Gamma_{\varphi}\,\,,\\
\cal L_{\tilde{\tilde{Y}}_{\theta}^{\left(N\right)}}\left(\varphi\right) & = & -x^{N}\tilde{\tilde{D}}_{\theta}^{\left(N\right)}+\left(x^{N+1}D_{\theta}^{\left(N+1\right)}\right)\circ\Gamma_{\varphi}\,\,,
\end{eqnarray*}
for a good choice of $D^{\left(N+1\right)},D_{\theta}^{\left(N+1\right)}$,
with 
\[
\Gamma_{\varphi}\left(x,\mathbf{y}\right):=\left(x,y_{1}\exp\left(-\varphi\left(x,\mathbf{y}\right)\right),y_{2}\exp\left(\varphi\left(x,\mathbf{y}\right)\right)\right)\qquad
\]
and 
\[
\varphi\left(x,\mathbf{y}\right)=x^{N-1}\varphi_{0}\left(y_{1}y_{2}\right)+x^{N}\varphi_{1}\left(\mathbf{y}\right)\qquad,
\]
where ${\displaystyle \varphi_{1}\left(\mathbf{y}\right)=\sum_{j_{1}\neq j_{2}}\varphi_{1,j_{1},j_{2}}y_{1}^{j_{1}}y_{2}^{j_{2}}}$.\\
Again, we have on a $\Gamma_{\varphi}\in\fdiffid$ and then also $\Gamma_{\varphi}^{-1}\in\fdiffid$.
If we consider $R^{\left(N+1\right)}$ and $R_{\theta}^{\left(N+1\right)}$
such that
\begin{eqnarray*}
R^{\left(N+1\right)}\left(x,\mathbf{y}\right) & := & \tilde{\tilde{R}}^{\left(N+1\right)}\circ\Gamma_{\varphi}^{-1}\left(x,\mathbf{y}\right)\\
R_{\theta}^{\left(N+1\right)}\left(x,\mathbf{z}\right) & := & \tilde{\tilde{R}}_{\theta}^{\left(N+1\right)}\circ\Gamma_{\varphi}^{-1}\left(x,\mathbf{y}\right)\qquad,
\end{eqnarray*}
then we have:
\begin{eqnarray*}
\left(\Gamma_{\varphi}\right)_{*}\left(\tilde{\tilde{Y}}^{\left(N\right)}\right) & = & Y_{0}+\left(d^{\left(N+1\right)}\left(y_{1}y_{2}\right)+x^{N+1}D^{\left(N+1\right)}\left(x,\mathbf{y}\right)\right)\overrightarrow{\cal C}\\
 &  & +x^{N+1}R^{\left(N+1\right)}\left(x,\mathbf{y}\right)\overrightarrow{\cal R}\\
\left(\Gamma_{\varphi}\right)_{*}\left(\tilde{\tilde{Y}}_{\theta}^{\left(N\right)}\right) & = & Y_{0}+\left(d^{\left(N+1\right)}\left(y_{1}y_{2}\right)+x^{N+1}D_{\theta}^{\left(N+1\right)}\left(x,\mathbf{y}\right)\right)\overrightarrow{\cal C}\\
 &  & +x^{N+1}R_{\theta}^{\left(N+1\right)}\left(x,\mathbf{y}\right)\overrightarrow{\cal R}\qquad.
\end{eqnarray*}
As above, notice that the degree of the monomial $x^{N+1}$ in front
of $R^{\left(N+1\right)}$ is indeed $N+1$. We choose:
\begin{eqnarray*}
D^{\left(N+1\right)}\left(x,\mathbf{y}\right) & := & \left(\frac{\cal L_{\tilde{\tilde{Y}}^{\left(N\right)}}\left(\varphi\right)+x^{N}\tilde{\tilde{D}}^{\left(N\right)}}{x^{N+1}}\right)\circ\Gamma_{\varphi}^{-1}\left(x,\mathbf{y}\right)\\
D_{\theta}^{\left(N+1\right)}\left(x,\mathbf{y}\right) & := & \left(\frac{\cal L_{\tilde{\tilde{Y}}_{\theta}^{\left(N\right)}}\left(\varphi\right)+x^{N}\tilde{\tilde{D}}_{\theta}^{\left(N\right)}}{x^{N+1}}\right)\circ\Gamma_{\varphi}^{-1}\left(x,\mathbf{y}\right)\quad.
\end{eqnarray*}
By Proposition \ref{prop: compositon summable}, $D^{\left(N+1\right)}$
$\big($\emph{resp.} $R^{\left(N+1\right)}$$\big)$ is 1-summable
in the direction $\theta$, of 1-sum $D_{\theta}^{\left(N\right)}$
$\big($\emph{resp.} $R_{\theta}^{\left(N+1\right)}$ $\big)$. We
finally define $Y^{\left(N+1\right)}:=\left(\Gamma_{\varphi}\right)_{*}\left(\tilde{\tilde{Y}}^{\left(N\right)}\right)$,
and $Y_{\theta}^{\left(N+1\right)}:=\left(\Gamma_{\varphi}\right)_{*}\left(\tilde{\tilde{Y}}_{\theta}^{\left(N\right)}\right)$.
\end{enumerate}
\end{itemize}
\end{proof}

\subsection{Proof of Proposition \ref{prop: forme pr=0000E9par=0000E9e ordre N}}

~

We now give a short proof of Proposition \ref{prop: forme pr=0000E9par=0000E9e ordre N},
using the different results proved in this section.
\begin{proof}[Proof of Proposition \ref{prop: forme pr=0000E9par=0000E9e ordre N}.]

We just have to use consecutively Proposition \ref{prop: preparation sur hypersurface invariante}
(applied to $Y_{0}:=Y_{\mid\acc{x=0}}$), Proposition \ref{prop: diag prep},
Proposition \ref{prop: preparation} and finally Proposition \ref{prop: ordre N avec eq homologique},
using at each time Corollary \ref{cor: summability push-forward}
in order to obtain the directional 1-summability.
\end{proof}

\section{\label{sec:Sectorial-isotropies-and}Sectorial analytic normalization}

The aim of this section is to prove that for any $Y\in\snodiag$ and
for any ${\displaystyle \eta\in\big[\pi,2\pi\big[}$, there exists
a unique pair 
\[
\left(\Phi_{+},\Phi_{-}\right)\in\diffsect[\arg\left(i\lambda\right)][\eta]\times\diffsect[\arg\left(-i\lambda\right)][\eta]
\]
whose elements analytically conjugate $Y$ to its normal form $\ynorm$
(given by Theorem \ref{thm: forme normalel formelle}) in sectorial
neighborhoods of the origin with wide opening. The existence of sectorial
normalizing maps $\Phi_{+}$ and $\Phi_{-}$ in domains of the form
$\cal S_{+}\in\cal S_{\arg\left(i\lambda\right),\eta}$ and $\cal S_{-}\in\cal S_{\arg\left(-i\lambda\right),\eta}$
for all ${\displaystyle \eta\in\big[\pi,2\pi\big[}$, is equivalent
to the existence of a sectorial normalizing map $\Phi_{\theta}$ in
domains $\cal S\in{\cal S}_{\theta,\pi}$, for all $\theta\in\ww R$
such that $\theta\neq\arg\left(\pm\lambda\right)$. At the end of
this section we will also prove that $\Phi_{+}$ and $\Phi_{-}$ both
admit the unique formal normalizing map $\hat{\Phi}$ (given by Theorem
\ref{thm: forme normalel formelle}) as weak Gevrey-1 asymptotic expansion
in domains $\cal S_{+}\in\cal S_{\arg\left(i\lambda\right),\eta}$
and $\cal S_{-}\in\cal S_{\arg\left(-i\lambda\right),\eta}$ respectively.
In particular, this will prove that $\hat{\Phi}$ is weakly 1-summable
in every direction $\theta\neq\arg\left(\pm\lambda\right)$.

We start with a vector field $Y^{\left(N\right)}$ normalized up to
order $N\geq2$ as in Proposition \ref{prop: forme pr=0000E9par=0000E9e ordre N}.
First of all, we prove the existence of germs of sectorial analytic
functions $\alpha_{+}\in\cal O\left(\cal S_{+}\right),\alpha_{-}\in\cal O\left(\cal S_{-}\right)$,
which are solutions to homological equations of the form: 
\[
\cal L_{Y^{\left(N\right)}}\left(\alpha_{\pm}\right)=x^{M+1}A_{\pm}\left(x,\mathbf{y}\right)\qquad,
\]
where $M\in\ww N_{>0}$ and $A_{\pm}\in{\cal O}\left(\cal S_{\pm}\right)$
is analytic in $\cal S_{\pm}$ (see Lemma \ref{lem: solution eq homo}).
In order to construct such solutions, we will integrate some appropriate
meromorphic 1-form on asymptotic paths (see subsection \ref{subsec: proof lemma}).
Once we have these solutions $\alpha_{+},\alpha_{-}$, we will construct
the desired germs of sectorial diffeomorphisms as the flows of some
elementary linear vector fields at ``time'' $\alpha_{\pm}\left(x,\mathbf{y}\right)$.
After that, we will prove in subsection \ref{sub :Sectorial isotropies in big sectors}
that there exist unique germs of sectorial fibered diffeomorphisms
tangent to the identity which conjugate $Y\in\snofib$ to its normal
form, by studying the sectorial isotropies in sectorial domains with
wide opening.

We go on using the notations introduced in subsection \ref{subsec:Summable-normal-form},
\emph{i.e.}
\begin{itemize}
\item $\lambda\in\ww C^{*}$
\item $a_{1}+a_{2}\notin\ww Q_{\leq0}$
\item $\overrightarrow{\cal C}:=-y_{1}\pp{y_{1}}+y_{2}\pp{y_{2}}$
\item $\overrightarrow{\cal R}:=y_{1}\pp{y1}+y_{2}\pp{y_{2}}$
\item $Y_{0}:=\lambda\overrightarrow{\cal C}+x\left(x\pp x+a_{1}y_{1}\pp{y_{1}}+a_{2}y_{2}\pp{y2}\right)$.
\end{itemize}
For ${\displaystyle \epsilon\in\left]0,\frac{\pi}{2}\right[}$ and
$r>0$, we will consider two sectors, namely 
\[
S_{+}\left(r,\epsilon\right):=S\left(r,\arg\left(i\lambda\right)-\frac{\pi}{2}-\epsilon,\arg\left(i\lambda\right)+\frac{\pi}{2}+\epsilon\right)
\]
 and 
\[
S_{-}\left(r,\epsilon\right)=S\left(r,\arg\left(-i\lambda\right)-\frac{\pi}{2}-\epsilon,\arg\left(-i\lambda\right)+\frac{\pi}{2}+\epsilon\right).
\]

Let us consider a (weakly) 1-summable non-degenerate div-integrable
doubly-resonant saddle-node normalized up to an order $N+2$, with
$N>0$: 
\begin{eqnarray*}
Y^{\left(N+2\right)} & = & Y_{0}+\left(c\left(y_{1}y_{2}\right)+x^{N+2}D^{\left(N+2\right)}\left(x,\mathbf{y}\right)\right)\overrightarrow{\cal C}+x^{N+2}R^{\left(N+2\right)}\left(x,\mathbf{y}\right)\overrightarrow{\cal R}\\
 &  & \mbox{(formal)}\\
Y_{\pm}^{\left(N+2\right)} & = & Y_{0}+\left(c\left(y_{1}y_{2}\right)+x^{N+2}D_{\pm}^{\left(N+2\right)}\left(x,\mathbf{y}\right)\right)\overrightarrow{\cal C}+x^{N+2}R_{\pm}^{\left(N+2\right)}\left(x,\mathbf{y}\right)\overrightarrow{\cal R}\\
 &  & \mbox{(analytic in }S_{\pm}\left(r,\epsilon\right)\times\mathbf{D\left(0,r\right)})
\end{eqnarray*}
where $D^{\left(N+2\right)},R^{\left(N+2\right)}$ are of order at
least one with respect to $\mathbf{y}$, and (weak) 1-summable in
every direction $\theta\in\ww R$ with $\theta\neq\arg\left(\pm\lambda\right)$:
their respective (weak) 1-sums in the direction $\arg\left(\pm i\lambda\right)$
are $D_{\pm}^{\left(N+2\right)},R_{\pm}^{\left(N+2\right)}$, which
can be analytically extended in $S_{\pm}\left(r,\epsilon\right)\times\left(\ww C^{2},0\right)$.
In order to have the complete sectorial normalizing map, we have to
assume now that our vector field is \textbf{\emph{strictly non-degenerate}},
\emph{i.e.}

\[
\fbox{\ensuremath{\Re\left(a_{1}+a_{2}\right)}>0}\,\,\,.
\]
\begin{prop}
\label{prop: norm sect}Under the assumptions above, for all $\eta\in\left]\pi,2\pi\right[$,
there exist two germs of sectorial fibered diffeomorphisms
\[
\begin{cases}
\Psi_{+}\in\diffsect[\arg\left(i\lambda\right)][\eta]\\
\Psi_{-}\in\diffsect[\arg\left(-i\lambda\right)][\eta]
\end{cases}
\]
 of the form 
\begin{eqnarray*}
\Psi_{\pm}:\left(x,\mathbf{y}\right) & \mapsto & \left(x,\mathbf{y}+\tx O\left(\norm{\mathbf{y}}^{2}\right)\right)\,\,,
\end{eqnarray*}
 which conjugate $Y_{\pm}^{\left(N+2\right)}$ to its formal normal
form 
\[
\ynorm=x^{2}\pp x+\left(-\lambda+a_{1}x-c\left(y_{1}y_{2}\right)\right)y_{1}\pp{y_{1}}+\left(\lambda+a_{2}x+c\left(y_{1}y_{2}\right)\right)y_{2}\pp{y_{2}}\qquad,
\]
where $c\left(v\right)\in v\ww C\acc v$ is the germ of an analytic
function in $v:=y_{1}y_{2}$ vanishing at the origin. Moreover, we
can choose $\Psi_{\pm}$ above such that 
\[
\Psi_{\pm}\left(x,\mathbf{y}\right)=\tx{Id}\left(x,\mathbf{y}\right)+x^{N}\mathbf{P}_{\pm}^{\left(N\right)}\left(x,\mathbf{y}\right)\,\,,
\]
where $\mathbf{P}_{\pm}^{\left(N\right)}=\left(0,P_{1,\pm},P_{2,\pm}\right)$
is analytic in $S_{\pm}\left(r,\epsilon\right)\times\left(\ww C^{2},0\right)$
(for some $r>0$ and $\epsilon>\frac{\eta}{2}$) and of order at least
two with respect to $\mathbf{y}$.
\end{prop}

By combining Propositions \ref{prop: forme pr=0000E9par=0000E9e ordre N}
and \ref{prop: norm sect} we immediately obtain the following result.
\begin{cor}
\label{cor: existence normalisations sectorielles}Let $Y\in\snofib$
be a strictly non-degenerate diagonal doubly-resonant saddle-node
which is div-integrable. Then, for all $\eta\in\left]\pi,2\pi\right[$,
there exist two germs of sectorial fibered diffeomorphisms
\[
\begin{cases}
\Phi_{+}\in\diffsect[\arg\left(i\lambda\right)][\eta]\\
\Phi_{-}\in\diffsect[\arg\left(-i\lambda\right)][\eta]
\end{cases}
\]
 tangent to the identity such that:
\begin{eqnarray*}
\left(\Phi_{\pm}\right)_{*}\left(Y\right) & = & x^{2}\pp x+\left(-\lambda+a_{1}x-c\left(y_{1}y_{2}\right)\right)y_{1}\pp{y_{1}}+\left(\lambda+a_{2}x+c\left(y_{1}y_{2}\right)\right)y_{2}\pp{y_{2}}\\
 & =: & \ynorm\,\,,
\end{eqnarray*}
where $\lambda\in\ww C^{*}$, $\Re\left(a_{1}+a_{2}\right)>0$, and
$c\left(v\right)\in v\ww C\acc v$ is the germ of an analytic function
in $v:=y_{1}y_{2}$ vanishing at the origin.
\end{cor}

As already mentioned, we will also prove at the end of this section
that $\Phi_{+}$ and $\Phi_{-}$ are unique as germs (see Proposition
\ref{prop: unique normalizations}), and that they are the weak 1-sums
of the unique formal normalizing map $\hat{\Phi}$ given by Theorem
\ref{thm: forme normalel formelle}.

\subsection{Proof of Proposition \ref{prop: norm sect}. }

~

We give here two consecutive propositions which allow to prove Proposition
\ref{prop: norm sect} as an immediate consequence. When we say that
a function $f:U\rightarrow\ww C$ is \emph{dominated} by another $g:U\rightarrow\ww R_{+}$
in $U$, it means that there exists $L>0$ such that for all $u\in U$,
we have $\abs{f\left(u\right)}\leq L.g\left(u\right)$.
\begin{prop}
\label{prop: radial part}

Let ${\displaystyle Y_{\pm}^{\left(N+2\right)}=Y_{0}+D_{\pm}\overrightarrow{\cal C}+R_{\pm}\overrightarrow{\cal R}},$
where 
\[
\begin{cases}
D_{\pm}\left(x,\mathbf{y}\right)=c\left(y_{1}y_{2}\right)+x^{N+2}D_{\pm}^{\left(N+2\right)}\left(x,\mathbf{y}\right)\\
R_{\pm}\left(x,\mathbf{y}\right)=x^{N+2}R_{\pm}^{\left(N+2\right)}\left(x,\mathbf{y}\right)
\end{cases}\qquad,
\]
with $N\in\ww N_{>0}$, $c\left(v\right)\in v\ww C\acc v$ of order
at least one, and $D_{\pm}^{\left(N+2\right)},R_{\pm}^{\left(N+2\right)}$
analytic in $S_{\pm}\left(r,\epsilon\right)\times\left(\ww C^{2},0\right)$
and dominated by $\norm{\mathbf{y}}_{\infty}$. Assume that $\Re\left(a_{1}+a_{2}\right)>0$.

Then, possibly by reducing $r>0$ and the neighborhood $\left(\ww C^{2},0\right)$,
there exist two germs of sectorial fibered diffeomorphisms $\varphi_{+}$
and $\varphi_{-}$ in $S_{+}\left(r,\epsilon\right)\times\left(\ww C^{2},0\right)$
and $S_{-}\left(r,\epsilon\right)\times\left(\ww C^{2},0\right)$
respectively, which conjugate $Y_{\pm}^{\left(N+2\right)}$ to 
\begin{eqnarray*}
Y_{\overrightarrow{\cal C},\pm} & := & Y_{0}+C_{\pm}\overrightarrow{\cal C}\qquad,
\end{eqnarray*}
where $C_{\pm}\left(x,\mathbf{y}\right)=D_{\pm}\circ\varphi_{\pm}^{-1}\left(x,\mathbf{z}\right)$.
Moreover we can chose $\varphi_{\pm}$ to be of the form 
\begin{eqnarray*}
\varphi_{\pm}\left(x,\mathbf{y}\right) & = & \left(x,y_{1}\exp\left(\rho_{\pm}\left(x,\mathbf{y}\right)\right),y_{2}\exp\left(\rho_{\pm}\left(x,\mathbf{y}\right)\right)\right)\qquad,
\end{eqnarray*}
where $\rho_{\pm}\left(x,\mathbf{y}\right)=x^{N+1}\tilde{\rho}_{\pm}\left(x,\mathbf{y}\right)$
and $\tilde{\rho}_{\pm}$ is analytic in $S_{\pm}\left(r,\epsilon\right)\times\left(\ww C^{2},0\right)$
and dominated by $\norm{\mathbf{y}}_{\infty}$.
\end{prop}

\begin{rem}
Notice that $\varphi_{\pm}^{-1}$ is of the form 
\[
\varphi_{\pm}^{-1}\left(x,\mathbf{z}\right)=\left(x,z_{1}\left(1+x^{N+1}\vartheta\left(x,\mathbf{z}\right)\right),z_{2}\left(1+x^{N+1}\vartheta\left(x,\mathbf{z}\right)\right)\right)\qquad,
\]
where $\vartheta$ is analytic in $S_{\pm}\left(r,\epsilon\right)\times\left(\ww C^{2},0\right)$
and dominated by $\norm{\mathbf{z}}_{\infty}$. Consequently: 
\[
C_{\pm}\left(x,\mathbf{z}\right)=c\left(z_{1}z_{2}\right)+x^{N+1}C_{\pm}^{\left(N+1\right)}\left(x,\mathbf{z}\right)\qquad,
\]
where $c$ is the same as above and $C_{\pm}$ is analytic in $S_{\pm}\left(r,\epsilon\right)\times\left(\ww C^{2},0\right)$
and dominated by $\norm{\mathbf{z}}_{\infty}$.
\end{rem}

\begin{prop}
\label{prop: tangential part}Let ${\displaystyle Y_{\cal C,\pm}:=Y_{0}+C_{\pm}\overrightarrow{\cal C}}$,
where 
\[
C_{\pm}\left(x,\mathbf{z}\right)=c\left(z_{1}z_{2}\right)+x^{N+1}C_{\pm}^{\left(N+1\right)}\left(x,\mathbf{z}\right)\qquad,
\]
with $N\in\ww N_{>0}$, $c\left(v\right)\in v\ww C\acc v$ of order
at least one, and $C_{\pm}^{\left(N+1\right)}$ analytic in $S_{\pm}\left(r,\epsilon\right)\times\left(\ww C^{2},0\right)$
and dominated by $\norm{\mathbf{z}}_{\infty}$ . Assume $\Re\left(a_{1}+a_{2}\right)>0$. 

Then, possibly by reducing $r>0$ and the neighborhood $\left(\ww C^{2},0\right)$,
there exist two germs of sectorial fibered diffeomorphisms $\psi_{+}$
and $\psi_{-}$ in $S_{+}\left(r,\epsilon\right)\times\left(\ww C^{2},0\right)$
and $S_{-}\left(r,\epsilon\right)\times\left(\ww C^{2},0\right)$
respectively, which conjugate $Y_{\cal C,\pm}$ to 
\begin{eqnarray*}
\ynorm & := & Y_{0}+c\left(v\right)\overrightarrow{\cal C}\qquad.
\end{eqnarray*}
Moreover, we can chose $\psi_{\pm}$ to be of the form 
\begin{eqnarray*}
\psi_{\pm}\left(x,\mathbf{z}\right) & = & \left(x,z_{1}\exp\left(-\chi_{\pm}\left(x,\mathbf{z}\right)\right),z_{2}\exp\left(\chi_{\pm}\left(x,\mathbf{z}\right)\right)\right)\qquad,
\end{eqnarray*}
where $\chi_{\pm}\left(x,\mathbf{z}\right)=x^{N}\tilde{\chi}_{\pm}\left(x,\mathbf{z}\right)$
and $\tilde{\chi}$ is analytic in $S_{\pm}\left(r,\epsilon\right)\times\left(\ww C^{2},0\right)$
and dominated by $\norm{\mathbf{z}}_{\infty}$.
\end{prop}

If we assume for a moment the two propositions above, the proof of
Proposition becomes obvious.
\begin{proof}[Proof of Proposition \ref{prop: norm sect}.]

It is an immediate consequence of the consecutive application of the
previous two propositions, just by taking $\Psi_{\pm}=\psi_{\pm}\circ\varphi_{\pm}$
with the notations above.
\end{proof}

\subsection{Proof of Propositions \ref{prop: radial part} and \ref{prop: tangential part}}

~

In order to prove Propositions \ref{prop: radial part} and \ref{prop: tangential part},
we will need the following lemmas. The first one gives the existence
of analytic solutions (in sectorial domains) to a homological equations
we need to solve.
\begin{lem}
\label{lem: solution eq homo}Let ${\displaystyle Z_{\pm}:=Y_{0}+C_{\pm}\left(x,\mathbf{y}\right)\overrightarrow{\cal C}+xR_{\pm}^{\left(1\right)}\left(x,\mathbf{y}\right)\overrightarrow{\cal R}}$,
with $C_{\pm},R_{\pm}^{\left(1\right)}$ analytic in $S_{\pm}\left(r,\epsilon\right)\times\left(\ww C^{2},0\right)$
and dominated by $\norm{\mathbf{y}}_{\infty}$ and let also $A_{\pm}\left(x,\mathbf{y}\right)$
be analytic in $S_{\pm}\left(r,\epsilon\right)\times\left(\ww C^{2},0\right)$
and dominated by $\norm{\mathbf{y}}_{\infty}$. Then for all $M\in\ww N_{>0}$,
possibly by reducing $r>0$ and the neighborhood $\left(\ww C^{2},0\right)$,
there exists a solution $\alpha_{\pm}$ to the homological equation
\begin{equation}
\cal L_{Z_{\pm}}\left(\alpha_{\pm}\right)=x^{M+1}A_{\pm}\left(x,\mathbf{y}\right)\qquad,\label{eq: eq homo lemme}
\end{equation}
such that $\alpha_{\pm}\left(x,\mathbf{y}\right)=x^{M}\tilde{\alpha}_{\pm}\left(x,\mathbf{y}\right)$,
where $\tilde{\alpha}_{\pm}$ is a germ of analytic function in $S_{\pm}\left(r,\epsilon\right)\times\left(\ww C^{2},0\right)$
and dominated by $\norm{\mathbf{y}}_{\infty}$.
\end{lem}

We will prove this lemma in subsection \ref{subsec: proof lemma}.
The following lemma proves that $\varphi_{\pm}$ and $\psi_{\pm}$
constructed in Propositions \ref{prop: radial part} and \ref{prop: tangential part}
are indeed germs of sectorial fibered diffeomorphisms in domains of
the form $S_{\pm}\left(r,\epsilon\right)\times\left(\ww C^{2},0\right)$.
\begin{lem}
\label{lem: diffeo sectoriel}Let $f_{\pm},g_{\pm}$ be two germs
of analytic and bounded functions in $S_{\pm}\left(r,\epsilon\right)\times\left(\ww C^{2},0\right)$,
which tend to $0$ as $\left(x,\mathbf{y}\right)\rightarrow\left(0,\mathbf{0}\right)$
in $S_{\pm}\left(r,\epsilon\right)\times\left(\ww C^{2},0\right)$.
Then 
\[
\phi_{\pm}:\left(x,\mathbf{y}\right)\mapsto\left(x,y_{1}\exp\left(f_{\pm}\left(x,\mathbf{y}\right)\right),y_{2}\exp\left(g_{\pm}\left(x,\mathbf{y}\right)\right)\right)
\]
 defines a germ of sectorial fibered diffeomorphism analytic in $S_{\pm}\left(r,\epsilon\right)\times\left(\ww C^{2},0\right)$
(possibly by reducing $r>0$ and the neighborhood $\left(\ww C^{2},0\right)$).
\end{lem}

Let us explain why these lemmas imply Propositions \ref{prop: radial part}
and \ref{prop: tangential part}.
\begin{proof}[Proof of both Propositions \ref{prop: radial part} and \ref{prop: tangential part}.]

It is sufficient to apply Lemma \ref{lem: solution eq homo} with
$M=N+1$ , $A_{\pm}=-R_{\pm}^{\left(N+2\right)}$, $\alpha_{\pm}=\rho_{\pm}$
and $Z_{\pm}=Y_{\pm}^{\left(N+2\right)}$ for Proposition \ref{prop: radial part},
and with $M=N$, $A_{\pm}=-C_{\pm}^{\left(N+1\right)}$, $\alpha_{\pm}=\chi_{\pm}$
and $Z_{\pm}=Y_{\overrightarrow{\cal C},\pm}$ for Proposition \ref{prop: tangential part}.
Then we use Lemma \ref{lem: diffeo sectoriel} to see that $\varphi_{\pm}$
and $\psi_{\pm}$ are germs of sectorial fibered diffeomorphisms on
the considered domains, and we finally check that they do the conjugacy
we want. With the notations above:
\begin{eqnarray*}
\mbox{D}\varphi_{\pm}\cdot Y_{\pm}^{\left(N+2\right)} & = & \left(\begin{array}{c}
\cal L_{Y_{\pm}^{\left(N+2\right)}}\left(x\right)\\
\cal L_{Y_{\pm}^{\left(N+2\right)}}\left(y_{1}\exp\left(\rho_{\pm}\left(x,\mathbf{y}\right)\right)\right)\\
\cal L_{Y_{\pm}^{\left(N+2\right)}}\left(y_{2}\exp\left(\rho_{\pm}\left(x,\mathbf{y}\right)\right)\right)
\end{array}\right)\\
 & = & \left(\begin{array}{c}
x^{2}\\
\left(\cal L_{Y_{\pm}^{\left(N+2\right)}}\left(y_{1}\right)+y_{1}\left(\cal L_{Y_{\pm}^{\left(N+2\right)}}\left(\rho_{\pm}\right)\right)\right)\exp\left(\rho_{\pm}\left(x,\mathbf{y}\right)\right)\\
\left(\cal L_{Y_{\pm}^{\left(N+2\right)}}\left(y_{2}\right)+y_{2}\left(\cal L_{Y_{\pm}^{\left(N+2\right)}}\left(\rho_{\pm}\right)\right)\right)\exp\left(\rho_{\pm}\left(x,\mathbf{y}\right)\right)
\end{array}\right)\\
 & = & \left(\begin{array}{c}
x^{2}\\
\left(-\lambda+a_{1}x-D_{\pm}\left(x,\mathbf{y}\right)\right)y_{1}\exp\left(\rho_{\pm}\left(x,\mathbf{y}\right)\right)\\
\left(\lambda+a_{2}x+D_{\pm}\left(x,\mathbf{y}\right)\right)y_{2}\exp\left(\rho_{\pm}\left(x,\mathbf{y}\right)\right)
\end{array}\right)\\
 &  & \big(\mbox{we have used }\cal L_{Y_{\pm}^{\left(N+2\right)}}\left(\rho_{\pm}\right)=-x^{N+2}R_{\pm}^{\left(N+2\right)}\big)\\
 & = & \left(Y_{0}+C_{\pm}\overrightarrow{\cal C}\right)\circ\varphi_{\pm}\left(x,\mathbf{y}\right)\\
 & = & Y_{\overrightarrow{\cal C},\pm}\circ\varphi_{\pm}\left(x,\mathbf{y}\right)\hfill,
\end{eqnarray*}
so that $\left(\varphi_{\pm}\right)_{*}\left(Y_{\pm}^{\left(N+2\right)}\right)=Y_{\overrightarrow{\cal C},\pm}$
and then
\begin{eqnarray*}
\mbox{D}\psi_{\pm}\cdot Y_{\overrightarrow{\cal C},\pm} & = & \left(\begin{array}{c}
\cal L_{Y_{\overrightarrow{\cal C},\pm}}\left(x\right)\\
\cal L_{Y_{\overrightarrow{\cal C},\pm}}\left(z_{1}\exp\left(-\chi\left(x,\mathbf{z}\right)\right)\right)\\
\cal L_{Y_{\overrightarrow{\cal C},\pm}}\left(z_{2}\exp\left(\chi\left(x,\mathbf{z}\right)\right)\right)
\end{array}\right)\\
 & = & \left(\begin{array}{c}
x^{2}\\
\left(\cal L_{Y_{\overrightarrow{\cal C},\pm}}\left(z_{1}\right)+z_{1}\left(\cal L_{Y_{\overrightarrow{\cal C},\pm}}\left(\chi\right)\right)\right)\exp\left(-\chi\left(x,\mathbf{z}\right)\right)\\
\left(\cal L_{Y_{\overrightarrow{\cal C},\pm}}\left(z_{2}\right)+z_{2}\left(\cal L_{Y_{\overrightarrow{\cal C},\pm}}\left(\chi\right)\right)\right)\exp\left(\chi\left(x,\mathbf{z}\right)\right)
\end{array}\right)\\
 & = & \left(\begin{array}{c}
x^{2}\\
\left(-\lambda+a_{1}x-c\left(z_{1}z_{2}\right)\right)z_{1}\exp\left(-\chi\left(x,\mathbf{z}\right)\right)\\
\left(\lambda+a_{2}x+c\left(z_{1}z_{2}\right)\right)z_{2}\exp\left(\chi\left(x,\mathbf{y}\right)\right)
\end{array}\right)\\
 &  & \big(\mbox{we have used }\cal L_{Y_{\overrightarrow{\cal C},\pm}}\left(\chi_{\pm}\right)=-x^{N+1}C_{\pm}^{\left(N+1\right)}\big)\\
 & = & \left(Y_{0}+c\left(u\right)\overrightarrow{\cal C}\right)\circ\psi_{\pm}\left(x,\mathbf{z}\right)\\
 & = & \ynorm\circ\psi_{\pm}\left(x,\mathbf{z}\right)\hfill,
\end{eqnarray*}
so that $\left(\psi_{\pm}\right)_{*}\left(Y_{\overrightarrow{\cal C},\pm}\right)=\ynorm$
.
\end{proof}

\subsection{Proof of Lemma \ref{lem: diffeo sectoriel}}

~
\begin{proof}[Proof of Lemma \ref{lem: diffeo sectoriel}.]

We consider two germs of analytic functions $f_{\pm},g_{\pm}$ in
$S_{\pm}\left(r,\epsilon\right)\times\left(\ww C^{2},0\right)$ which
which tend to $0$ as $\left(x,\mathbf{y}\right)\rightarrow\left(0,\mathbf{0}\right)$
in $S_{\pm}\left(r,\epsilon\right)\times\left(\ww C^{2},0\right)$,
and we define 
\[
\phi_{\pm}:\left(x,\mathbf{y}\right)\mapsto\left(x,y_{1}\exp\left(f_{\pm}\left(x,\mathbf{y}\right)\right),y_{2}\exp\left(g_{\pm}\left(x,\mathbf{y}\right)\right)\right)\,\,.
\]
Let us first prove that $\phi_{\pm}$ is into. Let $\mathbf{x}=\left(x,y_{1},y_{2}\right)$
and $\mathbf{x}'=\left(x',y'_{1},y'_{2}\right)$ in $S_{\pm}\left(r,\epsilon\right)\times\left(\ww C^{2},0\right)$
such that $\phi_{\pm}\left(\mathbf{x}\right)=\phi_{\pm}\left(\mathbf{x}'\right)$.
Since $\phi_{\pm}$ is fibered, necessarily $x=x'$. Then assume that
$\left(y_{1},y_{2}\right)\neq\left(y'_{1},y'_{2}\right)$, such that
\[
\left\Vert \left(y_{1}-y'_{1},y_{2}-y'_{2}\right)\right\Vert _{\infty}>0
\]
 and for instance $\left\Vert \left(y_{1}-y'_{1},y_{2}-y'_{2}\right)\right\Vert _{\infty}=\abs{y_{1}-y'_{1}}>0$
(the other case can be done similarly). We denote by $D_{\mathbf{y}}$
the derivative with respect to variables $\left(y_{1},y_{2}\right)$.
According to the mean value theorem: 
\begin{eqnarray*}
\abs{\frac{e^{f_{\pm}\left(\mathbf{x}\right)}-e^{f_{\pm}\left(\mathbf{x'}\right)}}{y_{1}-y'_{1}}} & \leq & \underset{\left(z_{1},z_{2}\right)\in[\left(y_{1},y_{2}\right),\left(y'_{1},y'_{2}\right)]}{\mbox{sup}}\left\Vert D_{\mathbf{y}}\left(e^{f\pm}\right)\left(x,z_{1},z_{2}\right)\right\Vert _{\infty}.
\end{eqnarray*}

Consequently we have: 
\begin{eqnarray*}
0 & = & \abs{y_{1}e^{f_{\pm}\left(\mathbf{x}\right)}-y'_{1}e^{f_{\pm}\left(\mathbf{x'}\right)}}\\
 & = & \abs{e^{f_{\pm}\left(\mathbf{x}\right)}}.\abs{y_{1}-y'_{1}}.\abs{1+\frac{y'_{1}}{e^{f_{\pm}\left(\mathbf{x}\right)}}.\frac{e^{f_{\pm}\left(\mathbf{x}\right)}-e^{f_{\pm}\left(\mathbf{x'}\right)}}{y_{1}-y'_{1}}}\\
 & \geq & \abs{e^{f_{\pm}\left(\mathbf{x}\right)}}.\abs{y_{1}-y'_{1}}.\left(1-\abs{\frac{y'_{1}}{e^{f_{\pm}\left(\mathbf{x}\right)}}}.\abs{\frac{e^{f_{\pm}\left(\mathbf{x}\right)}-e^{f_{\pm}\left(\mathbf{x'}\right)}}{y_{1}-y'_{1}}}\right)\\
 & \geq & \abs{e^{f_{\pm}\left(\mathbf{x}\right)}}.\abs{y_{1}-y'_{1}}.\left(1-\abs{\frac{y'_{1}}{e^{f_{\pm}\left(\mathbf{x}\right)}}}.\underset{\left(z_{1},z_{2}\right)\in[\left(y_{1},y_{2}\right),\left(y'_{1},y'_{2}\right)]}{\mbox{sup}}\left\Vert D_{\mathbf{y}}\left(e^{f_{\pm}}\right)\left(x,z_{1},z_{2}\right)\right\Vert _{\infty}.\right)
\end{eqnarray*}

Assume that we chose $\left(\ww C^{2},0\right)=\mathbf{D\left(0,r\right)}$
small enough such that $f_{\pm}$ is analytic in 
\[
S_{\pm}\left(r,\epsilon\right)\times D\left(0,3r_{1}+\delta\right)\times D\left(0,3r_{2}+\delta\right)
\]
with $\delta>0$ small. Without lost of generality we can take $r_{1}=r_{2}$.
We apply Cauchy \textquoteright s integral formula to $z_{1}\mapsto e^{f_{\pm}\left(x,z_{1},z_{2}\right)}$,
for all fixed $z_{2}$ , integrating on the circle of center $0$
and radius $3r_{1}=3r_{2}$. Similarly we also apply Cauchy \textquoteright s
integral formula to $z_{2}\mapsto e^{f_{\pm}\left(x,z_{1},z_{2}\right)}$,
for all fixed $z_{1}$, integrating on the circle of center $0$ and
radius $3r_{2}=3r_{1}$. Then we obtain 
\[
\underset{\left(z_{1},z_{2}\right)\in[\left(y_{1},y_{2}\right),\left(y'_{1},y'_{2}\right)]}{\mbox{sup}}\left\Vert D_{\mathbf{y}}\left(e^{f_{\pm}}\right)\left(x,z_{1},z_{2}\right)\right\Vert _{\infty}\leq\frac{3}{4r_{1}}.\exp\left(\underset{\mathbf{x}\in S_{\pm}\left(r,\epsilon\right)\times\mathbf{D\left(0,r\right)}}{\mbox{sup}}\left(\abs{f_{\pm}\left(\mathbf{x}\right)}\right)\right)\quad,
\]

such that: 
\begin{eqnarray*}
0 & = & \abs{y_{1}e^{f_{\pm}\left(\mathbf{x}\right)}-y'_{1}e^{f_{\pm}\left(\mathbf{x'}\right)}}\\
 & \geq & \abs{e^{f_{\pm}\left(\mathbf{x}\right)}}.\abs{y_{1}-y'_{1}}.\left(1-\frac{3}{4}\exp\left(\underset{\mathbf{x}\in S_{\pm}\left(r,\epsilon\right)\times\mathbf{D\left(0,r\right)}}{\mbox{sup}}\left(2\abs{f_{\pm}\left(\mathbf{x}\right)}\right)\right)\right)\qquad.
\end{eqnarray*}

Since $f_{\pm}\left(\mathbf{x}\right)\underset{\mathbf{x}\rightarrow\mathbf{0}}{\rightarrow}0$,
we can choose $r,r_{1}$ and $r_{2}$ small enough such that: 
\[
\exp\left(\underset{\mathbf{x}\in S_{\pm}\left(r,\epsilon\right)\times\mathbf{D\left(0,r\right)}}{\mbox{sup}}\left(2\abs{f_{\pm}\left(\mathbf{x}\right)}\right)\right)\leq\frac{5}{4}<\frac{4}{3}\qquad.
\]

Finally we obtain: 
\begin{eqnarray*}
0 & = & \abs{y_{1}e^{f_{\pm}\left(\mathbf{x}\right)}-y'_{1}e^{f_{\pm}\left(\mathbf{x'}\right)}}\\
 & \geq & \abs{e^{f_{\pm}\left(\mathbf{x}\right)}}\frac{\abs{y_{1}-y'_{1}}}{16}>0\qquad,
\end{eqnarray*}

and so, if $y_{1}\neq y'_{1}$, $0=\abs{y_{1}e^{\rho\left(\mathbf{x}\right)}-y'_{1}e^{\rho\left(\mathbf{x'}\right)}}>0$,
which is a contradiction.

Conclusion: $\left(y_{1},y_{2}\right)=\left(y'_{1},y'_{2}\right)$
and then $\phi_{\pm}$ is into in $S_{\pm}\left(r,\epsilon\right)\times\left(\ww C^{2},0\right)$.

Since $\phi_{\pm}$ is into and analytic in $S_{\pm}\left(r,\epsilon\right)\times\left(\ww C^{2},0\right)$,
it is a biholomorphism between $S_{\pm}\left(r,\epsilon\right)\times\left(\ww C^{2},0\right)$
and its image which is necessarily open (an analytic function is open),
and of the same form.
\end{proof}

\subsection{\label{subsec: proof lemma}Resolution of the homological equation:
proof of Lemma \ref{lem: solution eq homo}}

~

The goal of this subsection is to prove Lemma \ref{lem: solution eq homo}
by studying the existence of paths asymptotic to the singularity and
tangent to the foliation, and then to use them to construct the solution
to the homological equation $\left(\mbox{\ref{eq: eq homo lemme}}\right)$.

\emph{For convenience and without lost of generality we assume $\lambda=1$}
\emph{during this subsection} $\Big($otherwise we can divide our
vector field by $\lambda\neq0$, make $x\mapsto\lambda x$ and finally
consider $\exp\left(-i\arg\left(\lambda\right)\right).S_{\pm}\left(r,\epsilon\right)$
instead of $S_{\pm}\left(r,\epsilon\right)$: these modifications
do not change $a_{1}$ and $a_{2}$,$\Big)$.

\subsubsection{\label{subsec:Domain-of-stability}Domain of stability and asymptotic
paths}

~

We consider 
\begin{eqnarray*}
Z_{\pm} & = & Y_{0}+C_{\pm}\left(x,\mathbf{y}\right)\overrightarrow{\cal C}+xR_{\pm}^{\left(1\right)}\left(x,\mathbf{y}\right)\overrightarrow{\cal R}\\
 & = & \left(\begin{array}{c}
x^{2}\\
y_{1}\left(-\left(1+C_{\pm}\left(x,\mathbf{y}\right)\right)+a_{1}x+xR_{\pm}^{\left(1\right)}\left(x,\mathbf{y}\right)\right)\\
y_{2}\left(1+C_{\pm}\left(x,\mathbf{y}\right)+a_{2}x+xR_{\pm}^{\left(1\right)}\left(x,\mathbf{y}\right)\right)
\end{array}\right)
\end{eqnarray*}
with $\Re\left(a_{1}+a_{2}\right)>0$, and $C_{\pm},R_{\pm}^{\left(1\right)}$
analytic in $S_{\pm}\left(r,\epsilon\right)\times\mathbf{D\left(0,r\right)}$
and dominated by $\norm{\mathbf{y}}_{\infty}$. More precisely, we
consider the Cauchy problem of unknown $\mathbf{x}\left(t\right):=\left(x\left(t\right),y_{1}\left(t\right),y_{2}\left(t\right)\right)$,
with real and increasing time $t\geq0$, associated to 
\[
X_{\pm}:=\frac{\pm i}{1+\left(\frac{a_{2}-a_{1}}{2}\right)x+C_{\pm}}Z_{\pm}\,\,,
\]
\emph{i.e.} 
\begin{equation}
\begin{cases}
\ddd xt=\frac{\pm ix^{2}}{1+\left(\frac{a_{2}-a_{1}}{2}\right)x+C_{\pm}}\\
\ddd{y_{1}}t=\frac{\pm iy_{1}}{1+\left(\frac{a_{2}-a_{1}}{2}\right)x+C_{\pm}}\left(-\left(1+C_{\pm}\left(x,\mathbf{y}\right)\right)+a_{1}x+xR_{\pm}^{\left(1\right)}\left(x,\mathbf{y}\right)\right)\\
\ddd{y_{2}}t=\frac{iy_{2}}{1+\left(\frac{a_{2}-a_{1}}{2}\right)x+C_{\pm}}\left(1+C_{\pm}\left(x,\mathbf{y}\right)+a_{2}x+xR_{\pm}^{\left(1\right)}\left(x,\mathbf{y}\right)\right)\\
\mathbf{x}\left(t\right)=\mathbf{x}_{0}=\left(x_{0},y_{1,0},y_{2,0}\right)\in S_{\pm}\left(r,\epsilon\right)\times\times\mathbf{D\left(0,r\right)} & .
\end{cases}\label{eq: Probleme de cauchy}
\end{equation}
We denote by $\left(t,\mathbf{x}_{0}\right)\mapsto\Phi_{X_{\pm}}^{t}\left(\mathbf{x}_{0}\right)$
the flow of $X_{\pm}$ with increasing time $t\geq0$ and with initial
point $\mathbf{x}_{0}$: $\Phi_{X\pm}^{0}\left(\mathbf{x}_{0}\right)=\mathbf{x}_{0}$.

We will prove the following:
\begin{prop}
\label{prop: domaine stable}For all ${\displaystyle \epsilon\in\left]0,\frac{\pi}{2}\right[}$,
there exists finite sectors $S_{\pm}\left(r,\epsilon\right),S_{\pm}\left(r',\epsilon\right)$
with $r,r'>0$ and an open domain $\Omega_{\pm}$ stable by the flow
of $\left(\mbox{\ref{eq: Probleme de cauchy}}\right)$ with increasing
time $t\geq0$ such that 
\[
S_{\pm}\left(r',\epsilon\right)\times\mathbf{D\left(0,r'\right)}\subset\Omega_{\pm}\subset S_{\pm}\left(r,\epsilon\right)\times\mathbf{D\left(0,r\right)}\qquad,
\]
(\emph{cf. }figure \ref{fig: domaine stable}). Moreover, if $\mathbf{x}_{0}\in\Omega_{\pm}$
then the corresponding solution of $\left(\mbox{\ref{eq: Probleme de cauchy}}\right)$,
namely $\mathbf{x}\left(t\right):=\Phi_{X_{\pm}}^{t}\left(\mathbf{x}_{0}\right)$
exists for all $t\geq0$ and $\mathbf{x}\left(t\right)\rightarrow\mathbf{0}$
as $t\rightarrow+\infty$.
\end{prop}

\begin{figure}
\includegraphics[bb=0bp 0bp 750bp 440bp,scale=0.55]{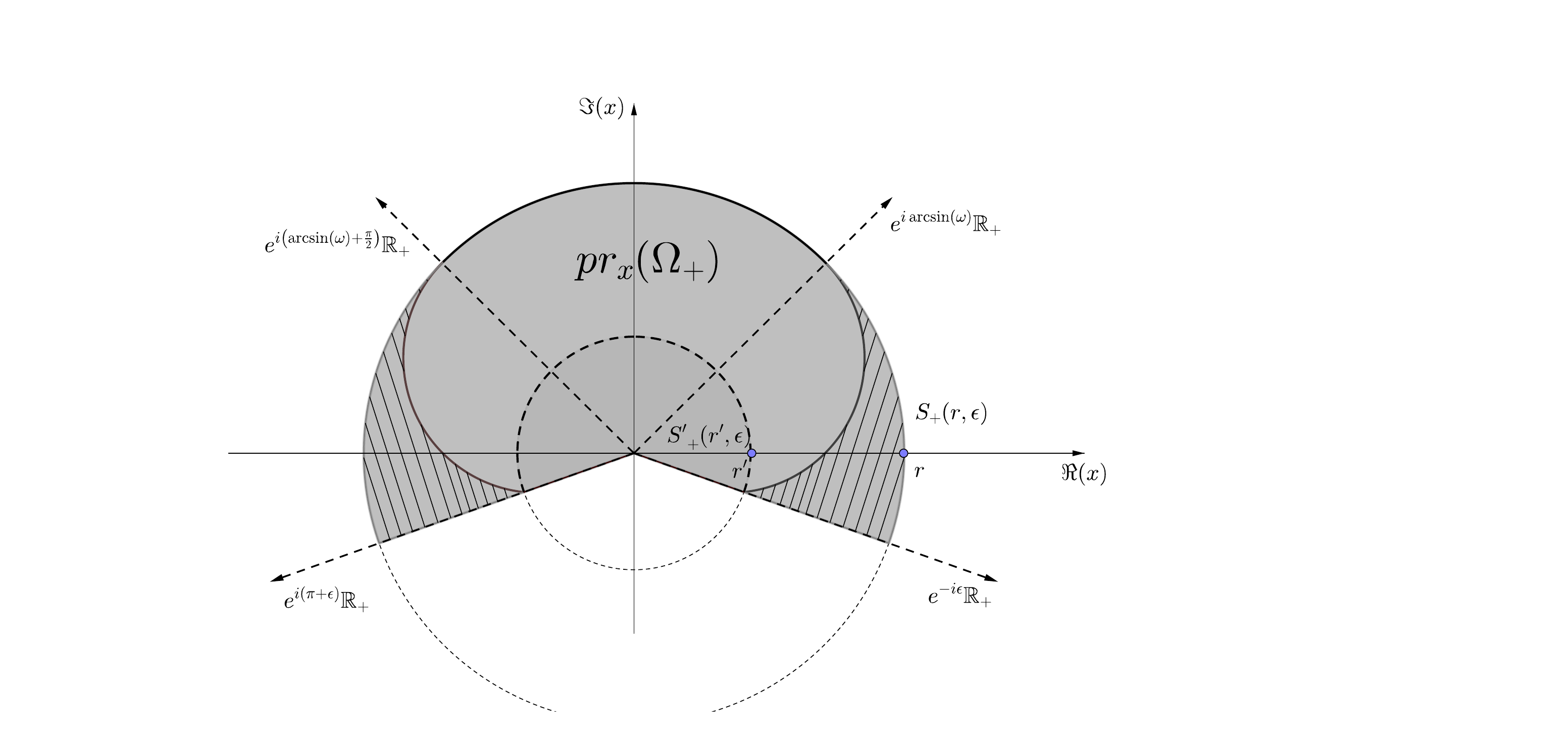}\caption{\foreignlanguage{french}{\label{fig: domaine stable}\foreignlanguage{english}{Representation
of the projection $pr_{x}\left(\Omega_{+}\right)$ of the stable domain
$\Omega_{+}$ in the $x$-space. }}}
\end{figure}
\begin{rem}
This will prove that the solution $\mathbf{x}\left(t\right)$ to $\left(\mbox{\ref{eq: Probleme de cauchy}}\right)$
exists for all $t\geq0$ and tends to the origin: it defines a path
tangent to the foliation and asymptotic to the origin. Moreover, notice
that the domain $\Omega_{\pm}$ depends on the choice of $r$ and
$r'>0$.
\end{rem}

\begin{defn}
\label{def: chemin asympto}We define the \emph{asymptotic path} with
base point $\mathbf{x}_{0}\in\Omega_{\pm}$ associated to $X_{\pm}$
the path $\gamma_{\pm,\mathbf{x}_{0}}:=\acc{\Phi_{X_{\pm}}^{t}\left(\mathbf{x}_{0}\right),\,t\geq0}$. 
\end{defn}

\emph{For convenience and without lost of generality we }only detail
the case where ``$\pm=+$'' (the case where ``$\pm=-$'' is totally
similar).

If we write $a:=a_{1}+a_{2}$ and $b:=\frac{a_{2}-a_{1}}{2}$, in
the case ``$\pm=+$'' we have:
\[
\begin{cases}
\ddd xt=\frac{ix^{2}}{1+bx+C_{+}}\\
\ddd{y_{1}}t=iy_{1}\left(-1+\left(\frac{\frac{a}{2}+R_{+}^{\left(1\right)}\left(x,\mathbf{y}\right)}{1+bx+C_{+}\left(x,\mathbf{y}\right)}\right)x\right)\\
\ddd{y_{2}}t=iy_{2}\left(1+\left(\frac{\frac{a}{2}+R_{+}^{\left(1\right)}\left(x,\mathbf{y}\right)}{1+bx+C_{+}\left(x,\mathbf{y}\right)}\right)x\right)\\
\mathbf{x}\left(t\right)=\mathbf{x}_{0}=\left(x_{0},y_{1,0},y_{2,0}\right)\in S_{+}\left(r,\epsilon\right)\times\mathbf{D\left(0,r\right)}.
\end{cases}
\]
We also consider the differential equations satisfied by $\abs{x\left(t\right)}$,
$\abs{y_{1}\left(t\right)}$, $\abs{y_{2}\left(t\right)}$ and $\theta\left(t\right):=\arg\left(x\left(t\right)\right)$:
\[
\begin{cases}
\ddd{\abs{x\left(t\right)}}t=\abs{x\left(t\right)}\Re\left(\frac{ix\left(t\right)}{1+bx\left(t\right)+C_{+}\left(\mathbf{x}\left(t\right)\right)}\right)\\
\ddd{\abs{y_{1}\left(t\right)}}t=\abs{y_{1}\left(t\right)}\Re\left(ix\left(t\right)\left(\frac{\frac{a}{2}+R_{+}^{\left(1\right)}\left(\mathbf{x}\left(t\right)\right)}{1+bx\left(t\right)+C_{+}\left(\mathbf{x}\left(t\right)\right)}\right)\right)\\
\ddd{\abs{y_{2}\left(t\right)}}t=\abs{y_{2}\left(t\right)}\Re\left(ix\left(t\right)\left(\frac{\frac{a}{2}+R_{+}^{\left(1\right)}\left(\mathbf{x}\left(t\right)\right)}{1+bx\left(t\right)+C_{+}\left(\mathbf{x}\left(t\right)\right)}\right)\right)\\
\ddd{\theta\left(t\right)}t=\Im\left(\frac{ix\left(t\right)}{1+bx\left(t\right)+C_{+}\left(\mathbf{x}\left(t\right)\right)}\right).
\end{cases}
\]

\medskip{}
For any non-zero complex number $\zeta$ and positive numbers $R,B>0$,
we denote by $\Sigma_{+}\left(\zeta,R,B\right)$ the sector of radius
$R$ bisected by $i\bar{\zeta}\ww R_{+}$ and of opening $\pi-2\arcsin\left(B\right)=2\arccos\left(B\right)$:
\begin{eqnarray*}
\Sigma_{+}\left(\zeta,R,B\right) & := & \acc{x\in D\left(0,R\right)\mid\Im\left(\zeta x\right)>B\abs{\zeta x}}\\
 & = & \acc{x\in D\left(0,R\right)\mid-\arccos\left(B\right)<\arg\left(x\right)-\arg\left(i\bar{\zeta}\right)<\arccos\left(B\right)}\,\,.
\end{eqnarray*}
For $T,R>0$, we denote by $\Theta_{+}\left(R,T\right)$ $\big($\emph{resp.}
$\Theta_{-}\left(R,T\right)$$\big)$ the sector of radius $R$ bisected
by $\ww R_{+}$ $\big($\emph{resp.} $\ww R_{-}$$\big)$ and of opening
$2\arccos\left(T\right)$:
\begin{eqnarray*}
\Theta_{+}\left(R,T\right) & := & \acc{x\in D\left(0,R\right)\mid\Re\left(x\right)>T\abs x}\\
 & = & \acc{x\in D\left(0,R\right)\mid-\arccos\left(T\right)<\arg\left(x\right)<\arccos\left(T\right)}\\
\Theta_{-}\left(R,T\right) & := & \acc{x\in D\left(0,R\right)\mid\Re\left(x\right)<-T\abs x}\\
 & = & \acc{x\in D\left(0,R\right)\mid-\arccos\left(T\right)<\arg\left(x\right)-\pi<\arccos\left(T\right)}\,\,
\end{eqnarray*}
Since $\Re\left(a\right)>0$ by assumption, we can choose $\omega'\in\left]0,\frac{\Re\left(a\right)}{\abs a}\right[$,
such that $\Sigma_{+}\left(a,r,\omega'\right)$ contains $i\ww R_{>0}$.
Indeed, we have 
\[
\abs{\arg\left(i\right)-\arg\left(i\overline{a}\right)}=\abs{\arg\left(a\right)}<\arccos\left(\omega'\right)\,\,.
\]
In particular, we have: 
\[
0<\arccos\left(\omega'\right)-\left|\arg\left(a\right)\right|<\frac{\pi}{2}
\]
so that 
\[
0<\cos\left(\arccos\left(\omega'\right)-\left|\arg\left(a\right)\right|\right)<1\,\,.
\]
Hence we take $\omega>0$ such that 
\begin{equation}
\omega\in\left]\cos\left(\arccos\left(\omega'\right)-\left|\arg\left(a\right)\right|\right),1\right[\,\,,\label{eq: choix de w}
\end{equation}
and then ${\displaystyle \Sigma_{+}\left(1,r,\omega\right)\subset\Sigma_{+}\left(a,r,\omega'\right)}$.
Indeed, if $x\in\Sigma_{+}\left(1,r,\omega\right)$, then:
\begin{equation}
-\arccos\left(\omega\right)<\arg\left(x\right)-\frac{\pi}{2}<\arccos\left(\omega\right)\,\,,\label{eq: ineg arg et arccos1}
\end{equation}
and therefore
\begin{eqnarray*}
\abs{\arg\left(x\right)-\arg\left(i.\overline{a}\right)} & < & \arccos\left(\omega\right)+\abs{\arg\left(a\right)}\\
 &  & (\mbox{by \eqref{eq: ineg arg et arccos1}})\\
 & < & \arccos\left(\omega'\right)\\
 &  & (\mbox{by \eqref{eq: choix de w}}).
\end{eqnarray*}
Finally, we fix $\mu\in\left]0,\sqrt{1-\omega^{2}}\right[$ small
enough such that 
\begin{eqnarray*}
\Theta_{+}\left(r,\mu\right)\cap\Sigma_{+}\left(1,r,\omega\right) & \neq & \emptyset\\
\Theta_{-}\left(r,\mu\right)\cap\Sigma_{+}\left(1,r,\omega\right) & \neq & \emptyset
\end{eqnarray*}
and 
\[
{\displaystyle S_{+}\left(r,\epsilon\right)\subset\Sigma_{+}\left(1,r,\omega\right)\cup\Theta_{+}\left(r,\mu\right)\cup\Theta_{-}\left(r,\mu\right)\,\,.}
\]
More precisely, we must have $0<\epsilon<\arccos\left(\mu\right)$.
The idea is now to study the behavior of $t\mapsto x\left(t\right)$
(where $t\mapsto\mathbf{x}\left(t\right)=\left(x\left(t\right),y_{1}\left(t\right),y_{2}\left(t\right)\right)$
is the solution of (\ref{eq: Probleme de cauchy})) over each domains
$\Sigma_{+}\left(1,r,\omega\right),\Theta_{+}\left(r,\mu\right),\Theta_{-}\left(r,\mu\right)$
(\emph{cf. }figure \ref{fig:Repr=0000E9sentation-des-domaines})\emph{.}
\begin{figure}
\includegraphics[bb=8cm 3bp 850bp 500bp,clip,scale=0.55]{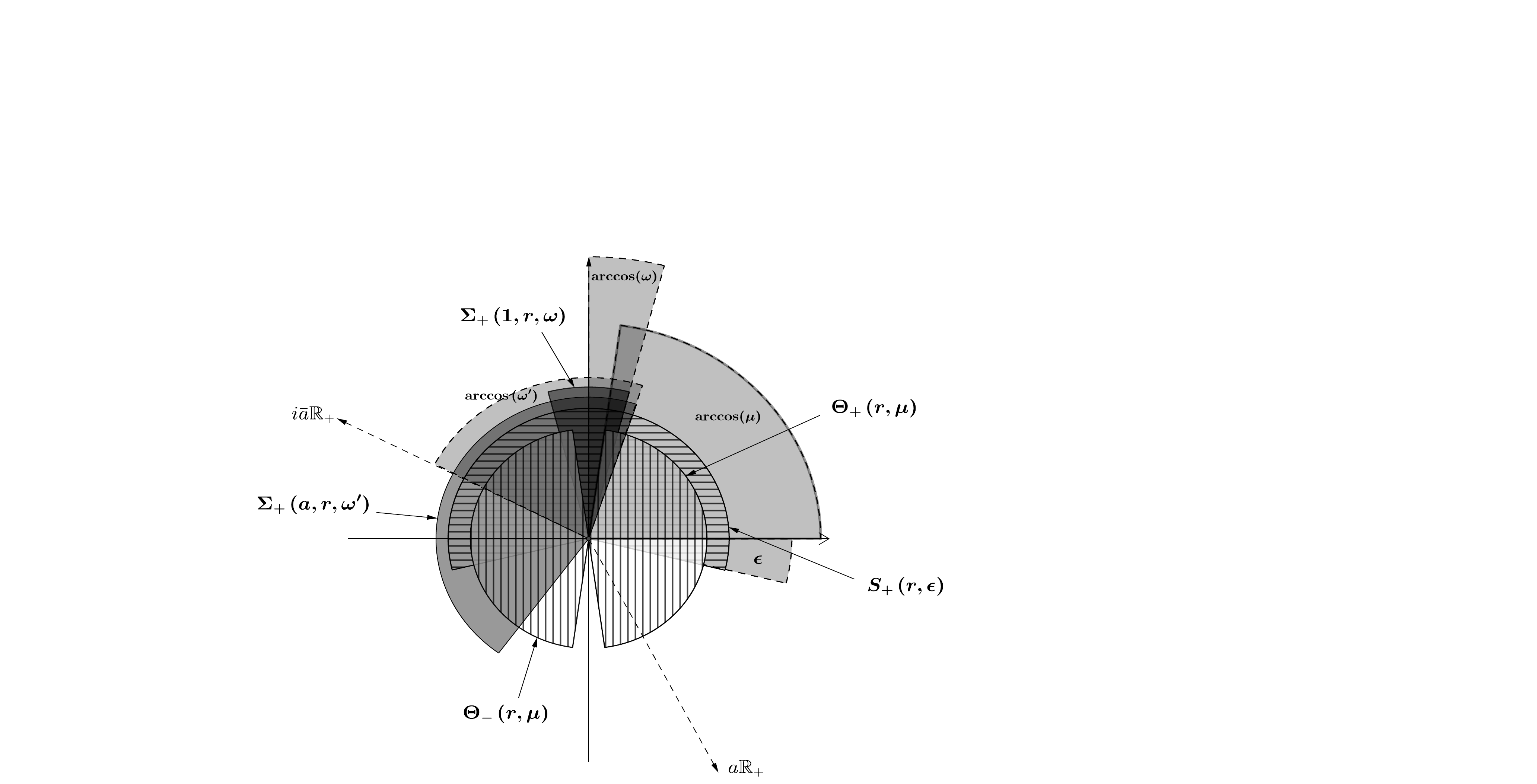}\caption{\label{fig:Repr=0000E9sentation-des-domaines}Representation of domains
$\Sigma_{+}\left(1,r,\omega\right),\Sigma_{+}\left(a,r,\omega'\right),\Theta_{+}\left(r,\mu\right),\Theta_{-}\left(r,\mu\right),S_{+}\left(r,\epsilon\right)$
(with modified radii for more clarity).}
\end{figure}

We can now prove the following result, which is a precision of Proposition
\ref{prop: domaine stable}.

~
\begin{lem}
~

\begin{enumerate}
\item There exists $r,r_{1},r_{2}>0$ such that $\Sigma_{+}\left(1,r,\omega\right)\times\mathbf{D\left(0,r\right)}$
is stable by the flow of $\left(\mbox{\ref{eq: Probleme de cauchy}}\right)$
with increasing time $t\geq0$. Moreover in this region $\abs{x\left(t\right)}$,
$\abs{y_{1}\left(t\right)}$ and $\abs{y_{2}\left(t\right)}$ decrease
and go to $0$ as $t\rightarrow+\infty$.
\item There exists $0<r'<r$, $0<r'_{1}<r_{1}$, $0<r'_{2}<r_{2}$ and an
open domain $\Omega_{+}$ stable under the action flow of $\left(\mbox{\ref{eq: Probleme de cauchy}}\right)$
with increasing time $t\geq0$ such that 
\[
S_{+}\left(r',\epsilon\right)\times\mathbf{D\left(0,r'\right)}\subset\Omega_{+}\subset S_{+}\left(r,\epsilon\right)\times\mathbf{D\left(0,r\right)}\,\,.
\]
Moreover, if $x_{0}\in\Theta_{+}\left(r',\mu\right)$ $\Big($\emph{resp.}
$x_{0}\in\Theta_{-}\left(r',\mu\right)$$\Big)$ , then $\theta\left(t\right)=\arg\left(x\left(t\right)\right),t\geq0$
is increasing (\emph{resp.} decreasing) as long as $x\left(t\right)$
remains in $\Theta_{+}\left(r',\mu\right)$ $\Big($\emph{resp.} $\Theta_{-}\left(r',\mu\right)$$\Big)$.
Finally, there exists $t_{0}\geq0$ such that for all $t\geq t_{0}$,
$x\left(t\right)\in\Sigma_{+}\left(1,r,\omega\right)$.
\end{enumerate}
\end{lem}

\begin{proof}
We fix $\delta\in\left]0,\min\left(\omega,\mu\right)\right[$, $\delta'\in\left]0,\omega'\right[$
and we take $r>0$ small enough such that for all $\mathbf{x}=\left(x,\mathbf{y}\right)\in S_{+}\left(r,\epsilon\right)\times\mathbf{D\left(0,r\right)}$,
we have
\[
\begin{cases}
\abs{\frac{1}{1+bx+C_{+}\left(\mathbf{x}\right)}-1}<\delta\\
\abs{\frac{\frac{a}{2}+R_{+}^{\left(1\right)}\left(\mathbf{x}\right)}{1+bx+C_{+}\left(\mathbf{x}\right)}-\frac{a}{2}}<\delta' & \,\,.
\end{cases}
\]
Consequently for all $\mathbf{x}\in S_{+}\times\mathbf{D\left(0,r\right)}$
we have the following estimations:
\[
\begin{cases}
-\abs x\left(1+\delta\right)<\Re\left(\frac{ix}{1+bx+C_{+}\left(\mathbf{x}\right)}\right)<\abs x\left(1+\delta\right)\\
-\abs x\left(\abs{\frac{a}{2}}+\delta'\right)<\Re\left(ix\left(\frac{\frac{a}{2}+R_{+}^{\left(1\right)}\left(x,\mathbf{y}\right)}{1+bx+C_{+}\left(x,\mathbf{y}\right)}\right)\right)<\abs x\left(\abs{\frac{a}{2}}+\delta'\right) & \,\,.
\end{cases}
\]
Moreover:

\begin{itemize}
\item if $x\in\Sigma_{+}\left(1,r,\omega\right)$ then 
\begin{eqnarray*}
\Re\left(\frac{ix}{1+bx+C_{+}\left(\mathbf{x}\right)}\right) & < & -\abs x\left(\omega-\delta\right)\qquad;
\end{eqnarray*}
\item if $x\in\Sigma_{+}\left(a,r,\omega'\right)$ $\Big($in particular
if $x\in\Sigma_{+}\left(1,r,\omega\right)$$\Big)$ then
\begin{eqnarray*}
\Re\left(ix\left(\frac{\frac{a}{2}+R_{+}^{\left(1\right)}\left(x,\mathbf{y}\right)}{1+bx+C_{+}\left(x,\mathbf{y}\right)}\right)\right) & < & -\abs x\left(\omega'-\delta'\right)\qquad;
\end{eqnarray*}
\item if $x\in\Theta_{-}\left(r,\mu\right)$ $\Big($\emph{resp.} $\Theta_{+}\left(r,\mu\right)$$\Big)$
then 
\begin{eqnarray*}
\Im\left(\frac{ix}{1+bx+C_{+}\left(\mathbf{x}\right)}\right) & < & -\abs x\left(\mu-\delta\right)\\
\Bigg(\mbox{\emph{resp. }}\Im\left(\frac{ix}{1+bx+C_{+}\left(\mathbf{x}\right)}\right) & > & \abs x\left(\mu-\delta\right)\Bigg)\qquad.
\end{eqnarray*}
\end{itemize}
Hence:

\begin{itemize}
\item for all $t\geq0$
\begin{eqnarray*}
-\left(1+\delta\right)\abs{x\left(t\right)}^{2} & <\ddd{\abs{x\left(t\right)}}t< & -\left(1+\delta\right)\abs{x\left(t\right)}^{2}
\end{eqnarray*}
and then, as long as $\mathbf{x}\left(t\right)\in S_{+}\left(r,\epsilon\right)\times\mathbf{D\left(0,r\right)}$,
we have 
\[
\abs{x\left(t\right)}>\frac{\abs{x_{0}}}{1+\left(1+\delta\right)\abs{x_{0}}t}\qquad;
\]
\item for all $t\geq0$, if $x\left(t\right)\in\Sigma_{+}\left(1,r,\omega\right)$,
then 
\begin{eqnarray}
\ddd{\abs{x\left(t\right)}}t & < & -\left(\omega-\delta\right)\abs{x\left(t\right)}^{2}\label{eq: majoration derivee de x}
\end{eqnarray}
and 
\begin{equation}
\begin{cases}
\ddd{\abs{y_{1}\left(t\right)}}t<-\left(\omega'-\delta'\right)\abs{y_{1}\left(t\right)}\abs{x\left(t\right)}\\
\ddd{\abs{y_{2}\left(t\right)}}t<-\left(\omega'-\delta'\right)\abs{y_{2}\left(t\right)}\abs{x\left(t\right)}
\end{cases}\label{eq: majoration derivee de y}
\end{equation}
so that $\abs{x\left(t\right)},\abs{y_{1}\left(t\right)}\mbox{ and }\abs{y_{2}\left(t\right)}$
are decreasing as long as $x\left(t\right)\in\Sigma_{+}\left(1,r,\omega\right)$;
\item for all $t\geq0$, if $x\left(t\right)\in\Theta_{-}\left(r,\mu\right)$
$\Big($\emph{resp.} $\Theta_{+}\left(r,\mu\right)$$\Big)$ then
\begin{eqnarray*}
\ddd{\theta}t\left(t\right) & < & -\left(\mu-\delta\right)\abs{x\left(t\right)}<\frac{-\left(\mu-\delta\right)\abs{x_{0}}}{1+\left(1+\delta\right)\abs{x_{0}}t}\\
\Bigg(\mbox{\emph{resp. }}\ddd{\theta}t\left(t\right) & > & \left(\mu-\delta\right)\abs{x\left(t\right)}>\frac{\left(\mu-\delta\right)\abs{x_{0}}}{1+\left(1+\delta\right)\abs{x_{0}}t}\Bigg)
\end{eqnarray*}
so that $t\mapsto\theta\left(t\right)$ is strictly decreasing (\emph{resp.
}increasing) as long as $x\left(t\right)\in\Theta_{-}\left(r,\mu\right)$
$\Big($\emph{resp.} $\Theta_{+}\left(r,\mu\right)$$\Big)$. Moreover,
if $\theta_{0}=\theta\left(0\right)$ is such that $x_{0}=x\left(0\right)\in\Theta_{-}\left(t,\mu\right)\backslash\Sigma_{+}\left(1,r,\omega\right)$
$\Big($\emph{resp. }$\Theta_{+}\left(r,\mu\right)\backslash\Sigma_{+}\left(1,r,\omega\right)$$\Big)$,
then as long as $x\left(t\right)\in\Theta_{-}\left(r,\mu\right)$
$\Big($\emph{resp. }$\Theta_{+}\left(r,\mu\right)$$\Big)$ we have:
\begin{eqnarray*}
\theta\left(t\right) & < & \theta_{0}-\left(\frac{\mu-\delta}{1+\delta}\right)\ln\left(1+\left(1+\delta\right)\abs{x_{0}}t\right)\\
\Bigg(\mbox{\emph{resp. }}\theta\left(t\right) & > & \theta_{0}+\left(\frac{\mu-\delta}{1+\delta}\right)\ln\left(1+\left(1+\delta\right)\abs{x_{0}}t\right)\Bigg)\qquad.
\end{eqnarray*}
We see that $x\left(t\right)\in\Sigma_{+}\left(1,r,\omega\right)$
for all 
\[
t\geq t_{0}:=\frac{\left(\exp\left(\frac{1+\delta}{\mu-\delta}\left(\theta_{0}-\frac{\pi}{2}-\arccos\left(\omega\right)\right)\right)-1\right)}{\left(1+\delta\right)\abs{x_{0}}}
\]
 
\[
\Bigg(\mbox{\emph{resp. }}t_{0}:=\frac{\left(\exp\left(\frac{1+\delta}{\mu-\delta}\left(\frac{\pi}{2}-\arccos\left(\omega\right)-\theta_{0}\right)\right)-1\right)}{\left(1+\delta\right)\abs{x_{0}}}\Bigg)\qquad.
\]
Indeed, if $t\geq t_{0}$, with $t_{0}$ as above, and if $x\left(t\right)\in\Theta_{+}\left(r,\mu\right)$,
then we have: 
\begin{eqnarray*}
\theta\left(t\right) & > & \theta_{0}+\left(\frac{\mu-\delta}{1+\delta}\right)\ln\left(1+\left(1+\delta\right)\abs{x_{0}}t\right)\\
 & > & \theta_{0}+\left(\frac{\mu-\delta}{1+\delta}\right)\ln\left(\exp\left(\frac{1+\delta}{\mu-\delta}\left(\theta_{0}-\frac{\pi}{2}-\arccos\left(\omega\right)\right)\right)\right)\\
 & = & \theta_{0}+\frac{\pi}{2}-\arccos\left(\omega\right)-\theta_{0}=\frac{\pi}{2}-\arccos\left(\omega\right)
\end{eqnarray*}
and therefore 
\[
-\arccos\left(\omega\right)<\arg\left(x\left(t\right)\right)-\frac{\pi}{2}<0\,\,.
\]
Hence, we have $x\left(t\right)\in\Sigma_{+}\left(1,r,\omega\right)$.
Moreover, notice that 
\begin{equation}
t_{0}\leq\frac{\exp\left(\left(\frac{1+\delta}{\mu-\delta}\right)\left(\epsilon+\arcsin\left(\omega\right)\right)\right)}{\left(1+\delta\right)\abs{x_{0}}}\qquad.\label{eq: temps critique}
\end{equation}
\end{itemize}
\medskip{}
On the one hand $\Sigma_{+}\left(1,r,\omega\right)\times\mathbf{D\left(0,r\right)}$
is stable by the flow of $\left(\mbox{\ref{eq: Probleme de cauchy}}\right)$
with increasing time $t\geq0$. Indeed in this region $\abs{x\left(t\right)},\abs{y_{1}\left(t\right)}\mbox{ and }\abs{y_{2}\left(t\right)}$
are decreasing, and as soon as $x\left(t\right)$ goes in $\Sigma_{+}\left(1,r,\omega\right)\cap\Theta_{-}\left(r,\mu\right)$
$\Big($\emph{resp. }$\Sigma_{+}\left(1,r,\omega\right)\cap\Theta_{+}\left(r,\mu\right)$$\Big)$,
which is non-empty and contains a part of the boundary of $\Sigma_{+}\left(1,r,\omega\right)$
with constant argument, $\theta\left(t\right)$ is decreasing\emph{
}(\emph{resp.} increasing). Then, $x\left(t\right)$ remains in $\Sigma_{+}\left(1,r,\omega\right)$.

On the other hand, as long as we are $x\left(t\right)$ belongs to
$\Theta_{-}\left(r,\mu\right)$ $\Big($\emph{resp.} $\Theta_{+}\left(r,\mu\right)$$\Big)$
we can re-parametrized the solutions by $\left(-\theta\right)$ ($resp$
$\theta$) (we are now going to make an abuse of notation, writing
when needed $x\left(\theta\right)$ or $x\left(t\right)$):
\[
\begin{cases}
\ddd{\abs x}{\left(-\theta\right)}=-\abs x\frac{\Re\left(\frac{ix}{1+bx+C_{+}\left(\mathbf{x}\right)}\right)}{\Im\left(\frac{ix}{1+bx+C_{+}\left(\mathbf{x}\right)}\right)}\leq\abs x.\frac{1+\delta}{\mu-\delta}\\
\Bigg(\mbox{\emph{resp. }}\ddd{\abs x}{\theta}=\abs x\frac{\Re\left(\frac{ix}{1+bx+C_{+}\left(\mathbf{x}\right)}\right)}{\Im\left(\frac{ix}{1+bx+C_{+}\left(\mathbf{x}\right)}\right)}\leq\abs x.\frac{1+\delta}{\mu-\delta}\Bigg)\\
\ddd{\abs{y_{1}}}{\left(-\theta\right)}=-\abs{y_{1}}\frac{\Re\left(ix\left(\frac{\frac{a}{2}+R_{+}^{\left(1\right)}\left(x,\mathbf{y}\right)}{1+bx+C_{+}\left(x,\mathbf{y}\right)}\right)\right)}{\Im\left(\frac{ix}{1+bx+C_{+}\left(\mathbf{x}\right)}\right)}\leq\abs{y_{1}}.\frac{\abs{\frac{a}{2}}+\delta'}{\mu-\delta}\\
\Bigg(\mbox{\emph{resp. }}\ddd{\abs{y_{1}}}{\theta}=\abs{y_{1}}\frac{\Re\left(ix\left(\frac{\frac{a}{2}+R_{+}^{\left(1\right)}\left(x,\mathbf{y}\right)}{1+bx+C_{+}\left(x,\mathbf{y}\right)}\right)\right)}{\Im\left(\frac{ix}{1+bx+C_{+}\left(\mathbf{x}\right)}\right)}\leq\abs{y_{1}}.\frac{\abs{\frac{a}{2}}+\delta'}{\mu-\delta}\Bigg)\\
\ddd{\abs{y_{2}}}{\left(-\theta\right)}=-\abs{y_{2}}\frac{\Re\left(ix\left(\frac{\frac{a}{2}+R_{+}^{\left(1\right)}\left(x,\mathbf{y}\right)}{1+bx+C_{+}\left(x,\mathbf{y}\right)}\right)\right)}{\Im\left(\frac{ix}{1+bx+C_{+}\left(\mathbf{x}\right)}\right)}\leq\abs{y_{2}}.\frac{\abs{\frac{a}{2}}+\delta'}{\mu-\delta}\\
\Bigg(\mbox{\emph{resp. }}\ddd{\abs{y_{2}}}{\theta}=\abs{y_{2}}\frac{\Re\left(ix\left(\frac{\frac{a}{2}+R_{+}^{\left(1\right)}\left(x,\mathbf{y}\right)}{1+bx+C_{+}\left(x,\mathbf{y}\right)}\right)\right)}{\Im\left(\frac{ix}{1+bx+C_{+}\left(\mathbf{x}\right)}\right)}\leq\abs{y_{2}}.\frac{\abs{\frac{a}{2}}+\delta'}{\mu-\delta}\Bigg).
\end{cases}
\]
Hence, if $\theta_{0}:=\theta\left(0\right)$ is such that $x_{0}:=x\left(0\right)\in\Theta_{-}\left(r,\mu\right)$
$\Big($\emph{resp.} $\Theta_{+}\left(r,\mu\right)$$\Big)$, for
$t\leq t_{0}$ we have: 
\begin{equation}
\begin{cases}
\abs{x\left(t\right)}\leq\abs{x_{0}}\exp\left(\frac{1+\delta}{\mu-\delta}\left(\theta_{0}-\theta\left(t\right)\right)\right)\\
\Bigg(\mbox{\emph{resp. }}\abs{x\left(t\right)}\leq\abs{x_{0}}\exp\left(\frac{1+\delta}{\mu-\delta}\left(\theta\left(t\right)-\theta_{0}\right)\right)\Bigg)\\
\abs{y_{1}\left(t\right)}\leq\abs{y_{1,0}}\exp\left(\frac{\abs{\frac{a}{2}}+\delta'}{\mu-\delta}\left(\theta_{0}-\theta\left(t\right)\right)\right)\\
\Bigg(\mbox{\emph{resp. }}\abs{y_{1}\left(t\right)}\leq\abs{y_{1,0}}\exp\left(\frac{\abs{\frac{a}{2}}+\delta'}{\mu-\delta}\left(\theta\left(t\right)-\theta_{0}\right)\right)\Bigg)\\
\abs{y_{2}\left(t\right)}\leq\abs{y_{2,0}}\exp\left(\frac{\abs{\frac{a}{2}}+\delta'}{\mu-\delta}\left(\theta_{0}-\theta\left(t\right)\right)\right)\\
\Bigg(\mbox{\emph{resp. }}\abs{y_{1}\left(t\right)}\leq\abs{y_{1,0}}\exp\left(\frac{\abs{\frac{a}{2}}+\delta'}{\mu-\delta}\left(\theta\left(t\right)-\theta_{0}\right)\right)\Bigg).
\end{cases}\label{eq: estimation zone croissante}
\end{equation}

\begin{defn}
\label{def: stable domain}We define the domain $\Omega_{+}$ as the
set of all 
\[
\begin{array}{c}
\mathbf{x}=\left(x,y_{1},y_{2}\right)\in S_{+}\left(r,\epsilon\right)\times\mathbf{D\left(0,r\right)}\end{array}
\]
 such that:
\end{defn}

\begin{itemize}
\item $\mbox{ if }\Im\left(x\right)\geq\omega\abs x\mbox{ then }\begin{cases}
\abs x\leq r\exp\left(\frac{1+\delta}{\mu-\delta}\left(\arg\left(x\right)-\arcsin\left(\omega\right)\right)\right)\\
\abs{y_{1}}\leq r_{1}\exp\left(\frac{\abs{\frac{a}{2}}+\delta'}{\mu-\delta}\left(\arg\left(x\right)-\arcsin\left(\omega\right)\right)\right)\\
\abs{y_{2}}\leq r_{2}\exp\left(\frac{\abs{\frac{a}{2}}+\delta'}{\mu-\delta}\left(\arg\left(x\right)-\arcsin\left(\omega\right)\right)\right)
\end{cases}$;
\item $\mbox{if }\Im\left(x\right)\leq-\omega\abs x\mbox{ then }\begin{cases}
\abs x\leq r\exp\left(\frac{1+\delta}{\mu-\delta}\left(\pi-\arcsin\left(\omega\right)-\arg\left(x\right)\right)\right)\\
\abs{y_{1}}\leq r_{1}\exp\left(\frac{\abs{\frac{a}{2}}+\delta'}{\mu-\delta}\left(\pi-\arcsin\left(\omega\right)-\arg\left(x\right)\right)\right)\\
\abs{y_{2}}\leq r_{2}\exp\left(\frac{\abs{\frac{a}{2}}+\delta'}{\mu-\delta}\left(\pi-\arcsin\left(\omega\right)-\arg\left(x\right)\right)\right)
\end{cases}$.
\end{itemize}
We see that $\Omega_{+}$ is stable by the flow of $\left(\mbox{\ref{eq: Probleme de cauchy}}\right)$
with increasing time $t\geq0$. We have seen that for any initial
condition in $\Omega_{+}$, the solution exists for any $t\geq0$,
stays in $\Omega_{+}$, and after a finite time $t_{0}\geq0$ enters
and remains in $\Sigma_{+}\left(1,r,\omega\right)$. Finally, we have:
\[
S_{+}\left(r',\epsilon\right)\times\mathbf{D}\left(\mathbf{0},\mathbf{r'}\right)\subset\Omega_{+}\subset S_{+}\left(r,\epsilon\right)\times\mathbf{D\left(0,r\right)}\,\,,
\]
where 
\[
\begin{cases}
r'=r\exp\left(-\left(\frac{1+\delta}{\mu-\delta}\right)\left(\epsilon+\arcsin\left(\omega\right)\right)\right) & <r\\
r'_{1}=r_{1}\exp\left(-\left(\frac{\abs{\frac{a}{2}}+\delta'}{\mu-\delta}\right)\left(\epsilon+\arcsin\left(\omega\right)\right)\right) & <r_{1}\\
r'_{2}=r_{2}\exp\left(-\left(\frac{\abs{\frac{a}{2}}+\delta'}{\mu-\delta}\right)\left(\epsilon+\arcsin\left(\omega\right)\right)\right) & <r_{2}\,\,.
\end{cases}
\]
\medskip{}
Let $\mathbf{x}_{0}=\left(x_{0},\mathbf{y}_{0}\right)\in\Sigma_{+}\left(1,r,\omega\right)\times\mathbf{D\left(0,r\right)}$.
From $\left(\mbox{\ref{eq: majoration derivee de x}}\right)$ and
$\left(\mbox{\ref{eq: majoration derivee de y}}\right)$we have for
all $t\geq0$:
\begin{equation}
\begin{cases}
\abs{x\left(t\right)}\leq\frac{\abs{x_{0}}}{1+\left(\omega-\delta\right)\abs{x_{0}}t}\\
\abs{y_{1}\left(t\right)}\leq\frac{\abs{y_{1,0}}}{\left(1+\left(1+\delta\right)\abs{x_{0}}t\right)^{\frac{\omega'-\delta'}{1+\delta}}}\\
\abs{y_{1}\left(t\right)}\leq\frac{\abs{y_{2,0}}}{\left(1+\left(1+\delta\right)\abs{x_{0}}t\right)^{\frac{\omega'-\delta'}{1+\delta}}} & \,\,,
\end{cases}\label{eq: estimation zone decroissante}
\end{equation}
which proves that the solutions goes to $\mathbf{0}$ as $t\rightarrow+\infty$. 
\end{proof}
\begin{rem}
Another stable domain $\Omega_{-}$ is defined similarly when dealing
with the case ``$\pm=-$''
\end{rem}

\subsubsection{Construction of a sectorial analytic solution to the homological
equation}

~

We consider the meromorphic 1-form $\tau:=\frac{\mbox{d}x}{x^{2}}$
, which satisfies $\tau\cdot\left(Z_{\pm}\right)=1$. Let also $A_{\pm}\left(x,\mathbf{y}\right)$
be analytic in $S_{\pm}\left(r,\epsilon\right)\times\left(\ww C^{2},0\right)$
and dominated by $\norm{\mathbf{y}}_{\infty}$, and $M\in\ww N_{>0}$.
The following proposition is a precision of Lemma \ref{lem: solution eq homo}.
\begin{prop}
For all $\mathbf{x}_{0}\in\Omega_{\pm}$ (see Definition \ref{def: stable domain}),
the integral defined by 
\[
{\displaystyle \alpha_{\pm}\left(\mathbf{x}_{0}\right):=-\int_{\gamma_{\pm,\mathbf{x}_{0}}}x^{M+1}A_{\pm}\left(\mathbf{x}\right)\,\tau}
\]
 is absolutely convergent (the integration path $\gamma_{\pm,\mathbf{x_{0}}}$
is the one of Definition \ref{def: chemin asympto}). Moreover, the
function $\mathbf{x}_{0}\mapsto\alpha_{\pm}\left(\mathbf{x}_{0}\right)$
is analytic in $\Omega_{\pm}$, satisfies 
\[
\cal L_{Z_{\pm}}\left(\alpha_{\pm}\right)=x^{M+1}A_{\pm}\left(\mathbf{x}\right)
\]
and $\alpha_{\pm}\left(x,\mathbf{y}\right)=x^{M}\tilde{\alpha}_{\pm}\left(x,\mathbf{y}\right)$,
where $\tilde{\alpha}_{\pm}$ is analytic on $\Omega_{\pm}$ and dominated
by $\norm{\mathbf{y}}_{\infty}$.
\end{prop}

\begin{proof}
We are going to use the estimations obtained in the previous paragraph.

\begin{itemize}
\item Let us start by proving that the integral above is convergent. We
begin with: 
\begin{eqnarray*}
\alpha_{\pm}\left(\mathbf{x}_{0}\right) & = & -\int_{0}^{+\infty}\frac{x\left(t\right)^{M+1}A_{\pm}\left(\mathbf{x}\left(t\right)\right)}{x\left(t\right)^{2}}\frac{ix\left(t\right)^{2}}{1+bx\left(t\right)+C_{+}\left(\mathbf{x}\left(t\right)\right)}\mbox{d}t\\
 & = & -i\int_{0}^{+\infty}\frac{x\left(t\right)^{M+1}A_{\pm}\left(\mathbf{x}\left(t\right)\right)}{1+bx\left(t\right)+C_{+}\left(\mathbf{x}\left(t\right)\right)}\mbox{d}t\,\,.
\end{eqnarray*}
Since $\mathbf{x}\left(t\right)\in\Omega_{\pm}$ for all $t\geq0$
and $A_{\pm}\left(x,\mathbf{y}\right)$ is dominated by $\norm{\mathbf{y}}_{\infty}$,
we have then: 
\begin{eqnarray*}
\abs{\frac{x\left(t\right)^{M+1}A_{\pm}\left(\mathbf{x}\left(t\right)\right)}{1+bx\left(t\right)+C_{+}\left(\mathbf{x}\left(t\right)\right)}} & \leq & C\abs{x\left(t\right)}^{M+1}\norm{\mathbf{y}\left(t\right)}_{\infty}
\end{eqnarray*}
where $C>0$ is some constant, independent of $\mathbf{x}_{0}$ and
$t$. For $t\geq0$ big enough, we deduce from paragraph \ref{subsec:Domain-of-stability}
that: 
\begin{eqnarray*}
\abs{\frac{x\left(t\right)^{M+1}A_{\pm}\left(\mathbf{x}\left(t\right)\right)}{1+bx\left(t\right)+C_{+}\left(\mathbf{x}\left(t\right)\right)}} & \leq & C\norm{\mathbf{y}_{0}}\left(\frac{\abs{x_{0}}}{1+\left(\omega-\delta\right)\abs{x_{0}}t}\right)^{M+1}\frac{1}{\left(1+\left(1+\delta\right)\abs{x_{0}}t\right)^{\frac{\omega'-\delta'}{1+\delta}}}\\
 & = & \underset{t\rightarrow+\infty}{\tx O}\left(\frac{1}{t^{M+1}}\right)
\end{eqnarray*}
and then the integral is absolutely convergent. 
\item Let us prove the analyticity of $\alpha_{\pm}$ in $\Omega_{\pm}$:
it is sufficient to prove that it is analytic in every compact $K\subset\Omega_{\pm}$.
Let $K$ be such a compact subset. Let $L>0$ such that for all $\mathbf{x}\in K$,
we have:
\[
\abs{\frac{A_{\pm}\left(\mathbf{x}\right)}{1+bx+C_{+}\left(\mathbf{x}\right)}}\leq L.
\]
Since $K$ in a compact subset of $\Omega_{\pm}\subset S_{\pm}\left(r,\epsilon\right)\times\wcc$
and $S_{\pm}\left(r,\epsilon\right)$ is open ($0\notin S_{\pm}\left(r,\epsilon\right)$),
there exists $\delta>0$ such that for all $\mathbf{x}=\left(x,y_{1},y_{2}\right)\in K$,
we have $\delta<\abs x<r$. Finally, according to the several estimates
in paragraph \ref{subsec:Domain-of-stability}, there exists $B>0$
such that for all $\mathbf{x}_{0}\in K$ and $t\geq0$, we have:
\[
\abs{x\left(t\right)}\leq B\frac{\abs{x_{0}}}{1+\left(\omega-\delta\right)\abs{x_{0}}t}\qquad.
\]
Hence: 
\begin{eqnarray*}
\abs{\frac{x\left(t\right)^{M+1}A_{\pm}\left(\mathbf{x}\left(t\right)\right)}{1+bx\left(t\right)+C_{+}\left(\mathbf{x}\left(t\right)\right)}} & \leq & LB^{M+1}\frac{\abs{x_{0}}^{M+1}}{\left(1+\left(\omega-\delta\right)\abs{x_{0}}t\right)^{M+1}}\\
 & \leq & \frac{LB^{M+1}r^{M+1}}{\left(1+\left(\omega-\delta\right)\delta t\right)^{M+1}},
\end{eqnarray*}
and the classical theorem concerning the analyticity of integral with
parameters proves that $\alpha_{\pm}$ is analytic in any compact
$K\subset\Omega_{\pm}$, and consequently in $\Omega_{\pm}$.
\item Let us write $F\left(\mathbf{x}\right):=\frac{\pm ix^{M+1}A_{\pm}\left(\mathbf{x}\right)}{1+bx+C_{+}\left(\mathbf{x}\right)}$,
so that 
\[
\alpha_{\pm}\left(\mathbf{x}_{0}\right)=-\int_{0}^{+\infty}F\left(\Phi_{X_{\pm}}^{t}\left(\mathbf{x}_{0}\right)\right)\mbox{d}t\,\,.
\]
 For all $\mathbf{x}_{0}\in\Omega_{\pm}$, the function $t\mapsto\mathbf{x}\left(t\right)=\Phi_{X_{\pm}}^{t}\left(\mathbf{x}_{0}\right)$
satisfies: 
\[
\pp t\left(\Phi_{X_{\pm}}^{t}\left(\mathbf{x}_{0}\right)\right)=\frac{\pm i}{1+bx\left(\Phi_{X_{\pm}}^{t}\left(\mathbf{x}_{0}\right)\right)+C_{+}\left(\Phi_{X_{\pm}}^{t}\left(\mathbf{x}_{0}\right)\right)}Z_{\pm}\left(\Phi_{\pm}^{t}\left(\mathbf{x}\right)\right)\,\,.
\]
The classical theorem about the analyticity of integral with parameters
tells us that we can compute the derivatives inside the integral symbol:
\begin{eqnarray*}
\left(\cal L_{Z_{\pm}}\alpha_{\pm}\right)\left(\mathbf{x}_{0}\right) & = & -\int_{0}^{+\infty}\cal L_{Z_{\pm}}\left(F\circ\Phi^{s}\right)\left(\mathbf{x}_{0}\right)\mbox{d}s\\
 & = & -\int_{0}^{+\infty}\mbox{D}F\left(\Phi_{X_{\pm}}^{s}\left(\mathbf{x}_{0}\right)\right).\mbox{D}\Phi_{X_{\pm}}^{s}\left(\mathbf{x}_{0}\right).Z_{\pm}\left(\mathbf{x}_{0}\right)\mbox{d}s\\
 & = & -\int_{0}^{+\infty}\mbox{D}F\left(\Phi_{X_{\pm}}^{s}\left(\mathbf{x}\right)\right).\pp t\left(\Phi_{X_{\pm}}^{s+t}\left(\mathbf{x}_{0}\right)\right)_{\mid t=0}\left(\pm\frac{1+bx_{0}+C_{\pm}\left(\mathbf{x}_{0}\right)}{i}\right)\mbox{d}s\\
 & = & -\left(\pm\frac{1+bx_{0}+C_{\pm}\left(\mathbf{x}_{0}\right)}{i}\right).\int_{0}^{+\infty}\mbox{D}F\left(\Phi_{X_{\pm}}^{s}\left(\mathbf{x}_{0}\right)\right).\pp t\left(\Phi_{X_{\pm}}^{t}\left(\mathbf{x}_{0}\right)\right)_{\mid t=s}\mbox{d}s\\
 & = & -\left(\pm\frac{1+bx_{0}+C_{\pm}\left(\mathbf{x}_{0}\right)}{i}\right).\int_{0}^{+\infty}\pp s\left(F\circ\Phi_{X_{\pm}}^{s}\left(\mathbf{x}_{0}\right)\right)ds\\
 & = & -\left(\pm\frac{1+bx_{0}+C_{\pm}\left(\mathbf{x}_{0}\right)}{i}\right).\cro{F\circ\Phi_{X_{\pm}}^{s}\left(\mathbf{x}_{0}\right)}_{s=0}^{s=+\infty}\\
 & = & -\left(\pm\frac{1+bx_{0}+C_{\pm}\left(\mathbf{x}_{0}\right)}{i}\right).\left(-F\left(\mathbf{x}_{0}\right)\right)\\
 & = & x_{0}^{M+1}A_{\pm}\left(\mathbf{x}_{0}\right).
\end{eqnarray*}
\item Let us prove that $\tilde{\alpha}_{\text{\ensuremath{\pm}}}\left(x,\mathbf{y}\right):=\frac{\alpha_{\pm}\left(x,\mathbf{y}\right)}{x^{M}}$
is bounded and dominated by $\norm{\mathbf{y}}_{\infty}$ in $\Omega_{\pm}$.
The fact that it is analytic in $\Omega_{\pm}$ is clear because $\alpha_{\pm}$
is analytic there and $0\notin\Omega_{\pm}$. As above, there exists
there exists $C>0$ such that for all $\mathbf{x}_{0}:=\left(x_{0},\mathbf{y}_{0}\right)\in\Omega_{\pm}$
and for all $t\geq0$: 
\begin{eqnarray*}
\abs{\frac{x\left(\Phi_{X_{\pm}}^{t}\left(\mathbf{x}_{0}\right)\right)^{M+1}A_{\pm}\left(\Phi_{X_{\pm}}^{t}\left(\mathbf{x}_{0}\right)\right)}{\left(1+bx\left(\Phi_{X_{\pm}}^{t}\left(\mathbf{x}_{0}\right)\right)+C_{+}\left(\Phi_{X_{\pm}}^{t}\left(\mathbf{x}_{0}\right)\right)\right)}} & \leq & C\abs{x\left(\Phi_{X_{\pm}}^{t}\left(\mathbf{x}_{0}\right)\right)}^{M+1}\norm{\mathbf{y}\left(\Phi_{X_{\pm}}^{t}\left(\mathbf{x}_{0}\right)\right)}_{\infty}\qquad.
\end{eqnarray*}
We will only deal with the case where $x_{0}\in\Theta_{\pm}\left(r,\mu\right)$
(the case where $\Sigma_{\pm}\left(1,r,\omega\right)$ is easier and
can be deduced from that case). On the one hand from $\left(\mbox{\ref{eq: estimation zone croissante}}\right)$
we have for all $t\leq t_{0}$: 
\[
\begin{cases}
\abs{x\left(\Phi_{X_{\pm}}^{t}\left(\mathbf{x}_{0}\right)\right)}\leq D\abs{x_{0}} & ,\,\mbox{where }D:=\exp\left(\frac{1+\delta}{\mu-\delta}\left(\arccos\left(\mu\right)+\epsilon\right)\right)\\
\norm{\mathbf{y}\left(\Phi_{X_{\pm}}^{t}\left(\mathbf{x}_{0}\right)\right)}_{\infty}\leq D'\norm{\mathbf{y}_{0}}_{\infty} & ,\,\mbox{where }D':=\exp\left(\frac{\abs{\frac{a}{2}}+\delta'}{\mu-\delta}\left(\arccos\left(\mu\right)+\epsilon\right)\right)\,\,.
\end{cases}
\]
On the other hand we have seen in $\left(\mbox{\ref{eq: estimation zone decroissante}}\right)$
that for all $t\geq t_{0}$:
\[
\begin{cases}
\abs{x\left(\Phi_{X_{\pm}}^{t}\left(\mathbf{x}_{0}\right)\right)}\leq\frac{\abs{x\left(\Phi_{X_{\pm}}^{t_{0}}\left(\mathbf{x}_{0}\right)\right)}}{1+\left(\omega-\delta\right)\abs{x\left(\Phi_{X_{\pm}}^{t}\left(\mathbf{x}_{0}\right)\right)}\left(t-t_{0}\right)}\\
\norm{\mathbf{y}\left(\Phi_{X_{\pm}}^{t}\left(\mathbf{x}_{0}\right)\right)}_{\infty}\leq\norm{\mathbf{y}_{0}}_{\infty} & \,\,.
\end{cases}
\]
Hence, we use the Chasles relation and the estimations above to obtain:
\begin{eqnarray*}
\abs{\tilde{\alpha}_{\pm}\left(x_{0},\mathbf{y}_{0}\right)} & \leq & \frac{\abs{\alpha_{\pm}\left(x_{0},\mathbf{y}_{0}\right)}}{\abs{x_{0}}^{M}}\\
 & \leq & \frac{CD^{M+1}D'\norm{\mathbf{y}_{0}}_{\infty}\abs{x_{0}}^{M+1}\abs{t_{0}}}{\abs{x_{0}}^{M}}\\
 &  & +\frac{C\norm{\mathbf{y}_{0}}_{\infty}}{\abs{x_{0}}^{M}}\int_{t_{0}}^{+\infty}\frac{\mbox{d}t}{\left(1+\left(\omega-\delta\right)\abs{x\left(\Phi_{X_{\pm}}^{t}\left(\mathbf{x}_{0}\right)\right)}\left(t-t_{0}\right)\right)}\\
 & \leq & CD^{M+1}D'\norm{\mathbf{y}_{0}}_{\infty}\abs{x_{0}}\abs{t_{0}}+\frac{C\norm{\mathbf{y}_{0}}_{\infty}\abs{x\left(\Phi_{X_{\pm}}^{t_{0}}\left(\mathbf{x}_{0}\right)\right)}^{M+1}}{M\left(\omega-\delta\right)\abs{x_{0}}^{M}\abs{x\left(\Phi_{X_{\pm}}^{t_{0}}\left(\mathbf{x}_{0}\right)\right)}}\,\,;
\end{eqnarray*}
and according to $\left(\mbox{\ref{eq: temps critique}}\right)$ we
have
\begin{eqnarray*}
\abs{\tilde{\alpha}_{\pm}\left(x_{0},\mathbf{y}_{0}\right)} & \leq & \left(\frac{D^{2}D'}{\left(1+\delta\right)}+\frac{1}{M\left(\omega-\delta\right)}\right)CD^{M}\norm{\mathbf{y}_{0}}_{\infty}\qquad.
\end{eqnarray*}
\end{itemize}
\end{proof}

\subsection{\label{sub :Sectorial isotropies in big sectors}Sectorial isotropies
in ``wide'' sectors and uniqueness of the normalizations: proof
of Proposition \ref{prop: unique normalizations}.}

~

We consider a normal form $\ynorm$ as given by Corollary \ref{cor: existence normalisations sectorielles}.
We study here the germs of sectorial isotropies of the normal form
$\ynorm$ in $S_{\pm}\times\left(\ww C^{2},0\right)$, where $S_{\pm}\in\germsect{\arg\left(\pm i\lambda\right)}{\eta}$
is a sectorial neighborhood of the origin with opening $\eta\in\left]\pi,2\pi\right[$
in the direction $\arg\left(\pm i\lambda\right)$. Proposition \ref{prop: unique normalizations}
states that the normalizing maps $\left(\Phi_{+},\Phi_{-}\right)$
are unique as sectorial germs. It is a straightforward consequence
of Proposition \ref{prop: isot sect} below, which show that the only
sectorial fibered isotropy (tangent to the identity) of the normal
form in over ``wide'' sector (\emph{i.e. }of opening $>\pi$) is
the identity itself.
\begin{defn}
A germ of sectorial fibered diffeomorphism $\Phi_{\theta,\eta}$ in
the direction $\theta\in\ww R$ with opening $\eta\geq0$ and tangent
to the identity, is a germ of fibered sectorial \emph{isotropy }of
$\ynorm$ (in the direction $\theta\in\ww R$ with opening $\eta\geq0$
and tangent to the identity ) if $\left(\Phi_{\theta,\eta}\right)_{*}\left(\ynorm\right)=\ynorm$
in $\cal S\in\cal S_{\theta,\eta}$. We denote by $\isotsect Y{\theta}{\eta}\subset\diffsect$
the subset formed composed of these elements.
\end{defn}

Proposition \ref{prop: unique normalizations} is an immediate consequence
of the following one.
\begin{prop}
\label{prop: isot sect}~For all $\eta\in\left]\pi,2\pi\right[$:
\[
\isotsect{\ynorm}{\arg\left(\pm i\lambda\right)}{\eta}=\acc{\tx{Id}}\,\,.
\]
\end{prop}

\begin{proof}
Let 
\[
\phi:\left(x,\mathbf{y}\right)\mapsto\left(x,\phi_{1}\left(x,\mathbf{y}\right),\phi_{2}\left(x,\mathbf{y}\right)\right)\in\isotsect{\ynorm}{\arg\left(\pm i\lambda\right)}{\eta}
\]
 be a germ of a sectorial fibered isotropy (tangent to the identity)
of $\ynorm$ in $\cal S_{\pm}\in\germsect{\arg\left(\pm i\lambda\right)}{\eta}$
with $\eta\in\left]\pi,2\pi\right[$. Possibly by reducing our domain,
we can assume that $\cal S_{\pm}$ is bounded and of the form $S_{\pm}\times\mathbf{D\left(0,r\right)}$
(where, as usual, $S_{\pm}$ is an adapted sector and $\mathbf{D\left(0,r\right)}$
a polydisc), and that $\phi$ is bounded in this domain. We have
\[
\phi_{*}\left(\ynorm\right)=\ynorm
\]
\emph{i.e.} 
\[
\mbox{D}\phi\cdot\ynorm=\ynorm\circ\phi
\]
 which is also equivalent to: 
\begin{equation}
\begin{cases}
x^{2}\ppp{\phi_{1}}x+\left(-1-c\left(y_{1}y_{2}\right)+a_{1}x\right)y_{1}\ppp{\phi_{1}}{y_{1}}+\left(1+c\left(y_{1}y_{2}\right)+a_{2}x\right)y_{2}\ppp{\phi_{1}}{y_{2}}\\
=\phi_{1}\left(-1-c\left(\phi_{1}\phi_{2}\right)+a_{1}x\right)\\
x^{2}\ppp{\phi_{2}}x+\left(-1-c\left(y_{1}y_{2}\right)+a_{1}x\right)y_{1}\ppp{\phi_{2}}{y_{1}}+\left(1+c\left(y_{1}y_{2}\right)+a_{2}x\right)y_{2}\ppp{\phi_{2}}{y_{2}}\\
=\phi_{2}\left(1+c\left(\phi_{1}\phi_{2}\right)+a_{2}x\right) & \,\,.
\end{cases}\label{eq: isot sect}
\end{equation}
Let us consider $\psi:=\phi_{1}\phi_{2}$. Then
\[
x^{2}\ppp{\psi}x+\left(-1-c\left(y_{1}y_{2}\right)+a_{1}x\right)y_{1}\ppp{\psi}{y_{1}}+\left(1+c\left(y_{1}y_{2}\right)+a_{2}x\right)y_{2}\ppp{\psi}{y_{2}}=\left(a_{1}+a_{2}\right)x\psi\qquad.
\]
By assumption we can write
\[
{\displaystyle \psi\left(x,\mathbf{y}\right)=\sum_{j_{1}+j_{2}\geq2}}\psi_{j_{1},j_{2}}\left(x\right)y_{1}^{j_{1}}y_{2}^{j_{2}}\,\,,
\]
where $\psi_{j_{1},j_{2}}\left(x\right)$ is analytic and bounded
in $S_{\pm}$ for all $j_{1},j_{2}\geq0$ and such that 
\[
{\displaystyle \sum_{j_{1}+j_{2}\geq1}\left(\underset{x\in S_{\pm}}{\sup}\left(\abs{\psi_{j_{1},j_{2}}\left(x\right)}\right)\right)y_{1}^{j_{1}}y_{2}^{j_{2}}}
\]
 is convergent near the origin of $\ww C^{2}$ (\emph{e.g. }in $\mathbf{D\left(0,r\right)})$.
Consequently, with an argument of uniform convergence in every compact
subset, we have for all $j_{1},j_{2}\geq0$:
\begin{eqnarray*}
 &  & x^{2}\ddd{\psi_{j_{1};j_{2}}}x\left(x\right)+\left(j_{2}-j_{1}+\left(a_{1}\left(j_{1}-1\right)+a_{2}\left(j_{2}-1\right)\right)x\right)\psi_{j_{1},j_{2}}\left(x\right)\\
 &  & =\left(j_{1}-j_{2}\right)\sum_{l=1}^{\min\left(j_{1},j_{2}\right)}\psi_{j_{1}-l,j_{2}-l}\left(x\right)c_{l}\,\,\,.
\end{eqnarray*}
For $j_{1}=j_{2}=j\geq1$, we have 
\[
\psi_{j,j}\left(x\right)=b_{j,j}x^{-\left(j-1\right)\left(a_{1}+a_{2}\right)},\qquad b_{j,j}\in\ww C.
\]
Since $\Re\left(a_{1}+a_{2}\right)>0$, the function $x\mapsto\psi_{j,j}\left(x\right)$
is bounded near the origin if and only if $b_{j,j}=0$ or $j=1$.
For $j_{1}>j_{2}$, we see recursively that $\psi_{j_{1},j_{2}}\left(x\right)=0$.
Indeed, we obtain by induction that 
\[
\psi_{j_{1},j_{2}}\left(x\right)=b_{j_{1},j_{2}}\exp\left(\frac{j_{2}-j_{1}}{x}\right)x^{-\left(a_{1}\left(j_{1}-1\right)+a_{2}\left(j_{2}-1\right)\right)}\qquad,
\]
and since it has to be bounded on $S_{\pm}$, we necessarily have
$b_{j_{1},j_{2}}=0$. Similarly, for $j_{1}<j_{2}$, we see recursively
that $\psi_{j_{1},j_{2}}\left(x\right)=0$. As a conclusion, $\psi\left(x,\mathbf{y}\right)=b_{1,1}y_{1}y_{2}=y_{1}y_{2}$
(we must have $b_{1,1}=1$ since $\phi$ is tangent to the identity). 

We can now solve separately each equation in $\left(\mbox{\ref{eq: isot sect}}\right)$:
\[
\begin{cases}
x^{2}\ppp{\phi_{1}}x+\left(-1-c\left(y_{1}y_{2}\right)+a_{1}x\right)y_{1}\ppp{\phi_{1}}{y_{1}}+\left(1+c\left(y_{1}y_{2}\right)+a_{2}x\right)y_{2}\ppp{\phi_{1}}{y_{2}}\\
=\phi_{1}\left(-1-c\left(y_{1}y_{2}\right)+a_{1}x\right)\\
x^{2}\ppp{\phi_{2}}x+\left(-1-c\left(y_{1}y_{2}\right)+a_{1}x\right)y_{1}\ppp{\phi_{2}}{y_{1}}+\left(1+c\left(y_{1}y_{2}\right)+a_{2}x\right)y_{2}\ppp{\phi_{2}}{y_{2}}\\
=\phi_{2}\left(1+c\left(y_{1}y_{2}\right)+a_{2}x\right) & \,\,\,.
\end{cases}
\]
As above for $i=1,2$ we can write 
\[
{\displaystyle \phi_{i}\left(x,\mathbf{y}\right)=\sum_{j_{1}+j_{2}\geq1}}\phi_{i,j_{1},j_{2}}\left(x\right)y_{1}^{j_{1}}y_{2}^{j_{2}}\,\,,
\]
 where $\phi_{i,j_{1},j_{2}}\left(x\right)$ is analytic and bounded
in $S_{\pm}$ for all $j_{1},j_{2}\geq0$ and such that 
\[
{\displaystyle \sum_{j_{1}+j_{2}\geq1}\left(\underset{x\in S_{\pm}}{\sup}\left(\abs{\phi_{i,j_{1},j_{2}}\left(x\right)}\right)\right)y_{1}^{j_{1}}y_{2}^{j_{2}}}
\]
 is a convergent entire series near the origin of $\ww C^{2}$ (\emph{e.g.
}in $\mathbf{D\left(0,r\right)})$. As above, using the uniform convergence
in every compact subset and identifying terms of same homogeneous
degree $\left(j_{1},j_{2}\right)$, we obtain:
\[
\begin{cases}
{\displaystyle x^{2}\ddd{\phi_{1,j_{1};j_{2}}}x\left(x\right)+\left(j_{2}-j_{1}+1+\left(a_{1}\left(j_{1}-1\right)+a_{2}j_{2}\right)x\right)\phi_{1,j_{1},j_{2}}\left(x\right)}\\
{\displaystyle =\sum_{l=1}^{\min\left(j_{1},j_{2}\right)}\phi_{1,j_{1}-l,j_{2}-l}\left(x\right)\left(j_{1}-j_{2}-1\right)c_{l}}\\
{\displaystyle x^{2}\ddd{\phi_{2,j_{1};j_{2}}}x\left(x\right)+\left(j_{2}-j_{1}-1+\left(a_{1}j_{1}+a_{2}\left(j_{2}-1\right)\right)x\right)\phi_{2,j_{1},j_{2}}\left(x\right)}\\
{\displaystyle =\sum_{l=1}^{\min\left(j_{1},j_{2}\right)}\phi_{2,j_{1}-l,j_{2}-l}\left(x\right)\left(j_{1}-j_{2}+1\right)c_{l}\qquad.}
\end{cases}
\]
From this we deduce: 
\[
\begin{cases}
\phi_{1,1,0}\left(x\right)=p_{1,0}\in\ww C\backslash\acc 0\\
\phi_{2,0,1}\left(x\right)=q_{0,1}\in\ww C\backslash\acc 0
\end{cases}
\]
with $p_{1,0}q_{0,1}=1$. Then, using the assumption that $\phi_{i,j_{1},j_{2}}\left(x\right)$
is analytic and bounded in $S_{\pm}$ for all $j_{1},j_{2}\geq0$,
we see (by induction on $j\geq1$) that
\[
\forall j\geq1\,\,\begin{cases}
\phi_{1,j+1,j}=0\\
\phi_{2,j,j+1}=0
\end{cases}\qquad.
\]
Indeed, we show recursively that for all $j\geq1$, we have:
\[
x^{2}\ddd{\phi_{1,j+2,j+1}}x\left(x\right)+\left(j+1\right)\left(a_{1}+a_{2}\right)x\phi_{1,j+2,j+1}\left(x\right)=0\,\,,
\]
 and the general solution to this equation is:
\[
\phi_{1,j+2,j+1}\left(x\right)={\displaystyle p_{j+2,j+1}x^{-\left(j+1\right)\left(a_{1}+a_{2}\right)}}\,\,,\,\,\mbox{with }p_{j+2j+1}\in\ww C\,\,.
\]
The quantity $\phi_{1,j+2,j+1}\left(x\right)$ is bounded near the
origin if and only if $p_{j+2,j+1}=0$, since $\Re\left(a_{1}+a_{2}\right)>0$.
The same arguments work for $\phi_{2,j,j+1}$, $j\geq1$. Consequently:
\[
{\displaystyle \begin{cases}
{\displaystyle x^{2}\ddd{\phi_{1,j_{1};j_{2}}}x\left(x\right)+\left(j_{2}-j_{1}+1+\left(a_{1}\left(j_{1}-1\right)+a_{2}j_{2}\right)x\right)\phi_{1,j_{1},j_{2}}\left(x\right)}\\
{\displaystyle =\left(j_{1}-j_{2}-1\right)\sum_{l=1}^{\min\left(j_{1},j_{2}\right)}\phi_{1,j_{1}-l,j_{2}-l}\left(x\right)c_{l}}\\
{\displaystyle x^{2}\ddd{\phi_{2,j_{1};j_{2}}}x\left(x\right)+\left(j_{2}-j_{1}-1+\left(a_{1}j_{1}+a_{2}\left(j_{2}-1\right)\right)x\right)\phi_{2,j_{1},j_{2}}\left(x\right)}\\
{\displaystyle =\left(j_{1}-j_{2}+1\right)\sum_{l=1}^{\min\left(j_{1},j_{2}\right)}\phi_{2,j_{1}-l,j_{2}-l}\left(x\right)c_{l}\qquad.}
\end{cases}}
\]
 Once again, we see recursively that for $j_{1}>j_{2}+1$, $\phi_{1,j_{1},j_{2}}\left(x\right)=0$.
Indeed, we obtain by induction that 
\[
\phi_{1,j_{1},j_{2}}\left(x\right)=p_{j_{1},j_{2}}\exp\left(\frac{j_{2}-j_{1}+1}{x}\right)x^{-\left(a_{1}\left(j_{1}-1\right)+a_{2}j_{2}\right)}\qquad,
\]
and since this has to be bounded on $S_{\pm}$, we necessarily have
$p_{j_{1},j_{2}}=0$, and therefore $\phi_{1,j_{1},j_{2}}\left(x\right)=0$.
Similarly, for $j_{1}<j_{2}+1$, we prove that $\phi_{j_{1},j_{2}}\left(x\right)=0$.
As a conclusion, $\phi_{1}\left(x,\mathbf{y}\right)=y_{1}$. By exactly
the same kind of arguments we have $\phi_{2}\left(x,\mathbf{y}\right)=y_{2}$.
\end{proof}

\subsection{\label{subsec: Weak 1-summability of normalization}Weak 1-summability
of the normalizing map}

~

Let us consider the same data as in Lemma \ref{lem: solution eq homo}.
The following lemma states that an analytic solution to the considered
homological equation in $\cal S_{\pm}\in\cal S_{\arg\left(\pm i\lambda\right),\eta}$
with $\eta\in\big[\pi,2\pi\big[$, admits a weak Gevrey-1 asymptotic
expansion in this sector. In other words, it is the weak 1-sum of
a formal solution the homological equation. Let us re-use the notations
introduced at the beginning of the latter section.
\begin{lem}
\label{lem: weak summability homo equation}Let 

\[
Z:=Y_{0}+C\left(x,\mathbf{y}\right)\overrightarrow{\cal C}+xR^{\left(1\right)}\left(x,\mathbf{y}\right)\overrightarrow{\cal R}\,\,
\]
be a formal vector field weakly 1-summable in $\cal S_{\pm}\in\cal S_{\arg\left(\pm i\lambda\right),\eta}$,
with $\eta\in\big[\pi,2\pi\big[$ and $C,R^{\left(1\right)}$ of order
at least one with respect to $\mathbf{y}$. We denote by 
\[
Z_{\pm}:=Y_{0}+C_{\pm}\left(x,\mathbf{y}\right)\overrightarrow{\cal C}+xR_{\pm}^{\left(1\right)}\left(x,\mathbf{y}\right)\overrightarrow{\cal R}
\]
the associate weak 1-sum in $\cal S_{\pm}$. Let also $A\in\form{x,\mathbf{y}}$
be weakly 1-summable in $\cal S_{\pm}$, of 1-sum $A_{\pm}$ and of
order at least one with respect to $\mathbf{y}$. Then, any sectorial
germ of an analytic function of the form $\alpha_{\pm}\left(x,\mathbf{y}\right)=x^{M}\tilde{\alpha}_{\pm}\left(x,\mathbf{y}\right)$
, with $M\in\ww N_{>0}$ and $\tilde{\alpha}_{\pm}$ analytic in $\cal S_{\pm}$,
which is dominated by $\norm{\mathbf{y}}_{\infty}$ and satisfies
\[
\cal L_{Z_{\pm}}\left(\alpha_{\pm}\right)=x^{M+1}A_{\pm}\left(x,\mathbf{y}\right)\qquad,
\]
has a Gevrey-1 asymptotic expansion in $\cal S_{\pm}$, denoted by
$\alpha$. Moreover, $\alpha$ is a formal solution to
\[
\cal L_{Z}\left(\alpha\right)=x^{M+1}A\left(x,\mathbf{y}\right)\qquad.
\]
\end{lem}

\begin{proof}
Let us write $Z$ as follow: 
\begin{eqnarray*}
Z & = & x^{2}\pp x+\left(-\left(\lambda+d\left(y_{1}y_{2}\right)\right)+a_{1}x+F_{1}\left(x,\mathbf{y}\right)\right)y_{1}\pp{y_{1}}\\
 &  & +\left(\lambda+d\left(y_{1}y_{2}\right)+a_{2}x+F_{2}\left(x,\mathbf{y}\right)\right)y_{2}\pp{y_{2}}\,\,,
\end{eqnarray*}
with $F_{1},F_{2}$ weakly 1-summable in $\cal S_{\pm}\in\cal S_{\arg\left(\pm i\lambda\right),\eta}$,
with $\eta\in\big[\pi,2\pi\big[$, of weak 1-sums $F_{1,\pm},F_{2,\pm}$
respectively, which are dominated by $\norm{\mathbf{y}}$, and with
$d\left(v\right)\in v\germ v$ without constant term. Consider the
Taylor expansion with respect to $\mathbf{y}$ of $d$,$F_{1}$,$F_{2}$,$A$
and $\alpha$:
\[
\begin{cases}
{\displaystyle d\left(y_{1}y_{2}\right)=\sum_{k\geq1}d_{k}y_{1}^{k}y_{2}^{k}}\\
{\displaystyle F_{1}\left(x,\mathbf{y}\right)=\sum_{j_{1}+j_{2}\geq1}F_{1,\mathbf{j}}\left(x\right)\mathbf{y^{j}}}\\
{\displaystyle F_{2}\left(x,\mathbf{y}\right)=\sum_{j_{1}+j_{2}\geq1}F_{2,\mathbf{j}}\left(x\right)\mathbf{y^{j}}}\\
{\displaystyle A\left(x,\mathbf{y}\right)=\sum_{j_{1}+j_{2}\geq1}A_{\mathbf{j}}\left(x\right)\mathbf{y^{j}}}\\
{\displaystyle \alpha\left(x,\mathbf{y}\right)=\sum_{j_{1}+j_{2}\geq1}\alpha_{\mathbf{j}}\left(x\right)\mathbf{y^{j}}}
\end{cases}
\]
(same expansions are valid in $\cal S_{\pm}$ for the corresponding
weak 1-sums). As usual, possibly by reducing $\cal S_{\pm}$, we can
assume that $\cal S_{\pm}=S_{\pm}\times\mathbf{D\left(0,r\right)}$
(where $S_{\pm}$ is an adapted sector and $\mathbf{D\left(0,r\right)}$
a polydisc). The homological equation 
\[
\cal L_{Z}\left(\alpha\right)=x^{M+1}A_{\pm}\left(x,\mathbf{y}\right)
\]
can be re-written:
\begin{eqnarray*}
x^{2}\ppp{\alpha}x+\left(-\left(\lambda+d\left(y_{1}y_{2}\right)\right)+a_{1}x+F_{1,\pm}\left(x,\mathbf{y}\right)\right)y_{1}\ppp{\alpha}{y_{1}}\\
+\left(\lambda+d\left(y_{1}y_{2}\right)+a_{2}x+F_{2,\pm}\left(x,\mathbf{y}\right)\right)y_{2}\ppp{\alpha}{y_{2}} & = & x^{M+1}A_{\pm}\left(x,\mathbf{y}\right)\,\,.
\end{eqnarray*}
Using normal convergence in any compact subset of $\cal S_{\pm}$,
we can compute the partial derivatives of 
\[
{\displaystyle \alpha\left(x,\mathbf{y}\right)=\sum_{j_{1}+j_{2}\geq1}\alpha_{\mathbf{j}}\left(x\right)\mathbf{y^{j}}}
\]
 with respect to $x$, $y_{1}$ or $y_{2}$ term by term, in order
to obtain after identification: $\forall\mathbf{j}=\left(j_{1},j_{2}\right)\in\ww N^{2}$,
\begin{eqnarray*}
x^{2}\ddd{\alpha_{\mathbf{j},\pm}}x\left(x\right)+\left(\lambda\left(j_{2}-j_{1}\right)+\left(a_{1}j_{1}+a_{2}j_{2}\right)x\right)\alpha_{\mathbf{j},\pm}\left(x\right) & = & G_{\mathbf{j},\pm}\left(x\right)\,\,\,,
\end{eqnarray*}
where $G_{\mathbf{j},\pm}\left(x\right)$ depends only on $d_{k},F_{1,\mathbf{k},\pm},F_{2,\mathbf{k},\pm},\alpha_{\mathbf{k},\pm}$
and $A_{\mathbf{l},\pm}$, for $k\leq\min\left(j_{1},j_{2}\right)$,
$\abs{\mathbf{k}}\leq\mathbf{\abs j}-1$ and $\abs{\mathbf{l}}\leq\abs{\mathbf{j}}$.
We obtain a similar differential equation for the associated formal
power series. Let us prove by induction on $\abs{\mathbf{j}}\geq0$
that:

\begin{enumerate}
\item $G_{\mathbf{j},\pm}$ is the 1-sum of $G_{\mathbf{j}}$ in $S_{\pm}$,
\item $G_{j,j}\left(0\right)=0$ if $\mathbf{j}=\left(j,j\right)$
\item $\alpha_{\mathbf{j},\pm}$ is the 1-sum $\alpha_{\mathbf{j}}$ in
$S_{\pm}$.
\end{enumerate}
It is paramount to use the fact that for all $\mathbf{j}\in\ww N^{2}$,
$\alpha_{\mathbf{j},\pm}$ is bounded in $S_{\pm}$. 

\begin{itemize}
\item For $\mathbf{j}=\left(0,0\right)$, we have $G_{\left(0,0\right)}=0$
and then $\alpha_{\left(0,0\right)}=0$.
\item Let $\mathbf{j}=\left(j_{1},j_{2}\right)\in\ww N^{2}$ with $\abs{\mathbf{j}}=j_{1}+j_{2}\geq1$.
Assume the property holds for all $\mathbf{k}\in\ww N^{2}$ with $\abs{\mathbf{k}}\leq\abs{\mathbf{j}}-1$.

\begin{enumerate}
\item Since $G_{\mathbf{j}}\left(x\right)$ depends only on $d_{k},F_{1,\mathbf{k}},F_{2,\mathbf{k}},\alpha_{\mathbf{k}}$
and $A_{\mathbf{l}}$, for $k\leq\min\left(j_{1},j_{2}\right)$, $\abs{\mathbf{k}}\leq\mathbf{\abs j}-1$
and $\abs{\mathbf{l}}\leq\abs{\mathbf{j}}$, then $G_{\mathbf{j}}$
is 1-summable in $S_{\pm}$, of 1-sum $G_{\mathbf{j},\pm}$. 
\item We also see that $G_{j,j}\left(0\right)=0$, if $\mathbf{j}=\left(j,j\right)$.
\item If $j_{1}\neq j_{2}$, then point 1. in Proposition \ref{prop: solution borel sommable precise}
tells us that there exists a unique formal solution $\alpha_{\mathbf{j}}\left(x\right)$
to the irregular differential equation we are looking at, and such
that $\alpha_{\mathbf{j}}\left(0\right)=\frac{1}{\lambda\left(j_{2}-j_{1}\right)}G_{\mathbf{j}}\left(0\right)$
. Moreover, this solution is 1-summable in $S_{\pm}$ since the same
goes for $G_{\mathbf{j}}$.\\
\item If however $j_{1}=j_{2}=j\geq1$, since $G_{\left(j,j\right)}\left(0\right)=0$
we can write $G_{\left(j,j\right)}\left(x\right)=x\tilde{G}_{\left(j,j\right)}\left(x\right)$
with $\tilde{G}_{\left(j,j\right)}\left(x\right)$ 1-summable in $S_{\pm}$,
and then the differential equation becomes regular: 
\[
x\ddd{\alpha_{\left(j,j\right),\pm}}x\left(x\right)+\left(a_{1}+a_{2}\right)j\alpha_{\left(j,j\right),\pm}\left(x\right)=\tilde{G}_{\left(j,j\right),\pm}\left(x\right)\,\,.
\]
Since $\Re\left(a_{1}+a_{2}\right)>0$, according to point 2. in Proposition
\ref{prop: solution borel sommable precise}, the latter equation
has a unique formal solution $\alpha_{\left(j,j\right)}\left(x\right)$
such that $\alpha_{\left(j,j\right)}\left(0\right)=\frac{\tilde{G}_{\left(j,j\right)}\left(0\right)}{\left(a_{1}+a_{2}\right)j}$,
and this solution is moreover 1-summable in $S_{\pm}$, and its 1-sum
is the only solution to this equation bounded in $S_{\pm}$. Thus,
it is necessarily $\alpha_{\left(j,j\right),\pm}$.
\end{enumerate}
\end{itemize}
\end{proof}
We are now able to prove the weak 1-summability of the formal normalizing
map.
\begin{prop}
\label{prop: Weak sectorial normalizations}The sectorial normalizing
maps $\left(\Phi_{+},\Phi_{-}\right)$ in Corollary \ref{cor: existence normalisations sectorielles}
are the weak 1-sums in $\cal S_{\pm}\in\cal S_{\arg\left(\pm\lambda\right),\eta}$
of the formal normalizing map $\hat{\Phi}$ given by Theorem \ref{thm: forme normalel formelle},
for all $\eta\in\big[\pi,2\pi\big[$. In particular, $\hat{\Phi}$
is weakly 1-summable, except for $\arg\left(\pm\lambda\right)$.
\end{prop}

\begin{proof}
The normalizing map $\Phi_{\pm}$ from Corollary \ref{cor: existence normalisations sectorielles}
is constructed as the composition of two germs of sectorial diffeomorphisms,
using successively Propositions \ref{prop: forme pr=0000E9par=0000E9e ordre N}
and \ref{prop: norm sect}. The sectorial map obtained in Proposition
\ref{prop: forme pr=0000E9par=0000E9e ordre N} is 1-summable except
in directions ${\displaystyle \arg\left(\pm\lambda\right)}$. The
sectorial transformation in Proposition \ref{prop: norm sect} is
constructed as the composition of two germs of sectorial diffeomorphisms,
using successively Proposition \ref{prop: radial part} and \ref{prop: tangential part}.
Both of these two sectorial maps are built thanks to Lemma \ref{lem: solution eq homo}.
Lemma \ref{lem: weak summability homo equation} above justifies that
each of these maps admits in fact a weak Gevrey-1 asymptotic expansion
in a domain of the form $\cal S_{\pm}\in\cal S_{\arg\left(\pm\lambda\right),\eta}$,
for all $\eta\in\big[\pi,2\pi\big[$. Consequently, the same goes
for the sectorial diffeomorphisms of Proposition \ref{prop: norm sect},
and then for those of Corollary \ref{cor: existence normalisations sectorielles}
(we used here Proposition \ref{prop: composition faible} for the
composition). 

Using item \ref{enu: derivation faible} in Lemma \ref{lem: proprietes faibles},
we deduce that the weak Gevrey-1 asymptotic expansion of the sectorial
normalizing maps of Corollary \ref{cor: existence normalisations sectorielles}
is therefore a formal normalizing map, such as the one given by Theorem
\ref{thm: forme normalel formelle}. By uniqueness of such a normalizing
map, it is $\hat{\Phi}$.
\end{proof}

\section{\label{sec: analytic classification}Analytic classification}

In this section, we end up the proofs of both Theorems \ref{Th: Th drsn}
and \ref{thm: espace de module}. In order to do this, we prove that
the Stokes diffeomorphisms $\Phi_{\lambda}$ and $\Phi_{-\lambda}$
obtained from the germs of sectorial normalizing maps $\Phi_{+}$
and $\Phi_{-}$, which a priori admit identity as \emph{weak} Gevrey-1
asymptotic expansion, admit in fact identity as ``true'' Gevrey-1
asymptotic expansion. This will be done by studying more generally
germs of sectorial isotropies of the normal form $\ynorm$ in sectorial
domains with ``narrow'' opening, and by considering theses isotropies
in the of space of leaves. Using Theorem \ref{th: martinet ramis},
which is a ``non-abelian'' version of the Ramis-Sibuya theorem due
to Martinet and Ramis \cite{MR82}, this will has as consequence the
fact that the sectorial normalizing maps $\Phi_{+}$ and $\Phi_{-}$
both admit the formal normalizing map $\hat{\Phi}$ as Gevrey-$1$
asymptotic expansion in the corresponding sectorial domains. This
proves Theorem \ref{Th: Th drsn}. Moreover, another consequence will
be Theorem \ref{thm: espace de module}. Finally, we will describe
the moduli space of analytic classification in terms of some spaces
of power series.

\bigskip{}

From now on, we fix a normal form 
\[
\ynorm=x^{2}\pp x+\left(-\lambda+a_{1}x-c\left(y_{1}y_{2}\right)\right)y_{1}\pp{y_{1}}+\left(\lambda+a_{2}x+c\left(y_{1}y_{2}\right)\right)y_{2}\pp{y_{2}}\,\,\,\,,
\]
with $\lambda\in\ww C^{*},$$\Re\left(a_{1}+a_{2}\right)>0$ and $c\in v\germ v$
vanishing at the origin. We denote by ${\displaystyle \cro{\ynorm}}$
the set of germs of holomorphic doubly-resonant saddle-nodes in $\left(\ww C^{3},0\right)$,
formally conjugate to $\ynorm$ by formal fibered diffeomorphisms
tangent to the identity. We refer the reader to Definition \ref{def: sectorial germs}
for notions relating to sectors.
\begin{defn}
We define $\Lambda_{\lambda}^{\left(\mbox{\ensuremath{\tx{weak}}}\right)}\left(\ynorm\right)$
$\left(\mbox{\emph{resp.} }\Lambda_{-\lambda}^{\left(\mbox{\ensuremath{\tx{weak}}}\right)}\left(\ynorm\right)\right)$
as the group of germs of sectorial fibered isotropies of $\ynorm$,
admitting the identity as \textbf{weak} Gevrey-1 asymptotic expansion
in sectorial domains of the form $S_{\lambda}\times\left(\ww C^{2},0\right)$
$\left(resp.\,S_{-\lambda}\times\left(\ww C^{2},0\right)\right)$,
where:
\begin{eqnarray*}
S_{\lambda} & \in & \cal{AS}_{\arg\left(\lambda\right),\pi}\\
S_{-\lambda} & \in & AS_{\arg\left(-\lambda\right),\pi}
\end{eqnarray*}
(see Definition \ref{def: asymptotitc sector}).
\end{defn}

We recall the notations given in the introduction: we have defined
$\Lambda_{\lambda}\left(\ynorm\right)$ $\left(\mbox{\emph{resp.} }\Lambda_{-\lambda}\left(\ynorm\right)\right)$
as the group of germs of sectorial fibered isotropies of $\ynorm$,
admitting the identity as Gevrey-1 asymptotic expansion in sectorial
domains of the form $S_{\lambda}\times\left(\ww C^{2},0\right)$ $\left(resp.\,S_{-\lambda}\times\left(\ww C^{2},0\right)\right)$.
It is clear that we have:
\[
\Lambda_{\pm\lambda}\left(\ynorm\right)\subset\Lambda_{\pm\lambda}^{\left(\mbox{\ensuremath{\tx{weak}}}\right)}\left(\ynorm\right)\subset\isotsect Y{\arg\left(\pm\lambda\right)}{\eta},\,\,\forall\eta\in\left]0,\pi\right[\,\,.
\]

The main result of this section is the following.
\begin{prop}
\label{prop: isotropies plates}Any $\psi\in\Lambda_{\pm\lambda}^{\left(\tx{weak}\right)}\left(\ynorm\right)$
admits the identity as Gevrey-1 asymptotic expansion in $S_{\pm\lambda}\times\left(\ww C^{2},0\right)$.
In other words:
\[
\Lambda_{\pm\lambda}^{\left(\tx{weak}\right)}\left(\ynorm\right)=\Lambda_{\pm\lambda}\left(\ynorm\right)\,\,.
\]
\end{prop}

\subsection{Proofs of the main results (assuming Proposition \ref{prop: isotropies plates}
)}

~

In this subsection, we prove the main results of this paper, assuming
Proposition \ref{prop: isotropies plates} above holds.

\subsubsection{Analytic invariants: Stokes diffeomorphisms}

~

According to Corollary \ref{cor: existence normalisations sectorielles},
to any ${\displaystyle Y\in{\displaystyle \cro{\ynorm}}}$, we can
associate a pair of germs of sectorial fibered isotropies in $S_{\lambda}\times\left(\ww C^{2},0\right)$
and $S_{-\lambda}\times\left(\ww C^{2},0\right)$ respectively, denoted
by $\left(\Phi_{\lambda},\Phi_{-\lambda}\right)$:
\[
\begin{cases}
\Phi_{\lambda}:=\left(\Phi_{+}\circ\Phi_{-}^{-1}\right)_{\mid S_{\lambda}\times\left(\ww C^{2},0\right)}\in\isotsect Y{\arg\left(\lambda\right)}{\eta} & \,,\,\forall\eta\in\left]0,\pi\right[\\
\Phi_{-\lambda}:=\left(\Phi_{-}\circ\Phi_{+}^{-1}\right)_{\mid S_{-\lambda}\times\left(\ww C^{2},0\right)}\in\isotsect Y{\arg\left(-\lambda\right)}{\eta} & \,,\,\forall\eta\in\left]0,\pi\right[\,\,,
\end{cases}
\]
where $\left(\Phi_{+},\Phi_{-}\right)$ is the pair of the sectorial
normalizing maps given by Corollary \ref{cor: existence normalisations sectorielles}.
\begin{prop}
For any given $\eta\in\left]0,\pi\right[$ the map 
\begin{eqnarray*}
{\displaystyle {\displaystyle \cro{\ynorm}}} & \longrightarrow & \isotsect Y{\arg\left(\lambda\right)}{\eta}\times\isotsect Y{\arg\left(-\lambda\right)}{\eta}\\
Y & \longmapsto & \left(\Phi_{\lambda},\Phi_{-\lambda}\right)\,\,,
\end{eqnarray*}
actually ranges in $\Lambda_{\lambda}^{\left(\mbox{\ensuremath{\tx{weak}}}\right)}\left(\ynorm\right)\times\Lambda_{-\lambda}^{\left(\mbox{\ensuremath{\tx{weak}}}\right)}\left(\ynorm\right)$.
\end{prop}

\begin{proof}
The fact that the sectorial normalizing maps $\Phi_{+},\Phi_{-}$
given by Corollary \ref{cor: existence normalisations sectorielles}
 both conjugate ${\displaystyle Y\in{\displaystyle \cro{\ynorm}}}$
to $\ynorm$ in the corresponding sectorial domains proves that the
arrow above is well-defined, with values in $\isotsect Y{\arg\left(\lambda\right)}{\eta}\times\isotsect Y{\arg\left(-\lambda\right)}{\eta}$,
for all $\eta\in\left]0,\pi\right[$. The fact that $\Phi_{\pm\lambda}$
admits the identity as weak Gevrey-1 asymptotic expansion in $S_{\pm\lambda}\times\left(\ww C^{2},0\right)$
comes from Proposition \ref{prop: Weak sectorial normalizations}
($\Phi_{+}$ and $\Phi_{-}$ admits the same weak Gevrey-1 asymptotic
expansion in $S_{\lambda}\times\left(\ww C^{2},0\right)$ and $S_{-\lambda}\times\left(\ww C^{2},0\right)$)
and from Proposition \ref{prop: composition faible}.
\end{proof}
The subgroup $\fdiff[\ww C^{3},0,\tx{Id}]\subset\fdiff$ formed by
fibered diffeomorphisms tangent to the identity acts naturally on
${\displaystyle {\displaystyle {\displaystyle \cro{\ynorm}}}}$ by
conjugacy. Now we show that the uniqueness of germs of sectorial normalizing
maps $\left(\Phi_{+},\Phi_{-}\right)$ implies that the Stokes diffeomorphisms
$\left(\Phi_{\lambda},\Phi_{-\lambda}\right)$ of a vector field ${\displaystyle Y\in{\displaystyle {\displaystyle \cro{\ynorm}}}}$
is invariant under the action of $\fdiff[\ww C^{3},0,\tx{Id}]$. Furthermore,
this map is one-to-one.
\begin{prop}
\label{prop: espace de module injection}The map 
\begin{eqnarray*}
{\displaystyle {\displaystyle {\displaystyle \cro{\ynorm}}}} & \longrightarrow & \Lambda_{\lambda}^{\left(\mbox{\ensuremath{\tx{weak}}}\right)}\left(\ynorm\right)\times\Lambda_{-\lambda}^{\left(\mbox{\ensuremath{\tx{weak}}}\right)}\left(\ynorm\right)\\
Y & \longmapsto & \left(\Phi_{\lambda},\Phi_{-\lambda}\right)
\end{eqnarray*}
 factorizes through a one-to-one map
\begin{eqnarray*}
\quotient{{\displaystyle {\displaystyle \cro{\ynorm}}}}{{\fdiff[\ww C^{3},0,\tx{Id}]}} & \longrightarrow & \Lambda_{\lambda}^{\left(\mbox{\ensuremath{\tx{weak}}}\right)}\left(\ynorm\right)\times\Lambda_{-\lambda}^{\left(\mbox{\ensuremath{\tx{weak}}}\right)}\left(\ynorm\right)\\
Y & \longmapsto & \left(\Phi_{\lambda},\Phi_{-\lambda}\right)\,\,.
\end{eqnarray*}
\end{prop}

\begin{rem}
This very result means that the Stokes diffeomorphisms encode completely
the class of $Y$ in the quotient $\quotient{\left[\ynorm\right]}{{\fdiff[\ww C^{3},0,\tx{Id}]}}$
as they separate conjugacy classes.
\end{rem}

\begin{proof}
First of all, let us prove that the latter map is well-defined. Let
${\displaystyle Y,\tilde{Y}\in{\displaystyle {\displaystyle \cro{\ynorm}}}}$
and $\Theta\in\fdiff[\ww C^{3},0,\tx{Id}]$ be such that $\Theta_{*}\left(Y\right)=\tilde{Y}$.
We denote by $\Phi_{\pm}$ (\emph{resp.} $\tilde{\Phi}{}_{\pm}$)
the sectorial normalizing maps of $Y$ (\emph{resp. $\tilde{Y}$}),
and $\left(\Phi_{\lambda},\Phi_{-\lambda}\right)$ $\Big($\emph{resp.}
$\left(\tilde{\Phi}_{\lambda},\tilde{\Phi}_{-\lambda}\right)$$\Big)$
the Stokes diffeomorphisms of $Y$ (\emph{resp.} $\tilde{Y}$). By
assumption, $\tilde{\Phi}_{\pm}\circ\Theta$ is also a germ of a sectorial
fibered normalization of $Y$ in $S_{\pm}\times\left(\ww C^{2},0\right)$,
which is tangent to the identity. Thus, according to the uniqueness
statement in Theorem \ref{Th: Th drsn}: 
\[
\Phi_{\pm}=\tilde{\Phi}_{\pm}\circ\Theta\,\,.
\]
Consequently, in $S_{\pm\lambda}\times\left(\ww C^{2},0\right)$ we
have 
\begin{eqnarray*}
\Phi_{\lambda} & = & \left(\Phi_{+}\circ\Phi_{-}^{-1}\right)_{\mid S_{\lambda}\times\left(\ww C^{2},0\right)}\\
 & = & \tilde{\Phi}_{+}\circ\Theta\circ\Theta^{-1}\circ\tilde{\Phi}_{-}\\
 & = & \tilde{\Phi}_{\lambda}\,\,,
\end{eqnarray*}
and similarly
\begin{eqnarray*}
\Phi_{-\lambda} & = & \left(\Phi_{-}\circ\Phi_{+}^{-1}\right)_{\mid S_{-\lambda}\times\left(\ww C^{2},0\right)}\\
 & = & \tilde{\Phi}_{-}\circ\Theta\circ\Theta^{-1}\circ\left(\tilde{\Phi}_{+}\right)^{-1}\\
 & = & \tilde{\Phi}_{-\lambda}\,\,.
\end{eqnarray*}

Let us prove that the map is one-to-one. Let $Y,\tilde{Y}\in\cro{\ynorm}$
share the same Stokes diffeomorphisms $\left(\Phi_{\lambda},\Phi_{-\lambda}\right)$.
We denote by $\Phi_{\pm}$ (\emph{resp.} $\tilde{\Phi}{}_{\pm}$)
the germ of a sectorial fibered normalizing map of $Y$ (\emph{resp.
$\tilde{Y}$}) $S_{\pm}\times\left(\ww C^{2},0\right)$. We have:
\[
{\displaystyle \begin{cases}
\Phi_{+}\circ\left(\Phi_{-}\right)^{-1}=\Phi_{\lambda}=\tilde{\Phi}_{+}\circ\left(\tilde{\Phi}_{-}\right)^{-1} & \mbox{ in }S_{\lambda}\times\left(\ww C^{2},0\right)\\
\Phi_{-}\circ\left(\Phi_{+}\right)^{-1}=\Phi_{-\lambda}=\tilde{\Phi}_{-}\circ\left(\tilde{\Phi}_{+}\right)^{-1} & \mbox{ in }S_{-\lambda}\times\left(\ww C^{2},0\right)\,.
\end{cases}}
\]
Thus: 
\[
{\displaystyle \begin{cases}
\left(\tilde{\Phi}_{+}\right)^{-1}\text{\ensuremath{\circ}}\,\Phi_{+}=\left(\tilde{\Phi}_{-}\right)^{-1}\text{\ensuremath{\circ}}\,\Phi_{-} & \mbox{ in }S_{\lambda}\times\left(\ww C^{2},0\right)\\
\left(\tilde{\Phi}_{+}\right)^{-1}\text{\ensuremath{\circ}}\,\Phi_{+}=\left(\tilde{\Phi}_{-}\right)^{-1}\text{\ensuremath{\circ}}\,\Phi_{-} & \mbox{ in }S_{-\lambda}\times\left(\ww C^{2},0\right)\,.
\end{cases}}
\]
We can then define a map $\varphi$ analytic in a domain of the form
$\left(D\left(0,r\right)\backslash\acc 0\right)\times\mathbf{D\left(0,r\right)}$
by setting:
\[
{\displaystyle \begin{cases}
\varphi_{\mid S_{+}}=\left(\tilde{\Phi}_{+}\right)^{-1}\text{\ensuremath{\circ}}\,\Phi_{+} & \mbox{ in }S_{+}\\
\varphi_{\mid S_{-}}=\left(\tilde{\Phi}_{-}\right)^{-1}\text{\ensuremath{\circ}}\,\Phi_{-} & \mbox{ in }S_{-}\,.
\end{cases}}
\]
This map is analytic and bounded in $\left(D\left(0,r\right)\backslash\acc 0\right)\times\mathbf{D\left(0,r\right)}$,
and the Riemann singularity theorem tells us that this map can be
analytically extended to the entire poly-disc $D\left(0,r\right)\times\mathbf{D\left(0,r\right)}$.
As a conclusion, $\varphi\in\fdiff[\ww C^{3},0,\tx{Id}]$, $\Phi_{\pm}=\tilde{\Phi}_{\pm}\circ\varphi$
and $\varphi_{*}\left(Y\right)=\tilde{Y}$.
\end{proof}

\subsubsection{Proof of Theorem \ref{Th: Th drsn}: 1-summability of the formal
normalization}

~

As a first consequence of Proposition \ref{prop: isotropies plates},
we obtain Proposition \ref{prop: sommabilit=0000E9 normalisation formelle},
which states that the formal normalizing map from Theorem \ref{thm: forme normalel formelle}
\cite{bittmann1} is in fact 1-summable.
\begin{prop}
\label{prop: sommabilit=0000E9 normalisation formelle}The unique
formal normalizing map $\hat{\Phi}$ given in \ref{thm: forme normalel formelle}
is the Gevrey-1 asymptotic expansion of the unique germs of sectorial
normalizing maps $\Phi_{+}$ and $\Phi_{-}$ in $S_{+}\times\left(\ww C^{2},0\right)$
and $S_{-}\times\left(\ww C^{2},0\right)$ respectively. In particular,
$\hat{\Phi}$ is 1-summable in every direction $\theta\neq\arg\left(\pm\lambda\right)$,
and $\left(\Phi_{+},\Phi_{-}\right)$ is its Borel-Laplace 1-sum.
\end{prop}

\begin{proof}
Let us consider the unique germs of a sectorial normalizing map $\Phi_{+}$
and $\Phi_{-}$ in $S_{+}\times\left(\ww C^{2},0\right)$ and $S_{-}\times\left(\ww C^{2},0\right)$
respectively, and their associated Stokes diffeomorphisms:
\[
\begin{cases}
\Phi_{\lambda}=\left(\Phi_{+}\circ\Phi_{-}^{-1}\right)_{\mid S_{\lambda}\times\left(\ww C^{2},0\right)}\in\Lambda_{\lambda}^{\left(\tx{weak}\right)}\left(\ynorm\right)\\
\Phi_{-\lambda}=\left(\Phi_{-}\circ\Phi_{+}^{-1}\right)_{\mid S_{-\lambda}\times\left(\ww C^{2},0\right)}\in\Lambda_{-\lambda}^{\left(\tx{weak}\right)}\left(\ynorm\right) & \,\,\,.
\end{cases}
\]
According to Proposition \ref{prop: isotropies plates}, 
\[
\Lambda_{\pm\lambda}^{\left(\tx{weak}\right)}\left(\ynorm\right)=\Lambda_{\pm\lambda}\left(\ynorm\right)\,\,,
\]
so that $\Phi_{\lambda}$ and $\Phi_{-\lambda}$ both admit the identity
as Gevrey-1 asymptotic expansion, in $S_{\lambda}\times\left(\ww C^{2},0\right)$
and $S_{-\lambda}\times\left(\ww C^{2},0\right)$ respectively. Then,
Theorem \ref{th: martinet ramis} gives the existence of 
\[
\left(\phi_{+},\phi_{-}\right)\in\diffsect[\arg\left(i\lambda\right)][\eta]\times\diffsect[\arg\left(-i\lambda\right)][\eta]
\]
 for all $\eta\in\left]\pi,2\pi\right[$, such that:{\footnotesize{}
\[
\begin{cases}
\phi_{+}\circ\left(\phi_{-}\right)_{\mid S_{\lambda}\times\left(\ww C^{2},0\right)}^{-1}=\Phi_{\lambda}\\
\phi_{-}\circ\left(\phi_{+}\right)_{\mid S_{-\lambda}\times\left(\ww C^{2},0\right)}^{-1}=\Phi_{-\lambda} & \,,
\end{cases}
\]
}and the existence of a formal diffeomorphism $\hat{\phi}$ which
is tangent to the identity, such that $\phi_{+}$ and $\phi_{-}$
both admit $\hat{\phi}$ as Gevrey-1 asymptotic expansion in $S_{+}\times\left(\ww C^{2},0\right)$
and $S_{-}\times\left(\ww C^{2},0\right)$ respectively. In particular,
we have:
\[
\left(\left(\Phi_{+}\right)^{-1}\circ\phi_{+}\right)_{\substack{\mid\left(S_{\lambda}\cup S_{-\lambda}\right)\times\left(\ww C^{2},0\right)}
}=\left(\left(\Phi_{-}\right)^{-1}\circ\phi_{-}\right)_{\substack{\mid\left(S_{\lambda}\cup S_{-\lambda}\right)\times\left(\ww C^{2},0\right)}
}\,\,.
\]
This proves that the function $\Phi$ defined by $\left(\Phi_{+}\right)^{-1}\circ\phi_{+}$
in $S_{+}\times\left(\ww C^{2},0\right)$ and by $\left(\Phi_{-}\right)^{-1}\circ\phi_{-}$
in $S_{-}\times\left(\ww C^{2},0\right)$ is well-defined and analytic
in $D\left(0,r\right)\backslash\acc 0\times\mathbf{D\left(0,r\right)}$.
Since it is also bounded, it can be extended to an analytic map $\Phi$
in $D\left(0,r\right)\times\mathbf{D\left(0,r\right)}$ by Riemann's
theorem. Hence: 
\[
\begin{cases}
\phi_{+}=\Phi_{+}\circ\Phi\\
\phi_{-}=\Phi_{-}\circ\Phi & \,\,.
\end{cases}
\]
In particular, by composition, $\Phi_{+}$ and $\Phi_{-}$ both admit
$\hat{\phi}\circ\Phi^{-1}$ as Gevrey-1 asymptotic expansion in $S_{\lambda}\times\left(\ww C^{2},0\right)$
and $S_{-\lambda}\times\left(\ww C^{2},0\right)$ respectively. Since
$\Phi_{+}$ and $\Phi_{-}$ conjugates $Y$ to $\ynorm$ and since
the notion of asymptotic expansion commutes with the partial derivative
operators, the formal diffeomorphism $\hat{\phi}\circ\Phi^{-1}$ formally
conjugates $Y$ to $\ynorm$. Finally, notice that $\hat{\phi}\circ\Phi^{-1}$
is necessarily tangent to the identity. Hence, by uniqueness of the
formal normalizing map given by Theorem \ref{thm: forme normalel formelle},
we deduce that $\hat{\phi}\circ\Phi^{-1}=\hat{\Phi}$, the unique
formal normalizing map tangent to the identity.
\end{proof}
We are now ready to prove Theorem \ref{Th: Th drsn}.
\begin{proof}[Proof of Theorem \ref{Th: Th drsn}.]

It is a straightforward consequence of Proposition \ref{prop: sommabilit=0000E9 normalisation formelle}
above.
\end{proof}

\subsubsection{Proof of Theorem \ref{thm: espace de module}}

~
\begin{proof}[Proof of Theorem \ref{thm: espace de module}.]
\emph{}
\end{proof}
Propositions \ref{prop: espace de module injection}, together with
Proposition \ref{prop: isotropies plates}, tell us that the considered
map is well-defined and one-to-one. It remains to prove that this
map is onto. Let 
\[
\begin{cases}
\Phi_{\lambda}\in\Lambda_{\lambda}\left(\ynorm\right)\\
\Phi_{-\lambda}\in\Lambda_{-\lambda}\left(\ynorm\right) & \,\,\,.
\end{cases}
\]
According to Theorem \ref{th: martinet ramis}, there exists 
\[
\left(\phi_{+},\phi_{-}\right)\in\diffsect[\arg\left(i\lambda\right)][\eta]\times\diffsect[\arg\left(-i\lambda\right)][\eta]
\]
with $\eta\in\left]\pi,2\pi\right[$, which extend analytically to
$S_{+}\times\left(\ww C^{2},0\right)$ and $S_{-}\times\left(\ww C^{2},0\right)$
respectively, such that:
\[
\phi_{\pm}\circ\left(\phi_{\mp}\right)_{\mid S_{\pm\lambda}\times\left(\ww C^{2},0\right)}^{-1}=\Phi_{\pm\lambda}
\]
and there also exists a formal diffeomorphism $\hat{\phi}$ which
is tangent to the identity, such that $\phi_{\pm}$ both admit $\hat{\phi}$
as asymptotic expansion in $S_{\pm}\times\left(\ww C^{2},0\right)$.
Let us consider the two germs of sectorial vector fields obtained
as
\[
Y_{\pm}:=\left(\phi_{\pm}^{-1}\right)_{*}\left(\ynorm\right)
\]
In particular, since $\hat{\phi}$ is the Gevrey-1 asymptotic expansion
of $\phi_{\pm}$, the vector fields $Y_{\pm}$ both admit $\left(\hat{\phi}\right)_{*}\left(\ynorm\right)$
as Gevrey-1 asymptotic expansion. The fact that $\phi_{+}\circ\left(\phi_{-}\right)^{-1}$
is an isotropy of $\ynorm$ implies immediately that $Y_{+}=Y_{-}$
on 
\begin{eqnarray*}
\left(S_{+}\cap S_{-}\right)\times\left(\ww C^{2},0\right) & = & \left(S_{\lambda}\cup S_{-\lambda}\right)\times\left(\ww C^{2},0\right)\,\,.
\end{eqnarray*}
Then, the vector field $Y$, which coincides with $Y_{\pm}$ in $S_{\pm}\times\left(\ww C^{2},0\right)$,
defines a germ of analytic vector field in $\left(\ww C^{3},0\right)$
by Riemann's theorem. By construction, ${\displaystyle Y\in\fdiff[\ww C^{3},0,\tx{Id}]_{*}\left(\ynorm\right)}$
and admits $\left(\Phi_{\lambda},\Phi_{-\lambda}\right)$ as Stokes
diffeomorphisms.

\subsubsection{Proof of Theorem \ref{th: espace de module symplectic}}

~

In a similar way, we prove now Theorem \ref{th: espace de module symplectic}.
\begin{proof}[Proof of Theorem \ref{th: espace de module symplectic}.]
\emph{}

Let $\ynorm\in\snodiag$ be a normal form which is also transversally
symplectic. We refer to subsection \ref{subsec: transversally symplectic}
for the notations. It is clear from Theorems \ref{thm: espace de module}
and \ref{thm: Th ham} that the mapping is well-defined and one-to-one.
It remains to prove that it is also onto. Let 
\[
\begin{cases}
\Phi_{\lambda}\in\Lambda_{\lambda}^{\omega}\left(\ynorm\right)\\
\Phi_{-\lambda}\in\Lambda_{-\lambda}^{\omega}\left(\ynorm\right) & \,\,\,.
\end{cases}
\]
Since $\Lambda_{\lambda}^{\omega}\left(\ynorm\right)\subset\Lambda_{\lambda}\left(\ynorm\right)$
and $\Lambda_{-\lambda}^{\omega}\left(\ynorm\right)\subset\Lambda_{-\lambda}\left(\ynorm\right)$,
according to Theorem \ref{th: martinet ramis} there exists 
\[
\left(\phi_{+},\phi_{-}\right)\in\diffsect[\arg\left(i\lambda\right)][\eta]\times\diffsect[\arg\left(-i\lambda\right)][\eta]
\]
with $\eta\in\left]\pi,2\pi\right[$, which extend analytically in
$S_{+}\times\left(\ww C^{2},0\right)$ and $S_{-}\times\left(\ww C^{2},0\right)$
respectively, such that:
\[
\phi_{\pm}\circ\left(\phi_{\mp}\right)_{\mid S_{\pm\lambda}\times\left(\ww C^{2},0\right)}^{-1}=\Phi_{\pm\lambda}
\]
and there also exists a formal diffeomorphism $\hat{\phi}$ which
is tangent to the identity, such that $\phi_{\pm}$ both admit $\hat{\phi}$
as Gevrey-1 asymptotic expansion in $S_{\pm}\times\left(\ww C^{2},0\right)$.
According to Corollary \ref{cor: MR symplectic}, there exists a germ
of an analytic fibered diffeomorphism $\psi\in\fdiff[\ww C^{3},0,\tx{Id}]$
(tangent to the identity), such that
\[
\sigma_{\pm}:=\phi_{\pm}\circ\psi
\]
both are transversally symplectic. Then, we have: 
\[
\sigma_{\pm}\circ\left(\Psi_{\mp}\right)_{\mid S_{\pm\lambda}\times\left(\ww C^{2},0\right)}^{-1}=\Phi_{\pm\lambda}~.
\]
 The end of the proof goes exactly as at the end of the proof of the
previous theorem. 
\end{proof}

\subsection{\label{subsec:Sectorial-isotropies-in}Sectorial isotropies in narrow
sectors and space of leaves: proof of Proposition \ref{prop: isotropies plates}.}

~

A normal form 
\[
\ynorm=x^{2}\pp x+\left(-\lambda+a_{1}x-c\left(y_{1}y_{2}\right)\right)y_{1}\pp{y_{1}}+\left(\lambda+a_{2}x+c\left(y_{1}y_{2}\right)\right)y_{2}\pp{y_{2}}
\]
is fixed for some $\lambda\in\ww C^{*},$ $\Re\left(a_{1}+a_{2}\right)>0$
and $c\in v\germ v$ (vanishing at the origin). The aim of this subsection
is to prove Proposition \ref{prop: isotropies plates} stated at the
beginning of this section.

Let us denote $a:=\res{\ynorm}=a_{1}+a_{2}$, $m:={\displaystyle \frac{1}{a}}$
and 
\[
c\left(v\right)=\sum_{k=1}^{+\infty}c_{k}v^{k}\,\,.
\]
If $m\notin\ww N_{>0}$, we set $c_{m}:=0$. We also define the following
power series 
\begin{eqnarray*}
\tilde{c}\left(v\right) & = & m\sum_{k\neq m}\frac{c_{k}}{k-m}v^{k}\,\,,
\end{eqnarray*}
and we notice that $\tilde{c}\left(v\right)\in v\germ v$.

\subsubsection{Sectorial first integrals and the space of leaves}

~

In a sectorial neighborhood of the origin of the form $S_{\lambda}\times\left(\ww C^{2},0\right)$
$\left(resp.\,S_{-\lambda}\times\left(\ww C^{2},0\right)\right)$,with
$S_{\pm\lambda}\in\germsect{\arg\left(\pm\lambda\right)}{\epsilon}$
and $\epsilon\in\left]0,\pi\right[$, we can give three first integrals
of $\ynorm$ which are analytic in the considered domain. Let us start
with the following proposition.
\begin{prop}
The following quantities are first integrals of $\ynorm$, analytic
in $S_{\pm\lambda}\times\left(\ww C^{2},0\right)$:

\begin{equation}
\begin{cases}
{\displaystyle w_{\pm\lambda}:=\frac{y_{1}y_{2}}{x^{a}}}\\
{\displaystyle h_{1,\pm\lambda}\left(x,\mathbf{y}\right):=y_{1}\exp\left(\frac{-\lambda}{x}+\frac{c_{m}\left(y_{1}y_{2}\right)^{m}\log\left(x\right)}{x}+\frac{\tilde{c}\left(y_{1}y_{2}\right)}{x}\right)x^{-a_{1}}}\\
{\displaystyle h_{2,\pm\lambda}\left(x,\mathbf{y}\right):=y_{2}\exp\left(\frac{\lambda}{x}-\frac{c_{m}\left(y_{1}y_{2}\right)^{m}\log\left(x\right)}{x}-\frac{\tilde{c}\left(y_{1}y_{2}\right)}{x}\right)x^{-a_{2}}}
\end{cases}\label{eq: integrales premieres}
\end{equation}
(we fix here a branch of the logarithm analytic in $S_{\pm\lambda}$,
and we write simply $h_{j}$ and $w$ instead of $h_{j,\pm\lambda}$
and $w_{\pm\lambda}$ respectively, if there is no ambiguity on the
sector $S_{\pm\lambda}$).

Moreover, we have the relation:
\[
h_{1}h_{2}=w\,\,.
\]
\end{prop}

\begin{proof}
It is an elementary computation.
\end{proof}
\begin{rem}
In other words, in a sectorial domain, we can parametrize a leaf (which
is not in $\acc{x=0}$) of the foliation associated to $\ynorm$ by:
\begin{eqnarray}
 & \begin{cases}
{\displaystyle y_{1}\left(x\right)=h_{1}\exp\left(\frac{\lambda}{x}-c_{m}\left(h_{1}h_{2}\right)^{m}\log\left(x\right)-\frac{\tilde{c}\left(h_{1}h_{2}x^{a}\right)}{x}\right)x^{a_{1}}}\\
{\displaystyle y_{2}\left(x\right)=h_{2}\exp\left(-\frac{\lambda}{x}+c_{m}\left(h_{1}h_{2}\right)^{m}\log\left(x\right)+\frac{\tilde{c}\left(h_{1}h_{2}x^{a}\right)}{x}\right)x^{a_{2}}}
\end{cases}\label{eq: solutions parametrees}\\
 & \left(h_{1},h_{2}\right)\in\ww C^{2}\,.\nonumber 
\end{eqnarray}
\end{rem}

\begin{cor}
\label{cor: coordonn=0000E9es espace des feuiles}The map 
\begin{eqnarray*}
{\cal H_{\pm\lambda}}:S_{\pm\lambda}\times\left(\ww C^{2},0\right) & \rightarrow & S_{\pm\lambda}\times\ww C^{2}\\
\left(x,\mathbf{y}\right) & \mapsto & \left(x,h_{1,\pm\lambda}\left(x,\mathbf{y}\right),h_{2,\pm\lambda}\left(x,\mathbf{y}\right)\right)\,\,,
\end{eqnarray*}
(where $h_{1,\pm\lambda},h_{2,\pm\lambda}$ are defined in (\ref{eq: integrales premieres}))
is a sectorial germ of a fibered analytic map in $S_{\pm\lambda}\times\left(\ww C^{2},0\right)$,
which is into. Moreover, there exists an open neighborhood of the
origin in $\ww C^{2}$, denoted by $\ls{\pm}\subset\ww C^{2}$, such
that 
\[
{\cal H_{\pm\lambda}}\left(S_{\pm\lambda}\times\left(\ww C^{2},0\right)\right)=S_{\pm\lambda}\times\ls{\pm}\,\,.
\]
In particular, ${\cal H_{\pm}}$ induces a fibered biholomorphism
\[
S_{\pm\lambda}\times\left(\ww C^{2},0\right)\overset{{\cal H_{\pm\lambda}}}{\longrightarrow}S_{\pm\lambda}\times\ls{\pm}
\]
which conjugates $\ynorm$ to $x^{2}\pp x$, \emph{i.e.}
\[
\left({\cal H_{\pm\lambda}}\right)_{*}\left(\ynorm\right)=x^{2}\pp x\,\,.
\]
\end{cor}

\begin{defn}
\label{def: space of leaves}We call $\ls{\pm}$ \textbf{the space
of leaves of} $\ynorm$ in $S_{\pm\lambda}\times\left(\ww C^{2},0\right)$.
\end{defn}

\begin{rem}
The set $\ls{\pm}$ depends on the choice of the neighborhood $\left(\ww C^{2},0\right)$,
but also on the choice of the sectorial neighborhood $S_{\pm\lambda}\in\germsect{\arg\left(\pm\lambda\right)}{\epsilon}$.
\end{rem}

\subsubsection{Sectorial isotropies in the space of leaves}

~

Now, we consider a germ of a sectorial isotropy $\psi_{\pm\lambda}\in\Lambda_{\pm\lambda}^{\left(\tx{weak}\right)}\left(\ynorm\right)$
and we denote by $\Gamma'_{\pm\lambda}$ the (germ of an) open subset
of $\ww C^{2}$ such that: 
\[
\cal H_{\pm\lambda}\circ\psi_{\pm}\left(S_{\pm\lambda}\times\left(\ww C^{2},0\right)\right)=S_{\pm\lambda}\times\Gamma'_{\pm\lambda}\,\,.
\]
\begin{prop}
\label{prop: isotropie espace des feuilles} With the notations and
assumptions above, the map 
\begin{eqnarray*}
\Psi_{\pm\lambda}:=\cal H_{\pm\lambda}\circ\psi_{\pm}\circ\cal H_{\pm\lambda}^{-1}:S_{\pm\lambda}\times\Gamma_{\pm\lambda} & \longrightarrow & S_{\pm\lambda}\times\Gamma'_{\pm\lambda}
\end{eqnarray*}
is a sectorial germ of a fibered biholomorphism from $S_{\pm\lambda}\times\Gamma_{\pm\lambda}$
to $S_{\pm\lambda}\times\Gamma'_{\pm\lambda}$, which is of the form:
\[
\Psi_{\pm\lambda}\left(x,h_{1},h_{2}\right)=\left(x,\Psi_{1,\pm\lambda}\left(h_{1},h_{2}\right),\Psi_{2,\pm\lambda}\left(h_{1},h_{2}\right)\right)\,\,.
\]
In particular, $\Psi_{1,\pm\lambda}$ and $\Psi_{2,\pm\lambda}$ are
analytic and depend only on $\left(h_{1},h_{2}\right)\in\Gamma_{\pm\lambda}$,
while $\Psi_{\pm\lambda}$ induces a biholomorphism (still written
$\Psi_{\pm\lambda}$): 
\begin{eqnarray*}
\Psi_{\pm\lambda}:\Gamma_{\pm\lambda} & \rightarrow & \Gamma'_{\pm\lambda}\\
\left(h_{1},h_{2}\right) & \mapsto & \left(\Psi_{1,\pm\lambda}\left(h_{1},h_{2}\right),\Psi_{2,\pm\lambda}\left(h_{1},h_{2}\right)\right)\,\,.
\end{eqnarray*}
\end{prop}

\begin{proof}
We only have to prove that $\Psi_{1,\pm\lambda}$ and $\Psi_{2,\pm\lambda}$
depend only on $\left(h_{1},h_{2}\right)\in\Gamma_{\pm\lambda}$.
By assumption, $\Psi_{\pm\lambda}$ is an isotropy of $x^{2}\pp x$:
\begin{eqnarray*}
\left(\Psi_{\pm\lambda}\right)_{*}\left(x^{2}\pp x\right) & = & x^{2}\pp x\,\,.
\end{eqnarray*}
We immediately obtain:
\[
\ppp{\Psi_{1,\pm\lambda}}x=\ppp{\Psi_{2,\pm\lambda}}x=0\,\,.
\]
 
\end{proof}
In the space of leaves $\Gamma_{\pm\lambda}$ equipped with coordinates
$\left(h_{1},h_{2}\right)$, we denote by $w$ the product of $h_{1}$
and $h_{2}$:
\[
w\left(h_{1},h_{2}\right):=h_{1}h_{2}\,\,.
\]
We define the two following quantities: 
\begin{equation}
\begin{cases}
f_{1}\left(x,w\right):=\exp\left(\frac{\lambda}{x}-c_{m}w^{m}\log\left(x\right)-\frac{\tilde{c}\left(wx^{a}\right)}{x}\right)x^{a_{1}}\\
f_{2}\left(x,w\right):=\exp\left(-\frac{\lambda}{x}+c_{m}w^{m}\log\left(x\right)+\frac{\tilde{c}\left(wx^{a}\right)}{x}\right)x^{a_{2}} & ,
\end{cases}\label{eq: f1 et f2}
\end{equation}
such that the leaves of the foliations are parametrized by: 
\[
\begin{cases}
y_{1}\left(x\right)=h_{1}f_{1}\left(x,h_{1}h_{2}\right)\\
y_{2}\left(x\right)=h_{2}f_{2}\left(x,h_{1}h_{2}\right)
\end{cases},\,\left(h_{1},h_{2}\right)\in\ww C^{2}\,.
\]
Notice that:
\begin{eqnarray*}
f_{1}\left(x,w\right)f_{2}\left(x,w\right) & = & x^{a}\,\,.
\end{eqnarray*}
Moreover, one checks immediately the following statement.
\begin{fact}
\label{fact: limites}For all $w\in\ww C$:
\[
\begin{cases}
\underset{\substack{x\rightarrow0\\
x\in S_{\lambda}
}
}{\lim}\abs{f_{1}\left(x,w\right)}=\underset{\substack{x\rightarrow0\\
x\in S_{-\lambda}
}
}{\lim}\abs{f_{2}\left(x,w\right)}=+\infty\\
\underset{\substack{x\rightarrow0\\
x\in S_{-\lambda}
}
}{\lim}\abs{f_{1}\left(x,w\right)}=\underset{\substack{x\rightarrow0\\
x\in S_{\lambda}
}
}{\lim}\abs{f_{2}\left(x,w\right)}=0 & \,\,.
\end{cases}
\]
\end{fact}

Using notations of Proposition \ref{prop: isotropie espace des feuilles},
we also assume from now on that $\left(\ww C^{2},0\right)=\mathbf{D\left(0,r\right)}$,
with $\mathbf{r}=\left(r_{1},r_{2}\right)\in\left(\ww R_{>0}\right)^{2}$
and $r_{1},r_{2}>0$ small enough so that 
\[
\psi_{\pm\lambda}\left(S_{\pm\lambda}\times\mathbf{D\left(0,r\right)}\right)\subset S_{\pm\lambda}\times\mathbf{D\left(0,r'\right)}
\]
for some $\mathbf{r'}=\left(r'_{1},r'_{2}\right)\in\left(\ww R_{>0}\right)^{2}$.
Let us now define in a general way the following set associated to
the sector $S_{\pm\lambda}$ and to a polydisc $\mathbf{D}\left(\mathbf{0},\tilde{\mathbf{r}}\right)$,
with $\tilde{\mathbf{r}}:=\left(\tilde{r}_{1},\tilde{r}_{2}\right)$.
\begin{defn}
\label{def: espace des feuilles =0000E0 rayon fixe}For all $x\in S_{\pm\lambda}$
et $\tilde{\mathbf{r}}:=\left(\tilde{r}_{1},\tilde{r}_{2}\right)\in\left(\ww R_{>0}\right)^{2}$,
we define
\[
\lsr{\pm}{x,\tilde{\mathbf{r}}}:=\acc{\left(h_{1},h_{2}\right)\in\ww C^{2}\mid\abs{h_{j}}\leq\frac{\tilde{r}_{j}}{\abs{f_{j}\left(x,h_{1}h_{2}\right)}}~,\,\mbox{for }j\in\left\{ 1,2\right\} }\,\,.
\]
We also consider the: 
\begin{eqnarray*}
\lsr{\pm}{\tilde{\mathbf{r}}} & := & \bigcup_{\substack{x\in S_{\pm\lambda}}
}\lsr{\pm}{x,\tilde{\mathbf{r}}}\\
 & = & \acc{\left(h_{1},h_{2}\right)\in\ww C^{2}\mid\exists x\in S_{\pm\lambda}\mbox{ s.t. }\abs{h_{j}}\leq\frac{\tilde{r}_{j}}{\abs{f_{j}\left(x,h_{1}h_{2}\right)}}~,~\mbox{for }j\in\left\{ 1,2\right\} }
\end{eqnarray*}
(\emph{cf. }figure \ref{fig:Illustration-de-l'espace}).
\end{defn}

Since we assume now that $\left(\ww C^{2},0\right)=\mathbf{D\left(0,r\right)}$,
then we have:
\[
\ls{\pm}=\lsr{\pm}{\mathbf{r}}\,\,,
\]
and 
\[
\lsp{\pm}\subset\lsr{\pm}{\mathbf{r}'}\,\,.
\]

\begin{figure}
\includegraphics[scale=0.22]{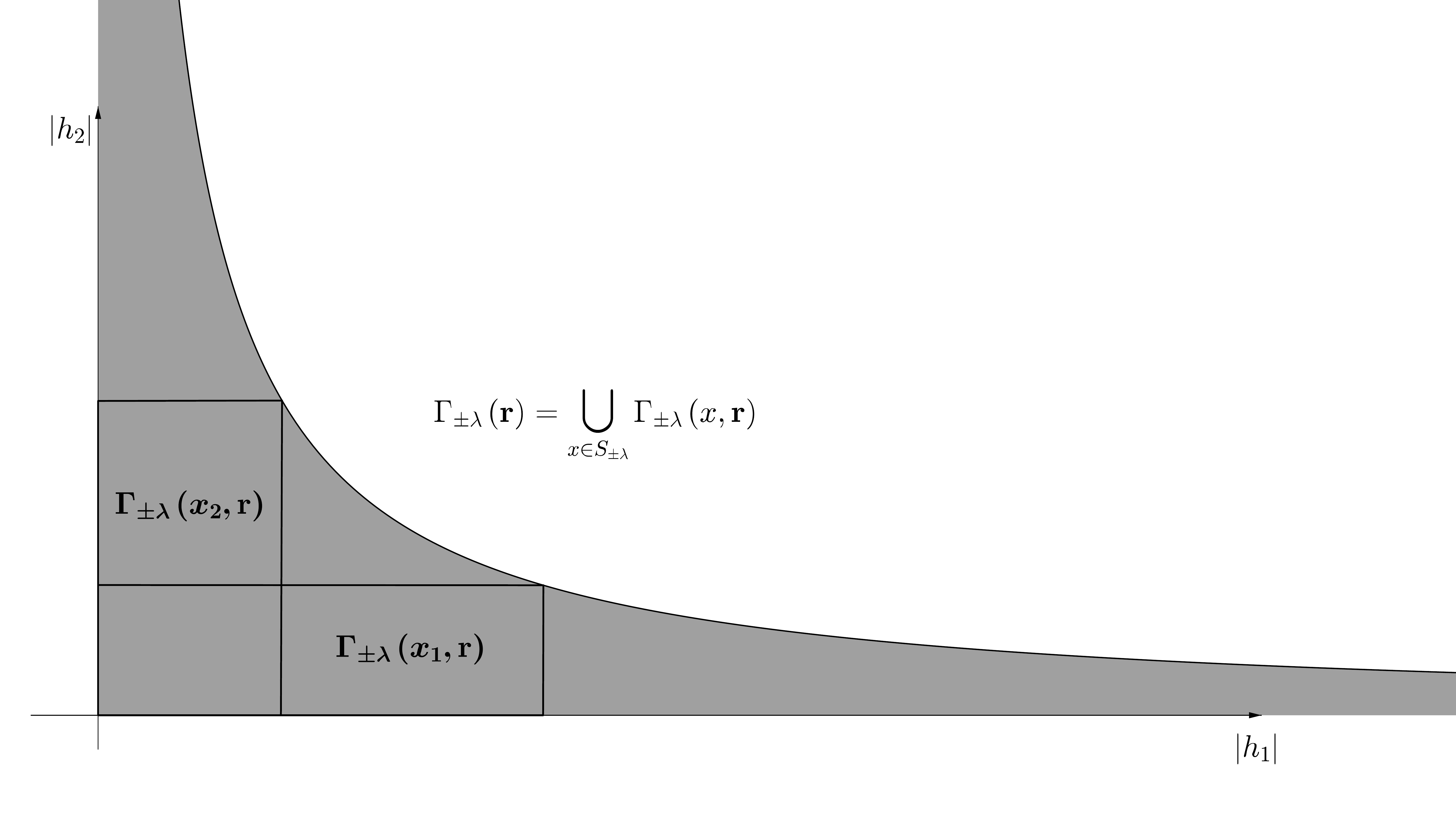}\caption{\foreignlanguage{french}{\label{fig:Illustration-de-l'espace}\foreignlanguage{english}{Representation
of the space of leaves in terms of $\protect\abs{h_{1}}$ and $\protect\abs{h_{2}}$
when $c=0$: in this case, it is a\emph{ Reinhardt domain }(\emph{cf.
}\cite{hormander1973introduction})}.}}
\end{figure}
\begin{rem}
~

\begin{enumerate}
\item It is important to notice that the particular form of $\Psi_{\pm\lambda}$
implies that the image of any fiber 
\[
\acc{x=x_{0}}\times\lsr{\pm}{x_{0},\mathbf{r}}
\]
by $\Psi_{\pm\lambda}$ is included in a fiber of the form 
\[
\acc{x=x_{0}}\times\lsr{\pm}{x_{0},\mathbf{r'}}\,\,.
\]
\item If $\left(h_{1},h_{2}\right)\in\lsr{\pm}{x,\mathbf{r}}$, then 
\begin{eqnarray*}
\abs{h_{1}h_{2}} & < & \frac{r_{1}r_{2}}{\abs{x^{a}}}\,\,.
\end{eqnarray*}
\item As $\left(h_{1},h_{2}\right)\in\Gamma_{\pm\lambda}$ varies the values
of $w=h_{1}h_{2}$ cover the whole $\ww C$.
\end{enumerate}
\end{rem}

\subsubsection{Action on the resonant monomial in the space of leaves}

~

Let us study the the action of $\Psi_{\pm\lambda}$ on the resonant
monomial $w=h_{1}h_{2}$ in the space of leaves.
\begin{lem}
\label{lem: isotropie monomone res esp des feuilles}We consider a
biholomorphism 
\begin{eqnarray*}
\Psi_{\pm\lambda}:\ls{\pm} & \tilde{\rightarrow} & \lsp{\pm}\\
\left(h_{1},h_{2}\right) & \mapsto & \left(\Psi_{1,\pm}\left(h_{1},h_{2}\right),\Psi_{2,\pm}\left(h_{1},h_{2}\right)\right)\,\,,
\end{eqnarray*}
such that for all $x\in S_{\pm\lambda}$, we have 
\[
\Psi_{\pm\lambda}\left(\lsr{\pm}{x_{0},\mathbf{r}}\right)\subset\lsr{\pm}{x_{0},\mathbf{r'}}\,\,.
\]
We also define $\Psi_{w,\pm\lambda}:=\Psi_{1,\pm\lambda}\Psi_{2,\pm\lambda}$.
Then, for all $n\in\ww N$, there exists entire (\emph{i.e. }analytic
over $\ww C$) functions $\Psi_{w,\lambda,n}$ and $\Psi_{w,-\lambda,n}$
such that 
\[
\begin{cases}
{\displaystyle \Psi_{w,\lambda}\left(h_{1},h_{2}\right)=\sum_{n\geq0}\Psi_{w,\lambda,n}\left(h_{1}h_{2}\right)h_{1}^{n}}\\
{\displaystyle \Psi_{w,-\lambda}\left(h_{1},h_{2}\right)=\sum_{n\geq0}\Psi_{w,-\lambda,n}\left(h_{1}h_{2}\right)h_{2}^{n}} & \,.
\end{cases}
\]
Moreover, the series above uniformly converge (for the $\sup$-norm)
in every subset of $\ls{\pm}$ of the form $\lsr{\pm}{\tilde{\mathbf{r}}}$,
with $\tilde{\mathbf{r}}:=\left(\tilde{r}_{1}\tilde{r}_{2}\right)$
and 
\[
0<\tilde{r}_{j}<r_{j}~~~,~j\in\left\{ 1,2\right\} 
\]
(\emph{cf.} Definition\emph{ }\ref{def: espace des feuilles =0000E0 rayon fixe}).
More precisely, for all $\tilde{r}_{1},\tilde{r}_{2},\delta>0$ such
that 
\[
0<\tilde{r}_{j}+\delta<r_{j}~~~,~j\in\left\{ 1,2\right\} 
\]
for all $x\in S_{\pm\lambda}$ and $w\in\ww C$ we have 
\begin{eqnarray*}
\abs{wx^{a}}\leq\tilde{r}_{1}\tilde{r}_{2} & \Longrightarrow & \begin{cases}
\abs{\Psi_{w,\lambda,n}\left(w\right)}\leq\frac{r'_{1}r'_{2}}{\abs{x^{a}}}\abs{\frac{f_{1}\left(x,w\right)}{\tilde{r}_{1}+\delta}}^{n}\\
\abs{\Psi_{w,-\lambda,n}\left(w\right)}\leq\frac{r'_{1}r'_{2}}{\abs{x^{a}}}\abs{\frac{f_{2}\left(x,w\right)}{\tilde{r}_{2}+\delta}}^{n}
\end{cases},\,\forall n\geq0\,.
\end{eqnarray*}
\end{lem}

\begin{proof}
Let us give the proof for $\Psi_{w,\lambda},\Psi_{1,\lambda}$ and
$\Psi_{2,\lambda}$ in $\Gamma_{\lambda}$ (the same proof applies
also for $\Psi_{w,-\lambda}$ in $\Gamma_{-\lambda}$ by exchanging
the role played by $h_{1}$ and $h_{2}$). We fix some $0<\tilde{r}_{j}<r_{j}$,
$j\in\left\{ 1,2\right\} $, and $\delta>0$ such that 
\[
0<\tilde{r}_{j}+\delta<r_{j}~~~,~j\in\left\{ 1,2\right\} .
\]
For a fixed value $w\in\ww C$, we consider the restriction of $\Psi_{w,\lambda}$
to the hypersurface $M_{w}:=\acc{h_{1}h_{2}=w}\cap\Gamma_{\lambda}$:
this restriction is analytic in $M_{w}$. The map
\[
\varphi_{w}:h_{1}\mapsto\Psi_{w,\lambda}\left(h_{1},\frac{w}{h_{1}}\right)
\]
 is analytic in 
\[
M_{w,1}:=\bigcup_{\substack{x\in S_{\lambda}\\
\abs{wx^{a}}<r_{1}r_{2}
}
}\Omega_{x,w}\,\,,
\]
where for all $x\in S_{\lambda}$ with $\abs{wx^{a}}<r_{1}r_{2}$,
the set $\Omega_{x,w}$ is the following annulus: 
\[
\Omega_{x,w}:=\acc{h_{1}\in\ww C\mid\abs{\frac{wf_{2}\left(x,w\right)}{r_{2}}}<\abs{h_{1}}<\abs{\frac{r_{1}}{f_{1}\left(x,w\right)}}}\,\,.
\]
In particular, $\varphi_{w}$ admits a Laurent expansion 
\[
\varphi_{w}\left(h_{1}\right)=\Psi_{w,+}\left(h_{1},\frac{w}{h_{1}}\right)=\sum_{n\geq-L}\Psi_{w,+,n}\left(w\right)h_{1}^{n}
\]
in every annulus $\Omega_{x,w}$, with $x\in S_{\lambda}$ such that
$\abs{wx^{a}}<r_{1}r_{2}$. Moreover for all $x\in S_{\lambda}$ such
that $\abs{wx^{a}}<r_{1}r_{2}$, Cauchy's formula gives 
\begin{eqnarray*}
\Psi_{w,\lambda,n}\left(w\right) & = & \frac{1}{2i\pi}\oint_{\gamma\left(x,w\right)}\frac{\Psi_{w,\lambda}\left(h_{1},\frac{w}{h_{1}}\right)}{h_{1}^{n+1}}\mbox{d}h_{1}\,\,,\mbox{ for all }n\in\wn,
\end{eqnarray*}
where $\gamma\left(x,w\right)$ is any circle (oriented positively)
centered at the origin with a radius $\rho\left(x,w\right)$ satisfying
\[
\abs{\frac{wf_{2}\left(x,w\right)}{r_{2}}}<\rho\left(x,w\right)<\abs{\frac{r_{1}}{f_{1}\left(x,w\right)}}\,\,.
\]
If $\abs{wx^{a}}<\left(\tilde{r}_{1}+\delta\right)\left(\tilde{r}_{2}+\delta\right)$,
we can take for instance 
\[
\rho\left(x,w\right)=\abs{\frac{\tilde{r}_{1}+\delta}{f_{1}\left(x,w\right)}}\,\,.
\]
Therefore, for all $x\in S_{\lambda}$ and all $w\in\ww C$ such that
$\abs{wx^{a}}\leq\tilde{r}_{1}\tilde{r}_{2}$, for all $\xi\in\ww C$
with $\abs{\xi}<\delta$, we also have: 
\begin{eqnarray*}
\Psi_{w,\lambda,n}\left(w+\xi\right) & = & \frac{1}{2i\pi}\oint_{\gamma\left(x,w\right)}\frac{\Psi_{w,\lambda}\left(h_{1},\frac{w+\xi}{h_{1}}\right)}{h_{1}^{n+1}}\mbox{d}h_{1}\,\,,\mbox{ for all }n\in\ww Z,
\end{eqnarray*}
where $\gamma\left(x,w\right)$ is the same circle $\Big($of radius
${\displaystyle \rho\left(x,w\right)=\abs{\frac{\tilde{r}_{1}+\delta}{f_{1}\left(x,w\right)}}}$$\Big)$
for all $\abs{\xi}<\delta$. Moreover, since for all $x\in S_{\lambda}$,
we have 
\[
\Psi_{\lambda}\left(\lsr{}{x,\mathbf{r}}\right)\subset\lsr{}{x,\mathbf{r'}}\,\,,
\]
and since for all $\left(h'_{1},h'_{2}\right)\in\lsr{}{x,\mathbf{r'}}$
we have 
\[
\abs{h'_{1}h'_{2}}\leq\frac{r'_{1}r'_{2}}{\abs{x^{a}}}\,\,,
\]
then for all $x\in S_{\lambda}$ and $w\in\ww C$ such that $\abs{wx^{a}}\leq\tilde{r}_{1}\tilde{r}_{2}$,
the following inequality holds for all $h_{1}$ with $\abs{h_{1}}<\frac{r_{1}}{f_{1}\left(x,w\right)}$:
\begin{eqnarray*}
\abs{\Psi_{w,\lambda}\left(h_{1},\frac{w}{h_{1}}\right)} & < & \frac{r'_{1}r'_{2}}{\abs{x^{a}}}\,.
\end{eqnarray*}
The well-known theorem regarding integrals depending analytically
on a parameter asserts that for all $n\in\ww Z$ the mapping $\Psi_{w,\lambda,n}$
is analytic near any point $w\in\ww C$. Hence, it is an entire function
(\emph{i.e.} analytic over $\wc$). Moreover, the inequality above
and the Cauchy's formula together imply that for all $n\in\ww Z$
and for all $\left(x,w\right)\in S_{\lambda}\times\ww C$ such that
$\abs{wx^{a}}\leq\tilde{r}_{1}\tilde{r}_{2}$, we have: 
\begin{eqnarray*}
\abs{\Psi_{w,\lambda,n}\left(w\right)} & < & \frac{r'_{1}r'_{2}}{\abs{x^{a}}\rho\left(x,w\right)^{n}}=\frac{r'_{1}r'_{2}}{\abs{x^{a}}}\abs{\frac{f_{1}\left(x,w\right)}{\tilde{r}_{1}+\delta}}^{n}\,.
\end{eqnarray*}
According to Fact \ref{fact: limites}, for a fixed value $w\in\ww C$,
if $n<0$, the right hand-side tends to $0$ as $x$ tends to $0$
in $S_{\lambda}$. This implies in particular that $\Psi_{w,\lambda,n}=0$
for all $n<0$. Consequently: 
\begin{eqnarray*}
\Psi_{w,\lambda}\left(h_{1},\frac{w}{h_{1}}\right) & = & \sum_{n\geq0}\Psi_{w,\lambda,n}\left(w\right)h_{1}^{n}\,\,.
\end{eqnarray*}
Moreover, for all $w\in\ww C$ the series converges normally in every
domain of the form 
\[
\Omega_{x,w}:=\acc{h_{1}\in\ww C\mid\abs{h_{1}}\leq\abs{\frac{\tilde{r}_{1}}{f_{1}\left(x,w\right)}}}\,\,,\mbox{ for all }x\in S_{\lambda}\,\,,\,0<\tilde{r}_{1}<r_{1},
\]
since the Laurent expansion's range is $n\geq0$. This actually means
that the series converges normally in an entire neighborhood of the
origin in $\ww C$. In particular, for all fixed $w\in\ww C$, the
map 
\[
h_{1}\mapsto\Psi_{w,\lambda}\left(h_{1},\frac{w}{h_{1}}\right)=\sum_{n\geq0}\Psi_{w,\lambda,n}\left(w\right)h_{1}^{n}
\]
 is analytic in a neighborhood of the origin. Finally, the series
\begin{eqnarray*}
\Psi_{w,\lambda}\left(h_{1},h_{2}\right) & = & \sum_{n\geq0}\Psi_{w,\lambda,n}\left(h_{1}h_{2}\right)h_{1}^{n}
\end{eqnarray*}
converges normally, and hence its sum is analytic in every domain
of the form $\lsr{}{\tilde{\mathbf{r}}}$, with $0<\tilde{r}_{1}<r_{1}$
and $0<\tilde{r}_{2}<r_{2}$.
\end{proof}

\subsubsection{Action on the resonant monomial}

~

Since $\psi_{\pm\lambda}\in\Lambda_{\pm\lambda}^{\left(\tx{weak}\right)}\left(\ynorm\right)$,
the mapping $\psi_{\pm\lambda}$ is of the form 
\begin{eqnarray*}
\psi_{\pm\lambda}\left(x,\mathbf{y}\right) & = & \left(x,\psi_{1,\pm\lambda}\left(x,\mathbf{y}\right),\psi_{2,\pm\lambda}\left(x,\mathbf{y}\right)\right)\,\,,
\end{eqnarray*}
with $\psi_{1,\pm\lambda},\psi_{2,\pm\lambda}$ analytic and bounded
in $S_{\pm\lambda}\times\mathbf{D\left(0,r\right)}$. Moreover, by
assumption $\psi_{\pm\lambda}$ admits the identity as weak Gevrey-1
asymptotic expansion, \emph{i.e. }we have a normally convergent expansion:
\begin{eqnarray*}
\psi_{i,\pm\lambda}\left(x,\mathbf{y}\right) & = & y_{i}+\sum_{\mathbf{k}\in\ww N^{2}}\psi_{i,\mathbf{\pm\lambda,k}}\left(x\right)\mathbf{y^{k}}\,\,,
\end{eqnarray*}
where $\psi_{i,\pm\lambda,\mathbf{k}}$ is holomorphic in $S_{\pm\lambda}$
and admits $0$ as Gevrey-1 asymptotic expansion, for $i=1,2$ and
all $\mathbf{k}=\left(k_{1},k_{2}\right)\in\ww N^{2}$.
\begin{lem}
\label{lem: istropie monomiale}With the notations and assumptions
above, let us define $\psi_{v,\pm\lambda}:=\psi_{1,\pm\lambda}\psi_{2,\pm\lambda}$.
Then $\psi_{v,\lambda}$ and $\psi_{v,-\lambda}$ can be expanded
as the series 
\[
\begin{cases}
{\displaystyle \psi_{v,\lambda}\left(x,\mathbf{y}\right)=y_{1}y_{2}+x^{a}\sum_{n\ge1}\Psi_{w,\lambda,n}\left(\frac{y_{1}y_{2}}{x^{a}}\right)\left(\frac{y_{1}}{f_{1}\left(x,\frac{y_{1}y_{2}}{x^{a}}\right)}\right)^{n}}\\
{\displaystyle \psi_{v,-\lambda}\left(x,\mathbf{y}\right)=y_{1}y_{2}+x^{a}\sum_{n\ge1}\Psi_{w,-\lambda,n}\left(\frac{y_{1}y_{2}}{x^{a}}\right)\left(\frac{y_{2}}{f_{2}\left(x,\frac{y_{1}y_{2}}{x^{a}}\right)}\right)^{n}}
\end{cases}
\]
which are normally convergent in every subset of $S_{\pm\lambda}\times\mathbf{D\left(0,r\right)}$
of the form $S_{\pm\lambda}\times\mathbf{\overline{D}\left(0,\tilde{r}\right)},$
where $\mathbf{\overline{D}\left(0,\tilde{r}\right)}$ is a closed
poly-disc with $\mathbf{\tilde{r}}=\left(\tilde{r}_{1},\tilde{r}_{2}\right)$
such that 
\[
0<\tilde{r}_{j}<r_{j}~~~,~j\in\left\{ 1,2\right\} .
\]
Here $\Psi_{w,\lambda,n}$ and $\Psi_{w,-\lambda,n}$ , for $n\in\ww N$,
are the ones appearing in Lemma \ref{lem: isotropie monomone res esp des feuilles}.
Moreover, for all closed sub-sector $S'\subset S_{\pm\lambda}$ and
for all closed poly-disc $\mathbf{\overline{D}}\subset\mathbf{D\left(0,r\right)}$,
there exists $A,B>0$ such that:
\begin{eqnarray*}
\abs{\psi_{v,\pm\lambda}\left(x,y_{1},y_{2}\right)-y_{1}y_{2}} & \leq & A\exp\left(-\frac{B}{\abs x}\right)\,\,,\,\,\forall\left(x,\mathbf{y}\right)\in S'\times\mathbf{\overline{D}}\,\,.
\end{eqnarray*}
In particular, $\psi_{v,\pm\lambda}$ admits $y_{1}y_{2}$ as Gevrey-1
asymptotic expansion in $S_{\pm\lambda}\times\mathbf{D\left(0,r\right)}$.
\end{lem}

\begin{proof}
By definition, we have 
\[
\Psi_{\pm\lambda}\circ\cal H_{\pm\lambda}=\cal H_{\pm\lambda}\circ\psi_{\pm\lambda}\,.
\]
In particular, for all $\left(x,\mathbf{y}\right)\in S_{\pm\lambda}\times\mathbf{D\left(0,r\right)}$:
\begin{eqnarray*}
\Psi_{w,\pm}\left(x,\frac{y_{1}}{f_{1}\left(x,\frac{y_{1}y_{2}}{x^{a}}\right)},\frac{y_{2}}{f_{2}\left(x,\frac{y_{1}y_{2}}{x^{a}}\right)}\right) & = & \frac{\psi_{v,\pm}\left(x,y_{1},y_{2}\right)}{x^{a}}\,\,.
\end{eqnarray*}
Thus, according to Lemma \ref{lem: isotropie monomone res esp des feuilles}
we have: 
\begin{equation}
\begin{cases}
{\displaystyle \psi_{v,\lambda}\left(x,\mathbf{y}\right)=x^{a}\sum_{n\ge0}\Psi_{w,\lambda,n}\left(\frac{y_{1}y_{2}}{x^{a}}\right)\left(\frac{y_{1}}{f_{1}\left(x,\frac{y_{1}y_{2}}{x^{a}}\right)}\right)^{n}}\\
{\displaystyle \psi_{v,-\lambda}\left(x,\mathbf{y}\right)=x^{a}\sum_{n\ge0}\Psi_{w,-\lambda,n}\left(\frac{y_{1}y_{2}}{x^{a}}\right)\left(\frac{y_{2}}{f_{2}\left(x,\frac{y_{1}y_{2}}{x^{a}}\right)}\right)^{n}} & \,.
\end{cases}\label{eq: psi_v et Psi_v}
\end{equation}
Besides we know that $\psi_{v,\pm\lambda}$ admits $y_{1}y_{2}$ as
weak Gevrey-1 asymptotic expansion in $S_{\pm\lambda}\times\mathbf{D\left(0,r\right)}$:
\begin{eqnarray}
\psi_{v,\pm\lambda}\left(x,y_{1},y_{2}\right) & = & y_{1}y_{2}+\sum_{\mathbf{k}\in\ww N^{2}}\psi_{v,\pm\lambda,\mathbf{k}}\left(x\right)\mathbf{y^{k}}\,\,,\label{eq: psi_v developpement}
\end{eqnarray}
where for all $\mathbf{k}=\left(k_{1},k_{2}\right)\in\ww N^{2}$ the
mapping $\psi_{v,\pm\lambda,\mathbf{k}}$ is holomorphic in $S_{\pm\lambda}$
and admits $0$ as Gevrey-1 asymptotic expansion. Let us compare both
expressions of $\psi_{v,\pm\lambda}$ above. Looking at monomials
$\mathbf{y^{k}}$ with $k_{1}=k_{2}$ in (\ref{eq: psi_v developpement}),
and at terms corresponding to $n=0$ on the right-hand side of (\ref{eq: psi_v et Psi_v}),
we must have for all $x\in S_{\pm\lambda}$ and $v\in\ww C$ with
$\abs v<r_{1}r_{2}$:
\begin{eqnarray*}
v+\sum_{k\ge0}\psi_{v,\lambda,\mathbf{k}\left(k,k\right)}\left(x\right)v^{k} & = & x^{a}\Psi_{w,\lambda,0}\left(\frac{v}{x^{a}}\right)\,\,.
\end{eqnarray*}
Since $\Psi_{w,\pm\lambda,0}$ is analytic in $\ww C$, there exists
$\left(\alpha_{\text{\ensuremath{\pm\lambda},}k}\right)_{k\in\ww N}\subset\ww C$
such that
\begin{eqnarray*}
\Psi_{w,\pm\lambda,0}\left(\frac{v}{x^{a}}\right) & = & \sum_{k\ge0}\alpha_{\pm\lambda,k}\left(\frac{v}{x^{a}}\right)^{k}\,\,.
\end{eqnarray*}
This can only happen if $\alpha_{\pm\lambda,k}=0$ whenever $k\neq1$,
for $\psi_{v,\pm\lambda,\mathbf{k}}$ is holomorphic in $S_{\pm\lambda}$
and admits $0$ as Gevrey-1 asymptotic expansion. A further immediate
identification yields 
\begin{align*}
\Psi_{v,\pm\lambda,0}\left(w\right) & =w~.
\end{align*}
Thus
\[
\begin{cases}
{\displaystyle \psi_{v,\lambda}\left(x,\mathbf{y}\right)=y_{1}y_{2}+x^{a}\sum_{n\ge1}\Psi_{w,\lambda,n}\left(\frac{y_{1}y_{2}}{x^{a}}\right)\left(\frac{y_{1}}{f_{1}\left(x,\frac{y_{1}y_{2}}{x^{a}}\right)}\right)^{n}}\\
{\displaystyle \psi_{v,-\lambda}\left(x,\mathbf{y}\right)=y_{1}y_{2}+x^{a}\sum_{n\ge1}\Psi_{w,-\lambda,n}\left(\frac{y_{1}y_{2}}{x^{a}}\right)\left(\frac{y_{2}}{f_{2}\left(x,\frac{y_{1}y_{2}}{x^{a}}\right)}\right)^{n}} & \,\,.
\end{cases}
\]

Let us prove that $\psi_{v,\pm\lambda}$ admits $y_{1}y_{2}$ as Gevrey-1
asymptotic expansion in $S_{\pm\lambda}\times\left(\ww C^{2},0\right)$.
We have to show that ${\displaystyle \abs{\psi_{v,\pm\lambda}\left(x,y_{1},y_{2}\right)-y_{1}y_{2}}}$
is exponentially small with respect to $x\in S_{\pm\lambda}$, uniformly
in $\mathbf{y}\in\mathbf{D\left(0,r\right)}$. As for the previous
lemma, we perform the proof for $\psi_{v,\lambda}$ only (the same
proof applies for $\psi_{v,-\lambda}$ by exchanging $y_{1}$ and
$y_{2}$). 

From the computations above we derive 
\begin{eqnarray*}
\abs{\psi_{v,\lambda}\left(x,y_{1},y_{2}\right)-y_{1}y_{2}} & \leq & \sum_{n\geq1}\abs{x^{a}\Psi_{w,\lambda,n}\left(\frac{y_{1}y_{2}}{x^{a}}\right)\left(\frac{y_{1}}{f_{1}\left(x,\frac{y_{1}y_{2}}{x^{a}}\right)}\right)^{n}}\,\,.
\end{eqnarray*}
Let us fix $\tilde{r}_{1},\tilde{r}_{2},\delta>0$ in such a way that
\[
0<\tilde{r}_{j}+\delta<r_{j}~,~j\in\left\{ 1,2\right\} ~.
\]
Let us take $\abs x$, $\abs{y_{1}}$ and $\abs{y_{2}}$ small enough
so that 
\[
{\displaystyle 2x\in S_{\lambda}}
\]
 and 
\[
{\displaystyle \abs{y_{1}y_{2}}<\frac{\tilde{r}_{1}\tilde{r}_{2}}{\abs{2^{a}}}<r_{1}r_{2}\,\,.}
\]
According to Lemma \ref{lem: isotropie monomone res esp des feuilles},
for all $\tilde{x}\in S_{\lambda}$ and all $w\in\ww C$: 
\begin{eqnarray*}
\abs{w\tilde{x}^{a}}\leq\tilde{r}_{1}\tilde{r}_{2} & \Longrightarrow & \abs{\Psi_{w,\lambda,n}\left(w\right)}\leq\frac{r'_{1}r'_{2}}{\abs{\tilde{x}^{a}}}\abs{\frac{f_{1}\left(\tilde{x},w\right)}{\tilde{r}_{1}+\delta}}^{n}\,\,.
\end{eqnarray*}
In particular for $\tilde{x}=2x$ and ${\displaystyle w}=\frac{y_{1}y_{2}}{x^{a}}$
we derive $\abs{w\tilde{x}^{a}}<\tilde{r}_{1}\tilde{r}_{2}$, from
which we conclude 
\begin{eqnarray*}
\abs{\Psi_{w,\lambda,n}\left(\frac{y_{1}y_{2}}{x^{a}}\right)} & \leq & \frac{r'_{1}r'_{2}}{\abs{2^{a}x^{a}}}\abs{\frac{f_{1}\left(2x,\frac{y_{1}y_{2}}{x^{a}}\right)}{\tilde{r}_{1}+\delta}}^{n}\,\,.
\end{eqnarray*}
Consequently, for all $\left(x,y_{1},y_{2}\right)\in S_{\lambda}\times\mathbf{D\left(0,\tilde{r}\right)}$
with
\[
\begin{cases}
{\displaystyle 2x\in S_{\lambda}}\\
\abs{y_{1}y_{2}}<\frac{\tilde{r}_{1}\tilde{r}_{2}}{\abs{2^{a}}}<r_{1}r_{2} & ,
\end{cases}
\]
we have 
\begin{eqnarray*}
\abs{\psi_{v,\lambda}\left(x,y_{1},y_{2}\right)-y_{1}y_{2}} & \leq & \sum_{n\geq1}\abs{x^{a}\frac{r'_{1}r'_{2}}{2^{a}x^{a}}\left(\frac{f_{1}\left(2x,\frac{y_{1}y_{2}}{x^{a}}\right)}{\tilde{r}_{1}+\delta}\right)^{n}\left(\frac{y_{1}}{f_{1}\left(x,\frac{y_{1}y_{2}}{x^{a}}\right)}\right)^{n}}\\
 & \leq & \frac{r'_{1}r'_{2}}{\abs{2^{a}}}\sum_{n\geq1}\abs{\left(\frac{y_{1}}{\tilde{r}_{1}+\delta}\right)^{n}\left(\frac{f_{1}\left(2x,\frac{y_{1}y_{2}}{x^{a}}\right)}{f_{1}\left(x,\frac{y_{1}y_{2}}{x^{a}}\right)}\right)^{n}}\,\,.
\end{eqnarray*}
Since $\tilde{c}\left(v\right)$ is the germ of an analytic function
near the origin which is null at the origin, we can take $r_{1}$,$r_{2}>0$
small enough in order that for all closed sub-sector $S'\subset S_{\lambda}$
, for all $\tilde{r}_{1}\in\left]0,r_{1}\right[$ and $\tilde{r}_{2}\in\left]0,r_{2}\right[$,
there exist $A,B>0$ satisfying:
\begin{eqnarray*}
\left(x,y_{1},y_{2}\right)\in S'\times\mathbf{D\left(0,\tilde{r}\right)} & \Longrightarrow & \abs{\psi_{v,\lambda}\left(x,y_{1},y_{2}\right)-y_{1}y_{2}}A\exp\left(-\frac{B}{\abs x}\right)\,\,.
\end{eqnarray*}
Let us prove this. We need here to estimate the quantity:
\begin{eqnarray*}
\abs{\frac{f_{1}\left(2x,\frac{y_{1}y_{2}}{x^{a}}\right)}{f_{1}\left(x,\frac{y_{1}y_{2}}{x^{a}}\right)}} & = & \abs{2^{a_{1}}\exp\left(-\frac{\lambda}{2x}-c_{m}\frac{\left(y_{1}y_{2}\right)^{m}}{x}\log\left(2\right)-\frac{\tilde{c}\left(y_{1}y_{2}2^{a}\right)}{2x}+\frac{\tilde{c}\left(y_{1}y_{2}\right)}{x}\right)}\,\,.
\end{eqnarray*}
On only have tot deal with the case where $x\in S'$ is such that
$2x\in S'$ (otherwise, $x$ is ``far from the origin'', and we
conclude without difficulty). We have:
\begin{eqnarray*}
\left(x,y_{1},y_{2}\right)\in S'\times\mathbf{D\left(0,\tilde{r}\right)}\tx{\,et\,}2x\in S & \Longrightarrow & \abs{\frac{f_{1}\left(2x,\frac{y_{1}y_{2}}{x^{a}}\right)}{f_{1}\left(x,\frac{y_{1}y_{2}}{x^{a}}\right)}}\leq\abs{2^{a_{1}}}\exp\left(-\frac{B}{\abs x}\right)<1\,\,.
\end{eqnarray*}
Hence 
\begin{eqnarray*}
\abs{\psi_{v,\lambda}\left(x,y_{1},y_{2}\right)-y_{1}y_{2}} & \leq & \frac{r'_{1}r'_{2}}{\abs{2^{a}}}\sum_{n\geq1}\abs{\frac{2^{a_{1}}y_{1}}{\tilde{r}_{1}+\delta}\exp\left(-\frac{B}{\abs x}\right)}^{n}\\
 & \leq & \frac{r'_{1}r'_{2}}{\abs{2^{a}}}\frac{\abs{\frac{2^{a_{1}}y_{1}}{\tilde{r}_{1}+\delta}\exp\left(-\frac{B}{\abs x}\right)}}{1-\abs{\frac{2^{a_{1}}y_{1}}{\tilde{r}_{1}+\delta}\exp\left(-\frac{B}{\abs x}\right)}}\\
 & \leq & A\exp\left(-\frac{B}{\abs x}\right)\,\,,
\end{eqnarray*}
for a convenient $A>0$. 
\end{proof}
The latter lemma implies $\Psi_{v,\pm\lambda,0}\left(w\right)=w$,
having for consequence the next result.
\begin{cor}
\label{cor: quantit=0000E9 born=0000E9e}For all closed sub-sector
$S'\subset S_{\pm\lambda}$ and for all $\tilde{r}_{1}\in\left]0,r_{1}\right[$
and $\tilde{r}_{2}\in\left]0,r_{2}\right[$, there exists $A,B>0$
such that for all $x\in S'$: 
\begin{eqnarray*}
\left.\begin{array}{c}
{\displaystyle \abs{h_{1}}\leq\frac{\tilde{r}_{1}}{\abs{f_{1}\left(x,h_{1}h_{2}\right)}}}\\
{\displaystyle \abs{h_{2}}\leq\frac{\tilde{r}_{2}}{\abs{f_{2}\left(x,h_{1}h_{2}\right)}}}
\end{array}\right\}  & \Longrightarrow & \abs{\Psi_{w,\pm}\left(x,h_{1},h_{2}\right)-h_{1}h_{1}}\leq\frac{A\exp\left(-\frac{B}{\abs x}\right)}{\abs{x^{a}}}\,\,.
\end{eqnarray*}
In particular, there exists $C>0$ such that: 
\begin{eqnarray*}
\left.\begin{array}{c}
{\displaystyle \abs{h_{1}}\leq\frac{\tilde{r}_{1}}{\abs{f_{1}\left(x,h_{1}h_{2}\right)}}}\\
{\displaystyle \abs{h_{2}}\leq\frac{\tilde{r}_{2}}{\abs{f_{2}\left(x,h_{1}h_{2}\right)}}}
\end{array}\right\}  & \Longrightarrow & \frac{\abs{\exp\left(c_{m}\left(h_{1}h_{2}\right)^{m}\log\left(x\right)+\frac{\tilde{c}\left(x^{a}\left(h_{1}h_{2}\right)^{m}\right)}{x}\right)}}{\abs{\exp\left(c_{m}\left(\Psi_{w}\left(x,h_{1},h_{2}\right)\right)^{m}\log\left(x\right)+\frac{\tilde{c}\left(x^{a}\left(\Psi_{w}\left(x,h_{1},h_{2}\right)\right)^{m}\right)}{x}\right)}}<C\,.
\end{eqnarray*}
\end{cor}

\subsubsection{Power series expansion of sectorial isotropies in the space of leaves}

~

Now, we give a power series expansion of $\Psi_{1,\pm\lambda}$ and
$\Psi_{2,\pm\lambda}$ in the space of leaves. Let us introduce the
following notations:
\[
\begin{cases}
N\left(1,+\right):=N\left(2,-\right):=1\\
N\left(1,-\right):=N\left(2,+\right):=-1 & .
\end{cases}
\]
\begin{lem}
\label{lem: istropies espace des feuilles 2}With the notations and
assumptions above, there exists entire functions (\emph{i.e. }analytic
over $\ww C$) denoted by $\Psi_{j,\pm\lambda,n}$, $j\in\acc{1,2}$,
$n\geq N\left(j,\pm\right)$, such that for $j\in\acc{1,2}:$
\[
\begin{cases}
{\displaystyle \Psi_{j,\lambda}\left(h_{1},h_{2}\right)=\sum_{n\geq N\left(j,+\right)}\Psi_{j,\lambda,n}\left(h_{1}h_{2}\right)h_{1}^{n}}\\
{\displaystyle \Psi_{j,-\lambda}\left(h_{1},h_{2}\right)=\sum_{n\geq N\left(j,-\right)}\Psi_{j,\lambda,n}\left(h_{1}h_{2}\right)h_{2}^{n}} & .
\end{cases}
\]
These series converge normally in every subset of $\Gamma_{\pm\lambda}$
of the form $\lsr{\pm}{\tilde{\mathbf{r}}}$ with $0<\tilde{r}_{1}<r_{1}$
and $0<\tilde{r}_{2}<r_{2}$ (\emph{cf.} Definition\emph{ }\ref{def: espace des feuilles =0000E0 rayon fixe}).
More precisely, for all $\tilde{r}_{1},\tilde{r}_{2},\delta>0$ such
that 
\[
0<\tilde{r}_{j}+\delta<r_{j}~,~j\in\left\{ 1,2\right\} 
\]
there exists $C>0$ such that for all $x\in S_{\pm\lambda}$ and for
all $w\in\ww C$, we have: 
\begin{eqnarray*}
\abs{wx^{a}}\leq\tilde{r}_{1}\tilde{r}_{2} & \Longrightarrow & \begin{cases}
{\displaystyle \abs{\Psi_{1,\lambda,n}\left(w\right)}<Cr'_{1}\frac{\abs{f_{1}\left(x,w\right)}^{n-1}}{\left(\tilde{r}_{1}+\delta\right)^{n}}} & ,\,n\geq1\\
{\displaystyle \abs{\Psi_{2,\lambda,n}\left(w\right)}<\frac{Cr'_{2}}{\abs{x^{a}}}\frac{\abs{f_{1}\left(x,w\right)}^{n+1}}{\left(\tilde{r}_{1}+\delta\right)^{n}}} & ,\,n\geq-1\\
{\displaystyle \abs{\Psi_{1,-\lambda,n}\left(w\right)}<\frac{Cr'_{1}}{\abs{x^{a}}}\frac{\abs{f_{2}\left(x,w\right)}^{n+1}}{\left(\tilde{r}_{2}+\delta\right)^{n}}} & ,\,n\geq-1\\
{\displaystyle \abs{\Psi_{2,-\lambda,n}\left(w\right)}<Cr'_{2}\frac{\abs{f_{2}\left(x,w\right)}^{n-1}}{\left(\tilde{r}_{2}+\delta\right)^{n}}} & ,\,n\geq1\,\,.
\end{cases}
\end{eqnarray*}
Moreover: 
\[
\Psi_{1,-\lambda,-1}\left(0\right)=\Psi_{2,\lambda,-1}\left(0\right)=0\,\,.
\]
\end{lem}

\begin{proof}
We use the same notations as in the proof of Lemma \ref{lem: isotropie monomone res esp des feuilles},
and as usual, we give the proof only for $\Psi_{\lambda}$ (the proof
for $\Psi_{-\lambda}$ is analogous, by exchanging the role played
by $h_{1}$ and $h_{2}$). For fixed $w\in\ww C$, the maps
\[
\varphi_{1}:h_{1}\mapsto\Psi_{1,\lambda}\left(h_{1},\frac{w}{h_{1}}\right)
\]
and 
\[
\varphi_{2}:h_{1}\mapsto\Psi_{2,\lambda}\left(h_{1},\frac{w}{h_{1}}\right)
\]
are analytic in 
\[
M_{w,1}=\bigcup_{\substack{x\in S_{\lambda}\\
\abs{wx^{a}}<r_{1}r_{2}
}
}\Omega_{x,w}
\]
(see the proof of Lemma \ref{lem: isotropie monomone res esp des feuilles}).
In particular, $\varphi_{1}$ and $\varphi_{2}$ admit Laurent expansions
\[
\begin{cases}
{\displaystyle \varphi_{1}\left(h_{1}\right)=\Psi_{1,\lambda}\left(h_{1},\frac{w}{h_{1}}\right)=\sum_{n\geq-L_{1}}\Psi_{1,\lambda,n}\left(w\right)h_{1}^{n}}\\
{\displaystyle \varphi_{2}\left(h_{1}\right)=\Psi_{2,\lambda}\left(h_{1},\frac{w}{h_{1}}\right)=\sum_{n\geq-L_{2}}\Psi_{2,\lambda,n}\left(w\right)h_{1}^{n}}
\end{cases}
\]
in every annulus $\Omega_{x,w}$, with $x\in S_{\lambda}$ such that
$\abs{wx^{a}}<r_{1}r_{2}$. Using the same method as in the proof
of Lemma \ref{lem: isotropie monomone res esp des feuilles}, we prove
without additional difficulties that for all $n\in\ww Z$, $\Psi_{1,\lambda,n}$
and $\Psi_{2,\lambda,n}$ are analytic in any point $w\in\ww C$,
and thus are entire functions (\emph{i.e.} analytic over $\wc$).
Moreover, we also show in the same way as earlier that for all $\tilde{r}_{1},\tilde{r}_{2},\delta>0$
with 
\[
0<\tilde{r}_{j}+\delta<r_{j}~,~j\in\left\{ 1,2\right\} \,\,,
\]
for all $n\in\ww Z$ and for all $\left(x,w\right)\in S_{\lambda}\times\ww C$
such that $\abs{wx^{a}}\leq\tilde{r}_{1}\tilde{r}_{2}$, we have:
\[
\begin{cases}
{\displaystyle \abs{\Psi_{1,\lambda,n}\left(w\right)}<\frac{r'_{1}}{\abs{f_{1}\left(x,\Psi_{w,\lambda}\left(x,h_{1},\frac{w}{h_{1}}\right)\right)}}\abs{\frac{f_{1}\left(x,w\right)}{\tilde{r}_{1}+\delta}}^{n}}\\
{\displaystyle \abs{\Psi_{2,\lambda,n}\left(w\right)}<\frac{r'_{2}}{\abs{f_{2}\left(x,\Psi_{w,\lambda}\left(x,h_{1},\frac{w}{h_{1}}\right)\right)}}\abs{\frac{f_{1}\left(x,w\right)}{\tilde{r}_{1}+\delta}}^{n}} & \,.
\end{cases}
\]
According to Corollary \ref{cor: quantit=0000E9 born=0000E9e}, there
exists $C>0$ such that for all $\left(x,w\right)\in S_{\lambda}\times\ww C$
with $\abs{wx^{a}}\leq\tilde{r}_{1}\tilde{r}_{2}$, we have: 
\[
\begin{cases}
{\displaystyle \abs{\Psi_{1,\lambda,n}\left(w\right)}<Cr'_{1}\frac{\abs{f_{1}\left(x,w\right)}^{n-1}}{\left(\tilde{r}_{1}+\delta\right)^{n}}}\\
{\displaystyle \abs{\Psi_{2,\lambda,n}\left(w\right)}<\frac{Cr'_{2}}{\abs{x^{a}}}\frac{\abs{f_{1}\left(x,w\right)}^{n+1}}{\left(\tilde{r}_{1}+\delta\right)^{n}}} & \,.
\end{cases}
\]
According to the statement in Fact \ref{fact: limites}, for a fixed
value $w\in\ww C$, if we look at the limit as $x$ tends to $0$
in $S_{\lambda}$ of the right hand-sides above we deduce that: 
\[
\begin{cases}
\abs{\Psi_{1,\lambda,n}\left(w\right)}=0 & ,\,\forall n\leq0\\
\abs{\Psi_{2,\lambda,n}\left(w\right)}=0 & ,\,\forall n\leq-2\,.
\end{cases}
\]
Consequently: 
\[
\begin{cases}
{\displaystyle \Psi_{1,\lambda}\left(h_{1},h_{2}\right)=\sum_{n\geq1}\Psi_{1,\lambda,n}\left(h_{1}h_{2}\right)h_{1}^{n}}\\
{\displaystyle \Psi_{2,\lambda}\left(h_{1},h_{2}\right)=\sum_{n\geq-1}\Psi_{2,\lambda,n}\left(h_{1}h_{2}\right)h_{1}^{n}} & \,\,.
\end{cases}
\]
These function series converges normally (and are analytic) in every
domain of the form$\lsr{}{\tilde{\mathbf{r}}}$ with $\tilde{\mathbf{r}}:=\left(\tilde{r}_{1},\tilde{r}_{2}\right)$
and 
\[
0<\tilde{r}_{j}+\delta<r_{j}~,~j\in\left\{ 1,2\right\} \,\,
\]
(\emph{cf.} Definition\emph{ }\ref{def: espace des feuilles =0000E0 rayon fixe}).
Moreover, for any fixed value of $h_{2}$, on the one hand the function
series 
\[
h_{1}\mapsto\Psi_{2,\lambda}\left(h_{1},h_{2}\right)=\sum_{n\geq-1}\Psi_{2,\lambda,n}\left(h_{1}h_{2}\right)h_{1}^{n}
\]
is analytic in a punctured disc, since 
\begin{eqnarray*}
{\displaystyle \abs{f_{2}\left(x,h_{1},h_{2}\right)}} & \underset{\substack{x\rightarrow0\\
x\in S_{\lambda}
}
}{\longrightarrow} & 0\,,
\end{eqnarray*}
and on the other hand, we already know that the function $h_{1}\mapsto\Psi_{2,\lambda}\left(h_{1},h_{2}\right)$
is analytic in a neighborhood of the origin. Thus, we must have $\Psi_{2,\lambda,-1}\left(0\right)=0$.
\end{proof}

\subsubsection{Sectorial isotropies: proof of Proposition \ref{prop: isotropies plates}}

~

The following lemma is a more precise version of Proposition \ref{prop: isotropies plates}.
We recall the notations:
\[
\begin{cases}
N\left(1,+\right)=N\left(2,-\right)=1\\
N\left(1,-\right)=N\left(2,+\right)=-1 & .
\end{cases}
\]
\begin{lem}
\label{lem: isotropies exp plates}With the notations and assumptions
above, we consider $\psi_{\pm\lambda}\in\Lambda_{\pm\lambda}^{\left(\tx{weak}\right)}\left(\ynorm\right)$,
with
\begin{eqnarray*}
\psi_{\pm\lambda}\left(x,\mathbf{y}\right) & = & \left(x,\psi_{1,\pm\lambda}\left(x,\mathbf{y}\right),\psi_{2,\pm\lambda}\left(x,\mathbf{y}\right)\right)\,\,.
\end{eqnarray*}
 Then, for $i\in\acc{1,2}$, $\psi_{i,\lambda}$ and $\psi_{i,-\lambda}$
can be written as power series as follows: 
\[
\begin{cases}
{\displaystyle \psi_{i,\lambda}\left(x,\mathbf{y}\right)=y_{i}+f_{i}\left(x,\frac{\psi_{v,\lambda}\left(x,\mathbf{y}\right)}{x^{a}}\right)\sum_{n\ge N\left(i,+\right)+1}\Psi_{i,\lambda,n}\left(\frac{y_{1}y_{2}}{x^{a}}\right)\left(\frac{y_{1}}{f_{1}\left(x,\frac{y_{1}y_{2}}{x^{a}}\right)}\right)^{n}}\\
{\displaystyle \psi_{i,-\lambda}\left(x,\mathbf{y}\right)=y_{i}+f_{i}\left(x,\frac{\psi_{v,-\lambda}\left(x,\mathbf{y}\right)}{x^{a}}\right)\sum_{n\ge N\left(i,-\right)+1}\Psi_{i,-\lambda,n}\left(\frac{y_{1}y_{2}}{x^{a}}\right)\left(\frac{y_{2}}{f_{2}\left(x,\frac{y_{1}y_{2}}{x^{a}}\right)}\right)^{n}} & .
\end{cases}
\]
which are normally convergent in every subset of $S_{\pm\lambda}\times\mathbf{D\left(0,r\right)}$
of the form $S_{\pm\lambda}\times\mathbf{\overline{D}\left(0,\tilde{r}\right)},$
where $\mathbf{\overline{D}\left(0,\tilde{r}\right)}$ is a closed
poly-disc with $\mathbf{\tilde{r}}=\left(\tilde{r}_{1},\tilde{r}_{2}\right)$
such that 
\[
0<\tilde{r}_{j}<r_{j}~~~~,~j\in\left\{ 1,2\right\} ~.
\]
Here $\Psi_{i,\lambda,n}$ , $\Psi_{i,-\lambda,n}$ (for $i=1,2$
and $n\in\ww N$) are given in Lemma \ref{lem: istropies espace des feuilles 2}.
Moreover, for all closed sub-sector $S'\subset S_{\pm\lambda}$ and
for all closed poly-disc $\mathbf{\overline{D}}\subset\mathbf{D\left(0,r\right)}$,
there exists $A,B>0$ such that for $j=1,2$:
\begin{eqnarray*}
\abs{\psi_{j,\pm\lambda}\left(x,y_{1},y_{2}\right)-y_{j}} & \leq & A\exp\left(-\frac{B}{\abs x}\right)\,\,,\,\,\forall\left(x,\mathbf{y}\right)\in S'\times\mathbf{\overline{D}}\,\,.
\end{eqnarray*}
As a consequence, $\psi_{j,\pm\lambda}$ admits $y_{j}$ as Gevrey-1
asymptotic expansion in $S_{\pm\lambda}\times\mathbf{D\left(0,r\right)}$.
\end{lem}

\begin{rem}
In particular, $\Psi_{1,\lambda,1}\left(w\right)=\Psi_{2,-\lambda,1}\left(w\right)=1$
and $\Psi_{1,-\lambda,-1}\left(w\right)=\Psi_{2,\lambda,-1}\left(w\right)=w$.
\end{rem}

\begin{proof}
By definition, we have 
\[
\Psi_{\pm\lambda}\circ\cal H_{\pm\lambda}=\cal H_{\pm\lambda}\circ\psi_{\pm}\,.
\]
In particular, for $j=1,2$ and all $\left(x,\mathbf{y}\right)\in S_{\pm\lambda}\times\mathbf{D\left(0,r\right)}$:
\begin{eqnarray*}
\Psi_{j,\pm\lambda}\left(x,\frac{y_{1}}{f_{1}\left(x,\frac{y_{1}y_{2}}{x^{a}}\right)},\frac{y_{2}}{f_{2}\left(x,\frac{y_{1}y_{2}}{x^{a}}\right)}\right) & = & \frac{\psi_{j,\pm\lambda}\left(x,y_{1},y_{2}\right)}{f_{j}\left(x,\frac{\psi_{v,\pm}\left(x,y_{1},y_{2}\right)}{x^{a}}\right)}\,\,.
\end{eqnarray*}
Thus, according to Lemma \ref{lem: istropies espace des feuilles 2}
we have for $i=1,2$: 
\begin{equation}
\begin{cases}
{\displaystyle \psi_{i,\lambda}\left(x,\mathbf{y}\right)=f_{i}\left(x,\frac{\psi_{v,\lambda}\left(x,\mathbf{y}\right)}{x^{a}}\right)\sum_{n\ge N\left(i,+\right)}\Psi_{i,\lambda,n}\left(\frac{y_{1}y_{2}}{x^{a}}\right)\left(\frac{y_{1}}{f_{1}\left(x,\frac{y_{1}y_{2}}{x^{a}}\right)}\right)^{n}}\\
{\displaystyle \psi_{i,-}\left(x,\mathbf{y}\right)=f_{i}\left(x,\frac{\psi_{v,-\lambda}\left(x,\mathbf{y}\right)}{x^{a}}\right)\sum_{n\ge N\left(i,-\right)}\Psi_{i,-\lambda,n}\left(\frac{y_{1}y_{2}}{x^{a}}\right)\left(\frac{y_{2}}{f_{2}\left(x,\frac{y_{1}y_{2}}{x^{a}}\right)}\right)^{n}} & ,
\end{cases}\label{eq:psi et Psi}
\end{equation}
and these series are normally convergent (and then define analytic
functions) in any domain of the form $S'\times\mathbf{\overline{D}\left(0,\tilde{r}\right)},$
where $S'$ is a closed sub-sector of $S_{\pm\lambda}$ and $\mathbf{\overline{D}\left(0,\tilde{r}\right)}$
is a closed poly-disc with $\mathbf{\tilde{r}}=\left(\tilde{r}_{1},\tilde{r}_{2}\right)$
such that 
\[
0<\tilde{r}_{j}<r_{j}~~~~,j\in\left\{ 1,2\right\} ~.
\]
Let us compare the different expressions of $\psi_{j,\pm\lambda}$,
$j=1,2$. We know that $\psi_{j,\pm\lambda}\left(x,y_{1},y_{2}\right)$
admits $y_{j}$ as weak Gevrey-1 asymptotic expansion in $S_{\pm\lambda}\times\mathbf{D\left(0,r\right)}$.
Thus, we can write:
\begin{eqnarray*}
\psi_{j,\pm\lambda}\left(x,y_{1},y_{2}\right) & = & y_{j}+\sum_{\mathbf{k}\in\ww N^{2}}\psi_{j,\pm\lambda,\mathbf{k}}\left(x\right)\mathbf{y^{k}}\,\,,
\end{eqnarray*}
where for all $\mathbf{k}=\left(k_{1},k_{2}\right)\in\ww N^{2}$,
$\psi_{j,\pm\lambda,\mathbf{k}}$ is analytic in $S_{\pm\lambda}$
and admits $0$ as Gevrey-1 asymptotic expansion. As usual, let us
deal with the case of $\psi_{1,\lambda}$ and $\psi_{2,\lambda}$
(the other one being similar by exchanging $y_{1}$ and $y_{2}$).

According to the expressions given by Lemmas \ref{lem: isotropie monomone res esp des feuilles}
and \ref{lem: istropies espace des feuilles 2}, we can be more precise
on the index sets in the sums above:
\begin{equation}
\begin{cases}
{\displaystyle \psi_{1,\lambda}\left(x,y_{1},y_{2}\right)=y_{1}+\sum_{\substack{\mathbf{k}=\left(k_{1},k_{2}\right)\in\ww N^{2}\\
k_{1}\geq k_{2}+1
}
}\psi_{1,\lambda,\mathbf{k}}\left(x\right)y_{1}^{k_{1}}y_{2}^{k_{2}}}\\
\psi_{2,\lambda}\left(x,y_{1},y_{2}\right)=y_{2}+\sum_{\substack{\mathbf{k}=\left(k_{1},k_{2}\right)\in\ww N^{2}\\
k_{1}\geq k_{2}
}
}\psi_{2,\lambda,\mathbf{k}}\left(x\right)y_{1}^{k_{1}}y_{2}^{k_{2}} & {\displaystyle .}
\end{cases}\label{eq: psi_1 et psi_2}
\end{equation}
 Let us deal with $\psi_{1,\lambda}$ (a similar proof holds for $\psi_{2,\lambda}$).
Looking at terms for $n=1$ in (\ref{eq:psi et Psi}) and at monomials
terms $\mathbf{y^{k}}$ such that $k_{1}\leq k_{2}+1$ in (\ref{eq: psi_1 et psi_2}),
we must have for all $x\in S_{\lambda}$, $y_{1},y_{2}\in\ww C$ with
$\abs{y_{1}}<r_{1}$, $\abs{y_{2}}<r_{2}$:
\begin{eqnarray*}
1+\sum_{k\ge0}\psi_{1,\lambda,\left(k+1,k\right)}\left(x\right)y_{1}^{k}y_{2}^{k} & = & \frac{f_{1}\left(x,\frac{\psi_{v,\lambda}\left(x,\mathbf{y}\right)}{x^{a}}\right)}{f_{1}\left(x,\frac{y_{1}y_{2}}{x^{z}}\right)}\Psi_{1,\lambda,1}\left(\frac{y_{1}y_{2}}{x^{a}}\right)\,\,.
\end{eqnarray*}
According to Lemma \ref{lem: istropie monomiale} and Corollary\ref{cor: quantit=0000E9 born=0000E9e},
we have: 
\begin{eqnarray*}
\frac{f_{1}\left(x,\frac{\psi_{v,\lambda}\left(x,\mathbf{y}\right)}{x^{a}}\right)}{f_{1}\left(x,\frac{y_{1}y_{2}}{x^{z}}\right)} & = & 1+\sum_{\substack{j_{1}\geq j_{2}+1\geq1}
}F_{j_{1},j_{2}}\left(x\right)y_{1}^{j_{1}}y_{2}^{j_{2}}\\
 & = & 1+\underset{\substack{\left(x,\mathbf{y}\right)\longrightarrow0\\
\left(x,\mathbf{y}\right)\in S_{\lambda}\times\mathbf{D\left(0,r\right)}
}
}{\tx O}\left(\abs{y_{1}}\right)\,\,,
\end{eqnarray*}
for some analytic and bounded functions $F_{j_{1},j_{2}}\left(x\right)$,
$j_{1}\ge j_{2}$. As in the proof of Lemma \ref{lem: istropie monomiale},
using the fact that $\psi_{\lambda}$ admits the identity as weak
Gevrey-1 asymptotic expansion, we deduce that $\Psi_{1,\lambda,1}\left(w\right)=1$,
and then:
\begin{eqnarray*}
\psi_{1,\lambda}\left(x,\mathbf{y}\right) & = & y_{1}+f_{1}\left(x,\frac{\psi_{v,\lambda}\left(x,\mathbf{y}\right)}{x^{a}}\right)\sum_{n\ge2}\Psi_{1,\lambda,n}\left(\frac{y_{1}y_{2}}{x^{a}}\right)\left(\frac{y_{1}}{f_{1}\left(x,\frac{y_{1}y_{2}}{x^{a}}\right)}\right)^{n}\\
 & = & y_{1}+\sum_{\substack{\mathbf{k}=\left(k_{1},k_{2}\right)\in\ww N^{2}\\
k_{1}\geq k_{2}+2
}
}\psi_{1,\lambda,\mathbf{k}}\left(x\right)y_{1}^{k_{1}}y_{2}^{k_{2}}\,\,.
\end{eqnarray*}
It remains to show that $\psi_{1,\lambda}$ admits $y_{1}$ as Gevrey-1
asymptotic expansion in $S_{\lambda}\times\mathbf{D\left(0,r\right)}$.
From the computations above, we deduce: 
\begin{eqnarray*}
\abs{\psi_{1,\lambda}\left(x,y_{1},y_{2}\right)-y_{1}} & \leq & \sum_{n\geq2}\abs{\Psi_{1,\lambda,n}\left(\frac{y_{1}y_{2}}{x^{a}}\right)\left(\frac{y_{1}}{f_{1}\left(x,\frac{y_{1}y_{2}}{x^{a}}\right)}\right)^{n-1}\frac{f_{1}\left(x,\frac{\psi_{v,\lambda}\left(x,\mathbf{y}\right)}{x^{a}}\right)}{f_{1}\left(x,\frac{y_{1}y_{2}}{x^{a}}\right)}y_{1}}\,\,.
\end{eqnarray*}
Using Lemma \ref{lem: istropies espace des feuilles 2}, Corollary
\ref{cor: quantit=0000E9 born=0000E9e} and the same method as at
the end of the proof of Lemma \ref{lem: istropie monomiale}, we can
show the following: we can take $r_{1},r_{2}>0$ small enough such
that for all closed sub-sector $S'$ of $S_{\lambda}$ for all $\tilde{r}_{1}\in\left]0,r_{1}\right[$
and $\tilde{r}_{2}\in\left]0,r_{2}\right[$, there exists $A,B>0$
satisfying:
\begin{eqnarray*}
\left(x,y_{1},y_{2}\right)\in S'\times\mathbf{D\left(0,\tilde{r}\right)} & \Longrightarrow & \abs{\psi_{1,\lambda}\left(x,y_{1},y_{2}\right)-y_{1}}\leq A\exp\left(-\frac{B}{\abs x}\right)\,\,.
\end{eqnarray*}

A similar proof holds for $\psi_{2,\lambda},\psi_{2,-\lambda}$ and
$\psi_{1,-\lambda}$.
\end{proof}
\begin{rem}
It should be noticed that in the expressions 
\[
\begin{cases}
{\displaystyle \psi_{1,\lambda}\left(x,\mathbf{y}\right)=y_{1}+f_{1}\left(x,\frac{\psi_{v,\lambda}\left(x,\mathbf{y}\right)}{x^{a}}\right)\sum_{n\ge2}\Psi_{1,\lambda,n}\left(\frac{y_{1}y_{2}}{x^{a}}\right)\left(\frac{y_{1}}{f_{1}\left(x,\frac{y_{1}y_{2}}{x^{a}}\right)}\right)^{n}}\\
{\displaystyle \psi_{1,-\lambda}\left(x,\mathbf{y}\right)=y_{1}+f_{1}\left(x,\frac{\psi_{v,-\lambda}\left(x,\mathbf{y}\right)}{x^{a}}\right)\sum_{n\ge0}\Psi_{1,-\lambda,n}\left(\frac{y_{1}y_{2}}{x^{a}}\right)\left(\frac{y_{2}}{f_{2}\left(x,\frac{y_{1}y_{2}}{x^{a}}\right)}\right)^{n}}
\end{cases}
\]
given by Lemma \ref{lem: isotropies exp plates}, the expansion of
$\psi_{1,\lambda}$ with respect to $\mathbf{y}=\left(y_{1},y_{2}\right)$
starts with a term of order $1$, namely $y_{1}$, followed by terms
of order at least $2$, while in the expansion of $\psi_{1,-\lambda}$,
the term of lowest order is a constant, namely $\Psi_{1,-\lambda,0}\left(0\right)$.
Similarly, the expansion of $\psi_{2,-\lambda}$ (with respect to
$\mathbf{y}=\left(y_{1},y_{2}\right)$) starts with $y_{2}$, while
the expansion of $\psi_{1,-\lambda}$ starts with the constant $\Psi_{2,\lambda,0}\left(0\right)$.
\end{rem}

\subsection{Description of the moduli space and some applications}

~

From Lemmas \ref{lem: istropies espace des feuilles 2} and \ref{lem: isotropies exp plates},
we can give a description of the moduli space $\Lambda_{\lambda}\left(\ynorm\right)\times\Lambda_{-\lambda}\left(\ynorm\right)$
of a fixed analytic normal form $\ynorm$.

\subsubsection{\label{subsec: description moduli space}A power series presentation
of the moduli space}

~

We use the notations introduced in section \ref{sec:Sectorial-isotropies-and}.
We denote by $\cal O\left(\text{\ensuremath{\ww C}}\right)$ the set
of entire functions, \emph{i.e.} of functions holomorphic in $\ww C$.
We consider the functions $f_{1}$ and $f_{2}$ defined in $\left(\mbox{\ref{eq: f1 et f2}}\right)$
and introduce four subsets of $\left(\cal O\left(\ww C\right)\right)^{\ww N}$,
denoted by $\cal E_{1,\lambda}\left(\ynorm\right)$, $\cal E_{2,\lambda}\left(\ynorm\right)$,
$\cal E_{1,-\lambda}\left(\ynorm\right)$ and $\cal E_{2,-\lambda}\left(\ynorm\right)$,
defined as follows. On remind the notations 
\[
\begin{cases}
N\left(1,+\right)=N\left(2,-\right)=1\\
N\left(1,-\right)=N\left(2,+\right)=-1 & .
\end{cases}
\]
\begin{defn}
For $j\in\acc{1,2}$, a sequence $\left(\psi_{n}\left(w\right)\right)_{n\geq N\left(j,\pm\right)+1}\in\left(\cal O\left(\ww C\right)\right)^{\ww N}$
belongs to $\cal E_{j,\pm\lambda}\left(\ynorm\right)$ if there exists
an open polydisc $\mathbf{D\left(0,r\right)}$ and an open asymptotic
sector 
\begin{eqnarray*}
S_{\pm\lambda} & \in & \cal{AS}_{\arg\left(\pm\lambda\right),2\pi}
\end{eqnarray*}
such that for all $\tilde{r}_{1},\tilde{r}_{2},\delta>0$ with 
\[
0<\tilde{r}_{i}+\delta<r_{i}~~~~,~i\in\left\{ 1,2\right\} 
\]
there exists $C>0$ such that for all $x\in S_{\lambda}$ (\emph{resp.
$x\in S_{-\lambda}$}) and for all $w\in\ww C$:
\begin{eqnarray*}
\abs{wx^{a}}\leq\tilde{r}_{1}\tilde{r}_{2} & \Longrightarrow & \begin{cases}
{\displaystyle \abs{\psi_{n}\left(w\right)}<C\frac{\abs{f_{1}\left(x,w\right)}^{n-1}}{\left(\tilde{r}_{1}+\delta\right)^{n}}\mbox{ , }\forall n\geq2,} & {\displaystyle \mbox{ if }\left(\psi_{n}\left(w\right)\right)_{n\geq2}\in\cal E_{1,\lambda}\left(\ynorm\right)}\\
{\displaystyle \abs{\psi_{n}\left(w\right)}<\frac{C}{\abs{x^{a}}}\frac{\abs{f_{1}\left(x,w\right)}^{n+1}}{\left(\tilde{r}_{1}+\delta\right)^{n}}\mbox{ , }\forall n\geq0,} & {\displaystyle \mbox{ if }\left(\psi_{n}\left(w\right)\right)_{n\geq2}\in\cal E_{2,\lambda}\left(\ynorm\right)}\\
{\displaystyle \abs{\psi_{n}\left(w\right)}<\frac{C}{\abs{x^{a}}}\frac{\abs{f_{2}\left(x,w\right)}^{n+1}}{\left(\tilde{r}_{2}+\delta\right)^{n}}\mbox{ , }\forall n\geq0,} & {\displaystyle \mbox{ if }\left(\psi_{n}\left(w\right)\right)_{n\geq2}\in\cal E_{1,-\lambda}\left(\ynorm\right)}\\
{\displaystyle \abs{\psi_{n}\left(w\right)}<C\frac{\abs{f_{2}\left(x,w\right)}^{n-1}}{\left(\tilde{r}_{2}+\delta\right)^{n}}\mbox{ , }\forall n\geq2,} & {\displaystyle \mbox{ if }\left(\psi_{n}\left(w\right)\right)_{n\geq2}\in\cal E_{2,-\lambda}\left(\ynorm\right)\,\,.}
\end{cases}
\end{eqnarray*}
\end{defn}

As explained in section \ref{sec:Sectorial-isotropies-and}, we can
associate to any pair 
\[
\left(\psi_{\lambda},\psi_{-\lambda}\right)\in\Lambda_{\lambda}\left(\ynorm\right)\times\Lambda_{-\lambda}\left(\ynorm\right)
\]
two germs of sectorial biholomorphisms of the space of leaves corresponding
to each ``narrow'' sector, which we denote by $\Psi_{\lambda}$
and $\Psi_{-\lambda}$, defined by: 
\begin{equation}
\Psi_{\pm\lambda}:=\cal H_{\pm\lambda}\circ\psi_{\pm\lambda}\circ\cal H_{\pm\lambda}^{-1}\,\,,\label{eq: isotropie dans espace des feuilles}
\end{equation}
where $\cal H_{\pm\lambda}$ is given by Corollary \ref{cor: coordonn=0000E9es espace des feuiles}.
According to Lemmas \ref{lem: istropies espace des feuilles 2} and
\ref{lem: isotropies exp plates}, if we write $\Psi_{\pm\lambda}=\left(x,\Psi_{1,\pm\lambda},\Psi_{2,\pm\lambda}\right)$,
then for $j=1,2$ we have: 
\begin{eqnarray}
{\displaystyle \Psi_{j,\lambda}\left(h_{1},h_{2}\right)} & = & h_{j}+\sum_{n\geq N\left(j,+\right)+1}\Psi_{j,\lambda,n}\left(h_{1}h_{2}\right)h_{1}^{n}\label{eq: description moduli space}\\
{\displaystyle \Psi_{j,-\lambda}\left(h_{1},h_{2}\right)} & = & h_{j}+\sum_{n\geq N\left(j,-\right)+1}\Psi_{j,-\lambda,n}\left(h_{1}h_{2}\right)h_{2}^{n}\nonumber 
\end{eqnarray}
 $\left(\Psi_{j,\pm\lambda,n}\right)_{n}\in{\cal E}_{j,\pm\lambda}$
. Conversely, given $\left(\Psi_{j,\pm\lambda}\right)_{n}\in\cal E_{j,\pm\lambda}$
for $j=1,2$, the estimates made in section \ref{sec:Sectorial-isotropies-and}
show that 
\[
\psi_{\pm\lambda}:=\cal H_{\pm\lambda}^{-1}\circ\Psi_{\pm\lambda}\circ\cal H_{\pm\lambda}~,
\]
where $\Psi_{\pm\lambda}\left(x,\mathbf{h}\right)=\left(x,\Psi_{1,\pm\lambda}\left(\mathbf{h}\right),\Psi_{2,\pm\lambda}\left(\mathbf{h}\right)\right)$,
belongs to $\Lambda_{\pm\lambda}\left(\ynorm\right)$. Consequently,
we can state:
\begin{prop}
\label{prop: description espace de module}We have the following bijections:
\begin{eqnarray*}
\Lambda_{\lambda}\left(\ynorm\right) & \tilde{\rightarrow} & \cal E_{1,\lambda}\left(\ynorm\right)\times\cal E_{2,\lambda}\left(\ynorm\right)\\
\psi_{\lambda} & \mapsto & \left(\Psi_{1,\lambda},\Psi_{2,\lambda}\right)
\end{eqnarray*}
and
\begin{eqnarray*}
\Lambda_{-\lambda}\left(\ynorm\right) & \tilde{\rightarrow} & \cal E_{1,-\lambda}\left(\ynorm\right)\times\cal E_{2,-\lambda}\left(\ynorm\right)\\
\psi_{-\lambda} & \mapsto & \left(\Psi_{1,-\lambda},\Psi_{2,-\lambda}\right)
\end{eqnarray*}
$\Big($we identify here $\Psi_{\pm\lambda}\left(x,\mathbf{h}\right)=\left(x,\Psi_{1,\pm\lambda}\left(\mathbf{h}\right),\Psi_{2,\pm\lambda}\left(\mathbf{h}\right)\right)$
with $\left(\Psi_{1,\pm\lambda}\left(\mathbf{h}\right),\Psi_{2,\pm\lambda}\left(\mathbf{h}\right)\right)$$\Big)$.
\end{prop}

\subsubsection{Analytic invariant varieties and two-dimensional saddle-nodes}

~

We can give a necessary and sufficient condition for the existence
of analytic invariant varieties in terms of the moduli space described
above.

We recall that for any vector field ${\displaystyle Y\in\cro{\ynorm}}$
as in (\ref{eq: intro}) (\emph{cf.} Definition \ref{def: ynorm class}),
there always exist three formal invariant varieties: $\mathscr{C}=\acc{\left(y_{1},y_{2}\right)=\left(g_{1}\left(x\right),g_{2}\left(x\right)\right)}$,
$\mathscr{H}_{1}=\acc{{\displaystyle y_{1}=f_{1}\left(x,y_{2}\right)}}$
and $\mathscr{H}_{2}=\acc{{\displaystyle y_{2}}=f_{2}\left(x,y_{1}\right)}$,
where $g_{1},g_{2},f_{1},f_{2}$ are formal power series with null
constant term. The first one is classically called the \emph{center
variety, }and we have $\mathscr{C}=\mathscr{H}_{1}\cap\mathscr{H}_{2}$.
If $Y=\ynorm$, then:
\[
\begin{cases}
\mathscr{C}=\acc{y_{1}=y_{2}=0}\\
\mathscr{H}_{1}=\acc{y_{1}=0}\\
\mathscr{H}_{2}=\acc{y_{2}=0} & .
\end{cases}
\]
\begin{prop}
Let ${\displaystyle Y\in\cro{\ynorm}}$ and $\left(\Phi_{\lambda},\Phi_{-\lambda}\right)\in\Lambda_{\lambda}\left(\ynorm\right)\times\Lambda_{-\lambda}\left(\ynorm\right)$
be its Stokes diffeomorphisms. We consider $\Psi_{\pm}=\cal H_{\pm\lambda}\circ\Phi_{\pm\lambda}\circ\cal H_{\pm\lambda}^{-1}$
as above. Then:

\begin{enumerate}
\item the center variety $\mathscr{C}$ is convergent (analytic in the origin)
if and only if $\Psi_{2,\lambda,0}\left(0\right)=\Psi_{1,-\lambda,0}\left(0\right)=0$;
\item the invariant hypersurface $\mathscr{H}_{1}$ is convergent (analytic
in the origin) if and only if for all $n\geq0$, we have $\Psi_{1,-\lambda,n}\left(0\right)=0$;
\item the invariant hypersurface $\mathscr{H}_{2}$ is convergent (analytic
in the origin) if and only if for all $n\geq0$, we have $\Psi_{2,\lambda,n}\left(0\right)=0$.
\end{enumerate}
\end{prop}

\begin{proof}
It is a direct consequence of the power series representation (\ref{eq: description moduli space})
of the Stokes diffeomorphisms $\left(\Phi_{\lambda},\Phi_{-\lambda}\right)$.
Let us explain item $2.$ (the same arguments hold for $1.$ and $3.$
with minor adaptation). The fact that $\Psi_{1,-\lambda,n}\left(0\right)=0$
for all $n\geq0$ means that $\Psi_{1,-\lambda}$ is divisible by
$h_{1}$. Equivalently, both $\Phi_{1,\lambda}$ and $\Phi_{1,-\lambda}$
are divisible by $y_{1}$, so that the analytic hypersurface $\acc{y_{1}=0}$
has the same pre-image by the sectorial normalizing maps $\Phi_{+}$
and $\Phi_{-}$. These pre-images glue together in order to define
an analytic invariant hypersurface $\mathscr{H}_{1}$.
\end{proof}
Notice that if we consider the restriction of a formal normal form
$\ynorm$ to one of the formal invariant hypersurfaces, we obtain
precisely the normal form for two-dimensional saddle-nodes as given
in \cite{MR82}. When one of these hypersurfaces is convergent (\emph{i.e.
}analytic), we recover the Martinet-Ramis invariants by restriction
to this hypersurface, as we present below.
\begin{prop}
Suppose that the formal invariant hypersurface $\mathscr{H}_{1}$
is convergent (\emph{i.e. }analytic in the origin). Then, the Martinet-Ramis
invariants for the saddle-node $Y_{\mid\mathscr{H}_{1}}$ are given
by:
\[
\begin{cases}
{\displaystyle \Psi_{2,\lambda}\left(0,h_{2}\right)=h_{2}+\Psi_{2,\lambda,0}\left(0\right)\in Aff\left(\ww C\right)}\\
{\displaystyle \Psi_{2,-\lambda}\left(0,h_{2}\right)=h_{2}+\sum_{n\geq2}\Psi_{2,-\lambda,n}\left(0\right)h_{2}^{n}\in\diff[\ww C,0]}.
\end{cases}
\]
\end{prop}

Similar result holds for the hypersurface $\mathscr{H}_{2}$.

\subsubsection{The transversally symplectic case and quasi-linear Stokes phenomena
in the first Painlevé equation }

~

Let us now focus on the transversally symplectic case studied in Theorem
\ref{th: espace de module symplectic}. Let $\ynorm\in\snodiag$ be
transversally symplectic (\emph{i.e.} its residue is $\tx{res}\left(\ynorm\right)=1$).
Using the notations introduced in paragraph \ref{subsec: description moduli space},
we define the following sets:
\begin{eqnarray*}
\left(\cal E_{1,\lambda}\left(\ynorm\right)\times\cal E_{2,\lambda}\left(\ynorm\right)\right)_{\omega} & := & \acc{\begin{array}{c}
\Psi_{\lambda}=\left(\Psi_{1,\lambda},\Psi_{2,\lambda}\right)\in\cal E_{1,\lambda}\left(\ynorm\right)\times\cal E_{2,\lambda}\left(\ynorm\right)\\
\mbox{such that: }\det\left(\tx D\Psi_{\lambda}\right)=1
\end{array}}\\
\left(\cal E_{1,-\lambda}\left(\ynorm\right)\times\cal E_{2,-\lambda}\left(\ynorm\right)\right)_{\omega} & := & \acc{\begin{array}{c}
\Psi_{-\lambda}=\left(\Psi_{1,-\lambda},\Psi_{2,-\lambda}\right)\in\cal E_{1,-\lambda}\left(\ynorm\right)\times\cal E_{2,-\lambda}\left(\ynorm\right)\\
\mbox{such that: }\det\left(\tx D\Psi_{-\lambda}\right)=1
\end{array}}.
\end{eqnarray*}
According to Proposition \ref{prop: description espace de module},
the map
\begin{eqnarray*}
\Lambda_{\pm\lambda}\left(\ynorm\right) & \longrightarrow & \cal E_{1,\pm\lambda}\left(\ynorm\right)\times\cal E_{2,\pm\lambda}\left(\ynorm\right)\\
\psi_{\pm\lambda} & \mapsto & \Psi_{\pm\lambda}:=\cal H_{\pm\lambda}\circ\psi_{\pm\lambda}\circ\cal H_{\pm\lambda}^{-1}
\end{eqnarray*}
given in (\ref{eq: isotropie dans espace des feuilles}) is a bijection
$\Big($we identify here $\Psi_{\pm\lambda}\left(x,\mathbf{h}\right)=\left(x,\Psi_{1,\pm\lambda}\left(\mathbf{h}\right),\Psi_{2,\pm\lambda}\left(\mathbf{h}\right)\right)$
with $\left(\Psi_{1,\pm\lambda}\left(\mathbf{h}\right),\Psi_{2,\pm\lambda}\left(\mathbf{h}\right)\right)$$\Big)$.
An easy computation based on (\ref{eq: integrales premieres}) gives:
\begin{eqnarray*}
\left(\cal H_{\pm\lambda}^{-1}\right)^{*}\left(\frac{\tx dy_{1}\wedge\tx dy_{2}}{x}\right) & = & \tx dh_{1}\wedge\tx dh_{2}+\ps{\tx dx}\,\,.
\end{eqnarray*}
This means in particular that $\psi_{\pm\lambda}$ is transversally
symplectic with respect to ${\displaystyle \omega=\frac{\tx dy_{1}\wedge\tx dy_{2}}{x}}$,
\emph{i.e.} 
\begin{eqnarray*}
\left(\psi_{\pm\lambda}\right)^{*}\left(\omega\right) & = & \omega+\ps{\tx dx}\,\,,
\end{eqnarray*}
if and only if $\Psi_{\pm\lambda}=\left(\Psi_{1,\pm\lambda},\Psi_{2,\pm\lambda}\right)$
preserves the standard symplectic form $\tx dh_{1}\wedge\tx dh_{2}$
in the space of leaves, \emph{i.e. $\det\left(\tx D\Psi_{\pm\lambda}\right)=1$}.
In other words:
\begin{prop}
We have the following bijections:
\begin{eqnarray*}
\Lambda_{\lambda}^{\omega}\left(\ynorm\right) & \tilde{\rightarrow} & \left(\cal E_{1,\lambda}\left(\ynorm\right)\times\cal E_{2,\lambda}\left(\ynorm\right)\right)_{\omega}\\
\psi_{\lambda} & \mapsto & \left(\Psi_{1,\lambda},\Psi_{2,\lambda}\right)
\end{eqnarray*}
and
\begin{eqnarray*}
\Lambda_{-\lambda}^{\omega}\left(\ynorm\right) & \tilde{\rightarrow} & \left(\cal E_{1,-\lambda}\left(\ynorm\right)\times\cal E_{2,-\lambda}\left(\ynorm\right)\right)_{\omega}\\
\psi_{-\lambda} & \mapsto & \left(\Psi_{1,-\lambda},\Psi_{2,-\lambda}\right)
\end{eqnarray*}
$\Big($we identify here $\Psi_{\pm\lambda}\left(x,\mathbf{h}\right)=\left(x,\Psi_{1,\pm\lambda}\left(\mathbf{h}\right),\Psi_{2,\pm\lambda}\left(\mathbf{h}\right)\right)$
with $\left(\Psi_{1,\pm\lambda}\left(\mathbf{h}\right),\Psi_{2,\pm\lambda}\left(\mathbf{h}\right)\right)$$\Big)$.
\end{prop}

\subsubsection{Quasi-linear Stokes phenomena in the first Painlevé equation}

~

In \cite{bittmann:tel-01367968}, we link the study of \emph{quasi-linear}
\emph{Stokes phenomena }(see \cite{Kapaev} for the first Painlevé
equation)\emph{ }to our Stokes diffeomorphisms. For instance, in the
case of the first Painlevé equation, we show that the quasi-linear
Stokes phenomena formula found by Kapaev in \cite{Kapaev} allows
to compute the terms $\Psi_{2,\lambda,0}\left(0\right)$ and $\Psi_{1,-\lambda,0}\left(0\right)$
in (\ref{eq: description moduli space}). More precisely, elementary
computations (using Kapaev's connection formula) give: 
\[
\Psi_{2,\lambda,0}\left(0\right)=i\Psi_{1,-\lambda,0}\left(0\right)=\frac{e^{\frac{i\pi}{8}}}{\sqrt{\pi}}2^{\frac{3}{8}}3^{\frac{1}{8}}\,\,.
\]
Moreover, our description of the Stokes diffeomorphisms implies a
more precise estimate of the order of the remaining terms in Kapaev's
formula. In a forthcoming paper, we will use the study of some \emph{non-linear
Stokes phenomena }for the second Painlevé equations\emph{ }(see \emph{e.g.
}\cite{ClarksonMcLeod}) in order to compute coefficients of the $\Psi_{i,\pm\lambda}$'s.

\bibliographystyle{alphaurl}
\bibliography{references_preprint_AIF}

\end{document}